\documentclass[a4paper]{amsart}

\pdfoutput=1

\usepackage{amsmath,amssymb,amsthm,url,scalerel}
\usepackage[utf8]{inputenc}
\usepackage[T1]{fontenc}
\usepackage{libertine}
\usepackage[libertine,cmintegrals,cmbraces]{newtxmath}
\usepackage{verbatim}
\usepackage[shortlabels]{enumitem}
\usepackage{stmaryrd}
\usepackage{mathtools}
\usepackage{microtype}
\usepackage{tabto}
\usepackage{xcolor}
\definecolor{darkred}{RGB}{139,0,0}
\definecolor{darkblue}{RGB}{0,0,139}
\definecolor{darkgreen}{RGB}{0,100,0}
\usepackage[pagebackref]{hyperref}
\hypersetup{
  colorlinks   = true, 
  urlcolor     = darkred, 
  linkcolor    = darkred, 
  citecolor   = darkgreen}
\usepackage{tikz}
\usepackage{tikz-cd}
\usepackage[capitalise,noabbrev]{cleveref}
\usepackage{nicefrac}
\usepackage{tabularx}

\AtBeginDocument{%
   \def\MR#1{}
}
\setenumerate[1]{label=(\roman*),}
\setitemize[1]{label=\raisebox{0.25ex}{\tiny$\bullet$}}

\newcommand{\Diff}{\ensuremath{\mathrm{Diff}}}
\newcommand{\BlockDiff}{\ensuremath{\mathrm{\widetilde{Diff}}}} 
\newcommand{\BDiff}{\ensuremath{B\mathrm{Diff}}}
\newcommand{\BHomeo}{\ensuremath{B\mathrm{Homeo}}}
\newcommand{\Homeo}{\ensuremath{\mathrm{Homeo}}}
\newcommand{\BlockBDiff}{\ensuremath{B\mathrm{\widetilde{Diff}}}}

\newcommand{\hAut}{\ensuremath{\mathrm{hAut}}}
\newcommand{\BhAut}{\ensuremath{B\mathrm{hAut}}}

\newcommand{\BEmb}{\ensuremath{B\mathrm{Emb}}}
\newcommand{\Emb}{\ensuremath{\mathrm{Emb}}}
\newcommand{\Imm}{\ensuremath{\mathrm{Imm}}}
\newcommand{\BlockEmb}{\ensuremath{\mathrm{\widetilde{Emb}}}}

\newcommand{\tohofib}{\ensuremath{\mathrm{tohofib}}}
\newcommand{\bP}{\mathrm{bP}}
\newcommand{\Sp}{\mathrm{Sp}}
\newcommand{\OSp}{\mathrm{OSp}}

\newcommand{\OO}{\mathrm{O}}

\newcommand{\Hom}{\mathrm{Hom}}
\newcommand{\Der}{\mathrm{Der}}
\newcommand{\GL}{\mathrm{GL}}
\newcommand{\SL}{\mathrm{SL}}
\newcommand{\BGL}{B\mathrm{GL}}
\newcommand{\BOSp}{B\mathrm{OSp}}
\newcommand{\Map}{\mathrm{Map}}
\newcommand{\coker}{\mathrm{coker}}
\newcommand{\id}{\mathrm{id}}
\newcommand{\fr}{\mathrm{fr}}
\newcommand{\sfr}{\mathrm{sfr}}
\newcommand{\onefr}{\mathrm{1\text{-}fr}}
\newcommand{\rk}{\mathrm{rk}}
\newcommand{\im}{\mathrm{im}}
\newcommand{\ext}{\mathrm{ext}}
\newcommand{\Bun}{\mathrm{Bun}}
\newcommand{\Sym}{\mathrm{Sym}}
\newcommand{\hofib}{\mathrm{hofib}}

\newcommand{\inc}{\mathrm{inc}}
\newcommand{\pr}{\mathrm{pr}}

\DeclareMathOperator*{\hocolim}{hocolim}
\DeclareMathOperator*{\colim}{colim}
\newcommand{\ev}{\mathrm{ev}}
\newcommand{\even}{\mathrm{even}}
\newcommand{\odd}{\mathrm{odd}}

\newcommand{\Wh}{\mathrm{Wh}}
\newcommand{\gr}{\mathrm{gr}}
\newcommand{\Ind}{\mathrm{Ind}}

\newcommand{\std}{\mathrm{std}}

\newcommand{\Fr}{\mathrm{Fr}}

\newcommand{\spl}{\mathrm{split}}

\newcommand{\BC}{B\mathrm{C}}
\newcommand{\BG}{B\mathrm{G}}
\newcommand{\BSG}{B\mathrm{SG}}
\newcommand{\SG}{\mathrm{SG}}

\newcommand{\G}{\mathrm{G}}

\newcommand{\TOP}{\mathrm{Top}}

\newcommand{\BTOP}{B\mathrm{Top}}
\newcommand{\BSTOP}{B\mathrm{STop}}
\newcommand{\BSO}{B\mathrm{SO}}
\newcommand{\BO}{B\mathrm{O}}
\newcommand{\EO}{\mathrm{EO}}
\newcommand{\SO}{\mathrm{SO}}

\DeclareMathAlphabet{\mathpzc}{OT1}{pzc}{m}{it}
\newcommand{\catsingle}[1]{\ensuremath{\mathcal{#1}}}
\newcommand{\cat}[1]{\ensuremath{\mathsf{#1}}}

\newcommand{\fe}{\mathsf{E}}
\newcommand{\ff}{\mathsf{F}}
\newcommand{\fbt}{\mathsf{Bt}}
\newcommand{\fbo}{\mathsf{Bo}}
\newcommand{\fbg}{\mathsf{Bg}}
\newcommand{\fto}{\mathsf{t}\mskip-2mu/\mskip-2mu\mathsf{o}\mskip1mu}
\newcommand{\fbst}{\mathsf{Bst}}

\newcommand{\oC}{\ensuremath{\mathrm{C}}}

\newcommand{\oE}{\ensuremath{\mathrm{E}}}
\newcommand{\oH}{\ensuremath{\mathrm{H}}}
\newcommand{\oG}{\ensuremath{\mathrm{G}}}
\newcommand{\oO}{\ensuremath{\mathrm{O}}}

\newcommand{\bfC}{\ensuremath{\mathbf{C}}}

\newcommand{\bfL}{\ensuremath{\mathbf{L}}}

\newcommand{\bfR}{\ensuremath{\mathbf{R}}}
\newcommand{\bfZ}{\ensuremath{\mathbf{Z}}}
\newcommand{\bfQ}{\ensuremath{\mathbf{Q}}}
\newcommand{\bfS}{\ensuremath{\mathbf{S}}}

\newcommand{\bfk}{\ensuremath{\mathbf{k}}}

\newcommand{\cA}{\ensuremath{\catsingle{A}}}

\newcommand{\cG}{\ensuremath{\catsingle{G}}}

\newcommand{\cL}{\ensuremath{\catsingle{L}}}

\newcommand{\interior}{\mathrm{int}}

\newcommand{\ra}{\rightarrow}
\newcommand{\lra}{\longrightarrow}
\newcommand{\lla}{\longleftarrow}
\newcommand{\xra}[1]{\xrightarrow{#1}}
\newcommand{\xlra}[1]{\overset{#1}{\longrightarrow}}

\newcommand{\xlratwohead}[1]{\overset{#1}{\relbar\joinrel\twoheadrightarrow}}
\newcommand{\longhookrightarrow}{\lhook\joinrel\longrightarrow}

\newcommand{\Mod}[1]{\ (\mathrm{mod}\ #1)}

\newcommand\dslash{/\mkern-6mu/}
\newcommand{\dl}{(\!(}
\newcommand{\dr}{)\!)}

\newcommand{\circled}[1]{\raisebox{.5pt}{\textcircled{\raisebox{-.9pt} {#1}}}}
\makeatletter
\renewcommand{\boxed}[1]{\text{\fboxsep=.2em\fbox{\m@th$\displaystyle#1$}}}
\makeatother

\newcommand{\map}[5]{\ensuremath{#1\colon\begin{array}{rcl} 
      #2 & \longrightarrow & #3 \\[0.3em] 
      #4 & \longmapsto & #5
    \end{array}}}   
    
\newcommand{\maphook}[5]{\ensuremath{#1\colon\begin{array}{rcl} 
      #2 & \longhookrightarrow & #3 \\[0.3em] 
      #4 & \longmapsto & #5
    \end{array}}}

\newtheorem{bigthm}{Theorem}
\newtheorem{bigcor}[bigthm]{Corollary}

\newtheorem{thm}{Theorem}[section]

\newtheorem{lem}[thm]{Lemma}
\newtheorem{prop}[thm]{Proposition}
\newtheorem{cor}[thm]{Corollary}

\theoremstyle{definition}
\newtheorem{dfn}[thm]{Definition}

\newtheorem*{notation}{Notation}
\newtheorem{conv}[thm]{Convention}

\theoremstyle{remark}
\newtheorem{ex}[thm]{Example}

\newtheorem{rem}[thm]{Remark}

\newcommand{\half}{\nicefrac{1}{2}}

\newcounter{P-counter}
\newcounter{PW-counter}

\begin{document}

\title[Diffeomorphisms of discs and the second Weiss derivative of BTop(--)]{Diffeomorphisms of discs and the second\\ Weiss derivative of BTop(--)}
\author{Manuel Krannich}
\address{Department of Mathematics, Karlsruhe Institute of Technology, 76131 Karlsruhe, Germany}
\email{krannich@kit.edu}

\author{Oscar Randal-Williams}
\address{Centre for Mathematical Sciences, Wilberforce Road, Cambridge CB3 0WB, UK}
\email{o.randal-williams@dpmms.cam.ac.uk}


\begin{abstract}
We compute the rational homotopy groups in degrees up to approximately $\tfrac{3}{2}d$ of the group of diffeomorphisms of a closed $d$-dimensional disc fixing the boundary. Based on this we determine the optimal rational concordance stable range for high-dimensional discs, describe the rational homotopy type of $\BTOP(d)$ in a range, and calculate the second rational derivative of the functor $\BTOP(-)$ in the sense of Weiss' orthogonal calculus. 
\end{abstract}

\maketitle

\tableofcontents

\section{Introduction}
\subsection{Smoothing fibre bundles}\label{intro-smoothing}The study of smooth and topological fibre bundles with a compact smooth $d$-manifold $M$ as fibre lies at the heart of geometric topology. Through the lens of homotopy theory, it becomes the study of the homotopy types of the classifying spaces $\BDiff_\partial(M)$ and $\BHomeo_\partial(M)$ of the topological groups of diffeomorphisms and homeomorphisms of $M$, fixing the boundary. In this work, we are mostly concerned with the case $M=D^d$ of a closed disc, which plays a distinguished role in the theory as it, informally speaking, measures the ``difference'' between smooth and topological fibre bundles. Let us explain why.

Formally, this ``difference'' is the homotopy type of the fibre $\Homeo_\partial(M)/\Diff_\partial(M)=\hofib(\BDiff_\partial(M)\ra \BHomeo_\partial(M))$ since this homotopy type classifies smooth fibre bundles together with a topological trivialisation. It is the source of a \emph{scanning map}
\begin{equation}\label{equ:intro-smoothing-theory}\Homeo_\partial(M)/\Diff_\partial(M)\lra \Gamma_\partial\big(\Fr(M)\times_{\OO(d)}\TOP(d)/\OO(d)\ra M\big)\end{equation}
whose target is the space of sections, standard near the boundary, of the bundle associated to the frame bundle $\Fr(M)$ using the left $\oO(d)$-action on the fibre $\TOP(d)/\oO(d)=\hofib(\BO(d)\ra \BTOP(d))$ where $\TOP(d)=\Homeo(\bfR^d)$ is the topological group of homeomorphisms of $\bfR^d$. Smoothing theory asserts that the scanning map \eqref{equ:intro-smoothing-theory} is a weak equivalence onto the path-components it hits as long as $d\neq 4$ \cite[Essay V]{KS} so, by obstruction theory, the main limiting factor in understanding the difference between smooth and topological fibre bundles is sufficient knowledge of the homotopy groups of $\TOP(d)/\OO(d)$. In degrees up to $d$, these groups are classically well understood, via surgery theory and work of Kirby--Siebenmann \cite[p.\ 246]{KS}. Beyond this degree these groups are less understood, but admit the following alternative description: specialising \eqref{equ:intro-smoothing-theory} to $M=D^d$, the Alexander trick $\BHomeo_\partial(D^d)\simeq*$ results in \emph{Morlet's equivalence}
\begin{equation}\label{equ:intro-Morlet}\BDiff_\partial(D^d)\simeq \Omega^d_0\TOP(d)/\OO(d),\end{equation}
so we have $\pi_{d+k}(\TOP(d)/\OO(d))\cong \pi_k(\BDiff_\partial(D^d))$, leaving us wishing to understand the latter.

\subsection{Rational homotopy groups of $\BDiff_\partial(D^d)$}
As a first approximation, one might attempt to determine $\pi_k(\BDiff_\partial(D^d))$ after rationalisation (indicated by a $(-)_\bfQ$-subscript). There are three known sources of nontrivial classes in these rational homotopy groups:

\subsubsection*{\circled{1} Algebraic $K$-theory}Based on Waldhausen's work on concordance and algebraic $K$-theory \cite{Waldhausen}, Farrell--Hsiang \cite{FarrellHsiang} expressed $\pi_*(\BDiff_\partial(D^d))_\bfQ$ partially in terms of $K_*(\bfZ)$:
\begin{equation}\label{equ:into-farrell-hsiang}
\pi_k(\BDiff_\partial(D^d))_\bfQ\cong\begin{cases} K_{k+1}(\bfZ)_\bfQ&d\text{ odd}\\0&d\text{ even}\end{cases}\quad\text{for }k\le \phi(D^d).
\end{equation}
Here $\phi(D^d)$ is the \emph{concordance stable range}, so $d/3 \lesssim\phi(D^d)$ by work of Igusa \cite{IgusaStability} (but see \cref{intro-pseudoisotopy}). From Borel's work \cite{Borel1} it is known (and used by Farell--Hsiang) that $K_{k+1}(\bfZ)\cong\bfQ$ for positive $k\equiv0\Mod{4}$, and $K_{k+1}(\bfZ)=0$ otherwise. Via different methods, the isomorphism \eqref{equ:into-farrell-hsiang} was later improved to $k\le d-5$ for even $d$ and to $k\le d-6$ for odd $d$, by Randal-Williams \cite[Theorem 4.1]{RWUpperBound} and Krannich \cite[Corollary B]{KrannichConc1} respectively. 

\subsubsection*{\circled{2} Watanabe's graph classes}In \cite{WatanabeI,WatanabeI,Watanabe4,Watanabe4Addendum}, Watanabe constructed classes in the homotopy groups $\pi_{k(d-3)}(\BDiff_\partial(D^d))_\bfQ$ for $k\ge1$ and $d\ge4$, indexed by trivalent graphs, and showed that many of them are nontrivial by evaluating them against characteristic classes which had been described by Kontsevich, in terms of \emph{configuration space integrals}. The classes of lowest degree Watanabe finds show that one has 
\begin{equation}\label{equ:intro-Watanabe-first-class}\pi_{2n-2}(\BDiff_\partial(D^{2n+1)}))_\bfQ\neq 0 \text{ for ``many'' }n\quad \text{and}\quad\pi_{4n-6}(\BDiff_\partial(D^{2n}))_\bfQ\neq0\text{ for all }n\ge2;\end{equation}
``many'' refers to an numerical condition on $n\ge2$ that is often satisfied. The \emph{Kontsevich classes} responsible for \eqref{equ:intro-Watanabe-first-class} are indexed by the ``theta'' graph and the complete graph with four vertices.

\subsubsection*{\circled{3} Pontryagin--Weiss classes}
By work of Sullivan \cite{Sullivan} and Kirby--Siebenmann \cite{KS}, the map $\BO\ra \BTOP$, obtained by stabilising with $(-)\times\bfR$ the maps $\BO(d)\ra \BTOP(d)$ that classify the underlying topological $\bfR^d$-bundle of a vector bundle, is a rational equivalence, so we have $\oH^*(\BTOP;\bfQ)=\oH^*(\BO;\bfQ)=\bfQ[p_1,p_2,\ldots]$. Pulling back $p_i$ along the map $\BTOP(d)\ra \BTOP$, one obtains Pontryagin classes $p_i\in\oH^{4i}(\BTOP(d);\bfQ)$ for topological $\bfR^d$-bundles which extend the usual classes for vector bundles, so one cannot help but wonder whether the relations 
\begin{enumerate}[align=parleft,label=(P\arabic*), ]
\item\label{P1}\tabto{0cm}$p_i=0$\tabto{1.2cm} in $\oH^{4i}(\BO(d);\bfQ)$\tabto{4.1cm} for $i> \lfloor \tfrac{d}{2}\rfloor$
\item\label{P2} \tabto{0cm}$p_n=e^2$\tabto{1.2cm} in $\oH^{4n}(\BSO(2n);\bfQ)$
\setcounter{P-counter}{\value{enumi}}
\end{enumerate}
one is used to for vector bundles  are also satisfied in $\oH^*(\BSTOP(d);\bfQ)$; here $e\in\oH^{2n}(\BSTOP(2n);\bfQ)$ denotes the Euler class. Remarkably the answer is no: Weiss \cite{WeissDalian} showed that \ref{P1} and \ref{P2} do \emph{not} hold in $\oH^*(\BSTOP(d);\bfQ)$ for ``many'' values of $d$ (see loc.cit.\,for precise statements). Moreover, using \ref{P1} and \ref{P2}, one sees that there are unique classes 
\[p_i^\tau\in\oH^{4i-1}(\OO(d)/\TOP(d);\bfQ)\text{ for }i>\lfloor \tfrac{d}{2}\rfloor\quad\text{and}\quad (p_n-e^2)^\tau\in\oH^{4n-1}(\OO(2n)/\TOP(2n);\bfQ)\] 
which transgress to $p_i$ and $p_n-e^2$  in the Serre spectral sequence of the fibration $\OO(d)/\TOP(d)\ra \BSO(d)\ra\BSTOP(d)$. Looping these $d$ times and using \eqref{equ:intro-Morlet} thus yields classes
\begin{enumerate}[align=parleft,label=(PW\arabic*)]
\item\label{PW1}\tabto{0.5cm}$\Omega^{d}p_i^\tau$\tabto{2cm}$\in \oH^{4i-d-1}(\BDiff_\partial(D^{d});\bfQ)$\tabto{6cm} for $i>\lfloor\tfrac{d}{2}\rfloor$
\item\label{PW2} \tabto{0cm}$\Omega^{2n}(p_n-e^2)^\tau$\tabto{2cm}$\in \oH^{2n-1}(\BDiff_\partial(D^{2n});\bfQ)$.\tabto{6cm}
\setcounter{PW-counter}{\value{enumi}}
\end{enumerate}
Weiss' work also shows that \ref{PW1} and \ref{PW2} ``often'' evaluate nontrivially against classes in the image of the rational Hurewicz map, which in particular gives (see Proposition 6.2.1.\,loc.cit.)
\[\pi_{4i-d-1}(\BDiff_\partial(D^{d}))_\bfQ\neq 0\quad\text{for }d>c_1\text{ and }\lfloor \tfrac{d}{2}\rfloor<i<\tfrac{9}{8}d-c_2,\]
for certain positive constants $c_1$ and $c_2$. We will refer to classes in $\pi_*(\BDiff_\partial(D^d))$ which evaluate nontrivially against \ref{PW1} or \ref{PW2} as \emph{Pontryagin--Weiss classes}.

\subsubsection{First main result}
The first result we state determines the groups $\pi_k(\BDiff_\partial(D^d))_\bfQ$ completely for $k\lesssim \tfrac{3}{2}d$, and it in particular shows that the three sources of classes $\circled{1}$--$\circled{3}$ exhaust these groups in this range. We also discover that Watanabe's first graph class from $\circled{2}$ is in fact a Pontryagin--Weiss class, in a sense to be explained (see \cref{intro-relation-to-Watanabe}).

\begin{bigthm}\label{bigthm:diff-discs}
There is an isomorphism
\begin{align*}
\pi_k(\BDiff_\partial(D^{2n+1}))_\bfQ&\cong \begin{cases}
\bfQ & k \equiv 2n-2\Mod{4}, k \geq 2n-2\\
0 & \text{else}
\end{cases} \oplus\ K_{k+1}(\bfZ)_\bfQ
\intertext{in degrees $k<3n-7$, and an isomorphism}
\pi_k(\BDiff_\partial(D^{2n}))_\bfQ&\cong \begin{cases}
\bfQ&k\equiv 2n-1\Mod{4}, k\ge 2n-1\\
0&\text{else} \end{cases}
\end{align*}
in degrees $k<3n-6$.
\end{bigthm}

\begin{rem}\label{rem:even-not-new} The new part of \cref{bigthm:diff-discs} is the result for odd-dimensional discs; the even case partially recovers a result of Kupers and Randal-Williams \cite[Theorem A]{KR-WDisc}.
\end{rem}
\begin{rem}\label{rem:stronger-disc-thm}Our result is stronger than stated here in that we (i) construct a specific map from the left-hand to the right-hand side in \cref{bigthm:diff-discs}, (ii) show that the surjectivity range of this map is one degree better, and (iii) determine the effect of the involution on $\BDiff_\partial(D^{d})$ induced by conjugation with a reflection of $D^{d}$ in one of the coordinates (see Corollaries \ref{cor:homotopy-diff-even-disc} and \ref{cor:homotopy-diff-odd-disc}).
\end{rem}

\subsubsection{The class $E$}\label{intro-E-sec}
In the statement of \cref{bigthm:diff-discs}, the first summand of $\pi_k(\BDiff_\partial(D^{2n+1}))_\bfQ$ and all of $\pi_k(\BDiff_\partial(D^{2n}))_\bfQ$ consists of Pontryagin--Weiss classes in the sense of $\circled{3}$. Note however, that for $d=2n+1$ the first class in \ref{PW1} has degree $2n+2$, but the first summand in \cref{bigthm:diff-discs} contains a class of degree $2n-2$. This is due to an odd-dimensional analogue of \ref{P2}, which seems to be less well known. Namely, writing $\G(d)=\hAut(S^{d-1})$ and $\SG(d)\subset \G(d)$ for its orientation-preserving variant, there is a unique class 
\begin{equation}\label{intro-E}
E\in \oH^{4n}(\BSG(2n+1);\bfQ)
\end{equation}
which pulls back to the square of the Euler class $e^2\in \oH^{4n}(\BSG(2n);\bfQ)$ along the stabilisation map $\BSG(2n)\ra \BSG(2n+1)$ induced by suspension. Pulled back along the composition $\BSO(2n+1)\ra\BSTOP(2n+1)\ra\BSG(2n+1)$ involving the map classifying the spherical fibration of an $\bfR^d$-bundle, the class $E$ agrees with $p_n$, so we can add to \ref{P1} and \ref{P2} the relation
\begin{enumerate}[label=(P\arabic*),align=parleft]
\setcounter{enumi}{\value{P-counter}}
\item\label{P3}\tabto{0cm}$p_n=E$\tabto{1.2cm} in $\oH^{4n}(\BSO(2n+1);\bfQ)$
\end{enumerate}
which, as with \ref{PW1} and \ref{PW2}, gives rise to a characteristic class
\begin{enumerate}[label=(PW\arabic*),align=parleft]
\setcounter{enumi}{\value{PW-counter}}
\item\label{PW3}  $\Omega^{2n+1}(p_n-E)^\tau\in \oH^{2n-2}(\BDiff_\partial(D^{2n+1});\bfQ)$.
\end{enumerate}
This invariant detects the class of degree $2n-2$ in the first summand in \cref{bigthm:diff-discs}.

\begin{rem}\label{intro-relation-to-Watanabe}
In addition to its appearance in \cref{bigthm:diff-discs}, the class $E$ features in two further discoveries we made while contemplating \ref{P3} and \ref{PW3}. These are not used in the proof of \cref{bigthm:diff-discs} but are perhaps of independent interest (see \cref{app:Watanabe}).

\begin{enumerate}
\item\label{enum:intro-wat-i}Kontsevich's class associated to the ``theta'' graph featuring in $\circled{2}$ can be expressed purely in terms of the class $E$---without reference to configuration space integrals (see \cref{thm:KontsevichClass}). This has as consequence that Watanabe's first graph class in \eqref{equ:intro-Watanabe-first-class} evaluates nontrivially against the class in \ref{PW3}, so is a Pontryagin--Weiss class in this sense.
\item\label{enum:intro-wat-ii} Inspired by \ref{P3} and \cite{WatanabeI}, we explain a remarkably simple proof (which could have been discovered in the '$70$s) that $\pi_{2n-2}(\BDiff_\partial(D^{2n+1}))_\bfQ$ is often nontrivial.
\end{enumerate}
\end{rem}

We proceed this introduction by explaining two applications of \cref{bigthm:diff-discs}, one to the concordance stable range as featuring in $\circled{1}$, and one to the rational homotopy type of $\BTOP(d)$. 

\subsection{The rational concordance stable range}\label{intro-pseudoisotopy}Closely related to the diffeomorphism group $\Diff_\partial(M)$ of a compact smooth manifold $M$ is its space of \emph{concordances}
\[ \oC(M)=\{\phi\colon M\times [0,1]\xra{\cong}M\times [0,1]\mid \phi_{M\times \{0\}\cup \partial M\times [0,1]}=\id\},\]
which in turn relates to the concordance space of $M\times [0,1]$ via the \emph{Hatcher stabilisation map}
\begin{equation}\label{equ:stabilsation}s_M\colon\oC(M)\lra \oC(M\times [0,1]),\end{equation}
given by taking products with $[0,1]$ and ``bending'' the resulting diffeomorphism appropriately to make it satisfy the boundary condition (see e.g.\,\cite[Chapter II §1]{IgusaStability}). The \emph{concordance stable range} of $M$ is defined in terms of these stabilisation maps as the quantity
\[\phi(M)\coloneq \min\{k\in \bfZ \mid s_{M\times [0,1]^m}\text{ is }k\text{-connected for all }m\ge0\},\]
and the \emph{rational concordance stable range} $\phi(M)_\bfQ$ is defined analogously in terms of the rational connectivity of the stabilisation maps. This range plays a crucial role in the study of $\BDiff_\partial(M)$ and $\BHomeo_\partial(M)$ in that it is the main limitation when trying to access these homotopy types via the classical route in terms of unstable homotopy theory and $K$- and $L$-theory (see e.g.\,\cite{WWSurvey} for a survey). As briefly mentioned in $\circled{1}$, the main result on $\phi(M)$ is due to Igusa \cite{IgusaStability} who established a lower bound $\phi(M)\gtrsim\dim(M)/3$ (more precisely $\phi(M)\ge \min(\tfrac{1}{3}(\dim(M)-4),\tfrac{1}{2}(\dim(M)-7))$), but the precise concordance stable range is the subject of speculation (see e.g.\,\cite[p.\,6]{IgusaStability}, \cite[p.\,4]{HatcherSurvey}, \cite[p.\,605]{BurgheleaRational}) and is still not known for a single manifold $M$. Based on a strengthening of \cref{bigthm:diff-discs}, we determine the rational concordance stable range for high-dimensional discs and compute the first rational obstruction to stability.

\begin{bigcor}\label{bigcor:relative-concordance-groups}There is a morphism
\[\pi_k\big(\oC(D^{d}\times[0,1]),\oC(D^{d})\big)_\bfQ\lra \begin{cases}
\bfQ&\text{if }k=d-3\\
0&\text{otherwise}
\end{cases}\]
which is an isomorphism for $k< \lfloor \tfrac{3}{2} d\rfloor -8$ and an epimorphism for $k=\lfloor \tfrac{3}{2} d\rfloor -8$.
\end{bigcor}

\begin{rem}
\cref{bigcor:relative-concordance-groups} in particular implies that $\phi(D^d)_\bfQ=d-4$ for $d>9$. Previously, it was known from Watanabe's work mentioned in \circled{2} that $\phi(D^d)_\bfQ\lesssim d$, and the above mentioned improvements of \eqref{equ:into-farrell-hsiang} contained in \cite{RWUpperBound} and \cite{KrannichConc1} suggested that $\phi(D^d)_\bfQ\approx d$.
\end{rem}

\begin{rem}In addition to \cref{bigcor:relative-concordance-groups}, we compute the individual groups $\pi_*(\oC(D^d))_\bfQ$ in a slightly larger range  (see \cref{thm:homotopy-framed-even-disc} and \cref{cor:homotopy-concordances-odd-disc}) and determine the effect of two involutions on these groups (see \cref{thm:concordances-both-involutions}).
\end{rem}

\subsection{The rational homotopy type of $\BSTOP(d)$}Combining \eqref{equ:intro-Morlet} and the known rational homotopy type of $\BO(d)$ with a variant of \cref{bigthm:diff-discs}, one can easily calculate $\pi_*(\BTOP(d))_\bfQ$ in a range. But unlike $\BDiff_\partial(D^d)$, the space $\BTOP(d)$ is not a loop space, so its rational homotopy type need not be determined by its rational homotopy groups. Our second main result describes this rational homotopy type in a $(\approx\tfrac{5}{2}d)$-range, using a strong form of \cref{bigthm:diff-discs} (see \cref{rem:stronger-disc-thm}). We state the answer in terms of the orientation-preserving variant $\BSTOP(d)$, and it involves (i) the Euler class $e\in \oH^{2n}(\BSTOP(2n);\bfQ)$, (ii) the odd-dimensional analogue of its square $E\in \oH^{4n}(\BSTOP(2n+1);\bfQ)$ from \cref{intro-E-sec}, (iii) the stabilisation map $s\colon \BSTOP(d)\ra \BSTOP\simeq_\bfQ\prod_{i\ge1}K(\bfQ,4i)$, and (v) a certain map $\tilde{k}\colon \BSTOP(2n+1)\ra  \Omega^{\infty-2n-1}(K(\bfZ)/\bfS)_\bfQ$ that factors a canonical map $k\colon \TOP(2n+2)/\TOP(2n+1)\ra \Omega^{\infty-2n-1}K(\bfZ)$ (more about this map in \cref{intro-first-derivative}) followed by quotienting out the unit $\bfS\ra K(\bfZ)$ and rationalising.

\begin{bigthm}\label{bigthm:top-d}
For $n\ge3$, the following maps are rationally  $(5n-5)$-connected
\vspace{-0.1cm}
\[
\def\arraystretch{1.5}
\begin{array}{c@{\hskip 0.1cm} l@{\hskip 0.1cm} c@{\hskip 0.1cm} l@{\hskip 0.1cm} l@{\hskip 0cm} l@{\hskip 0cm}}
(E, s , \tilde{k})&\colon& \BSTOP(2n+1)&\lra&K(\bfQ, 4n)\times \BSTOP\times \Omega^{\infty-2n-1}(K(\bfZ)/\bfS)_\bfQ\\
(e,s)&\colon&\BSTOP(2n)&\lra&K(\bfQ, 2n)\times \BSTOP.\\
\end{array}
\]
\end{bigthm}

\begin{rem}As with \cref{bigthm:diff-discs}, the new part of this result is the odd-dimensional case; the statement about even dimensions partially recovers \cite[Corollary D]{KR-WDisc}.
\end{rem}

\begin{rem}
In \cref{sec:obstruction-theory-yoga}, we reinterpret \cref{bigthm:top-d} in terms of cohomological obstructions to finding a vector bundle structure on a topological $\bfR^d$-bundle. Via smoothing theory \eqref{equ:intro-smoothing-theory}, this has direct applications to the problem of smoothing topological manifold bundles.
\end{rem}

\subsection{Orthogonal calculus and the second derivative of $\BTOP(-)$}
All of the above can be recast (and strengthened) in terms of Weiss' \emph{orthogonal calculus}, which provides a unifying perspective that greatly conceptualises our results. This part of the introduction serves to explain this viewpoint, and what we prove about it.

\subsubsection{The Taylor tower}The spaces $\BO(d)$, $\BTOP(d)$, and $\TOP(d)/\oO(d)$ may be viewed as the values at $\bfR^d$ of continuous functors $\fbo\coloneq \BO(-)$, $\fbt\coloneq \BTOP(-)$, and $\fto\coloneq \TOP(-)/\oO(-)$ from the category of finite-dimensional inner product spaces and isometries to the category of pointed spaces. Functors of this type are called \emph{orthogonal}, and they are the input for Weiss' apparatus of \emph{orthogonal calculus} \cite{WeissOrthogonal}. To an orthogonal functor $\fe$ this calculus associates for each $k\ge1$ a (na{\"i}ve) $\OO(k)$-spectrum $\Theta \fe^{(k)}$---the $k$th \emph{derivative}---and a tower
\[
\begin{tikzcd}[column sep=2cm,row sep=0.3cm,/tikz/column 3/.append style={anchor=base west}]
&\smash{\vdots} \dar\\
&T_2\fe(V) \dar &[-1.5cm]\lar\Omega^\infty\big((S^{2\cdot V}\wedge \Theta \fe^{(2)})_{h\oO(2)}\big) \\
&T_1\fe(V) \dar &\lar\Omega^\infty\big((S^{1\cdot V}\wedge \Theta \fe^{(1)})_{h\oO(1)}\big)\\
\fe(V) \arrow[r,swap, pos=0.7 ] \arrow[ru,swap, pos=0.69] \arrow[ruu, swap,pos=0.67]\arrow[ruuu,]& T_0\fe(V)&\arrow[l,equal]\fe(\bfR^\infty)
\end{tikzcd}
\]
of orthogonal functors---the \emph{Taylor tower}. Some explanation is in order: the right-hand horizontal maps---the \emph{layers}---indicate the homotopy fibre inclusions of the vertical maps between the \emph{stages} $T_k\fe$, the zeroth stage $T_0\fe$ (also called the zeroth layer) admits a preferred equivalence to the constant functor on the \emph{value at infinity} $\fe(\bfR^\infty)\coloneq \hocolim_d\fe(\bfR^d)$, the space $S^{k\cdot V}$ is the one-point compactification of $k\cdot V\coloneq \bfR^k\otimes V$, and $\oO(k)$ acts on $S^{k\cdot V}\wedge \Theta E^{(k)}$ diagonally. 

\subsubsection{Derivatives}Forgetting the $\OO(k)$-actions for a moment, let us outline how the derivatives $\Theta\fe^{(k)}$ arise: setting $\fe^{(1)}(V)\coloneq \hofib(\fe(V)\ra \fe(\bfR\oplus V))$, rotating in a $2$-plane induces maps 
\begin{equation}\label{intro-first-unstable-derivative}\fe^{(1)}(V)\lra \Omega \fe^{(1)}(\bfR\oplus V),\end{equation} 
which define a sequential spectrum with $d$th space $\fe^{(1)}(\bfR^d)$; this is the first derivative $\Theta \fe^{(1)}$ as constructed in  \cite{WeissOrthogonal} (see also \cite[Section A.1]{MunozEchaniz} for an explicit formula for \eqref{intro-first-unstable-derivative} involving the aforementioned rotation in a plane). Moreover, setting $\fe^{(2)}(V)\coloneq \hofib(\fe^{(1)}(V)\ra \Omega \fe^{(1)}(\bfR\oplus V))$, there are further natural maps
\[\fe^{(2)}(V)\lra \Omega^2 \fe^{(2)}(\bfR\oplus V)\] 
which give rise to a sequential spectrum with $2d$th space $\fe^{(2)}(\bfR^d)$; this is $\Theta \fe^{(2)}$. Continuing in this manner by taking iterated homotopy fibres gives rise to the higher derivatives $\Theta \fe^{(k)}$ as spectra, though this iterative description obscures the $\OO(k)$-action (see \cite{WeissOrthogonal} for details).

\subsubsection{An example: the tower for $\BO(-)$}\label{intro-bo-example}In the example $\fbo=\BO(-)$, the derivatives $\Theta \fbo^{(k)}$ are well understood. Here the map \eqref{intro-first-unstable-derivative} giving rise to $\Theta\fbo^{(1)}$ agrees with the loop-suspension map
\begin{equation}\label{intro-suspension-map}
S^d=\OO(d+1)/\OO(d)\lra \Omega\OO(d+2)/\OO(d+1)=\Omega S^{d+1},
\end{equation}
so $\Theta\fbo^{(1)}$ is the sphere spectrum $\bfS$. Moreover, from the metastable EHP-description of the fibre of \eqref{intro-suspension-map} one sees $\Theta \fbo^{(2)}\simeq \bfS^{-1}$. Moreover, the $\oO(1)$-action on $\bfS$ can be seen to be trivial, and the $\OO(2)$-action on $\bfS^{-1}$ to be rationally trivial. Based on this, it is an instructive exercise to compute the layers of the tower of $\fbo$ evaluated at $\bfR^d$ after rationalisation (fibrewise over $B\pi_1(-)$, indicated by a $(-)_\bfQ$-subscript) as:
\begin{equation}\label{equ:tower-o}
\begin{tikzcd}[ampersand replacement=\&, column sep=2cm, row sep=0.5cm, /tikz/column 3/.append style={anchor=base west}]
\&T_2\BO(d)_\bfQ\dar\&[-1.5cm]\lar \prod_{i=\lceil d/2\rceil}^{
\infty}K(\bfQ,4i-1)\\
\&T_1\BO(d)_\bfQ \dar \&\lar \smash{\begin{cases}K(\bfQ,d)&d\text{ even}\\ *&d\text{ odd}\end{cases}}\\
\BO(d)_\bfQ \arrow[r ] 
\arrow[ru,swap, pos=0.69] \arrow[ruu,]
\& \BO_\bfQ.\&
\end{tikzcd}
\end{equation}
Vaguely speaking, the zeroth layer creates the Pontryagin classes, the first layer creates the Euler class in even dimensions, the second layer imposes the relations \ref{P1} and \ref{P2} (the relation \ref{P3} need not be imposed as $E$ is never created), and the higher layers are rationally trivial. In particular, the map to the second stage $\BO(d)\ra T_2\BO(d)$ is a rational equivalence.
\begin{rem}Arone \cite{Arone} gave a complete description of $\Theta\fbo^{(k)}$ for all $k$, as $\OO(k)$-spectra.
\end{rem}

We now turn from the tower of $\fbo=\BO(-)$ to that of its topological cousin $\fbt=\BTOP(-)$.

\subsubsection{The first derivative of $\BTOP(-)$}\label{intro-first-derivative}As mentioned in \circled{3}, the map between the zeroth stages $\BO\ra \BTOP$ of the towers of $\fbo$ and $\fbt$ is a rational equivalence. Waldhausen's work on concordance theory  \cite[Thm 2 Add.\,(4)]{WaldhausenManifold} shows that the spectrum $\Theta\fbt^{(1)}$ defined by the topological incarnation of the loop-suspension map \eqref{intro-suspension-map}
\begin{equation}\label{intro-topological-loop-suspension}\TOP(d+1)/\TOP(d)\lra \Omega \TOP(d+2)/\TOP(d+1)\end{equation}
agrees with his $\mathrm{A}(*)=K(\bfS)$. Combined with the linearisation $K(\bfS)\ra K(\bfZ)$ this yields the map $k : \TOP(d+1)/\TOP(d)\ra \Omega^\infty(S^d\wedge K(\bfZ))$ mentioned before the statement of \cref{bigthm:top-d}. Moreover, the map $\bfS\simeq\Theta \fbo^{(1)}\ra \Theta \fbt^{(1)}\simeq K(\bfS)$ induced by the inclusion $\OO(-)\subset \TOP(-)$ is given by the unit, so $\Theta \fto^{(1)}\simeq \Omega K(\bfS)/\bfS=\Omega\Wh^{\Diff}(*)$ since taking derivatives preserves fibre sequences. Furthermore, using smoothing theory and a bit more of the theory of orthogonal calculus, one sees that the first stage $\TOP(d)/\oO(d)\ra T_1(\TOP(d)/\oO(d))$ of the tower for $\fto=\TOP(-)/\oO(-)$ is $(\phi(D^d)+d+1)$-connected where $\phi(D^d)$ is the stable range from \cref{intro-pseudoisotopy}. Taking fibres over the zeroth stage $\TOP/\oO$, using $\BlockDiff_\partial(D^d)/\Diff_\partial(D^d)\simeq \Omega^d\hofib\big(\TOP(d)/\oO(d)\ra \TOP/\oO\big)$, and looping $d$ times, this yields  the $(\phi(D^d)+1)$-connected map
\[
\BDiff_\partial(D^d)\simeq_\bfQ \BlockDiff_\partial(D^d)/\Diff_\partial(D^d) \lra \Omega^{\infty+d}((S^{1\cdot \bfR^d})\wedge \Omega\Wh^{\Diff}(*))_{h\oO(1)}),
\]
featuring in the work of Weiss--Williams \cite{WW1}. Using that the $\OO(1)$-action on $\pi_*(\Omega\Wh^{\Diff}(*))_\bfQ$ is by multiplication with $-1$, this recovers Farrell--Hsiang's result from \circled{1}.

\begin{rem}\label{intro-first-derivative-spectrum}We reproduce Waldhausen's identification $\Theta \fbt^{(1)}\simeq A(*)$ rationally from our point of view (and hence also $\Theta \fto^{(1)}\simeq \Omega\Wh^\Diff(*)$), including the $\OO(1)$-action (see \cref{lem:FirstDerivatives})\end{rem}

\begin{rem}From \cref{bigthm:top-d} one can deduce that in high dimensions the topological loop-suspension map \eqref{intro-topological-loop-suspension} is rationally $(2n-2)$-connected and injective on $\pi_{2n-2}(-)_\bfQ$. Rationally this proves what Burghelea--Lashof considered ``natural to conjecture'' \cite[p.450]{BurgheleaLashofStability}.
\end{rem}

\subsubsection{The second derivative of $\BTOP(-)$}Unlike for orthogonal groups, the second derivative $\Theta\fbt^{(2)}$ of $\fbt$ has proved to be difficult to access (see e.g.\,\cite[p.\,3745, 3748]{WeissOrthogonal}, \cite[p.\,228]{WeissReview}, \cite[p.\,35]{AroneChing}). Prior to this work, the one piece of nontrivial information on $\Theta\fbt^{(2)}$ available resulted from Weiss' work explained in \circled{3} which in a sense is all about $\Theta\fbt^{(2)}$. This slogan is already indicated by the above discussion of the rational Taylor tower for $\fbo$ which suggests that the (non)validity of \ref{P1} and \ref{P2} in the cohomology of $\BTOP(d)$ ought to be closely related to $\Theta\fbt^{(2)}$. Reis--Weiss made this precise in \cite{ReisWeiss} by showing that the map of $\oO(2)$-spectra $ \bfS^{-1}\simeq \Theta\fbo^{(2)}\ra \Theta \fbt^{(2)}$ induced by the inclusion has a rational left homotopy inverse if and only if \ref{P2} holds in $\BSTOP(2n)$ for all $n$, which we know is not the case by Weiss' theorem. 

\medskip

In the final part of this work, we completely determine the second derivative $\Theta\fbt^{(2)}$ rationally.

\begin{bigthm}\label{bigthm:second-derivative}There is a rational equivalence of $\oO(2)$-spectra
\[\Theta\fbt^{(2)}\simeq_\bfQ \Map(S^1_+,\bfS^{-1})\]
where $\oO(2)$ acts on the right-hand side via its canonical action on $S^1$. In other words, $\Theta\fbt^{(2)}$ is rationally equivalent to the coinduction along $\oO(1)\subset\oO(2)$ of $\bfS^{-1}$ viewed as trivial $\oO(1)$-spectrum.
\end{bigthm}
 
Combining \cref{bigthm:second-derivative} with the equivalence $\Theta\fbt^{(1)}\simeq \mathrm{A}(*) \simeq_\bfQ K(\bfZ)$ explained in \cref{intro-first-derivative}, the first two layers of the rationalised tower for $\BTOP(d)$ can be computed as follows: 

\begin{equation}\label{equ:tower-top}
\begin{tikzcd}[ampersand replacement=\&, column sep=2cm, row sep=0.6cm,, /tikz/column 3/.append style={anchor=base west}]
\&T_2\BTOP(d)_\bfQ\dar\&[-1.5cm]\lar \smash{\begin{cases}*&d\text{ even}\\ K(\bfQ,2d-2)&d\text{ odd}\end{cases}}\\
\&T_1\BTOP(d)_\bfQ \dar \&\lar \smash{\begin{cases}K(\bfQ,d)&d\text{ even}\\ \Omega^{\infty-d}(K(\bfZ)/\bfS)_\bfQ &d\text{ odd}\end{cases}}\\
\BTOP(d)_\bfQ \arrow[r] 
\arrow[ru] \arrow[ruu]
\& \BTOP_\bfQ.\&
\end{tikzcd}
\end{equation}
Vaguely speaking---as for $\BO(d)$---the zeroth layer creates the Pontryagin classes and the first layer creates the Euler class but---unlike for $\BO(d)$---the first layer also creates $K$-theory classes and instead of \emph{imposing} \ref{P1} and \ref{P2} the second layer \emph{creates} the class $E$ described in \cref{intro-E-sec}; this difference in behaviour of the second layer explains Weiss' theorem from \circled{3}. 

From the point of view of the tower \eqref{equ:tower-top}, \cref{bigthm:top-d} says equivalently that the second Taylor approximation $\BTOP(d)\ra T_2\BTOP(d)$ is rationally $(5\lfloor\tfrac{d}{2}\rfloor-5)$-connected for $d\ge6$, and that the second stage splits rationally into its layers when passing to the orientation preserving variant $\BSTOP(d)$. To fit \cref{bigthm:diff-discs} into this picture, one uses that with respect to the equivalence in \cref{bigthm:second-derivative}, the natural map $\bfS^{-1}\simeq \Theta\fbo^{(2)}\ra  \Theta\fbt^{(2)}\simeq \Map(S^1_+,\bfS^{-1})$ is given by the inclusion of the constant loops, which results in an equivalence of rational $\OO(2)$-spectra \[\Theta\fto^{(2)}\simeq_\bfQ \bfS^{-3},\] where the $\oO(2)$-action on $\bfS^{-3}$ is via the determinant. Together with $\Theta\fbt^{(1)}\simeq_\bfQ \Omega K(\bfS)/\bfS$ this allows for a computation of the first two layers of the Taylor tower for $\fto=\TOP(-)/\OO(-)$ which  reproduces \cref{bigthm:diff-discs} via Morlet's equivalence \eqref{equ:intro-Morlet}: the $K$-theory classes are created by the first layer, the Pontryagin--Weiss classes by the second, and the rational connectivity of $\BDiff_\partial(D^d)\simeq \Omega^d_0\TOP(d)/\oO(d)\ra \Omega_0^dT_2\TOP(d)/\oO(d)$ follows from that of $\BO(d)\ra T_2\BO(d)$ and $\BTOP(d)\ra T_2\BTOP(d)$ discussed above.

\subsection{Structure of the proof of \cref{bigthm:diff-discs}}\label{intro-proof-overview}
To provide some guidance, we give an overview of our strategy to prove the first main result, \cref{bigthm:diff-discs}, on which all other results rely.
\smallskip

\begin{center}\textit{Some steps may be of independent interest. We highlight them with Roman numerals.}\end{center}

\smallskip

\noindent Our approach is not to directly calculate the rational homotopy groups of $B\Diff_\partial(D^{2n+1})$ and $B\Diff_\partial(D^{2n})$, but rather calculate those of $B\oC(D^{2n})$ and $B\Diff_\partial(D^{2n+1})$ and then use the fibration relating these three spaces. To do this we use certain ``delooped Weiss fibre sequences''
\vspace{-0.1cm}
\begin{equation}\label{intro-weiss-fs}
\def\arraystretch{1.5}
\begin{array}{l@{\hskip 0.1cm} c@{\hskip 0.1cm} c@{\hskip 0.1cm}  c@{\hskip 0.1cm} l@{\hskip 0.1cm}}
\BDiff^{\fr}_{D^{2n}}(V_{g})_{\ell_V}  &\lra &\BEmb^{\fr, \cong}_{\nicefrac{1}{2}D^{2n}}(V_g,W_{g,1})_{\ell_V}&\lra&B(\BC(D^{2n}))\\
\BDiff^{\fr}_\partial(W_{g,1})_{\ell_W} &\lra &\BEmb^{\fr, \cong}_{\half\partial}(W_{g,1})_{\ell_W}&\lra&B(\BDiff^{\fr}_\partial(D^{2n})_B)
\end{array}
\end{equation}
Without explaining these in full detail (see \cref{sec:weiss-fs}),  we have 
\[V_g \coloneq \natural^g (D^{n+1}\times S^n)\quad\text{ and }\quad W_{g,1} \coloneq \partial V_g \backslash \interior(D^{2n}),\] the fibres in \eqref{intro-weiss-fs} are certain path-components of classifying spaces for framed smooth manifold bundles subject to certain boundary conditions, $\BDiff^{\fr}_\partial(D^{2n})_B$ differs from $B\Diff_\partial(D^{2n})$ (up to issues with path-components) by $\Omega^{2n}\OO(2n)$ whose rational homotopy groups are known, and the total spaces are classifying spaces of certain group-like monoids of framed self-embeddings of the manifolds $W_{g,1}$ and $V_g$, again subject to certain boundary conditions. Results of Botvinnik--Perlmutter and Galatius--Randal-Williams on parametrised surgery show that for $g \gg 0$ the fibres are rationally acyclic (see \cref{sec:stable-cohomology}), so the right-hand maps in \eqref{intro-weiss-fs} become rational homology isomorphisms for $g\gg 0$. As the base spaces are loop spaces, their rational homotopy groups can be extracted from the rational homology, so the problem becomes to determine the rational homology of the total spaces for $g \gg 0$ in degrees $*\lesssim 3n$. To do this, we first compute the rational homotopy groups of these spaces, as modules over the fundamental groups. In principle one may attempt this using embedding calculus, but this approach leads to various ambiguities that seem hard to resolve. We rather employ a different strategy, taking advantage of the simple handle-structures of the manifolds $V_g$ and $W_{g,1}$:

\begin{enumerate}[label={(\Roman*)},leftmargin=0.8cm]
\item instead of applying the machinery of embedding calculus to the total spaces of \eqref{intro-weiss-fs}, we take a more ``hands-on'' approach to access such spaces of self-embeddings (see \cref{sec:Disj}).
\end{enumerate}
In a bit more detail: to compute the rational homotopy groups of the total spaces of \eqref{intro-weiss-fs}, we  first replace framings by stable framings (which has a negligible effect) and replace the space $\BEmb^{\sfr, \cong}_{\nicefrac{1}{2}D^{2n}}(V_g,W_{g,1})_{\ell_V}$ by a variant $\BEmb^{\sfr, \cong}_{\nicefrac{1}{2}\partial}(V_g)_{\ell_V}$ involving a different boundary condition; these are related (up to issues with path-components) by a fibre sequence 
\[\BEmb^{\sfr, \cong}_{\nicefrac{1}{2}\partial}(V_g)_{\ell_V} \lra \BEmb^{\sfr, \cong}_{\nicefrac{1}{2}D^{2n}}(V_g,W_{g,1})_{\ell_V}\lra \BEmb^{\sfr, \cong}_{\half\partial}(W_{g,1})_{\ell_W}.\]
Then we calculate the rational homotopy groups of the homotopy fibres of certain maps
\begin{equation}\label{equ:intro-map-to-haut}\BEmb^{\sfr,\cong}_{\half\partial}(W_{g,1}) \lra \BhAut_\partial(W_{g,1}) \quad\text{and}\quad
\BEmb^{\sfr,\cong}_{\half\partial}(V_{g}) \lra \BhAut_\partial(V_g)\end{equation}
to the delooped spaces of homotopy automorphisms fixing the full boundary by an induction over handles. This step in particular relies on Goodwillie's work on multiple disjunction and calculations of Arone, Fresse, Turchin, and Willwacher on the rational homotopy groups of spaces of long embeddings. It also involves

\begin{enumerate}[label={(\Roman*)},leftmargin=0.8cm,resume]
\item\label{enum:intro-bonus-ii} a general method to compute rational homotopy groups of spaces of block embeddings of a point into simply-connected high-dimensional manifolds (see \cref{sec:block-embeddings}).
\end{enumerate}
The handle-induction breaks the symmetry of the manifolds: it gives the rational homotopy groups of the fibres in \eqref{equ:intro-map-to-haut} only as vector spaces instead of as modules over (a variant of) the mapping class groups. In the case of $W_{g,1}$ this is irrelevant as the answer turns out to be zero, but in the case of $V_g$ it does matter. There is however an obvious guess of what the equivariant answer ought to be, and we confirm this guess based on
\begin{enumerate}[label={(\Roman*)},leftmargin=0.8cm,resume]
\item a new perspective on Watanabe's parametrised form of the Goussarov--Habiro theory of claspers, offered as part of \cref{sec:clasper}. In particular, we give an axiomatic description of the family of genus $3$ handlebodies central to Watanabe's work in odd dimensions, which he described explicitly in terms of a Borromean ring construction.
\end{enumerate}
Combining the resulting equivariant computation with calculations of the rational homotopy groups of the targets of \eqref{equ:intro-map-to-haut} due to Berglund--Madsen and Krannich eventually yields the required computation of the higher \emph{homotopy} groups of the total spaces of \eqref{intro-weiss-fs}, as modules over their fundamental groups. To calculate the rational \emph{homology} groups, in \cref{sec:HEmbSpace}, we first determine the Lie algebra structure on the rational homotopy groups up to some ambiguities, then calculate the rational homology of the universal covers, and finally analyse the Serre spectral sequences of the universal covers and show that the ambiguities do not matter. This final step in particular involves a number of computations of stable cohomology groups of certain arithmetic groups with coefficients in various algebraic representations.

\subsection*{Acknowledgements}
We would like to thank S{\o}ren Galatius, Erik Lindell, Samuel Mu{\~n}oz-Echaniz, Arthur Souli{\'e}, and the referee for their comments on an earlier version of this paper. We were partially supported by the ERC under the European Union’s Horizon 2020 research and innovation programme (grant agreement No.\ 756444), and by a Philip Leverhulme Prize from the Leverhulme Trust. MK was also supported by the European Union through an ERC grant (MaFC, 101221003).

\section{Automorphisms, self-embeddings, and mapping class groups}\label{sec:manifold-preliminaries}
As outlined in \cref{intro-proof-overview}, the overall strategy of proof of \cref{bigthm:diff-discs} is to compare the homotopy type of $\Diff_\partial(D^d)$ to that of different types of automorphism and embedding spaces of the auxiliary manifolds
\begin{equation}\label{equ:the-stars}
V_g\coloneq \natural^g (D^{n+1}\times S^n),\quad W_g\coloneq \partial V_g\cong\sharp^g(S^n\times S^n),\quad\text{and}\quad W_{g,1}\coloneq W_g\backslash \interior(D^{2n}).
\end{equation}
where $D^{2n}\subset W_g$ is a fixed embedded disc. The discs fit into this as the case $g=0$: we have
\[V_0=D^{2n+1},\quad W_0= S^{2n},\quad\text{and}\quad W_{0,1}=D^{2n}.\]
This section serves to define the various automorphism spaces we shall we use, and to discuss some of their properties. The most important result in this section is \cref{cor:reduction-to-self-embeddings} which enables us to access the rational homotopy type of diffeomorphisms of discs by studying certain spaces of self-embeddings of $W_{g,1}$ and $V_g$ for large $g$. 

\begin{conv}Unless said otherwise, all manifolds considered are assumed to be smooth, as are all embeddings between them.
\end{conv}

\subsection{(Block) diffeomorphisms, the derivative map, and tangential structures}\label{sec:block-diffeos}For a compact $d$-manifold $M$ and subspaces $K,L\subset M$, we denote by $\Diff_{K}(M,L)\subset \Diff(M)$ the subgroup of the group of diffeomorphisms $\Diff(M)$ of those diffeomorphisms that fix $L$ setwise and a neighbourhood of $K\subset M$ pointwise, equipped with the smooth topology. We use the following abbreviations.
\begin{enumerate}
\item If $K$ or $L$ is empty, we omit it from the notation.
\item If $K=\partial M$, we use $\partial$ as subscript.
\item If $K=\partial M\backslash \interior(D^{d-1})$ for an embedded disc $D^{d-1}\subset \partial M$, we use $\half \partial$ as subscript.
\end{enumerate}
Similarly to $\Diff_K(M,L)$, we write $\hAut_K(M,L)\subset \hAut(M)$ and $\BlockDiff_{K}(M,L)\subset \BlockDiff(M)$ for the corresponding subspaces of the monoid of homotopy automorphisms in the compact open topology and of the realisation of the semi-simplicial group of block diffeomorphisms. Recall that these spaces are related by a preferred factorisation
\begin{equation}\label{equ:hAut-factorisation}
\Diff_K(M,L)\ra \BlockDiff_{K}(M,L) \ra \hAut_{K}(M,L)
\end{equation}
of the map that assigns a diffeomorphism its underlying homotopy equivalence. By definition, the first map is surjective on path components. As explained in \cite[Section\,4]{BerglundMadsen} (see \cite[Section\,1]{KrannichConc1} for an alternative point-set model), the second map factors further as
\begin{equation}\label{equ:block-bundle-factorisation}
\BlockDiff_{K}(M,L) \ra \hAut_{K}(\tau_M^s,L)\lra \hAut_{K}(M,L),
\end{equation}
where $\hAut_{K}(\tau_M^s,L)$ is the topological monoid of bundle maps of the stable tangent bundle $\tau^s_M$ of $M$ that are the identity on $\tau_M^s|_K$ and preserve $\tau_M^s|_L$, equipped with the compact open topology. The first map in the composition is induced by taking derivatives and the second map by forgetting bundle data (see the above references for precise definitions). We denote by
\[\hAut^\cong_{K}(\tau_M^s,L)\subset\hAut_{K}(\tau_M^s,L)\quad\text{and}\quad \hAut^\cong_{K}(M,L)\subset \hAut_{K}(M,L)\]
the components in the image of the map from $\BlockDiff_{K}(M,L)$, or equivalently from $\Diff_{K}(M,L)$.

\subsubsection{Tangential structures}\label{sec:tangential-structures}
A tangential structure is a map $\theta\colon B\ra\BO(d)$ from some space $B$ to the classifying space $\BO(d)$ of $d$-dimensional vector bundles. Given a bundle map $\ell_K\colon \tau_M|_K\ra \theta^*\gamma_d$ where $\tau_M$ is the tangent bundle of $M$ and $\gamma_d\ra \BO(d)$ the universal bundle, a \emph{$\theta$-structure} on $M$ is a bundle map $\ell\colon \tau_M\ra \theta^*\gamma_d$ extending $\ell_K$. We call $\ell_K$ a \emph{boundary condition}, since $K$ agrees with the boundary of $M$ in many of the situations we have have in mind. The set of $\theta$-structures extending $\ell_K$ forms a space $\Bun_K(\tau_M,\theta^*\gamma_d;\ell_K)$ with the compact open topology and it is acted upon by the diffeomorphism group $\Diff_{K}(M,L)$ via precomposition with the derivative. We denote the resulting homotopy quotient by
\[
\BDiff^{\theta}_K(M,L;\ell_K)\coloneq \Bun_K(\tau_M,\theta^*\gamma_d;\ell_K)\dslash \Diff_K(M,L).
\]
Recall from \cite[Section 1]{KrannichConc1} that if $\theta$ is \emph{stable}, i.e.\,pulled back from a map $\theta'\colon B'\ra\BO$ along the canonical map $\BO(d)\ra \BO$, then there is an action of $\BlockDiff_K(M,L)$ and $\hAut^\cong_K(\tau_M^s,L)$ on $\Bun_K(\tau_M,\theta^*\gamma_d;\ell_K)$ in the $A_\infty$-sense (i.e.\ an actual action on a space which comes equipped with a preferred equivalence to $\Bun_K(\tau_M,\theta^*\gamma_d;\ell_K)$) whose homotopy quotients 
\begin{align*}
\BlockBDiff_K^\theta(M,L;\ell_K)&\coloneq \Bun_K(\tau_M,\theta^*\gamma_d;\ell_K)\dslash \BlockDiff_K(M,L)\\ 
\BhAut_K^{\theta,\cong}(\tau_M^s,L;\ell_K)&\coloneq\Bun_K(\tau_M,\theta^*\gamma_d;\ell_K)\dslash \hAut^\cong_K(\tau_M^s,L)
\end{align*}
fit into a diagram of homotopy cartesian squares 
\begin{equation}\label{equ:comparison-tangential}
\begin{tikzcd}
\BDiff^\theta_K(M,L;\ell_K)\rar\dar&\BlockBDiff^\theta_K(M,L;\ell_K)\dar\rar&\BhAut_K^{\theta,\cong}(\tau_M^s,L;\ell_K)\dar\\
\BDiff_K(M,L)\rar&\BlockBDiff_K(M,L)\rar&\BhAut^\cong_K(\tau^s_M,L)
\end{tikzcd}
\end{equation}
induced by the maps in \eqref{equ:hAut-factorisation} and \eqref{equ:block-bundle-factorisation}. From the long exact sequence in homotopy groups of the fibration associated to a homotopy quotient, we see that sets of path components of the three spaces forming the upper row all agree with the orbits of the $\pi_0(\Diff_K(M,L))$-action on $\pi_0(\Bun_K(\tau_M,\theta^*\gamma_d;\ell_K))$. We denote the component corresponding to a fixed tangential structure $\ell\in \Bun_K(\tau_M,\theta^*\gamma_d;\ell_K)$ by a subscript $(-)_\ell$, e.g.\,$\BDiff^\theta_K(M,L;\ell_K)_\ell\subset \BDiff^\theta_K(M,L;\ell_K)$ in the case of diffeomorphisms. We also write
\[
\hAut^{\cong,\ell}_K(M,L)\subset \hAut^\cong_K(M,L)
\]
for the collection of components in the image of the map
\[
\pi_1(\BlockBDiff^{\theta}_{K}(M,L;\ell_K);\ell)\lra \pi_1(\BhAut^{\cong}_{K}(M,L))\cong \pi_0\hAut^{\cong}_{K}(M,L).
\]

\begin{ex}\label{ex:tangential-structures}The tangential structures relevant to this work are
\begin{enumerate}
\item\label{item:no-tangential-structure} $\theta=\id_{\BO(d)}$ in which case $\Bun_K(\tau_M,\theta^*\gamma_d;\ell_K)$ is contractible by the universality of $\gamma_d$, so in this case the vertical maps in \eqref{equ:comparison-tangential} are equivalences,
\item \label{item:framings-tangential-structure}the canonical map $\fr\colon \EO(d)\ra \BO(d)$; this encodes framings, 
\item\label{item:oneframings-tangential-structure} the map $\onefr\colon \oO(d+1)/\oO(d)\ra \BO(d)$ from the homotopy fibre of the map $\BO(d)\ra\BO(d+1)$; this encodes framings of the once stabilised bundle, 
\item \label{item:pm-framings-tangential-structure}the canonical map $\pm\fr\colon \BO(1)\ra \BO(d)$; this encodes ``unoriented framings'', and 
\item\label{item:stable-tangential-structure} the map $\sfr\colon \oO/\oO(d)\ra \BO(d)$ from the homotopy fibre of the map $\BO(d)\ra\BO$; this encodes stable framings.
\end{enumerate}
Among these, \ref{item:no-tangential-structure} and \ref{item:stable-tangential-structure} are stable tangential structures whereas \ref{item:framings-tangential-structure},  \ref{item:oneframings-tangential-structure}, and \ref{item:pm-framings-tangential-structure} are not.
\end{ex}

\subsubsection{A rational model for block diffeomorphisms}
In \cite[Corollary 4.21]{BerglundMadsen}, a variant of the bottom right horizontal map in \eqref{equ:comparison-tangential} was shown to be rational equivalence on universal covers for $K=\partial M\cong S^{d-1}$ with $d\ge5$ and simply-connected $M$. We shall need a version of this result, established in \cite{KrannichConc1}, for triads satisfying a $\pi$-$\pi$-condition. To state it, recall that a \emph{manifold triad} $(M;\partial_0M,\partial_1M)$ is a compact $d$-manifold $M$ together with a decomposition of its boundary $\partial M=\partial_0 M\cup \partial_1M$ into (possibly empty) submanifolds with $\partial(\partial_0M)=\partial_0M\cap\partial_1M=\partial(\partial_1M)$. It satisfies the \emph{$\pi$-$\pi$-condition} if $\partial_1M\subset M$ induces an equivalence on fundamental groupoids. The following is part of Theorem 2.2 and Corollary 2.4 of \cite{KrannichConc1}.

\begin{thm}[Krannich]\label{thm:blockdiff-identification}
For a manifold triad $(M;\partial_0M, \partial_1M)$ of dimension $d\ge6$ that satisfies the $\pi$-$\pi$-condition, the homotopy fibre of the map
\[
\BlockBDiff_{\partial_0}(M)\lra \BhAut^{\cong}_{\partial_0}(\tau^s_M,\partial_1M)
\]
is nilpotent and has finite homotopy groups. Moreover, if $\ell$ is a stable framing of $M$ and $\ell_{\partial_0}$ is its restriction to $\partial_0M$, then the same conclusion holds for the map
\[
\BlockBDiff^{\sfr}_{\partial_0}(M;\ell_{\partial_0})_\ell\lra \BhAut^{\cong,\ell}_{\partial_0}(M,\partial_1M).
\]
\end{thm}

\subsection{Diffeomorphisms and self-embeddings of handlebodies}\label{sec:diff-and-self-emb}
We now turn to the specific manifolds \eqref{equ:the-stars}. Restricting diffeomorphisms to the moving part of the boundary induces a fibre sequence of the form
\begin{equation}\label{equ:diff-restriction-sequence}
\Diff_\partial(V_g)\lra \Diff_{D^{2n}}(V_g)\lra \Diff_\partial(W_{g,1}),
\end{equation}
which we shall compare to a certain fibre sequence of self-embedding spaces via a diagram \begin{equation}\label{equ:diff-emb-comparison}
\begin{tikzcd}
\Diff_\partial(V_g)\rar\dar& \Diff_{D^{2n}}(V_g)\rar\dar& \Diff_\partial(W_{g,1})\dar\\
\Emb_{\half \partial}(V_g)\rar&\Emb_{\half D^{2n}}(V_g,W_{g,1})\rar&\Emb_{\half \partial}(W_{g,1}).
\end{tikzcd}
\end{equation}
To define the bottom row, we decompose the disc $D^{2n}=D^{2n}_+\cup D^{2n}_-\subset \partial V_g$ into two half-discs intersecting in a $(2n-1)$-disc and define the space $\Emb_{\half D^{2n}}(V_g,W_{g,1})$ as the space of self-embeddings of $V_g$ that restrict to a self-embedding of $W_{g,1}\subset \partial V_g$ and agree with the identity on a neighbourhood of $D^{2n}_+$ (but may send $\interior(D^{2n}_-)\subset \partial V_g$ to the interior). Restricting such embeddings to $W_{g,1}$ gives a map to the space $\Emb_{\half \partial}(W_{g,1})$ of self-embeddings of $W_{g,1}$ that agree with the identity on a neighbourhood of $W_{g,1}\cap D^{2n}_+\cong D^{2n-1}$ (but may send $\interior(W_{g,1}\cap D^{2n}_-)\cong\interior(D^{2n-1})\subset \partial W_{g,1}$ to the interior). By isotopy extension, the homotopy fibre of this map agrees with the space $\Emb_{\half \partial}(V_g)$ of self-embeddings of $V_g$ which agree with the identity on a neighbourhood of $W_{g,1}\cup D^{2n}_+=\partial V_g\backslash \interior(D^{2n}_-)$ (but may send $\interior(D^{2n}_-)\subset \partial V_g$ to the interior).

We denote the components in the image of the two right vertical maps in \eqref{equ:diff-emb-comparison} by
\[\Emb^\cong_{\half D^{2n}}(V_g,W_{g,1})\subset \Emb_{\half D^{2n}}(V_g,W_{g,1})\quad\text{and}\quad \Emb^{\cong}_{\half \partial}(W_{g,1})\subset \Emb_{\half \partial}(W_{g,1}),\]
denote by
\begin{equation}\label{equ:restricted-components-ext-1}
\Diff^{\ext}_\partial(W_{g,1})\subset \Diff_\partial(W_{g,1})\quad\text{and} \quad \Emb^{\cong,\ext}_{\half \partial}(W_{g,1})\subset \Emb^\cong_{\half \partial}(W_{g,1})
\end{equation}
the union of those path components hit by the maps from $\Diff_{D^{2n}}(V_g)$, and denote by 
\begin{equation}\label{equ:restricted-components-ext-2}
\Emb^\cong_{\half \partial}(V_g)\subset \Emb_{\half \partial}(V_g)
\end{equation}
those components that map to $\Emb^\cong_{\half D^{2n}}(V_g,W_{g,1})\subset \Emb_{\half D^{2n}}(V_g,W_{g,1})$.

\begin{rem}\label{rem:components-are-nice}
It does not play a role for our arguments, but the collection of components $\Emb^\cong_{\half \partial}(V_g)\subset \Emb_{\half \partial}(V_g)$ can be shown to agree with that of image of $\Diff_\partial(V_g)\ra \Emb_{\half \partial}(V_g)$, as suggested by the notation. This follows from a diagram chase in the ladder of long exact sequences induced by \eqref{equ:diff-emb-comparison} once one checks that the surjection \[\pi_0\Diff^{\ext}(W_{g,1})\lra \pi_0\Emb^{\cong,\ext}_{\half \partial}(W_{g,1})\] is in fact an isomorphism. The latter  follows from showing that the map $\pi_0\Diff_{D^{2n}}(V_g)\ra\pi_0\Emb_{\half \partial}(W_{g,1})$ is injective which follows by comparing the extensions appearing in \cite[Lemma 4.3]{KR-WAlg} and \cite[Theorem 3.3]{KrannichConc1}. 
\end{rem}

Given a tangential structure $\theta\colon B\ra \BO(2n+1)$ and a $\theta$-structure $\ell\colon \tau_{V_g}\ra \theta^*\gamma_{2n+1}$ for $V_g$, we fix an identification $\tau_{V_g}|_{W_{g,1}}\cong \tau_{W_{g,1}}\oplus \varepsilon$ induced by a choice of an inwards pointing vector field where $\varepsilon$ is the trivial line bundle. Using this, we consider the maps 
\[\hspace{-0.1cm}
\begin{tikzcd}[column sep=0.2cm]
\Bun_{\partial}(\tau_{V_g},\theta^*\gamma_{2n+1};\ell_{\partial V_g})\rar\dar&\dar \Bun_{D^{2n}}(\tau_{V_g},\theta^*\gamma_{2n+1};\ell_{D^{2n}})\rar\dar& \Bun_\partial(\tau_{W_{g,1}}\oplus\varepsilon,\theta^*\gamma_{2n+1};\ell_{\partial W_{g,1}})\dar\\
\Bun_{\half\partial}(\tau_{V_g},\theta^*\gamma_{2n+1};\ell_{\half\partial V_g})\rar& \Bun_{\half D^{2n}}(\tau_{V_g},\theta^*\gamma_{2n+1};\ell_{\half D^{2n}})\rar& \Bun_{\half\partial}(\tau_{W_{g,1}}\oplus\varepsilon,\theta^*\gamma_{2n+1};\ell_{\half\partial W_{g,1}})
\end{tikzcd}
\]
of spaces of $\theta$-structures, where the horizontal maps are induced by restriction and are fibre sequences, and the vertical maps are induced by relaxing the boundary conditions, all of which obtained by restriction from $\ell$. Write $(1\text{-}\theta)\colon B'\ra \BO(2n)$ for the tangential structure given by the pullback of $\theta$ along $\BO(2n)\ra \BO(2n+1)$, so that a $(1\text{-}\theta)$-structure on a $2n$-dimensional vector bundle is simply a $\theta$-structure on its one-fold stabilisation. Then the bases of these fibrations are equivalent to the space of $(1\text{-}\theta)$-structures on $W_{g,1}$ relative to $\partial W_{g,1}$ and to $\half\partial W_{g,1}$ respectively. Both rows of \eqref{equ:diff-emb-comparison} act compatibly on this fibre sequence by precomposition with the derivative, so after passing to the subsets of components specified in \eqref{equ:restricted-components-ext-1} and \eqref{equ:restricted-components-ext-2}, we may form homotopy quotients to obtain a diagram of horizontal fibre sequences
\[
\begin{tikzcd}
\BDiff^{\theta}_{\partial}(V_g;\ell_{\partial})\rar\dar& \BDiff^{\theta}_{D^{2n}}(V_g;\ell_{\partial})\rar\dar& \BDiff^{1\text{-}\theta,\ext}_{\partial}(W_{g,1};\ell_{\nicefrac{1}{2}\partial})\dar\\
\BEmb^{\theta,\cong}_{\nicefrac{1}{2}\partial}(V_g;\ell_{\nicefrac{1}{2}\partial})\rar& \BEmb^{\theta,\cong}_{\nicefrac{1}{2}D^{2n}}(V_g,W_{g,1};\ell_{\nicefrac{1}{2}D^{2n}})\rar& \BEmb^{1\text{-}\theta,\cong,\ext}_{\nicefrac{1}{2}\partial}(W_{g,1};\ell_{\nicefrac{1}{2}\partial}),
\end{tikzcd}
\]
with fibres taken over the points induced by $\ell$.

\subsection{Weiss fibre sequences}\label{sec:weiss-fs}
For a tangential structure $\theta\colon B\ra\BO(2n)$ and a $\theta$-structure $\ell$ on $W_{g,1}$, there is a fibre sequence of the form
\begin{equation}\label{equ:weiss-fs-Wg}
\BDiff^\theta_\partial(D^{2n};\ell_{\partial_0})\lra \BDiff^\theta_\partial(W_{g,1};\ell_{\partial})\lra \BEmb^{\theta, \cong}_{\half\partial}(W_{g,1};\ell_{\half \partial}),
\end{equation}
where the boundary conditions are obtained by restricting $\ell$. The first map is induced by extension by the identity along the inclusion $D^{2n}\subset W_{g,1}$ and the second by the forgetful map $\Diff_\partial(W_{g,1})\ra\Emb^\cong_{\half\partial}(W_{g,1})$. For $\theta=\id$, this sequence appeared in \cite[Remark 2.1.3]{WeissDalian} and was shown in \cite[Theorem 4.17]{KupersFiniteness} to deloop once to the right with respect to the $E_{2n}$-structure on the fibre induced by boundary connected sum. The version with tangential structures $\theta$ is noted in \cite[Proposition 8.7]{KR-WAlg} and deloops in the same way. The same arguments establish fibre sequences
\begin{equation}\label{equ:weiss-fs-Vg}
\begin{gathered}
\BDiff^\Xi_\partial(D^{2n+1};\ell_{\partial_0})\lra \BDiff^\Xi_\partial(V_{g};\ell_{\partial})\lra \BEmb^{\Xi, \cong}_{\half\partial}(V_{g};\ell_{\half \partial})\quad\text{and}\\
\BDiff^\Xi_{D^{2n}}(D^{2n+1};\ell_{D^{2n}})\lra \BDiff^{\Xi}_{D^{2n}}(V_{g};\ell_{D^{2n}})\lra \BEmb^{\Xi,\cong}_{\nicefrac{1}{2}D^{2n}}(V_g,W_{g,1};\ell_{\nicefrac{1}{2}D^{2n}}),
\end{gathered}
\end{equation} 
for tangential structures $\Xi\colon B'\ra\BO(2n+1)$, which both deloop once to the right with respect to the $E_{2n+1} / E_{2n}$-structure on the fibre induced by boundary connected sum. Similar to \eqref{equ:weiss-fs-Wg}, the right-hand maps are induced by the comparison maps between diffeomorphisms and embeddings and the left-hand maps are induced by extension by the identity along a fixed embedding $D^{2n+1}\subset V_g$ that restricts to the chosen embedding $D^{2n}\subset W_{g,1}$ on a hemisphere. Note that forgetting tangential structures induces an equivalence
\begin{equation}\label{equ:no-tangential-structure-for-moving-boundary}
\BDiff^\Xi_{D^{2n}}(D^{2n+1};\ell_{D^{2n}})\xlra{\simeq} \BDiff_{D^{2n}}(D^{2n+1}),
\end{equation}
 since the homotopy fibre is equivalent to the space of maps $D^{2n+1}\ra B'$ together with a homotopy from its composition with $\Xi$ to a fixed classifier $D^{2n+1}\ra \BO(2n+1)$ of the tangent bundle with fixed behaviour on the hemisphere $D^{2n}\subset D^{2n+1}$. As the latter inclusion is an equivalence, this space is contractible.

In \cite[Section 4.4.3]{KupersFiniteness}, Kupers explains how to set up similar sequences for the corresponding groups of homeomorphisms (his discussion is phrased without tangential structures), resulting in a commutative diagram of fibre sequences
\[
\begin{tikzcd}
\BDiff_\partial(D^{2n})\rar\dar& \BDiff_\partial(W_{g,1})\rar\dar& \BEmb^{\cong}_{\half\partial}(W_{g,1})\dar\\
\BHomeo_\partial(D^{2n})\rar& \BHomeo_\partial(W_{g,1})\rar& B\mathrm{TopEmb}^{\cong}_{\half\partial}(W_{g,1})
\end{tikzcd}
\]
and similarly for the extensions of \eqref{equ:weiss-fs-Vg} with $\Xi=\id$. Here the bottom row is defined in the same way as the top row, but using spaces of homeomorphisms and locally flat topological embeddings instead. As the spaces
\[\BHomeo_\partial(D^{2n}),\quad \BHomeo_\partial(D^{2n+1}),\quad\text{and}\quad\BHomeo_{D^{2n}}(D^{2n+1})\] are contractible by the Alexander trick, we obtain maps 
\[\begin{gathered}
\BEmb^{\cong}_{\half\partial}(V_{g})\lra \BHomeo_\partial(V_g)\quad\quad \BEmb^{\cong}_{\nicefrac{1}{2}D^{2n}}(V_g,W_{g,1})\lra  \BHomeo_{D^{2n}}(V_{g})\\
\BEmb^{\cong}_{\half\partial}(W_{g,1})\lra \BHomeo_\partial(W_{g,1})
\end{gathered}
\]
up to contractible choices, which we may postcompose with the comparison maps to homotopy automorphisms to obtain maps of the form
\begin{equation}\label{equ:self-embs-to-haut}
\begin{gathered}
\BEmb^{\cong}_{\half\partial}(V_{g})\lra \BhAut_\partial(V_g)\quad\quad \BEmb^{\cong}_{\nicefrac{1}{2}D^{2n}}(V_g,W_{g,1})\lra  \BhAut_{D^{2n}}(V_{g},W_{g,1})\\
\BEmb^{\cong}_{\half\partial}(W_{g,1})\lra \BhAut_\partial(W_{g,1}).
\end{gathered}
\end{equation}

\begin{rem}Even though phrased for the triad $(V_g;D^d,W_{g,1})$, the discussion of this and the previous subsection applies verbatim to any manifold triad $(M;\partial_0M,\partial_1M)$ with $\partial_0M\cong D^d$, except for \cref{rem:components-are-nice}.
\end{rem}

\subsection{Stable cohomology}\label{sec:stable-cohomology}
For  $n \geq 3$ and large values of $g$, the homology of $\BDiff^\theta_\partial(W_{g,1};\ell_{\partial})$ admits a homotopy theoretical description by work of Galatius--Randal-Williams \cite{GRWII}. Botvinnik--Perlmutter \cite{BotvinnikPerlmutter} proved a similar result for $\BDiff^{\Xi}_{D^{2n}}(V_{g};\ell_{D^{2n}})$ as long as the domain of $\Xi$ is $n$-connected and $n\ge4$. We will only use the case of framings, in which case these results read as follows. We denote the component of the constant map in a loop space by $\Omega_0X\subset \Omega X$ and the sphere spectrum by $\bfS$.

\begin{thm}[Botvinnik--Perlmutter, Galatius--Randal-Williams]
\label{thm:stable-homology}
For framings $\ell_{V_g}$ and $\ell_{W_{g,1}}$ of $V_g$ and $W_{g,1}$ respectively, there are maps
\vspace{-0.2cm}
\[
\def\arraystretch{1.5}
\begin{array}{r@{\hskip 0.1cm} c@{\hskip 0.1cm} l@{\hskip 0.1cm}  l@{\hskip 0.1cm}}
\BDiff^{\fr}_{D^{2n}}(V_g;\ell_{D^{2n}})_{\ell_{V_g}} &\lra &\Omega^\infty_0\bfS&\text{for }n\ge4\\
\BDiff^{\fr}_{\partial}(W_{g,1};\ell_{\partial W_{g,1}})_{\ell_{W_{g,1}}}&\lra &\Omega^\infty_0\bfS^{-2n}&\text{for }n\ge3
\end{array}
\]
that are homology isomorphisms in a range of degrees increasing with $g$. In particular, the reduced rational homology of the domains of these maps vanishes in a stable range.
\end{thm}
\begin{proof}
The statement about the first map follows from \cite[Theorem A*]{BotvinnikPerlmutter} and \cite[Theorem 8.2]{Perlmutter}. The asserted property of the second map is a special case of \cite[Corollary 1.8]{GRWII}.
\end{proof}

By specialising the sequence \eqref{equ:weiss-fs-Wg} and the second sequence of \eqref{equ:weiss-fs-Vg} to the case of framings, restricting to the components of given framings $\ell_{V_g}$ and $\ell_{W_{g,1}}$ of $W_{g,1}$ and $V_g$, and delooping we obtain for $n\ge3$ fibre sequences of the form
\begin{equation}\label{equ:delooped-framed-Weiss-fs}
\begin{gathered}
\BDiff^{\fr}_\partial(W_{g,1};\ell_{\partial})_{\ell_{W_{g,1}}}\lra \BEmb^{\fr, \cong}_{\half\partial}(W_{g,1};\ell_{\half \partial})_{\ell_{W_{g,1}}} \lra B(\BDiff^{\fr}_\partial(D^{2n};\ell_{\partial_0})_B)\\
\BDiff^{\fr}_{D^{2n}}(V_{g};\ell_{D^{2n}})_{\ell_{V_g}} \lra \BEmb^{\fr, \cong}_{\nicefrac{1}{2}D^{2n}}(V_g,W_{g,1};\ell_{\nicefrac{1}{2}D^{2n}})_{\ell_{V_g}} \lra B(\BDiff_{D^{2n}}(D^{2n+1}))
\end{gathered}
\end{equation}
where $\BDiff^{\fr}_\partial(D^{2n};\ell_{\partial_0})_B\subset \BDiff^{\fr}_\partial(D^{2n};\ell_{\partial_0})$ is a collection of components which is finite by \cite[Corollary 8.16]{KR-WAlg}. To arrive at the second sequence, we used the equivalence \eqref{equ:no-tangential-structure-for-moving-boundary}.

\begin{cor}\label{cor:reduction-to-self-embeddings}
For framings $\ell_{W_{g,1}}$ of $W_{g,1}$ and $\ell_{V_g}$ of $V_g$, the maps
\vspace{-0.2cm}
\[
\def\arraystretch{1.5}
\begin{array}{r@{\hskip 0.1cm} c@{\hskip 0.1cm} l@{\hskip 0.1cm}  l@{\hskip 0.1cm}}
\BEmb^{\fr,\cong}_{\half D^{2n}}(V_g,W_{g,1};\ell_{\half \partial})_{\ell_{V_g}}&\lra &B(\BDiff_{D^{2n}}(D^{2n+1}))&\text{for }n\ge4\\
\BEmb^{\fr,\cong}_{\half \partial}(W_{g,1};\ell_{\half \partial})_{\ell_{W_{g,1}}}&\lra &B(\BDiff^{\fr}_{\partial}(D^{2n};\ell_{\partial_0})_B)&\text{for }n\ge3
\end{array}
\]
are rational homology isomorphisms in a range of degrees increasing with $g$. Here the subspace $\BDiff^{\fr}_{\partial}(D^{2n};\ell_{\partial_0})_B\subset \BDiff^{\fr}_{\partial}(D^{2n};\ell_{\partial_0})$ is a finite union of components.
\end{cor}

\subsection{A standard model}\label{sec:standard-model}
We fix $g\ge0$ and write
\[
\bar{H}_W\coloneq \oH_n(W_{g,1};\bfQ),\quad \bar{H}_V\coloneq \oH_n(V_{g}; \bfQ),\quad\text{and}\quad \bar{K}_W\coloneqq\ker(\iota_*\colon \bar{H}_W\ra \bar{H}_V),
\] 
where $\iota$ is the inclusion $W_{g,1}\subset V_g$. Their integral cousins are denoted by \[
\bar{H}_W^\bfZ\coloneq H_n(W_{g,1};\bfZ),\quad \bar{H}_V^\bfZ\coloneq H_n(V_{g}; \bfZ),\quad\text{and}\quad \bar{K}_W^\bfZ\coloneqq\ker(\iota_*\colon \bar{H}_W^\bfZ\ra \bar{H}_V^\bfZ).
\]
The bar-superscript indicates that we consider these as plain vector spaces or abelian groups, whereas later we shall mostly consider graded objects (see \cref{sec:gradings}) 

At several points in this work, it will be convenient to have fixed a standard model of $V_g$ as a submanifold of $\bfR^{2n+1}$ with fixed embedded cores $e_i,f_i\colon S^n\hookrightarrow \partial V_g$ for $1\le i\le g$ representing a basis of $\bar{H}_W^\bfZ\cong\bfZ^{2g}$, denoted by the same symbols, such that
\begin{enumerate}
\item $(e_i,f_i)_{1\le i\le g}$ is a hyperbolic basis for the intersection form 
\[\lambda\colon \bar{H}_W^\bfZ\otimes \bar{H}_W^\bfZ\lra \bfZ,\]
i.e.\,$\lambda(e_i,e_j)=\lambda(f_i,f_j)=0$ and $\lambda(e_i,f_j)=\delta_{ij}$ for all $1\le i,j\le g$,
\item the $e_i$'s form a basis of the subspace $\bar{K}_W^\bfZ\subset \bar{H}_W^\bfZ$, and
\item the $f_i$'s map to a basis of $\bar{H}_V^\bfZ\cong \bfZ^g$ under the projection $\iota_*\colon \bar{H}_W^\bfZ\ra \bar{H}_V^\bfZ$.
\item the $V_g$'s come with inclusions $V_{h}\subset V_{g}$ for $h\le g$ which preserve the $e_i$'s and $f_i$'s.
\end{enumerate}
To this end, we use the standard embedding $S^n\subset \bfR^{n+1}$ and abbreviate by $\iota_{S^n}\colon S^n\hookrightarrow \bfR^{2n+1}$ the inclusion into the last $(n+1)$-coordinates. We consider the normal bundle $\nu(\iota_{S^n})$ of $\iota_{S^n}$ as a subspace of $S^n\times \bfR^{2n+1}$ using the standard metric and framing of $\bfR^{2n+1}$. Writing $\vec{e}_i$ for the $i$th standard unit vector, we view $V_1$ as the image of the embedding
\[
\maphook{\phi}{\{(x,y)\in \nu(\iota_{S^n})\subset S^n\times \bfR^{2n+1}\mid |y|\le 1/2\}}{\bfR^{2n}\times (0,2)}{(x,y)}
{\tfrac{1}{2}(\iota_{S^n}(x)+y)+\vec{e}_{2n+1},}
\]
equipped with the cores  $e_1\colon S^n\hookrightarrow \partial V_1$ and $f_1\colon S^n\hookrightarrow \partial V_1$  given by precomposing $\phi$ with
 \[
 S^n\ni x\mapsto (\vec{e}_{1},(\tfrac{1}{2}x,0))\in S^n\times \bfR^{2n+1} \quad\text{and}\quad S^n\ni x\mapsto (x,(0,-\tfrac{1}{2}x))\in S^n\times \bfR^{2n+1}
 \] respectively. The cores $e_1$ and $f_1$ intersect in the single point $\phi(\vec{e}_{1},-\tfrac{1}{2}\vec{e}_{n+1})=\tfrac{1}{4}\vec{e}_{n+1}+\vec{e}_{2n+1}$. Now consider the embedding $\psi\colon \sqcup^g V_1 \ra \bfR^{2n}\times (0,2g)$ given by translating the $i$th disjoint summand by $2(i-1)\cdot \vec{e}_{2n+1}$ so that it lies in $\bfR^{2n}\times (2(i-1),2i)$. The translates of the cores $e_1,f_1$ above give a pair of cores $(e_i,f_i)$ for each of the summands. We define $V_g\subset \bfR^{2n+1}$ as the union of the image of $\psi$ with suitable tubular neighborhoods of the arcs 
 $[0,1]\ni s\mapsto (2s+i+\tfrac{1}{2})\cdot \vec{e}_{2n+1}\in \bfR^{2n+1}$
 connecting the $i$th and $(i+1)$th disjoint summand, disjoint from the cores and axially symmetric around the $\vec{e}_{2n+1}$-axis. This embedding $V_g\subset\bfR^{2n+1}$ together with its cores $(e_i,f_i)_{1\le i\le g}$ satisfies the desired properties (i)--(iv). See \cref{equ:vtwo-figure} for a schematic figure.
 
In addition, we once and for all choose an embedded codimension $0$ disc $D^{2n+1}\subset V_g\cap \bfR^{2n}\times(2g-1,2g)$ which is disjoint from the cores and axially symmetric around the $\vec{e}_{2n+1}$-axis. Moreover, we choose the half-disc $D^{2n}_+\subset D^{2n}$ used in the definition of $\Emb_{\half D^{2n}}(V_g,W_{g,1})$ and $\Emb_{\half\partial}(W_{g,1})$ to have the same axial symmetry.

\begin{figure}[h!]
 \centering

\tikzset{every picture/.style={line width=0.75pt}} 

\begin{tikzpicture}[x=0.75pt,y=0.75pt,yscale=-1,xscale=1]

\draw  [color={rgb, 255:red, 0; green, 0; blue, 0 }  ,draw opacity=1 ] (32.33,125.05) .. controls (32.33,75.5) and (73.06,35.33) .. (123.31,35.33) .. controls (173.56,35.33) and (214.29,75.5) .. (214.29,125.05) .. controls (214.29,174.6) and (173.56,214.76) .. (123.31,214.76) .. controls (73.06,214.76) and (32.33,174.6) .. (32.33,125.05) -- cycle ;
\draw  [color={rgb, 255:red, 0; green, 0; blue, 0 }  ,draw opacity=1 ][line width=2.25]  (76.01,126.71) .. controls (76.01,102) and (96.33,81.96) .. (121.4,81.96) .. controls (146.46,81.96) and (166.78,102) .. (166.78,126.71) .. controls (166.78,151.42) and (146.46,171.46) .. (121.4,171.46) .. controls (96.33,171.46) and (76.01,151.42) .. (76.01,126.71) -- cycle ;
\draw  [color={rgb, 255:red, 0; green, 0; blue, 0 }  ,draw opacity=1 ] (259.06,127.28) .. controls (259.06,77.74) and (299.79,37.57) .. (350.04,37.57) .. controls (400.29,37.57) and (441.02,77.74) .. (441.02,127.28) .. controls (441.02,176.83) and (400.29,217) .. (350.04,217) .. controls (299.79,217) and (259.06,176.83) .. (259.06,127.28) -- cycle ;
\draw  [color={rgb, 255:red, 0; green, 0; blue, 0 }  ,draw opacity=1 ][line width=2.25]  (302.65,129.26) .. controls (302.65,104.55) and (322.97,84.52) .. (348.03,84.52) .. controls (373.09,84.52) and (393.41,104.55) .. (393.41,129.26) .. controls (393.41,153.98) and (373.09,174.01) .. (348.03,174.01) .. controls (322.97,174.01) and (302.65,153.98) .. (302.65,129.26) -- cycle ;
\draw  [draw opacity=0] (75.9,125.6) .. controls (75.43,134.73) and (65.86,142.01) .. (54.12,142.01) .. controls (42.15,142.01) and (32.44,134.44) .. (32.33,125.05) -- (54.12,124.89) -- cycle ; \draw  [color={rgb, 255:red, 0; green, 0; blue, 0 }  ,draw opacity=1 ] (75.9,125.6) .. controls (75.43,134.73) and (65.86,142.01) .. (54.12,142.01) .. controls (42.15,142.01) and (32.44,134.44) .. (32.33,125.05) ;  
\draw  [draw opacity=0][line width=2.25]  (120.91,81.95) .. controls (115.72,81.08) and (111.62,70.99) .. (111.62,58.67) .. controls (111.62,45.82) and (116.08,35.39) .. (121.59,35.33) -- (121.64,58.67) -- cycle ; \draw  [color={rgb, 255:red, 0; green, 0; blue, 0 }  ,draw opacity=1 ][line width=2.25]  (120.91,81.95) .. controls (115.72,81.08) and (111.62,70.99) .. (111.62,58.67) .. controls (111.62,45.82) and (116.08,35.39) .. (121.59,35.33) ;  
\draw  [draw opacity=0][line width=2.25]  (350.16,83.71) .. controls (344,82.97) and (339.12,72.88) .. (339.12,60.54) .. controls (339.12,47.76) and (344.36,37.39) .. (350.84,37.33) -- (350.89,60.54) -- cycle ; \draw  [color={rgb, 255:red, 0; green, 0; blue, 0 }  ,draw opacity=1 ][line width=2.25]  (350.16,83.71) .. controls (344,82.97) and (339.12,72.88) .. (339.12,60.54) .. controls (339.12,47.76) and (344.36,37.39) .. (350.84,37.33) ;  
\draw  [draw opacity=0] (213.7,127.88) .. controls (211.23,134.85) and (201.76,140.04) .. (190.44,140.04) .. controls (177.34,140.04) and (166.7,133.09) .. (166.56,124.47) -- (190.44,124.3) -- cycle ; \draw  [color={rgb, 255:red, 0; green, 0; blue, 0 }  ,draw opacity=1 ] (213.7,127.88) .. controls (211.23,134.85) and (201.76,140.04) .. (190.44,140.04) .. controls (177.34,140.04) and (166.7,133.09) .. (166.56,124.47) ;  
\draw  [draw opacity=0] (302.62,127.83) .. controls (302.06,135.37) and (292.53,141.37) .. (280.85,141.37) .. controls (268.89,141.37) and (259.19,135.08) .. (259.06,127.29) -- (280.85,127.13) -- cycle ; \draw  [color={rgb, 255:red, 0; green, 0; blue, 0 }  ,draw opacity=1 ] (302.62,127.83) .. controls (302.06,135.37) and (292.53,141.37) .. (280.85,141.37) .. controls (268.89,141.37) and (259.19,135.08) .. (259.06,127.29) ;  
\draw  [draw opacity=0] (441.02,127.3) .. controls (440.37,135.12) and (429.94,141.33) .. (417.17,141.33) .. controls (404.08,141.33) and (393.44,134.8) .. (393.29,126.71) -- (417.17,126.53) -- cycle ; \draw  [color={rgb, 255:red, 0; green, 0; blue, 0 }  ,draw opacity=1 ] (441.02,127.3) .. controls (440.37,135.12) and (429.94,141.33) .. (417.17,141.33) .. controls (404.08,141.33) and (393.44,134.8) .. (393.29,126.71) ;  
\draw [color={rgb, 255:red, 0; green, 0; blue, 0 }  ,draw opacity=1 ][line width=0.75]    (202,79.46) .. controls (229.43,111.05) and (242.64,113.06) .. (272.1,80.96) ;
\draw [color={rgb, 255:red, 0; green, 0; blue, 0 }  ,draw opacity=1 ][line width=0.75]    (200.98,172.22) .. controls (229.43,141.14) and (244.67,141.14) .. (271.09,172.22) ;
\draw [color={rgb, 255:red, 0; green, 0; blue, 0 }  ,draw opacity=1 ]   (259.06,127.28) ;
\draw [color={rgb, 255:red, 0; green, 0; blue, 0 }  ,draw opacity=1 ][line width=1.5]  [dash pattern={on 5.63pt off 4.5pt}]  (213.72,127.81) -- (259.06,127.28) ;
\draw  [draw opacity=0][line width=0.75]  (234.64,149.48) .. controls (230.74,147.37) and (227.81,138.35) .. (227.81,127.55) .. controls (227.81,115.17) and (231.65,105.14) .. (236.39,105.14) .. controls (236.66,105.14) and (236.93,105.17) .. (237.2,105.24) -- (236.39,127.55) -- cycle ; \draw  [color={rgb, 255:red, 0; green, 0; blue, 0 }  ,draw opacity=1 ][line width=0.75]  (234.64,149.48) .. controls (230.74,147.37) and (227.81,138.35) .. (227.81,127.55) .. controls (227.81,115.17) and (231.65,105.14) .. (236.39,105.14) .. controls (236.66,105.14) and (236.93,105.17) .. (237.2,105.24) ;  

\draw (119.19,56.61) node [anchor=north west][inner sep=0.75pt]   [align=left] {$\displaystyle e_{1}$};
\draw (345.92,58.85) node [anchor=north west][inner sep=0.75pt]   [align=left] {$\displaystyle e_{2}$};
\draw (81.97,121.17) node [anchor=north west][inner sep=0.75pt]   [align=left] {$\displaystyle f_{1}$};
\draw (307.38,122.52) node [anchor=north west][inner sep=0.75pt]   [align=left] {$\displaystyle f_{2}$};

\end{tikzpicture}

\caption{A figure of the standard embedding $V_2\subset \bfR^{2n+1}$, given as the union of two disjoint copies of $V_1=S^n\times D^{n+1}$ containing the cores $e_i$ and $f_i$ in the boundary (the thick circles) with an axially symmetric tubular neighbourhood of an arc connecting the two copies of $V_1$ (the dashed line).} \label{equ:vtwo-figure}

\end{figure}

\subsubsection{Standard (stable) framings}\label{sec:standard-framing}
The standard framing of $\bfR^{2n+1}$ induces by restriction a framing $\ell_{V_g}$ on $V_g\subset \bfR^{2n+1}$, which we call the \emph{standard framing} of $V_g$. Similarly, we fix a standard framing $\ell_{W_{g,1}}$ of $W_{g,1}$ which we require to (i) induce the $\pm$-structure on $W_{g,1}$ of \cite[Lemma 6.10]{KR-WDisc} and (ii) have trivial Arf invariant (see \cite[Section 2.5]{KR-WFramings} for an explanation of this condition). From now on, we shall always use the standard framings $\ell_{V_g}$ and $\ell_{W_{g,1}}$ and the induced boundary conditions on submanifolds on the boundaries on $\ell_{W_{g,1}}$ and $\ell_{W_{g,1}}$. When we consider stable framings of $V_g$ and $W_{g,1}$, we always use the ones induced by the standard framings. If there is no source of confusion we abbreviate the standard framings $\ell_{V_g}$ and $\ell_{W_{g,1}}$ as well as the induced (stable) framings on submanifolds of $V_g$ and $W_{g,1}$ by $\ell$, and we omit the reference to the induces boundary conditions from the notation. For instance we abbreviate 
\[\BDiff^\fr_{D^{2n}}(V_g)=\BDiff^\fr_{D^{2n}}(V_g;\ell_{D^{2n}})\quad\text{and}\quad\BEmb_{\half D^{2n}}^\fr(V_g,W_{g,1})=\BEmb_{\half D^{2n}}^\fr(V_g,W_{g,1};\ell_{\half D^{2n}}).\] 
Also, we consider the different variants of (stably) framed automorphism and self-embedding spaces of $V_g$ and $W_{g,1}$ as based spaces via the points induced by the standard framings. As in \cref{sec:tangential-structures}, a subscript $(-)_\ell$ indicates that we restrict to this basepoint component.

\subsection{Framings and stable framings} At various points in later sections, we will need to transition between framings and stable framings. The following lemma allows this.

\begin{lem}\label{lem:connectivity-framing-comparison}
The forgetful maps 
\vspace{-0.2cm}
\[
\begin{tikzcd}[row sep=0cm, column sep=0.3cm]
\BDiff_{D^{2n}}^{\fr}(V_{g})\rar& \BDiff_{D^{2n}}^{\sfr}(V_{g})&& \BEmb^{\fr,\cong}_{\nicefrac{1}{2}D^{2n}}(V_g,W_{g,1})\rar& \BEmb^{\sfr,\cong}_{\nicefrac{1}{2}D^{2n}}(V_g,W_{g,1})\\
\BEmb^{\fr,\cong}_{\half\partial}(V_g) \rar& \BEmb^{\sfr,\cong}_{\half\partial}(V_g)&&
\BEmb^{\onefr,\cong}_{\half\partial}(W_{g,1}) \rar& \BEmb^{\sfr,\cong}_{\half\partial}(W_{g,1})
\end{tikzcd}\vspace{-0.2cm}
\]
are $n$-connected and rationally $(3n+2)$-connected.
\end{lem}
\begin{proof}
Comparing the spaces in the statement to their analogous without (stable) framings, this follows from the fact that the maps the maps between mapping spaces of pairs
\begin{align*}
\Map((V_g,D^{2n}),(\SO(2n+1),*))&\lra \Map((V_g,D^{2n}),(\SO,*))\\
\Map((V_g,\nicefrac{1}{2}\partial V_g),(\SO(2n+1),*))&\lra \Map((V_g,\nicefrac{1}{2}\partial V_g),(\SO,*))\\
\Map((W_{g,1},\nicefrac{1}{2}\partial W_{g,1}), (\SO(2n+1),*))&\lra \Map((W_{g,1},\nicefrac{1}{2}\partial W_{g,1}), (\SO,*))
\end{align*}
induced by the inclusion $\SO(2n+1)\subset \SO$ are $n$-connected and rationally $(3n+2)$-connected by obstruction theory, since the pairs $(V_g,D^{2n})$, $(V_g, \nicefrac{1}{2}\partial V_g)$, and $(W_g, \nicefrac{1}{2}\partial W_g)$ have only relative $n$-cells and the inclusion $\SO(2n+1)\subset \SO$  is $2n$-connected and rationally $(4n+2)$-connected.
\end{proof}

\subsection{Mapping class groups}\label{sec:MCG}
In view of \cref{cor:reduction-to-self-embeddings}, in order to prove \cref{bigthm:diff-discs}, we are faced with the task of computing the rational cohomology of the spaces (see \cref{sec:standard-framing} for the shortened notation in use)
\begin{equation}\label{equ:self-embedding-spaces}
\BEmb^{\fr}_{\half D^{2n}}(V_g,W_{g,1})_{\ell}\quad\text{and}\quad\BEmb^{\fr,\cong}_{\half \partial}(W_{g,1})_{\ell}.
\end{equation}
To achieve this we need some understanding of the fundamental groups
\begin{equation}\label{equ:fundamental-groups-embedding-spaces}
\check{\Lambda}_{V_{g}}\coloneq \pi_1(\BEmb^{\fr}_{\half D^{2n}}(V_g,W_{g,1})_\ell)\quad\text{and}\quad
\check{\Lambda}_{W_{g,1}}\coloneq \pi_1(\BEmb^{\fr,\cong}_{\half \partial}(W_{g,1})_\ell)
\end{equation}
which come with comparison maps to $\GL(\bar{H}_W^\bfZ)$ (see \cref{sec:standard-model} for the notation)
\begin{equation}\label{equ:homology-action-mcg-emb-spaces}
\check{\Lambda}_{V_{g}}\lra \GL(\bar{H}_W^\bfZ)\quad\text{and}\quad\check{\Lambda}_{W_{g,1}}\lra \GL(\bar{H}_W^\bfZ)
\end{equation}
induced by the action on the homology of $W_{g,1}\subset \partial V_g$. These maps are almost  injective:
\begin{lem}\label{lem:finite-kernel}
For $n\ge3$, the maps \eqref{equ:homology-action-mcg-emb-spaces} have finite kernel. 
\end{lem}

\begin{proof}For the second map, this is \cite[Lemma 3.8]{KR-WDisc}. To deal with the first map, we consider the zig-zag
\begin{equation}\label{eq:SomeZigZag}
\check{\Lambda}_{V_{g}}\lla \pi_1(\BDiff^{\fr}_{D^{2n}}(V_g)_\ell)\lra \pi_1(\BDiff^{\sfr}_{D^{2n}}(V_g)_\ell)\lra \pi_1(\BlockBDiff^{\sfr}_{D^{2n}}(V_g)_\ell)
\end{equation}
where the rightmost arrow is the block-comparison map appearing in \eqref{equ:comparison-tangential}, the middle arrow is induced by passing from framings to stable framings, and the leftmost arrow forgets from framed diffeomorphisms to framed self-embeddings. All maps in this zig-zag are isomorphisms. For the leftmost arrow this holds as a consequence of  the second fibre sequence in \eqref{equ:weiss-fs-Vg} for $\Xi=\fr$ and the fact that $\BDiff^\fr_{D^{2n}}(D^{2n+1})\simeq \BDiff_{D^{2n}}(D^{2n+1})$ is simply connected since its loop space is equivalent to the space $\oC(D^{2n})$ of concordances of $D^{2n}$ by smoothing corners, and the latter is connected for $n\ge3$ by Cerf's work on pseudoisotopy \cite[Thèoreme 0]{Cerf}. For the middle arrow, this holds  by  \cref{lem:connectivity-framing-comparison}. For the rightmost arrow it holds because the comparison map $\Diff_{D^{2n}}(V_g)\ra \BlockDiff_{D^{2n}}(V_g)$ is $1$-connected as another consequence of Cerf's work just mentioned. It thus suffices to show that the composition 
\[
\pi_1(\BlockBDiff^{\sfr}_{D^{2n}}(V_g)_\ell)\lra \pi_1(\BhAut^{\cong,\ell}_{D^{2n}}(V_g,W_{g,1}))\lra \GL(\bar{H}_W^\bfZ)
\]
has finite kernel. In fact, the two maps each have finite kernel, the first by the second part of \cref{thm:blockdiff-identification} and the second map as a result of \cite[Lemma 3.10]{KrannichConc1}.
\end{proof}

Up to these finite kernels, the groups \eqref{equ:fundamental-groups-embedding-spaces} thus agree with the subgroups 
\begin{equation}\label{equ:linear-part-of-mcgs}
G_{V_g}\coloneqq \im\Big(\check{\Lambda}_{V_{g}}\ra \GL(\bar{H}_W^\bfZ) \Big)\quad \text{and}\quad G_{W_{g,1}}\coloneq \im\Big(\check{\Lambda}_{W_{g,1}}\ra \GL(\bar{H}_W^\bfZ)\Big)
\end{equation}
of $\GL(\bar{H}_W^\bfZ)$. If $g$ is clear from the context, we abbreviate these and the groups in \eqref{equ:fundamental-groups-embedding-spaces} as
\[\check{\Lambda}_{V}\coloneq \check{\Lambda}_{V_{g}},\quad \check{\Lambda}_{W}\coloneq \check{\Lambda}_{W_{g,1}},\quad G_V\coloneq G_{V_g}, \quad\text{and}\quad G_W\coloneq G_{W_{g,1}}.\]

\begin{rem}Our choice of notation differs slightly from that of other work in which these groups have appeared: 
\begin{enumerate}
\item $G_W=G_{W_{g,1}}$ is denoted $G_g^{\fr,[[\ell]]}$ in \cite{KR-WAlg}, 
\item  $\check{\Lambda}_W=\check{\Lambda}_{W_{g,1}}$ is denoted $\check{\Lambda}^{\fr,\ell}_{g}$ in \cite{KR-WAlg}, and
\item $G_V=G_{V_g}$ agrees with the subgroup denoted by $G_{g,\ell}^{\ext}$ in \cite{KrannichConc1}.
\end{enumerate}
\end{rem}

The group $G_V$ is closely related to $\GL(\bar{H}_V^\bfZ)$. Indeed, any self-embedding of the pair $(V_g,W_{g,1})$ preserves the subspace $\bar{K}^\bfZ_W\subset \bar{H}_W^\bfZ$, so we have an inclusion
\[
G_V\subset\{\Phi\in\GL(\bar{H}_W^\bfZ)\mid\Phi(\bar{K}^\bfZ_W)\subset \bar{K}^\bfZ_W\}
\]
and thus---using the canonical isomorphism $\bar{H}_W^\bfZ/\bar{K}_W^\bfZ\cong \bar{H}_V^\bfZ$---a canonical map \begin{equation}\label{equ:homology-action-Vg}G_V\lra \GL(\bar{H}_V^\bfZ).\end{equation} Precomposed with $\check{\Lambda}_{V}\ra G_V$, this agrees with the usual action on homology.

\begin{lem}\label{lem:mcg-is-semidirect-product}
For $n\ge 3$, the map \eqref{equ:homology-action-Vg} gives rise to an extension
\[
0 \lra M_V^\bfZ \coloneq \begin{cases}
 \Lambda^2(\bar{H}_V^\bfZ)^\vee&\text{for }n\text{ even}\\
\Gamma^2(\bar{H}_V^\bfZ)^\vee&\text{for }n\text{ odd}\\
\end{cases} \lra  G_V\lra \GL(\bar{H}_V^\bfZ)\lra 0.
\]
whose kernel is the integral dual of the exterior or divided power square of $\bar{H}_V^\bfZ$, acted upon by $\GL(\bar{H}_V^\bfZ)$ through $\bar{H}_V^\bfZ$. Moreover, this extension admits a preferred splitting, so 
\[
G_V\cong 
M_V^\bfZ \rtimes \GL(\bar{H}_V^\bfZ).
\]
\end{lem}
\begin{proof}
In view of the zig-zag of isomorphisms \eqref{eq:SomeZigZag} in the proof of \cref{lem:finite-kernel}, the extension has already been established as part of \cite[Proposition\,3.7 (ii)]{KrannichConc1}, so it remains to construct the splitting. For this, we consider the subgroup $\pi_0(\Diff_{D^{2n}}(V_g))_{\bfR^{2n+1}}\subset \pi_0(\Diff_{D^{2n}}(V_g))$ of diffeomorphisms of $V_g$ which admit an extension up to isotopy to a diffeomorphism of $\bfR^{2n+1}$ along the standard embedding $V_g\subset\bfR^{2n+1}$ of \cref{sec:standard-model}. Wall \cite[Lemma 17]{WallII} has shown that the map $\pi_0(\Diff_{D^{2n}}(V_g))_{\bfR^{2n+1}}\ra \GL(\bar{H}_V^\bfZ)$ is an isomorphism, so it suffices to  show that the subgroup $\pi_0(\Diff_{D^{2n}}(V_g))_{\bfR^{2n+1}}$ is contained in the image of the map induced by forgetting framings $\pi_1(\BDiff^\fr_{D^{2n}}(V_g)_\ell)\lra \pi_0(\Diff_{D^{2n}}(V_g))$ since the image of the latter in $\GL(H^\bfZ_W)$ is the subgroup $G_V$. Equivalently, the task is to show that the stabiliser of the standard framing $[\ell]\in \pi_0(\Bun_{D^{2n}}(\tau_{V_g},\varepsilon^{2n+1};\ell_{D^{2n}}))$ under the $\pi_0(\Diff_{D^{2n}}(V_g))$-action contains $\pi_0(\Diff_{D^{2n}}(V_g))_{\bfR^{2n+1}}$. This follows by observing that $\ell$ is by construction restricted from a framing of $\bfR^{2n+1}$ and that $\pi_0(\Bun_{D^{2n}}(\tau_{\bfR^{2n+1}},\varepsilon^{2n+1};\ell_{D^{2n}}))$ is a singleton.
\end{proof}

More explicitly, the subgroups $G_V$ and $G_W$ can be described in terms of the hyperbolic basis $\{e_i,f_i\}_{1\le i\le g}$ of $\bar{H}_W^\bfZ\cong\bfZ^{2g}$ fixed in \cref{sec:standard-model}, ordered by $(e_1,\ldots,e_g,f_1,\ldots,f_g)$. Namely, with respect to the induced identification $\GL(\bar{H}_W^\bfZ)\cong \GL_{2g}(\bfZ)$, we have 
\begin{align}\label{equ:matrix-description-Gg}
\begin{split}
G_W&\cong \begin{cases}
\oO_{g,g}(\bfZ)&\text{for }n\text{ even}\\
\Sp^q_{2g}(\bfZ)&\text{for }n\text{ odd},
\end{cases}\le \OSp_{g}(\bfZ)\quad\text{and}\\
G_V&\cong \{\Big(\begin{smallmatrix}A&AB\\0&(A^T)^{-1}\end{smallmatrix}\Big)\in \GL_{2g}(\bfZ),\mid A\in\GL_g(\bfZ), B\in M^{(-1)^{n+1}}_g\},
\end{split}
\end{align}
see \cite[Proposition 3.5]{KR-WFramings} for the first identification (this uses our condition on the standard framing in \cref{sec:standard-framing}) and \cite[Remarks 3.2, 3.5, 3.8]{KrannichConc1} for the second\footnote{There is a typo in \cite[Remark 3.5]{KrannichConc1}: The matrix should have $AM$ instead of $M$ as the upper right entry.}. Here
\begin{enumerate}
\item $\OSp_{g}(\bfZ)$ is defined as the automorphism group $\oO_{g,g}(\bfZ)$ of the symmetric form $\left(\begin{smallmatrix}0&I\\I&0\end{smallmatrix}\right)$ for $n$ even and the automorphism group $\Sp_{2g}(\bfZ)$ of the antisymmetric form $\left(\begin{smallmatrix}0&I\\-I&0\end{smallmatrix}\right)$ for $n$ odd. We omit the parameter $n$ from the notation; it will be clear from the context.
\item $\Sp^q_{2g}(\bfZ)\le \Sp_{2g}(\bfZ)$ is the finite index subgroup of those symplectic matrices which preserve the $\bfZ/2$-valued quadratic refinement
\[
\map{q}{\bfZ^{2g}}{\bfZ/2}{\sum_{i=1}^g(A_i e_i+B_i f_i)}{\sum_{i=1}^gA_iB_i \Mod{2}.}
\]
\item $M^{\pm}_g\subset \bfZ^{g\times g}$ is the additive subgroup of integral $(g\times g)$-matrices which are $(\pm1)$-symmetric and whose diagonal entries are even.
\end{enumerate}

\subsection{The reflection involution}\label{sec:Reflection}
It will be convenient for us to keep track of the following piece of additional structure. The manifold $V_g$ admits an orientation-reversing smooth involution $\rho$---the \emph{reflection involution}---which preserves setwise the subspaces $W_{g,1}\subset W_g\subset V_g$ and the discs $D^{2n+1}\subset V_g$ and $D^{2n}\subset W_g$, and agrees with a reflection on these discs. Using the standard model of $V_g\subset \bfR^{2n+1}$ from \cref{sec:standard-model}, we choose $\rho$ given by acting by $-1$ on the first coordinate (we could use any of the first $n$ coordinates). This fixes the cores $f_i\colon S^n\hookrightarrow V_g$ pointwise, and on the cores $e_i\colon S^n\hookrightarrow V_g$ restricts to a reflection of $S^n$. It also restricts to a reflection on the discs $D^{2n}\subset \partial V_g=W_g$, $D^{2n+1}\subset V_g$, and $D^{2n}_+\subset D^{2n}$ as chosen in \cref{sec:standard-model}.

\subsubsection{Action on automorphism spaces}
Conjugation by the reflection involution $\rho$ induces compatible pointed involutions $\rho$ on each of the spaces in the diagram
\begin{equation}\label{equ:big-diagram-with-the-players}
\begin{tikzcd}
 \BDiff_\partial(V_g) \dar\rar& \BDiff_{D^{2n}}(V_g) \dar \rar& \BDiff_\partial(W_{g,1}) \dar \\ 
\BEmb^\cong_{\half\partial}(V_g) \dar\rar&\BEmb^\cong_{\half D^{2n}}(V_g, W_{g,1})\rar\dar &\BEmb^\cong_{\half\partial}(W_{g,1}) \dar\\
    \BhAut_\partial(V_g)\rar& \BhAut_{D^{2n}}(V_g,W_{g,1})\rar& \BhAut_\partial(W_{g,1}).
\end{tikzcd}
\end{equation}
In \cite[Section\ 6.5]{KR-WDisc}, it is explained that these involutions lift along the forgetful maps
\[\BDiff^{\fr}_\partial(W_{g,1})_{\ell}\ra \BDiff_\partial(W_{g,1})\quad\text{and}\quad \BEmb_{\half\partial}^{\cong,\fr}(W_{g,1})_{\ell}\ra \BEmb_{\half\partial}^\cong(W_{g,1})\] to involutions in the pointed $A_\infty$-sense on $\BDiff^{\fr}_\partial(W_{g,1})_\ell$ and $\BEmb_{\half\partial}^{\cong,\ell}(W_{g,1})_\ell$, so in particular there are involutions on these spaces in the pointed homotopy category. It will be convenient for us to know that the same is true for the framed analogues of the spaces forming the upper left square of \eqref{equ:big-diagram-with-the-players}. We will explain this for $\BDiff_{D^{2n}}(V_g)$; the other cases are analogous. To this effect, we consider the map of fibre sequences 
\[\label{equ:framed-involution-lifts-diagram}
\begin{tikzcd}
\BDiff^\fr_{D^{2n}}(V_g)\rar\dar& \BDiff^\fr(V_g,D^{2n})\rar\dar& \BDiff^{\onefr}(D^{2n})\dar\\
\BDiff_{D^{2n}}(V_g)\rar& \BDiff(V_g,D^{2n})\rar& \BDiff(D^{2n}),
\end{tikzcd}
\]
induced by restriction, and view the spaces forming the upper row as based via the basepoints induced by the standard framing. As the loop space of the total space of any fibration of pointed spaces naturally acts (in the pointed $A_\infty$-sense) on the fibre, it suffices to prove:
\begin{lem}Writing $C_2=\{\pm1\}$, the pointed map $B C_2\ra \BDiff(V_g,D^{2n})$ induced by the reflection involution lifts along the map $\BDiff^\fr(V_g,D^{2n})_{\ell}\ra  \BDiff(V_g,D^{2n})$.
\end{lem}
\begin{proof}
By obstruction theory, the forgetful map $\BDiff^\fr(V_g,D^{2n})_{\ell}\to \BDiff^{\pm\fr}(V_g,D^{2n})_{\ell}$ is an equivalence (see \cref{ex:tangential-structures} for the $\pm\fr$-notation), so it suffices to solve the lifting problem for $\fr$ replaced by $\pm\fr$. This new lifting problem is equivalent to lifting along $\BO(1)\ra \BO(2n+1)$ the map $T_\pi\colon \oE C_2\times_{\rho}V_g\ra \BO(2n+1)$ classifying the vertical tangent bundle of the $V_g$-bundle $\oE C_2\times_{\rho} V_g\to B C_2$ corresponding to the involution $\rho$, ensuring that on a fibre $V_g$ the lift agrees with that induced by the standard framing of $V_g$. Using that the involution $\rho\colon V_g\ra V_g$ is restricted from the reflection of $\bfR^{2n+1}$ in the first coordinate, we see that this vertical tangent bundle is given by the pullback of the vertical tangent bundle of $\oE C_2\times_{\rho}\bfR^{2n+1}\ra B C_2$, which in turn agrees with the pullback of the bundle over $B C_2$ classified by the composition $B C_2\cong \BO(1)\ra \BO(2n+1)$. The latter tautologically has the required lifting property.
\end{proof}

\subsubsection{Action on mapping class groups}\label{sec:involution-on-mcg}
In particular, the reflection involution $\rho$ acts on the groups $\check{\Lambda}_W$, $\check{\Lambda}_V$ by conjugation, and also on the groups $G_W$, $G_V$, and $\GL(\bar{H}^\bfZ_V)$ (see \cref{sec:MCG}). As we chose the cores $f_i$ to form a basis of $\bar{H}_{V}^\bfZ$ and the cores $e_i$ to form a basis of $\bar{K}_W^\bfZ$ (see \cref{sec:standard-model} for the notation), the involution $\rho$ acts trivially on $\bar{H}_V^\bfZ$ and by multiplication by $-1$ on $\bar{K}_W^\bfZ$. From this we see that when considering $G_W$ and $G_V$ as subgroups of $\GL_{2g}(\bfZ)$ via the fixed basis the involution is given by conjugation by $\left(\begin{smallmatrix}-I&0\\0&I\end{smallmatrix}\right)$. In particular, the involution preserves the semi-direct product decomposition $G_V = M_V^\bfZ \rtimes \GL(\bar{H}_V^\bfZ)$: it is given by the identity on $\GL(\bar{H}_V^\bfZ)$ and by multiplication by $-1$ on $M_V^\bfZ$.

\subsubsection{Interaction with the intersection form}\label{sec:involution-and-intersection}As $\rho$ reverses the orientation, care is needed when it comes to its interaction with the intersection form $\lambda$ of $W_g$. Writing $\bfZ^\pm$ for $\bfZ$ acted upon by $\rho$ as $\pm1$, and similarly $\bfQ^\pm$, the intersection form is $\rho$-equivariant when considered as a morphism
$\lambda\colon \bar{H}_W^\bfZ \otimes \bar{H}_W^\bfZ \ra \bfZ^-,$
and similarly after rationalising, and it induces isomorphisms $\bar{H}_W^\vee\cong \bar{H}_W\otimes \bfQ^-$ as $G_W\rtimes\langle \rho\rangle$-modules and $\bar{H}_V^\vee\cong \bar{K}_W^\bfZ\otimes \bfQ^-$ as $G_V\rtimes\langle\rho\rangle$-modules. By dualising, the intersection form induces a map 
$\omega\colon \bfQ^- \ra \bar{H}_W \otimes \bar{H}_W$. Thinking of $\omega$ as an element of $\bar{H}_W \otimes \bar{H}_W$, the involution $\rho$ acts on it by $-1$. Moreover, as the $f_i$'s are fixed by $\rho$, they induce a splitting of the $G_V\rtimes\langle\rho\rangle$-equivariant extension
\begin{equation}\label{equ:filtration-H_W}0\lra \bar{H}_V^\vee\otimes\bfQ^-\lra \bar{H}_W\lra \bar{H}_V\lra 0\end{equation}
when restricted to the subgroup $\GL(\bar{H}^\bfZ_V)\times\langle \rho\rangle\le G_V\rtimes\langle\rho\rangle$, so we have an isomorphism $\bar{H}_W\cong \bar{H}_V\oplus (\bar{H}_V^\vee\otimes\bfQ^-)$ as $\GL(\bar{H}^\bfZ_V)\times\langle \rho\rangle$-modules.

\section{Homotopy automorphisms and mapping spaces}\label{sec:hAut}
Our strategy for computing the rational homology of the spaces of self-embeddings appearing in \cref{cor:reduction-to-self-embeddings} 
is to compare them to the relevant spaces of homotopy automorphisms forming the bottom row in \eqref{equ:big-diagram-with-the-players}. The rational homotopy Lie algebras of the universal covers of these spaces of homotopy automorphisms are known by work of Berglund--Madsen \cite{BerglundMadsen} and Krannich \cite{KrannichConc1}. We explain their answers in this section, alongside with some preliminaries and further information on the rational homotopy type of related spaces needed in later sections.

\subsection{Gradings and Schur functors}\label{sec-gradings-and-schur}
We begin by fixing some conventions on gradings and Schur functors used in the remainder of this work.

\subsubsection{Grading conventions}\label{sec:gradings}Unless said otherwise, throughout this and the following sections, we work in the symmetric monoidal category of $\bfZ$-graded vector spaces over $\bfQ$, where the symmetry includes the Koszul sign rule, i.e.\, is induced by $v\otimes w\mapsto (-1)^{\deg(v)\deg(w)} w\otimes v$. For a graded vector space $A$, we write $A_k$ for the subspace of elements of degree $k$ and $|x|$ for the degree of an homogenous element $x\in A$. For an ungraded vector space $B$, we write $B[k]$ for the graded vector space concentrated in degree $k$ with $B[k]_k=B$. The graded dual of a graded vector space $A$ is denoted $A^\vee\coloneq \Hom(A,\bfQ[0])$, and the $k$-fold suspension by $s^kA$. Total dimension is denoted $\dim(A)$. A superscript $(-)^+$ indicates that we pass to the subspace of elements of positive degree. We write $(-)_\bfQ\coloneq (-)\otimes_\bfZ\bfQ$ for rationalisation. Given a space $X$, we abbreviate its desuspended reduced homology by
\[H_X\coloneq s^{-1}\widetilde{\oH}_*(X;\bfQ).\]
For $X=V_g$ and $X=W_{g,1}$, we often neglect the genus $g$ from the notation and simply write
\[H_V\coloneq H_{V_g}=s^{-1}\widetilde{\oH}_*(V_g;\bfQ)\quad\text{and}\quad H_W\coloneq H_{W_{g,1}}=s^{-1}\widetilde{\oH}_*(W_{g,1};\bfQ).\]
This relates to the ungraded vector spaces $\bar{H}_V$ and $\bar{H}_W$ introduced in \cref{sec:standard-model} by \[H_V=\bar{H}_V[n-1]\quad\text{and}\quad H_W=\bar{H}_W[n-1].\]
Similarly, we write $K_W\coloneq \ker\big(\iota_*\colon H_W\ra H_V\big)$, so that $K_W=\bar{K}_W[n-1]$.

\subsubsection{Schur functors}\label{sec:Schur-functors}
Given a partition $\mu\vdash k$ of an integer $k\ge0$ the \emph{graded Schur functor} associated to $\mu$ is the endofunctor on graded vector spaces given by
\[S_\mu(A)\coloneq [(\mu)\otimes A^{\otimes k}]^{\Sigma_k}\]
where $(\mu)$ is the irreducible $\bfQ[\Sigma_k]$-module associated to $\mu$ and $[-]^{\Sigma_k}$ are the invariants with respect to the symmetric group action. The relation to usual Schur functors on ungraded vector spaces is as follows: because of the Koszul sign rule, given an ungraded vector space $B$, we have
\[
S_{\mu}(B[m])= [(\mu)\otimes (1^k)^{\otimes m}\otimes B^{\otimes k}]^{\Sigma_k}[km]=s^{km}\begin{cases}
S_{\mu}(B[0])&\text{for }m\text{ even}\\
S_{\mu'}(B[0])&\text{for }m\text{ odd}
\end{cases}
\]
where $\mu'\vdash k$ denotes the conjugate partition, so $(\mu)\otimes (1^k)=(\mu')$.
\begin{ex}
For $\mu=(k)$, a vector space $B$, and $m\ge0$, we have 
\[S_{k}(B[m])=s^{km}\begin{cases}
S_{k}(B[0])&\text{for }m\text{ even}\\
S_{1^k}(B[0])&\text{for }m\text{ odd}
\end{cases}=\begin{cases}
\Sym^k(B)&\text{for }m\text{ even}\\
\Lambda^k(B)&\text{for }m\text{ odd}
\end{cases}[k  m],\]
where $\Sym^k(-)$ and $\Lambda^k(-)$ are the usual (ungraded) symmetric and exterior powers.
\end{ex}

\subsubsection{Symplectic Schur functors}\label{sec:symplectic-schur-functors}Now suppose the graded vector space $A$ comes with an antisymmetric nondegenerate pairing of degree $k$, i.e.\ a map of graded $\bfQ[\Sigma_2]$-modules
\[\lambda\colon A\otimes A\lra (1^2)\otimes \bfQ[k]\]
 such that the adjoint $A\ra \Hom(A,\bfQ[k])=s^{k}A^\vee$ is an isomorphism. 
 
 \begin{ex}\label{ex:intersection-pairing}
For a compact oriented $d$-manifold $M$ with boundary a sphere, the intersection form induces a symmetric nondegenerate pairing of degree $d$ on the reduced homology, i.e. a map of graded $\bfQ[\Sigma_2]$-modules $\widetilde{\oH}_*(M;\bfQ) \otimes \widetilde{\oH}_*(M;\bfQ) \to \bfQ[d]$. Via the $\bfQ[\Sigma_2]$-module isomorphism $\bfQ[-1] \otimes \bfQ[-1] \cong (1^2) \otimes \bfQ[-2]$, this induces on the desuspension $H_M=s^{-1}\widetilde{\oH}_*(M;\bfQ)$ an antisymmetric nondegenerate pairing $\lambda$ of degree $d-2$ in the sense above.
\end{ex}

In this case, for $r\ge0$, we write
\[\textstyle{A^{[r]} \coloneq \ker\big(A^{\otimes r}\xra{\lambda_{i,j}}\bigoplus_{1\le i,j\le r}A^{\otimes r-2}\otimes \bfQ[k] \big)},\] where $\lambda_{i,j}$ applies $\lambda$ to the $i$th and $j$th coordinate, and define the \emph{graded symplectic Schur functor}
\[V_{\mu}(A) \coloneq [(\mu)\otimes A^{[k]}]^{\Sigma_k}\] for partitions $\mu\vdash k$. Similarly to the non-symplectic case discussed above, if $A=B[m]$ for an ungraded vector space $B$ and $m\ge0$, we have
\[V_\mu(B[m])\cong s^{km}\begin{cases}
V_{\mu}(B[0])&\text{for }m\text{ even}\\
V_{\mu'}(B[0])&\text{for }m\text{ odd}.
\end{cases}\]

\subsection{Lie algebras and their derivations}
\label{sec:derivations-reminder}
For a graded $\bfQ$-vector space $A$, we write $\bfL(A)$ for the free graded Lie algebra on $A$ and $\bfL^k(A)\subset \bfL(A)$ for the subspace spanned by $k$-fold brackets. For subspaces $C,D\subset L$ of a graded Lie algebra $L$, we write $[C,D]\subset L$ for the subspace spanned by elements $[c,d]$ with $c\in C$ and $d\in D$. For a morphism $f\colon L\ra L'$ of graded Lie algebras, we write $\Der^f(L,L')$ for the graded vector space of \emph{$f$-derivations} whose degree $k$ component consists of linear maps $\theta\colon L\ra L'$ that increase the degree by $k$ and satisfy the condition
$\theta([x,y])=[\theta(x),f(x)]+(-1)^{k|x|}[f(x),\theta(y)]$.
Having fixed an element $\omega\in L$, we write $\Der_\omega^f(L,L')\subset \Der^f(L,L')$ for the subspace of derivations that vanish on $\omega$. If $f=\id$, we abbreviate $\Der(L)=\Der^{\id}(L,L)$ and and consider this graded vector space as a graded Lie algebra via the bracket $[\theta,\psi] \coloneq \theta\circ \psi-(-1)^{|\psi|}\psi\circ\theta$ for homogeneous elements $\theta,\psi\in \Der(L)$. Note that for $\omega \in L$ the subspace $\Der_\omega(L)\subset \Der(L)$ is a graded sub Lie algebra.

\subsubsection{A dimension formula}In calculations within \cref{sec:block-embeddings}, we will use the following decomposition of the subspace of bracket lengths $\le4$ of the free graded Lie algebra $\bfL(A\oplus x)$ on a graded vector space $A$ concentrated either in even or in odd degrees and an additional generator $x$. We write $A^{\even}$ and $A^{\odd}$ for the subspaces of elements of even respectively odd degree.

\begin{lem}\label{lem:decomposition}For a graded vector space $A$ with $A=A^{\even}$ or $A=A^{\odd}$ and an additional generator $x$, we have direct sum decompositions
\begin{align*}
\bfL^1(A\oplus x)&=A\oplus \langle x\rangle\\
\bfL^2(A\oplus x)&=[A,A]\oplus [x,A]\oplus [x,x]\\
\bfL^3(A\oplus x)&=[A,[A,A]]\oplus [A,[x,A]]\oplus [x,[x,A]] \\
\bfL^4(A\oplus x)&=\bfL^4(A)\oplus [A,[A,[x,A]]\oplus \big\langle [x,[A,[x,A]]],[A,[x,[x,A]]] \big\rangle \oplus [x,[x,[x,A]]].
\end{align*}
Moreover, abbreviating $d=\dim(A)$ and setting $m=0$ if $A=A^{\even}$ and $m=1$ if $A=A^{\odd}$, the dimensions of the summands of bracket length $2$ and $3$ are given as
\[
\begin{tabular}{ c|c|c|c|c|c|c} 
 $[A,A]$ &$[x,A]$ & $[x,x]$ & $[A,[A,A]]$ & $[A,[x,A]]$ &  $[x,[x,A]]$ \\ 
\hline 
 $\tfrac{1}{2}(d^2-(-1)^md)$ &$d$ & $\tfrac{1}{2}(1-(-1)^{|x|})$ & $\tfrac{1}{3}(d^3-d)$ &  $d^2$ & $d$  \\
  \end{tabular}
\]
and those of the summands of bracket length $4$ are
\[
\begin{tabular}{c|c|c|c} $\bfL^4(A)$ & $[A,[A,[x,A]]]$ & $\big\langle [x,[A,[x,A]]],[A,[x,[x,A]]] \big\rangle$& $[x,[x,[x,A]]]$\\
 \hline
$\tfrac{1}{4}(d^4-d^2)$ & $d^3$  & $\tfrac{1}{2}(3d^2-(-1)^{m-|x|}d)$& $d$\\
\end{tabular}
\]
\end{lem}
\begin{proof}We will use two facts throughout the proof.
\begin{enumerate}
\item The free graded Lie algebra $\bfL(B)$ on a graded vector space $B$ is generated by right-normed words, i.e.\,elements of the form $[b_1,[b_2,[\ldots,[b_i,b_{i+1}]\ldots]]]$ for $b_i\in B$. This follows by induction on the word-length using the Jacobi identity.
\item Witt's classical dimension formula for free ungraded Lie algebras in terms of the Möbius function $\mu(-)$ generalises to the graded setting in that we have
\[
\textstyle{\dim(\bfL^k(B))=\tfrac{1}{k}\sum_{m\mid k}\mu(m)\cdot \big(\dim(B^{\even})-(-1)^m\dim(B^{\odd})\big)^{k/m},}
\]
for a graded vector space $B$ (see \cite[Corollary 2]{Petrogradsky}).
\end{enumerate}
We first establish the sum-decomposition. By counting occurrences of $x$'s, it is clear that the summands of the claimed composition have pairwise trivial intersection, so it suffices to show that their union span the subspaces of fixed bracket lengths. This is clear for $\bfL^1(A\oplus x)$ and $\bfL^2(A\oplus x)$. For $\bfL^3(A\oplus x)$, it follows from
\[
[A,[x,x]]\subset [x,[x,A]],\quad [x,[A,A]]\subset [A,[x,A]],\quad\text{and}\quad[x,[x,x]]=0
\]
which holds as a result of the (graded) Jacobi identity and (graded) anti-commutativity. In the case of $\bfL^4(A\oplus x)$, an application of the Jacobi identity to $[A,[x,[A,A]]]$, $[x,[A,[A,A]]]$, and to $[[A,A],[x,A]]$ shows that any Lie-word that contains one $x$ is contained in $[A,[A,[x,A]]$, two applications to $[A,[A,[x,x]]]$ and $[x,[x,A],A]]$ show that any Lie-word with two occurrences of $x$ is contained in $\langle [x,[A,[x,A]]],[A,[x,[x,A]]]\rangle$, and finally three applications to $[x,[A,[x,x]]]$, $[A,[x,[x,x]]]$, and $[[x,x][A,x]]$ show that any Lie-word with three $x$'s is contained in $[x,[x,[x,A]]]$. As $[x,[x,[x,x]]]$ is trivial since already $[x,[x,x]]$ is trivial by anti-commutativity and the Jacobi identity, this proves the claimed decomposition of $\bfL^4(A\oplus x)$. 

To calculate the entries in the two tables, we argue as follows: by the graded Witt formula and $\mu(1)=1$ as well as $\mu(2)=-1$ we have
\[
\dim([A,A])=\dim(\bfL^2(A))=\tfrac{1}{2}(d^2-(-1)^md),\quad\dim([x,x])=\dim(\bfL^2(x))=\tfrac{1}{2}(1-(-1)^{|x|}),
\]
and $[x,A] \cong A$ so $\dim([x,A]) = d$. 
This gives the first half of the first table. Another application of the graded Witt formula using $\mu(3)=-1$ shows 
\[
\dim([A,[A,A]])=\dim(\bfL^3(A))=\tfrac{1}{3}(d^3-d)\quad\text{and}\quad\dim(\bfL^3(A\oplus x))=\tfrac{1}{3}(d^3+3d^2+2d),
\] so we can combine this with the decomposition of $\bfL^3(A\oplus x)$ to conclude
\[
\dim([A,[x,A]])+\dim([x,[x,A]])=\dim(\bfL^3(A\oplus x))-\dim(\bfL^3(A))=d^2+d.
\]
This implies that the evident estimates $\dim([A,[x,A]])\le d^2$ and $\dim([x,[x,A]])\le d$ are equalities and thus confirms the second half of the first table. To deal with the second table, we use $\mu(4)=0$ and the graded Witt formula to compute
\[
\dim(\bfL^4(A))=\tfrac{1}{4}(d^4-d^2)\quad\text{and}\quad \dim(\bfL^4(A\oplus x))=\tfrac{1}{4}(d^4+4d^3+5a^2+4d-2(-1)^{m-|x|}d).
\]
Together with the above decomposition of $\bfL^4(A\oplus x)$ and the obvious estimates 
\begin{equation}\label{equ:first-dimension-estimate-Lie4}
\dim(\bfL^4([A,[A,[x,A]]])\le d^3\quad\text{and}\quad\dim(\bfL^4([A,[x,[x,A]]])\le d,
\end{equation}
this yields
\begin{equation}\label{equ:second-dimension-estimate-Lie4}
\dim\langle [x,[A,[x,A]]],[A,[x,[x,A]]] \big\rangle\ge \tfrac{1}{2}(3d^2-(-1)^{m-|x|}d).
\end{equation}
We claim that \eqref{equ:second-dimension-estimate-Lie4} is in fact an equality, which would imply that the inequalities in \eqref{equ:first-dimension-estimate-Lie4} have to be equalities as well and thus finish the proof. To prove this claim, we may equivalently show that the kernel of the epimorphism
\[
(A\otimes A)\oplus (A\otimes A)\xrightarrow{\big([x,[-,[x,-]],[-,[x,[x,-]] ]\big)}\big\langle [x,[A,[x,A]]],[A,[x,[x,A]]] \big\rangle)
\]
is at least $\tfrac{1}{2}(d^2+(-1)^{m-|x|}d)$-dimensional. A labourious but straight-forward sequence of applications of the Jacobi identity and anti-commutativity shows that this kernel contains
\[
\begin{cases}
(v\otimes w-w\otimes v,w\otimes v-v\otimes w)&\text{if }|x|\not\equiv m\Mod{2}\\
(v\otimes w+w\otimes v,-(w\otimes v+v\otimes w))&\text{if }|x|\equiv m\equiv0\Mod{2}\\
(v\otimes w-w\otimes v,w\otimes v+v\otimes w)&\text{if }|x|\equiv m\equiv1\Mod{2}
\end{cases}
\]
for $v,w\in A$. Applying this to a basis of $A$ implies the claimed dimension bound on the kernel and thus finishes the proof.
\end{proof}

\subsection{Homotopy groups of mapping spaces}\label{sec:homotopy-groups-haut}
We will need some information on the rational homotopy groups of $\BhAut_\partial(M)$ for various compact $d$-manifolds $M$ and on $\Map_*(M,X)$ for spaces $X$. Most of what we need can be extracted from work of Berglund--Madsen \cite{BerglundMadsen}. Their results apply in larger generality, but to simplify the exposition, we restrict to

\begin{itemize}[leftmargin=5ex]
\item a $1$-connected compact oriented $d$-manifold $M$  equivalent to a wedge of spheres, together with an orientation-preserving equivalence $\partial M\simeq S^{d-1}$ and based at a point in $\partial M$,
\item a $1$-connected based space $X$ equivalent to a wedge of spheres.
\end{itemize}
Using the notation of \cref{sec:gradings}, the Hurewicz homomorphism induces maps of graded vector spaces $\pi_*(\Omega_0 M)\ra H_M$ and $\pi_*(\Omega_0 X)\ra H_X$ that are epimorphisms as $M$ and $X$ are equivalent to wedges of spheres. A choice of splittings 
\[s_M\colon H_M\ra \pi_*(\Omega_0 M)\quad\text{and}\quad s_X\colon  H_X\ra \pi_*(\Omega_0 X)\] 
induces isomorphisms of graded Lie algebras $\bfL(H_M) \cong \pi_{*}(\Omega_0M)$ and $\bfL(H_X) \cong \pi_{*}(\Omega_0X)$, the first being equivariant with respect to the subgroup \begin{equation}\label{equ:haut-split}\pi_0\hAut^\spl_*(M,\partial M)\subset \pi_0\hAut_*(M,\partial M)\end{equation}
of homotopy automorphisms of $(M,\partial M)$ that preserve the splitting $s_M$. Note that \eqref{equ:haut-split} is an equality whenever $s_M$ is unique, e.g.\,if $M$ is equivalent to a wedge of \emph{equi}dimensional spheres.

\begin{lem}\label{lem:boundary-restriction-on-mapping-spaces}
For $(M, s_M)$ and $(X, s_X)$ as above, and a basepoint $f\in \Map_*(M,X)$, there is a  diagram of graded vector spaces with vertical isomorphisms in positive degrees
\[
\begin{tikzcd}[column sep=2cm]
\pi_{*}(\Map_*(M,X),f)_\bfQ\arrow[r,"(-)\circ \iota_\partial"]\arrow[d,"\cong"]&\pi_{*}(\Map_*(\partial M,X),f \circ \iota_\partial)_\bfQ\arrow[d,"\cong"]\\
\Der^{f_*}\big(\bfL(H_M),\bfL(H_X)\big)\dar[d,"\cong"]\rar[r,"\ev_\omega"]&s^{-(d-2)}\bfL(H_X)\dar[d,equal]\\
s^{-(d-2)}\bfL(H_X)\otimes H_M\otimes \bfQ^-\arrow[r,"{\left[-,f_*(-)\right]}"]&s^{-(d-2)}\bfL(H_X)
\end{tikzcd}
\]
where $\bfQ^-\coloneq \oH_d(M,\partial M;\bfQ)$ and $\omega\in \bfL(H_M)$ is the element of degree $d-2$ induced by the inclusion $\iota_\partial\colon \partial M\hookrightarrow M$. This diagram enjoys the following properties:
\begin{enumerate}
\item Equivariance with respect to the evident action of $\pi_0\hAut^\spl_*(M,\partial M)$.
\item Naturality in pairs $(X,s_X)$.
\item Commutativity of the upper square and commutativity of the lower square up to a sign.
\end{enumerate}
\end{lem}
\begin{proof}
This is essentially contained in \cite{BerglundMadsen}: The upper square (with the claimed naturality) follows from Theorem 3.6 loc.cit.\ together with the isomorphism $\Der^{(f\circ \iota_\partial)}(\bfL(H_{\partial M}),\bfL(H_X))\cong s^{d-2}\bfL(H_X)$ given by evaluation at the fundamental class of $\partial M\simeq S^{d-1}$. The lower left vertical isomorphism is given by the composition
\[\Der^{f_*}\big(\bfL(H_M),\bfL(H_X)\big)\cong \Hom(H_M,\bfL(H_X))\cong \bfL(H_X)\otimes H_M^\vee\cong s^{-(d-2)}\bfL(H_X)\otimes H_M\otimes \bfQ^-\]
which is equivariant with respect to $\pi_0\hAut^\spl_*(M,\partial M)$. Here the first isomorphism restricts to generators and the second isomorphism is induced by the identification $H_M\cong s^{d-2}H_M^\vee$ via the intersection form (c.f. \cref{sec:symplectic-schur-functors}), which is equivariant with respect to the subgroup $\pi_0\hAut^\spl_*(M,\partial M)^+\subset \pi_0\hAut^\spl_*(M,\partial M)$ of orientation-preserving homotopy equivalences and needs to be twisted by $\bfQ^-$ to be equivariant for the action of all of $\pi_0\hAut^\spl_*(M,\partial M)$. This composition satisfies the claimed naturality, so it only remains to establish commutativity of the lower square up to a sign (see Theorem 3.11 loc.cit.). For $M=X$ and $f=\id$, this follows from Proposition 3.9 loc.cit. since the element $\omega$ in that proposition agrees with $\omega$ as above up to a sign (see Theorem 3.11 (2) loc.cit.). The general case follows from \cite[App.\ B Lem.\ 1]{KrannichConc1}.
\end{proof}

In the case $X=M$ and $f=\id$, the bottom map in the diagram of the previous lemma is surjective since $s^{-(d-2)}\bfL(H_M)$ has no indecomposable elements in positive degrees, so combining \cref{lem:boundary-restriction-on-mapping-spaces} with the long exact sequence induced by the fibration
\vspace{-0.1cm}\[
\hAut_\partial(M)\lra \hAut_*(M)\xlra{(-)\circ \iota_\partial} \Map_*(\partial M,M),
\]
one arrives at a chain of $\pi_0\hAut^\spl(M,\partial M)$-equivariant isomorphisms in positive degrees
\vspace{-0.1cm}\[
\pi_*(\Omega_0\BhAut_\partial(M))_\bfQ\cong \Der_\omega(\bfL(H_M))\cong \ker\big(s^{-(d-2)}\bfL(H_M)\otimes H_M\otimes\bfQ^-\xlra{[-,-]}s^{-(d-2)}H_M\big).
\]
Berglund--Madsen \cite{BerglundMadsen} moreover show that the first isomorphism is compatible with the Lie bracket on both sides (see Proposition 3.7 loc.cit.), and a minor extension of their Proposition 6.6 incorporating $\bfQ^-$ identifies the right-hand kernel canonically with $\cL ie\dl H_M,d\dr\otimes \bfQ^-$ where we write more generally for a graded vector space $A$ \begin{equation}\label{eq:interpretation-cyclic-lie}\cL ie\dl A,d\dr\coloneq\textstyle{ s^{-(d-2)}\bigoplus_{k \geq 2} Lie\dl k\dr  \otimes_{\Sigma_k} A^{\otimes k}}
\end{equation}
with $Lie\dl k\dr $ the arity $k$ part of the cyclic Lie operad. We summarise:

\begin{thm}[Berglund--Madsen]\label{thm:haut-of-wg}
For $M$ as above, there are canonical isomorphisms of graded $\bfQ[\pi_0\hAut^\spl(M,\partial M)]$-modules in positive degrees
\[\pi_{*}(\Omega_0\BhAut_\partial(M))_\bfQ \cong \Der_\omega(\bfL(H_M))\cong \cL ie\dl H_M,d\dr\otimes\bfQ^-,\]
with $\bfQ^-=\oH_d(M,\partial M;\bfQ)$. The first isomorphism respects the Lie algebra structure on both sides.
\end{thm}

\subsubsection{Relative homotopy automorphisms of $V_g$} Setting $M=W_{g,1}$, \cref{thm:haut-of-wg} describes the rational homotopy groups of the base space of the fibration
\begin{equation}\label{equ:fibre-sequence-haut}
\hAut_\partial(V_g)\lra \hAut_{D^{2n}}(V_g,W_{g,1})\lra\hAut_\partial(W_{g,1}).
\end{equation}
induced by restriction. The homotopy groups of base and fibre were determined as part of \cite{KrannichConc1}. The following is a minor extension of Theorem 4.6 loc.cit.\ to include the effect of the reflection involution $\rho\colon V_g\ra V_g$ from \cref{sec:Reflection}. 

\begin{thm}[Krannich]\label{thm:KrannichhAut}
For $n \geq 2$, there is commutative diagram
\[\hspace{-0.2cm}
\begin{tikzcd}[column sep=0.5cm]
\pi_*(\Omega_0\BhAut_{D^{2n}}(V_g,W_{g,1}))_\bfQ\rar\arrow[d,"\cong"]&[-15pt]\pi_*(\Omega_0\BhAut_{\partial}(W_{g,1}))_\bfQ\rar\arrow[d,"\cong"]&[-8pt]\pi_{*-1}(\Omega_0\BhAut_{\partial}(V_{g}))_\bfQ\arrow[d,"\cong"]\\
\ker\big(\cL ie\dl H_W,2n \dr\rar \to\cL ie\dl H_V,2n\dr\big)\otimes\bfQ^-\rar&\cL ie\dl H_W,2n\dr\otimes\bfQ^-\arrow[r,two heads]&\cL ie\dl H_V,2n\dr\otimes \bfQ^-
\end{tikzcd}
\]
of graded $\bfQ[\pi_0\hAut_{D^{2n}}(V_g,W_{g,1})\rtimes \langle \rho\rangle]$-modules whose vertical maps are isomorphisms in positive degrees. The upper row is induced by \eqref{equ:fibre-sequence-haut} and the bottom by the projection $H_{W}\ra H_{V}$. The action on the lower row is through $H_{V}$, $H_{W}$, and $\bfQ^-=\oH_{2n}(W_{g,1},\partial W_{g,1})$.
\end{thm}
\begin{proof}
Considered as a diagram of  $\bfQ[\pi_0\hAut_{D^{2n}}(V_g,W_{g,1})]$-modules, this follows from \cite[Theorem 4.6]{KrannichConc1} together with the discussion above \cref{thm:haut-of-wg}, so we are left to argue that the diagram is $\langle\rho\rangle$-equivariant. This is clear for the top and bottom row, and part of \cref{thm:haut-of-wg} for the middle column (as $\rho\in\pi_0\hAut_*(W_{g,1},\partial W_{g,1})$, and this group agrees with its $(-)^\spl$-subgroup as $W_{g,1}$ is equivalent to a wedge of $n$-spheres), so it also follows for the right and left column since the bottom row is a short exact sequence (and thus the top row as well).
\end{proof}

\begin{rem}\label{rem:LieChar}
For later calculations we recall (for instance from \cite[p.\,387–388]{Thrall}) the decomposition of first few Lie representations into irreducible $\bfQ[\Sigma_k]$-modules:
\[Lie(1) = (1) \quad Lie(2) = (1^2) \quad Lie(3) = (2,1) \quad Lie(4) = (3,1) \oplus (2,1^2).\]
Using $Res^{\Sigma_{n}}_{\Sigma_{n-1}} Lie\dl n\dr  = Lie(n-1)$ (and by direct calculation in the first case) one concludes 
\[Lie(\!(2)\!) = (2) \quad Lie\dl 3\dr  = (1^3) \quad Lie\dl 4\dr  = (2^2) \quad Lie\dl 5\dr  = (3,1^2).\]
In particular, using the norm to identify coinvariant with invariants, we have \[ \cL ie\dl H_{W},2n\dr=s^{-(2n-2)}\big(S_{1^3}(H_W)\oplus S_{2^2}(H_W)\oplus S_{3,1^2}(H_W)\big)\]
in degrees $*<4n-4$, and similarly for $\cL ie\dl H_{V},2n\dr$ with $H_W$ replaced by $H_V$.
\end{rem}

\subsection{Some characteristic classes}\label{sec:kappas}
Specialising Theorems \ref{thm:haut-of-wg} and \ref{thm:KrannichhAut} to degree $n-1$ and using \cref{rem:LieChar} gives isomorphisms for $n\ge2$ of the form
\[
\pi_n(\BhAut_\partial(W_{g,1}))_\bfQ \cong  S_{1^3}(H_W)_{3(n-1)} \otimes \bfQ^-\ \ \ \text{and}\ \ \ 
\pi_{n-1}(\BhAut_\partial(V_{g}))_\bfQ \cong S_{1^3}(H_V)_{3(n-1)} \otimes \bfQ^-
\]
but for some purposes we will need a more geometric description of these identifications. To this effect, we will interpret them as certain characteristic classes of $W_{g,1}$- or $V_g$-fibrations.

\subsubsection{The case of $W_{g,1}$}\label{sec:kappa-w}
This case has already been discussed in detail in \cite{KR-WDisc}, see Section 5.3.3. We briefly recall the essentials. A class $\xi \in \pi_n(\BhAut_\partial(W_{g,1}))$ classifies a relative fibration
\[\pi\colon (E, S^n \times \partial W_{g,1}) \lra S^n.\]
with an identification of the fibre over the basepoint $*\in S^n$ with $(W_{g,1}, \partial W_{g,1})$. Identifying $\partial W_{g,1}=S^{2n-1}$ and gluing in $S^n \times D^{2n}$ results in an oriented fibration $\bar{\pi}\colon \bar{E} \to S^n$ with closed fibre $W_g$, with $\bar{E}$ a Poincar{\'e} duality space, which comes with a section $s \colon S^n \to \bar{E}$ given by the centres of the discs we glued in. The Serre spectral sequence for $\bar{\pi}$ gives an exact sequence
\[0 \lra \oH^n(S^n;\bfQ) \overset{\bar{\pi}^*}\lra \oH^n(\bar{E};\bfQ) \lra \oH^n(W_{g};\bfQ) \lra 0\]
and the section $s$ provides a preferred splitting, inducing a map $\iota \colon \bar{H}_W^\vee = \oH^n(W_{g};\bfQ) \ra \oH^n(\bar{E};\bfQ)$. This can be used to define a map
\[\map{\kappa_W}{\pi_n(\BhAut_\partial(W_{g,1}))}{\mathrm{Hom}((\bar{H}_W^\vee)^{\otimes 3}, \bfQ) = \bar{H}_W^{\otimes 3}}{\xi}{\left(v_1 \otimes v_2 \otimes v_3 \mapsto \int_{\bar{E}} \iota(v_1) \cdot \iota(v_2) \cdot \iota(v_3)\right)}\]
As the $\iota(v_i)$'s have degree $n$, on permuting the $v_i$ the number $\kappa_W(\xi)(v_1\otimes v_2\otimes v_3)$ transforms as the $n$th tensor power of the sign representation of $\Sigma_3$, and hence taking the grading of $H_W$ into account the image of $\kappa_W$ lies in $[(1^3) \otimes {H}_W^{\otimes 3}]^{\Sigma_3}_{3(n-1)}= S_{1^3}(H_W)_{3(n-1)}$. Furthermore, the $\pi_0\hAut_\partial(W_{g,1})$-action on $\pi_n(\BhAut_\partial(W_{g,1}))$ is given by changing the identification of the fibre over the basepoint, so $\kappa_W$ is equivariant. Finally, the reflection $\rho$ has the effect of changing the identification of the fibre over the basepoint \emph{and} the orientation of $\bar{E}$, so $\kappa_W$ is $\rho$-equivariant when twisting the target by $\bfQ^-$. Together with Proposition 5.13 (and the subsequent discussion) of \cite{KR-WDisc}, this gives the following.
\begin{lem}[Kupers--Randal-Williams]\label{lem:kappaIso}
The map of $\bfQ[\pi_0\hAut_{\partial}(W_{g,1})\rtimes \langle \rho\rangle]$-modules
\[\kappa_W \colon \pi_n(\BhAut_\partial(W_{g,1}))_\bfQ \lra S_{1^3}(H_W)_{3(n-1)}\otimes\bfQ^-\]
is an isomorphism.
\end{lem}

\begin{rem}It will not be necessary for our purposes, but it follows from the cited proposition and the subsequent discussion that this isomorphism agrees up to a scalar with the isomorphism resulting from \cref{thm:haut-of-wg} and \cref{rem:LieChar}.
\end{rem}

\subsubsection{The case of $V_g$} An element $\zeta \in \pi_{n-1}(\BhAut_\partial(V_g))$ classifies a relative fibration
\[(E', S^{n-1} \times  W_{g}) \overset{\pi'}\lra S^{n-1}\]
with an identification of the fibre over the basepoint $*\in S^n$ with $(V_{g},  W_{g})$. Gluing to this the trivial $\tilde{V}_g$-bundle $S^{n-1} \times \tilde{V}_g \to S^{n-1}$ using the $\partial V_g=W_g$ yields an oriented fibration $\bar{\pi}' \colon \bar{E}' \to S^{n-1}$ with closed fibre $U_g \coloneq V_g\cup_{\partial V_g}\tilde{V}_g\cong \sharp^{g} S^n \times S^{n+1}$, and with section $s \colon S^{n-1} \to \bar{E}'$ given by the centre $*\in D^{2n}\subset W_g$. Here $\tilde{(-)}$ indicates reversal of orientation. The Serre spectral sequence for $\bar{\pi}'$ gives an isomorphism 
\[\oH^n(\bar{E}';\bfQ) \xlra{\cong} \oH^n(U_{g};\bfQ)\cong \oH^n(V_{g};\bfQ)=\bar{H}_V^\vee\]
with inverse $\iota' \colon \bar{H}_V^\vee\to \oH^n(\bar{E}';\bfQ)$. 
This can be used to define a map
\[\map{\kappa_V}{\pi_{n-1}(\BhAut_\partial(V_{g}))}{\mathrm{Hom}((\bar{H}_V^\vee)^{\otimes 3}, \bfQ) = \bar{H}_V^{\otimes 3}}{\zeta}{\left(v_1 \otimes v_2 \otimes v_3 \mapsto \int_{\bar{E}'} \iota'(v_1) \cdot \iota'(v_2) \cdot \iota'(v_3)\right),}\]
and the same degree considerations as in the $W_g$-case show that the image of $\kappa_V$ lands in the subspace $[(1^3) \otimes {H}_V^{\otimes 3}]^{\Sigma_3}_{3(n-1)}=S_{1^3}(H_V)_{3(n-1)}$.

\begin{lem}\label{lem:kappasAreCompatible}
The square of $\bfQ$-vector spaces
\[
\begin{tikzcd}
\pi_n(\BhAut_\partial(W_{g,1}))_\bfQ \rar{\kappa_W} \arrow[d,"\partial",swap]& S_{1^3}(H_W)_{3(n-1)} \otimes \bfQ^- \dar\\
\pi_{n-1}(\BhAut_\partial(V_{g}))_\bfQ \rar{\kappa_V} &  S_{1^3}(H_V)_{3(n-1)}\otimes \bfQ^-
\end{tikzcd}
\]
commutes. Here the left map is induced by \eqref{equ:fibre-sequence-haut} and the right map by the projection $H_W\ra H_V$.
\end{lem}
\begin{proof}
If $\zeta = \partial(\xi)$ and $\xi$ corresponds to a clutching map $\bar{\xi} \colon S^{n-1} \to \hAut_\partial(W_{g,1}) \to \hAut(W_{g})$ thought of as a homotopy equivalence $\bar{\xi}\colon S^{n-1} \times W_{g}\ra S^{n-1} \times W_{g}$ by taking adjoints,  
then the spaces $\bar{E}$ and $\bar{E}'$ in the description of $\kappa_W$ and $\kappa_B$ are given by the twisted doubles
\[\bar{E} = (D^n \times W_g) \cup_{\bar{\xi}} (D^n \times \tilde{W}_g)\quad\text{and}\quad\quad \bar{E}' = (S^{n-1} \times V_g) \cup_{\bar{\xi}} (S^{n-1} \times \tilde{V}_g).\]
In general, for an oriented Poincar{\'e} complex $M$ and an orientation-preserving homotopy equivalence $\phi\colon \partial M\ra \partial M$, the twisted double $M\cup_\phi\tilde{M}$ is oriented Poincar{\'e}-cobordant to the mapping torus $T(\phi)=([-2,2]\times\partial M)/\sim$, via the bordism $K$ given by the (equivalent) pushouts
\[
\begin{tikzcd}
\{1\}\times (\left[-1,1\right]\times \partial M)\arrow[d,hook]\arrow[d, phantom, shift left=2cm, "\square_{1}"]\arrow[r,hook]&\{1\}\times (\left[-1,1\right]\times M)\dar&[-0.5cm]\{1\}\times (M\sqcup \tilde{M})\arrow[d,hook]\arrow[d, phantom, shift left=2cm, "\square_{2}"]\arrow[r,hook]&\{1\}\times (M\times I)\dar\\
\left[0,1\right]\times T(\phi)\rar&K&\left[0,1\right]\times \big(M\cup_{\phi}\tilde{M}\big)\rar& K
\end{tikzcd}
\]
Applied to ${\bar{\xi}}$ this gives cobordisms of Poincar{\'e} complexes
$K_1 \colon \bar{E} \leadsto T(\bar{\xi})$ and $K_2 \colon T({\bar{\xi}}) \leadsto \bar{E}'$ (it will be helpful to think of $K_1$ and $K_2$ as instances of $\square_{1}$ and $\square_{2}$, respectively). The maps
\[\bar{H}_V^\vee\xlra{\iota'}\oH^n(\bar{E}';\bfQ)\ \substack{\lra\\ \lra}\ \oH^n(S^{n-1} \times V_g;\bfQ)\]
given by restriction to the two copies $S^{n-1} \times V_g$ and $S^{n-1} \times \tilde{V}_g$ inside $\bar{E}'$ are equal (as the map $\bar{\xi}$ is the identity on cohomology, and the inclusion $S^{n-1} \times W_g \subset S^{n-1} \times V_g$ induces an injection on cohomology), so by Mayer--Vietoris for $\square_{2}$ we may choose a lift $\bar{H}_V^\vee \to \oH^n(K_2;\bfQ)$ which by restriction gives a map $\bar{H}_V^\vee \ra \oH^n(T({\bar{\xi}});\bfQ)$. By Mayer--Vietoris for $\square_{1}$, the group $\oH^{n+1}(K_1, T({\bar{\xi}});\bfQ)=\oH^{n+1}(D^n\times W_g,S^{n-1}\times W_g;\bfQ)$ vanishes, so the lift lifts further to a map $\bar{H}_V^\vee \to \oH^n(K_1;\bfQ)$, which may be restricted to a map  $\bar{H}_V^\vee \to \oH^n(\bar{E};\bfQ)$. The latter agrees with the composition $\bar{H}_V^\vee \to \bar{H}_W^\vee \to \oH^n(\bar{E};\bfQ)$ of the dual of the projection with $\iota$ (this may be checked by restriction to $S^{n-1} \times W_g$). Writing $K = K_2 \circ K_1 \colon \bar{E} \leadsto \bar{E}'$, by Mayer--Vietoris we therefore have a map $\bar{H}_V^\vee \to H^n(K;\bfQ)$ whose restriction to $\bar{E}$ is the composition $\bar{H}_V^\vee \to \bar{H}_W^\vee \to \oH^n(\bar{E};\bfQ)$ of the dual of the projection with $\iota$, and whose restriction to $\bar{E}'$ is $\iota'$. For $v_1,v_2, v_3 \in \bar{H}_V^\vee$ it then follows from Stokes' theorem that
$\textstyle{\int_{\bar{E}} \iota(v_1) \cdot \iota(v_2) \cdot \iota(v_3) = \int_{\bar{E}'} \iota'(v_1) \cdot \iota'(v_2) \cdot \iota'(v_3)}$
as required.
\end{proof}

\begin{cor}\label{cor:lowest-htp-grp-haut-v}
The map of $\bfQ$-vector spaces
\[\kappa_V \colon \pi_{n-1}(\BhAut_\partial(V_{g}))_\bfQ \lra S_{1^3}(H_V)_{3(n-1)}\otimes\bfQ^-\]
is an isomorphism of $\bfQ[\pi_0\hAut_{D^{2n}}(V_g,W_{g,1})\rtimes \langle \rho\rangle]$-modules
\end{cor}
\begin{proof}
As a result of \cref{thm:KrannichhAut}, the left vertical map in the diagram of \cref{lem:kappasAreCompatible} is surjective and equivariant with respect to the canonical $\pi_0\hAut_{D^{2n}}(V_g,W_{g,1})\rtimes \langle \rho\rangle$-action, and the same holds evidently for the right map. As $\kappa_W$ is $(\pi_0\hAut_{\partial}(W_{g,1})\rtimes \langle \rho\rangle)$-equivariant and an isomorphism by \cref{lem:kappaIso}, it follows that $\kappa_V$ is $(\pi_0\hAut_{D^{2n}}(V_g,W_{g,1})\rtimes \langle \rho\rangle)$-equivariant and surjective, so it suffices to argue that source and the target of $\kappa_V$ are abstractly isomorphic as $\bfQ$-vector spaces. By \cref{thm:KrannichhAut} and \cref{rem:LieChar}, the source is given by the degree $n-1$ part of $s^{-(2n-2)}\big(S_{1^3}(H_V)\oplus S_{2^2}(H_V)\oplus S_{3,1^2}(H_V))$. As $H_V$ is concentrated in degree $(n-1)$, this is precisely $S_{1^3}(H_V)_{3(n-1)}$, which is also the target of $\kappa_V$, so the claim follows.
\end{proof}

\begin{rem}
As the target of $\kappa_V$ is zero or irreducible as a $\GL(\bar{H}^\bfZ_V)$-representation (see \cref{sec:some-rep-theory} below) and $\pi_0\hAut_{D^{2n}}(V_g,W_{g,1})$ surjects onto $\GL(\bar{H}^\bfZ_V)$ as a result of \cref{lem:mcg-is-semidirect-product}, the target is irreducible or trivial as a $\pi_0\hAut_{D^{2n}}(V_g,W_{g,1})$-representation, so it follows that $\kappa_V$ must be the same, up to a scalar, as the isomorphism resulting from \cref{thm:KrannichhAut}. 
\end{rem}

\cref{cor:lowest-htp-grp-haut-v} allows us to understand the functoriality of the left-hand side with respect to codimension $0$ embeddings:

\begin{lem}\label{lem:lowest-htp-grp-haut-v-naturality}
For $g,h\ge0$ and an orientation-preserving embedding $\phi \colon V_g \hookrightarrow \interior(V_h)$ the square
\[
\begin{tikzcd}
\pi_{n-1}(\BhAut_\partial(V_{g})) \arrow[r,"\kappa_V","\cong"'] \arrow["\phi_*",d,swap]& S_{1^3}(H_{V_g})_{3(n-1)} \otimes \bfQ^- \dar{\phi_*}\\
\pi_{n-1}(\BhAut_\partial(V_{h})) \arrow[r,"\kappa_V","\cong"'] &  S_{1^3}(H_{V_h})_{3(n-1)} \otimes \bfQ^-
\end{tikzcd}
\]
commutes. Here the left-hand map is induced by extension of homotopy automorphisms by the identity, and the right-hand map is induced by $\phi_* \colon \oH_n(V_g;\bfQ) \to \oH_n(V_h;\bfQ)$.
\end{lem}
\begin{proof}
Let $C \colon W_g \leadsto W_h$ be the cobordism given by complement of the interior of $\phi(V_g) \subset V_h$, so that $V_h \cong V_g \cup_{W_g} C$. For $\zeta \in \pi_{n-1}(\BhAut_\partial(V_{g}))$ corresponding to a relative fibration
\vspace{-0.1cm}\[(V_{g},  W_{g}) \lra (E', S^{n-1} \times  W_{g}) \overset{\pi'}\lra S^{n-1},\]
the element $\phi_*(\zeta)$ corresponds to the relative fibration
\vspace{-0.1cm}\[(V_{h},  W_{h}) \lra (E'' = E' \cup_{S^{n-1} \times  W_{g}} S^{n-1} \times C, S^{n-1} \times  W_{h}) \overset{\pi''}\lra S^{n-1}.\]
Thus
$\bar{E}'' = (E' \cup_{S^{n-1} \times  W_{g}} S^{n-1} \times C) \cup_{S^{n-1} \times {W_h}} (S^{n-1} \times (\tilde{V}_g \cup_{W_g} \tilde{C})).$
There is a cobordism from $C \cup_{W_h} \tilde{C}$ to $W_g \times [0,1]$ relative the its boundary, which applied fibrewise gives a Poincar{\'e} cobordism $K \colon \bar{E}'' \leadsto \bar{E}'$. It is not hard to see that there is a lift $H_{V_h}^\vee \to \oH^n(K;\bfQ)$ of $\iota'' \colon H_{V_h}^\vee \to \oH^n(\bar{E}'';\bfQ)$, and the composition $H_{V_h}^\vee \to \oH^n(K;\bfQ) \to \oH^n(\bar{E}';\bfQ)$ is easily checked to be equal to $(\iota'\circ\phi^*)\colon H_{V_h}^\vee \ra \oH^n(\bar{E}';\bfQ)$, by restriction to a fibre. The conclusion then follows by Stokes' theorem, as in the proof of Lemma \ref{lem:kappasAreCompatible}.
\end{proof}

\subsubsection{A vanishing result}
At a certain point in Section \ref{sec:universal-covers-embedding-spaces} we will need to know that the map 
$\pi_n(\BEmb^{\cong,\fr}_{\half \partial}(W_{g,1})_\ell)_\bfQ \ra \pi_n(\BhAut_\partial(W_{g,1}))_\bfQ$
induced by the bottom map in \eqref{equ:self-embs-to-haut} is not surjective. To analyse this problem, we consider the composition
\vspace{-0.1cm}
\[\chi \colon \pi_n(\BhAut_\partial(W_{g,1}))_\bfQ \overset{\kappa_W}\lra \bar{H}_W^{\otimes 3} \xra{(x_1, x_2, x_3) \mapsto \lambda(x_1, x_2) x_3} \bar{H}_W,\]
which is nontrivial for $g \geq 2$, as \cref{lem:kappaIso} shows that the image of $\kappa_W$ are the $(-1)^n$-symmetric tensors, and the second map sends the $(-1)^n$-symmetric tensor
\[e_1 \otimes f_1 \otimes e_2 + e_2 \otimes e_1 \otimes f_1 + f_1 \otimes e_2 \otimes e_1 + (-1)^n (f_1 \otimes e_1 \otimes e_2 + e_2 \otimes f_1 \otimes e_1 + e_1 \otimes e_2 \otimes f_1)\]
to $2 e_2 \neq 0$. This implies the second part of the following lemma, whose statement uses \eqref{equ:self-embs-to-haut} and the framing conventions from \cref{sec:standard-framing}.
\begin{lem}\label{lem:ChiVanishesOnEmbFr}
For $n\ge3$, the composition
\[\pi_n(\BEmb^{\cong,\fr}_{\half \partial}(W_{g,1})_\ell)_\bfQ \lra \pi_n(\BhAut_\partial(W_{g,1}))_\bfQ \overset{\chi}\lra \bar{H}_W\]
is zero, so the first map is not surjective for $g \geq 2$.
\end{lem}
\begin{proof}
As explained in Section \ref{sec:weiss-fs}, the map $\BEmb^{\cong}_{\half \partial}(W_{g,1}) \to \BhAut_\partial(W_{g,1})$ factors over $\BHomeo_\partial(W_{g,1})$. Thus the map from the framed self-embedding space may be factored up to homotopy over
$\mathrm{Microbun}_{\half \partial}(\tau^{\mathrm{top}}_{W_{g,1}},\varepsilon^{2n};\ell_{\half\partial})\dslash \Homeo_\partial(W_{g,1}),$
the homotopy orbits of the space of topological framings, i.e.\ microbundle maps from the tangent microbundle of $W_{g,1}$ to the trivial $2n$-dimensional microbundle (all constructed as simplicial sets and with appropriate boundary conditions). This space classifies topological $W_{g,1}$-bundles with trivialised boundary and a framing of the vertical tangent microbundle, standard on half the boundary. Now suppose that $\pi \colon (E, S^n \times \partial W_{g,1}) \to S^n$ is such a topological bundle, with framing $\zeta \colon \tau^{\mathrm{top}}_\pi  \to \varepsilon^{2n}$ of the vertical tangent microbundle. Neglecting for now the framing, we form $\bar{\pi} \colon \bar{E} = E \cup_{S^n \times \partial W_{g,1}} (S^n \times D^{2n}) \to S^n$. Alternatively to the description in \cref{sec:kappa-w}, the map $\kappa_W$ can be described in terms of a twisted Miller--Morita--Mumford classes in the sense of \cite[Section 3]{KR-WTorelli}: it corresponds to evaluating homotopy classes against the twisted cohomology class $\kappa_{\epsilon^3}(\bar{\pi}) \coloneq \bar{\pi}_!(\epsilon^3) = \int_{\bar{\pi}} \epsilon^3$ from that paper. This is because over the simply-connected space $S^n$ any local coefficient system is trivial and the cohomology class $\epsilon \in H^n(\bar{E} ; \bar{\pi}^*\mathcal{H}_n(W_g ; \bfQ))$ from that paper corresponds to the map $\iota : H_n(W_g;\bfQ)^\vee = H^n(W_g;\bfQ) \to H^n(\bar{E};\bfQ)$ we have used to define $\kappa_W$. Our construction of $\chi$ is what would be denoted $\lambda_{1,2}(\kappa_{\epsilon^3})$ in that paper, and using the Contraction Formula of \cite[Proposition 3.10]{KR-WTorelli} (which is stated for smooth bundles but its proof goes through without change for topological bundles) we see that we have
\begin{align*}
\chi(\pi)(-) & = \textstyle{\int_{S^n} \lambda_{1,2}\left(\int_{\bar{\pi}} \epsilon^3\right)} & \text{by definition}\\
  &\textstyle{= \int_{S^n} \int_{\bar{\pi}} e(\tau^\mathrm{top}_{\bar{\pi}}) \cdot\epsilon } & \text{by the Contraction Formula}\\
  &\textstyle{= \int_{\bar{E}} e(\tau^\mathrm{top}_{\bar{\pi}}) \cdot \iota(-) }& \text{by definition of $\epsilon$ and Fubini.}
\end{align*}
Because of the framing of the vertical tangent microbundle the restriction of $e(\tau^\mathrm{top}_{\bar{\pi}})$ to $E \subset \bar{E}$ is trivial, so this class is in the image of the composition
\[\oH^0(S^n ; \bfQ) \cong \oH^{2n}(S^n \times D^{2n}, S^n \times \partial D^{2n};\bfQ) \cong \oH^{2n}(\bar{E}, E ; \bfQ) \lra \oH^{2n}(\bar{E};\bfQ).\]
That is, the class  $e(\tau^\mathrm{top}_{\bar{\pi}})$ is a multiple of the Poincar{\'e} dual $u \in \oH^{2n}(\bar{E};\bfQ)$ of the section $s \colon S^n \to \bar{E}$, and by evaluating on a fibre we see that the multiple is the Euler characteristic $(2 + (-1)^n 2g)$ of $W_g$. It follows that $\chi(\pi)(-) = (2 + (-1)^n 2g) \int_{S^n}s^*(\iota(-))$, but this vanishes as $s^* \circ \iota$ is trivial by the definition of the map $\iota$ in \cref{sec:kappa-w}.
\end{proof}

\section{The fibre to homotopy automorphisms}\label{sec:Disj}
To compute the rational cohomology of the framed self-embedding spaces featuring in \cref{cor:reduction-to-self-embeddings}, we first consider the stably framed self-embedding spaces 
\[
\BEmb^{\sfr,\cong}_{\half\partial}(V_{g})_{\ell}\quad \text{and} \quad
\BEmb^{\sfr,\cong}_{\half\partial}(W_{g,1})_{\ell}.
\]
By forgetting stable framings and using \eqref{equ:self-embs-to-haut}, they admit maps (see \cref{sec:tangential-structures} for the notation)
\begin{equation}\label{sec:recall-maps-to-haut}
\BEmb^{\sfr,\cong}_{\half\partial}(W_{g,1})_\ell\ra \BhAut_\partial^{\cong,\ell}(W_{g,1}) \quad\text{and}\quad
\BEmb^{\sfr,\cong}_{\half\partial}(V_{g})_\ell\ra \BhAut^{\cong,\ell}_\partial(V_g)
\end{equation}
whose homotopy fibres we denote by
\begin{equation}\label{equ:hofibs-self-emb-to-haut}
\hAut^{\cong,\ell}_\partial(V_g)/\Emb^{\sfr,\cong}_{\nicefrac{1}{2}\partial}(V_g)_\ell\quad\text{and}\quad \hAut^{\cong,\ell}_\partial(W_{g,1})/\Emb^{\sfr,\cong}_{\nicefrac{1}{2}\partial}(W_{g,1})_\ell.
\end{equation}
These fibres are connected, since the maps \eqref{sec:recall-maps-to-haut} are surjective on fundamental groups by construction. In this section, we compute their rational homotopy groups in a range.

\begin{thm}\label{thm:fibre-to-haut}For $n\ge3$, the fundamental groups of the homotopy fibres of the maps \eqref{sec:recall-maps-to-haut} are finite and their higher rational homotopy groups are in the range $1<*<3n-5$ given by 
\begin{align*}
\pi_*\big(\hAut^{\cong,\ell}_\partial(V_g)/\Emb^{\sfr,\cong}_{\nicefrac{1}{2}\partial}(V_g)_\ell\big)_\bfQ&\cong \begin{cases}
\bfQ^{\frac{g(g-1)}{2}}[2n-2]&n\text{ even}\\
\bfQ^{\frac{g(g+1)}{2}}[2n-2]&n\text{ odd}
\end{cases}\\
\pi_*\big(\hAut^{\cong,\ell}_\partial(W_{g,1})/\Emb^{\sfr,\cong}_{\nicefrac{1}{2}\partial}(W_{g,1})_\ell\big)_\bfQ&=0.
\end{align*}
\end{thm}

\begin{conv}We assume $n\ge3$ throughout this section.
\end{conv}

\subsection{Reduction to a disjunction problem and separating the handles}
As a first step towards \cref{thm:fibre-to-haut}, we will express the homotopy fibres in consideration rationally in terms of the total homotopy fibres of the commutative squares
\begin{equation}\label{equ:diff-blockdiff-squares}
\begin{tikzcd}
\BDiff_\partial(D^{2n+1})\arrow[d, phantom, shift left=1.4cm, "\square_{V_g}"]\dar\rar & \BDiff_\partial(V_g)\dar\\
\BlockBDiff_\partial(D^{2n+1})\rar& \BlockBDiff_\partial(V_g)
\end{tikzcd}\quad\text{and}\quad
\begin{tikzcd}
\BDiff_\partial(D^{2n})\dar\arrow[d, phantom, shift left=1.4cm, "\square_{W_g}"] \rar& \BDiff_\partial(W_{g,1})\dar\\
\BlockBDiff_\partial(D^{2n})\rar& \BlockBDiff_\partial(W_{g,1})
\end{tikzcd}
\end{equation}
induced by extending (block)-diffeomorphisms of an embedded disc by the identity.

\begin{prop}\label{prop:disjunction-problem}For $n\ge3$, the total homotopy fibres of $\square_{V_g}$ and $\square_{W_g}$ are nilpotent, the homotopy fibres \eqref{equ:hofibs-self-emb-to-haut} have finite fundamental groups, and there are rational equivalences
\[\begin{gathered}
\tohofib\big(\square_{V_g} \big)\lra \Omega_0\left(\hAut^{\cong,\ell}_\partial(V_{g})/\Emb^{\sfr,\cong}_{\nicefrac{1}{2}\partial}(V_{g})_\ell\right)\text{\ \ and}
\\
\tohofib \big(\square_{W_g} \big)\lra \Omega_0\left(\hAut^{\cong,\ell}_\partial(W_{g,1})/\Emb^{\sfr,\cong}_{\nicefrac{1}{2}\partial}(W_{g,1})_\ell\right).
\end{gathered}
\]
\end{prop}

\begin{proof}
The proofs for $V_g$ and $W_{g,1}$ are essentially identical; we shall explain the argument for $V_g$. It follows from work of Cerf \cite[Thèoreme 0]{Cerf} that the map $\BDiff_\partial(V_g)\ra \BlockBDiff_\partial(V_g)$ is $2$-connected for all $g\ge0$, so the vertical homotopy fibres of the square $\square_{V_g}$ are $1$-connected. The total homotopy fibre of this square thus arises as the homotopy fibre of a map between $1$-connected spaces and is therefore nilpotent (in particular connected). To prove the second claim, we consider the fibre sequence $\BlockDiff_\partial(V_g)\ra \BlockDiff_{\half\partial}(V_g)\ra \BlockDiff_\partial(D^{2n})$ induced by restricting block diffeomorphisms to the moving part of the boundary. The induced map $\Omega \BlockDiff_\partial(D^{2n})\ra \BlockDiff_\partial(V_g)$ factors as
$\Omega \BlockDiff_\partial(D^{2n})\ra \BlockDiff_\partial(D^{2n+1}) \ra \BlockDiff_\partial(V_g)$ where the first map results from identifying $D^{2n}\times [0,1]$ with $D^{2n+1}$ and the second one is given by extending diffeomorphisms of a disc $D^{2n+1}\subset \interior(V_g)$ by the identity. As the first map in this composition is an equivalence, the sequence 
\[\BlockDiff_\partial(D^{2n+1}) \lra \BlockDiff_\partial(V_g)\lra \BlockDiff_{\half\partial}(V_g)\] 
is a fibre sequence. Restricting the rightmost space to the components $\BlockDiff^\cong_{\half\partial}(V_g)$ in the image of the right map and taking homotopy orbits of the action on the space of stable framings (see \cref{sec:tangential-structures}) yields the rightmost column in the commutative diagram
\begin{equation}\label{equ:blockdiff-diff-big-diagram}
\begin{tikzcd}[column sep=0.5cm]
\BlockDiff_\partial(D^{2n+1})/\Diff_\partial(D^{2n+1})\rar\dar&\BDiff^{\sfr}_\partial(D^{2n+1})\dar\rar&\BlockBDiff^{\sfr}_\partial(D^{2n+1})\dar\\
\BlockDiff_\partial(V_g)/\Diff_\partial(V_g)\rar&\BDiff^{\sfr}_\partial(V_g)\dar\rar&\BlockBDiff^{\sfr}_\partial(V_g)\dar\\
&\BEmb^{\sfr,\cong}_{\nicefrac{1}{2}\partial}(V_{g})&\BlockBDiff^{\sfr,\cong}_{\nicefrac{1}{2}\partial}(V_g)
\end{tikzcd}
\end{equation}
whose rows and columns are fibration sequences; the identification of the homotopy fibres of the rows follows from \eqref{equ:comparison-tangential}, the middle column is the case $\Xi=\sfr$ of the first sequence in \eqref{equ:weiss-fs-Vg}. Extending the diagram to the right by the fibre sequence $\BhAut_\partial(D^{2n+1})\ra\BhAut_{\partial}(V_g)\ra \BhAut_{\half \partial}(V_g,D^{2n})$,
taking vertical homotopy fibres in \eqref{equ:blockdiff-diff-big-diagram} at the base points induced by $\ell$ and using contractibility of $\hAut_\partial(D^{2n+1})$, we obtain up to canonical equivalences a diagram of horizontal fibre sequence
\[
\begin{tikzcd}
\tohofib\big(\square_{V_g} \big)\rar\dar&\Omega \BEmb^{\sfr,\cong}_{\nicefrac{1}{2}\partial}(V_{g})_\ell\rar\dar&\Omega\BlockBDiff^{\sfr,\cong}_{\nicefrac{1}{2}\partial}(V_g)_\ell\dar\\
*\rar&\Omega\BhAut^{\cong,\ell}_{\half \partial}(V_g,D^{2n})\arrow[r,equal]&\Omega\BhAut^{\cong,\ell}_{\half \partial}(V_g,D^{2n}).
\end{tikzcd}
\]
Taking vertical homotopy fibres, the claim would follow by showing that the homotopy fibre of the rightmost vertical map has finitely many components and that these are rationally trivial. As the triad $(V_g;\partial V_g\backslash \interior(D^{2n}),D^{2n})$ satisfies the $\pi$-$\pi$-condition, this holds by \cref{thm:blockdiff-identification}.
\end{proof}

Motivated by \cref{prop:disjunction-problem}, we now focus on analysing the total homotopy fibres of the squares $\square_{V_{g}}$ and $\square_{W_g}$. It will be convenient to factor $\square_{V_{g}}$ as
\begin{equation}\label{equ:factorisationVg}
\begin{tikzcd}
\BDiff_\partial(D^{2n+1})\dar\arrow[d, phantom, shift left=1.4cm, "\square_{V_{g-1}}"]\rar & \BDiff_\partial(V_{g-1})\dar\rar\arrow[d, phantom, shift left=1.4cm, "\square_{V_{g-1,g}}"] &\BDiff_\partial(V_{g})\dar\\
\BlockBDiff_\partial(D^{2n+1})\rar& \BlockBDiff_\partial(V_{g-1})\rar & \BlockBDiff_\partial(V_{g})
\end{tikzcd}
\end{equation}
by using the inclusion $V_{g-1}\subset V_g$ and extending diffeomorphisms by the identity. Moreover, viewing $W_{g,1}$ as $W_{g-1,1} \natural (D^{n}\times S^n\cup_{D^n\times D^n}S^n\times D^n)$, we factor $\square_{W_g}$ as
\begin{equation}\label{equ:factorisationWg}
\begin{tikzcd}[column sep=0.4cm]
\BDiff_\partial(D^{2n})\dar\rar\arrow[d, phantom, shift left=1.4cm, "\square_{W_{g-1}}"]& \BDiff_\partial(W_{g-1,1})\dar\rar \arrow[d, phantom, shift left=1.7cm, "\square_{W_{g-1,g-1/2}}"] &\BDiff_\partial(W_{g-1,1}\natural D^n\times S^n)\dar\rar\arrow[d, phantom, shift left=2cm, "\square_{W_{g-1/2,g}}"] &\BDiff_\partial(W_{g,1})\dar\\
\BlockBDiff_\partial(D^{2n})\rar& \BlockBDiff_\partial(W_{g-1,1})\rar& \BlockBDiff_\partial(W_{g,1}\natural D^n\times S^n)\rar&\BlockBDiff_\partial(W_{g,1}).
\end{tikzcd}
\end{equation}
This is advantageous since the manifolds featuring in the right columns of $\square_{V_{g-1,g}}$, $\square_{W_{g-1, g-1/2}}$, and $\square_{W_{g-1/2,g}}$ 
are obtained from those in the respective left columns by attaching a single handle. As we shall see in the next subsections, the total homotopy fibres of squares of this form can be expressed in terms of (block) embedding spaces of discs.

\subsection{Splitting the handles}\label{section:splitting-handles}This part of the argument is not particular to the manifolds $V_g$ and $W_{g,1}$, so we shall phrase it more generally.

\begin{notation}Given a manifold $M$ and a system submanifolds $K\subset L \subset M$, we denote by $\Emb_{K}(L,M)$ and $\BlockEmb_{K}(L,M)$ the space of (block) embeddings of $L$ in $M$ which agree with the inclusion on $K$ and on an open neighbourhood of $K\cap \partial M$ (see \cite[p.\,122]{BLR} for a definition of spaces of block embeddings). If $K=\partial L$, we abbreviate $\partial=\partial L$. We also write \[\Emb^{(\sim)}_K(L,M)\coloneq \hofib_{L\subset M}\big(\Emb_{K}(L,M)\lra \BlockEmb_{K}(L,M)\big)\] for the homotopy fibre at the inclusion of the canonical comparison map.
\end{notation}
Let $N$ be a compact $d$-manifold, possibly with boundary, and $M=N\cup \{(d-k)\text{-handle}\}$ a handle attachment of index $d-k\ge 3$. Smoothing corners, $N$ is diffeomorphic to the complement in $M$ of a tubular neighborhood $D^k\times D^{n-k}\subset M$ of the cocore $D^k\subset M$, so the isotopy extension theorem (see \cite[p.\,129]{BLR} for the block-case) provides a map of fibre sequences
\[
\begin{tikzcd}
\Emb_{\partial D^k\times D^{d-k}}(D^{k}\times D^{d-k},M)_N\dar\rar&\BDiff_{\partial}(N)\rar\dar&\BDiff_\partial(M)\dar\\
\BlockEmb_{\partial D^k\times D^{d-k}}(D^{k}\times D^{d-k},M)_N\rar&\BlockBDiff_{\partial}(N)\rar&\BlockBDiff_\partial(M),
\end{tikzcd}
\]
where $(-)_N$ indicates that we pass to the appropriate collection of components, so the total homotopy fibre of right-hand square is equivalent to \[\hofib\big(\Emb_{\partial D^k\times D^{d-k}}(D^{k}\times D^{d-k},M)_N\ra \BlockEmb_{\partial D^k\times D^{d-k}}(D^{k}\times D^{d-k},M)_N\big).\]
Inclusion induces a map of this homotopy fibre to $\Emb^{(\sim)}_{\partial D^k\times D^{n-k}}(D^{k}\times D^{d-k},M)$ which is an equivalence since the latter is connected (this uses the assumption $d-k\ge3$, c.f.\ \cite[Remark 2.7]{GoodwillieKlein}). Restricting embeddings to the cocore $D^k$ induces an equivalence 
\begin{equation}\label{equ:no-normal-data}
\Emb^{(\sim)}_{\partial D^k\times D^{d-k}}(D^{k}\times D^{d-k},M)\xlra{\simeq}\Emb^{(\sim)}_{\partial}(D^{k},M),
\end{equation}
since the homotopy fibres of the individual maps between the spaces of embeddings and between the spaces of block embeddings are, by taking differentials, both equivalent to the space of automorphisms of the normal bundle of the cocore $D^k\subset M$ which cover the identity and are standard near the boundary.

Returning to the factorisation \eqref{equ:factorisationVg}, we specialise to $N=V_{g-1}$ and $M=V_{g-1}\cup\{n\text{-handle}\}\cong V_g$ to derive an equivalence
\begin{equation}\label{equ:tohofib-embeddings-Vg}
\tohofib\big(\square_{V_{g-1,g}}\big)\simeq\Emb^{(\sim)}_\partial(D^{n+1},V_g),
\end{equation}
where $D^{n+1}\times\{*\}\subset V_{g-1}\natural (D^{n+1}\times S^n)=V_g$ is the cocore of the ``last'' handle. Similarly, we obtain 
\begin{align}\label{equ:tohofib-embeddings-Wg}
\begin{split}
\tohofib\big(\square_{W_{g-1,g-1/2}}\big) &\simeq\Emb^{(\sim)}_\partial(D^{n},W_{g-1,1}\natural (D^n\times S^n))\\
\tohofib\big(\square_{W_{g-1/2,g-1}}\big) &\simeq\Emb^{(\sim)}_\partial(D^{n},W_{g,1})
\end{split}
\end{align}
of the total homotopy fibre of the middle and right square of \eqref{equ:factorisationWg}, using the cocores 
\[
D^n\times\{*\}\subset W_{g-1,1}\natural (D^n\times S^n)\quad \text{and}\quad \{*\}\times D^n\subset W_{g-1,1}\natural (D^n\times S^n\cup_{D^n\times D^n}S^n\times D^n).
\]

\subsection{The delooping trick and disjunction}\label{sect:delooping-disjunction}
To analyse the right-hand sides of \eqref{equ:tohofib-embeddings-Vg} and \eqref{equ:tohofib-embeddings-Wg}, we use a convenient trick to relate embedding spaces of discs of different dimension, sometimes called the \emph{delooping trick}. It applies to several variants of embedding spaces and has various incarnations (see e.g.\,\cite[p.\,7-8]{GoodwillieThesis} or \cite[p.\,23-25]{BLR}). Here is a version suitable for our needs, which we later apply to \eqref{equ:tohofib-embeddings-Vg} and \eqref{equ:tohofib-embeddings-Wg}:

Let $M$ be a compact $d$-manifold with an embedded neat disc $D^k\subset M$, i.e.\,$D^k$ is transverse to $\partial M$ and $D^k\cap\partial M=\partial D^k$. We assume  $d-k\ge3$. Let $D^{k-1}\subset D^k$ be an equatorial disc separating $D^k$ into half-discs $D^k=D^k_+\cup D^k_-$. We write $S^{k-1}_{\pm}=D^k_{\pm}\cap \partial D^{k}$ for the boundary hemispheres. We have a map of fibre sequences
\[
\begin{tikzcd}
\Emb_{S^{k-1}_+\cup D^{k-1}}(D^k_+,M)\dar\rar&\Emb_{S^{k-1}_+}(D^k_+,M)\rar\dar& \Emb_{\partial}(D^{k-1},M)\dar\\
\BlockEmb_{S^{k-1}_+\cup D^{k-1}}(D^k_+,M)\rar&\BlockEmb_{S^{k-1}_+}(D^k_+,M)\rar& \BlockEmb_{\partial}(D^{k-1},M)
\end{tikzcd}
\]
by restricting embeddings of $D^k_+$ to the equatorial disc. Shrinking $D^k_+$ to a small  neighborhood of $S^{k-1}_+$ shows that the spaces forming the middle column are contractible, so taking vertical homotopy fibres yields a preferred (up to contractible choices) equivalence
\[
\Emb_{S^{k-1}_+\cup D^{k-1}}^{(\sim)}(D^k_+,M)\simeq \Omega\Emb_\partial^{(\sim)}(D^{k-1},M).
\]
Now let $T\subset M$ be a closed tubular neighborhood of $D^{k-1}$. Restricting embeddings induces a map of fibre sequences
\[
\begin{tikzcd}[column sep=0.25cm]
\Emb_{\partial}\big(D^k_+\backslash (\interior(T)\cap D^k_+),M\backslash \interior(T) \big)\dar\rar&\Emb_{S^{k-1}_+\cup D^{k-1}}(D^k_+,M)\rar\dar& \Emb_{D^{k-1}}(T\cap D^k_+,M)\arrow[d]\\
\BlockEmb_{\partial} \big(D^k_+ \backslash(\interior(T)\cap D^k_+),M\backslash \interior(T) \big)\rar&\BlockEmb_{S^{k-1}_+\cup D^{k-1}}(D^k_+,M)\rar& \BlockEmb_{D^{k-1}}(T\cap D^k_+,M)
\end{tikzcd}
\]
whose right column is an equivalence, since by taking derivatives both spaces are equivalent to the space of normal vector fields of $D^{k-1}$. Up to smoothing corners, $D^k_+\backslash (\interior(T)\cap D^k_+)$ is a $k$-disc and $M\backslash \interior(T)$ is up to diffeomorphism obtained from $M$ by attaching an $(d-k)$-handle $H$ (see \cite[p.\,8]{GoodwillieThesis} for an enlightening picture in the situation of concordance embeddings), so we arrive at equivalences
\begin{equation}\label{equ:delooping-trick}
\Emb^{(\sim)}_{\partial}(D^k,M\cup H)\simeq \Emb_{S^{k-1}_+\cup D^{k-1}}^{(\sim)}(D^k_+,M)\simeq \Omega\Emb_\partial^{(\sim)}(D^{k-1},M).
\end{equation}
Viewing $M\cong (M\cup H)\backslash \nu(D^k)$ as being obtained from $M\cup H$ by removing an open tubular neighborhood $\nu(D^k)\subset H$ of the cocore, we have an inclusion $\Emb^{(\sim)}_\partial (D^k,M)\subset \Emb^{(\sim)}_{\partial}(D^k,M\cup H)$, which gives, via the equivalence above, a map of the form
\begin{equation}
\label{equ:delooping}
\Emb^{(\sim)}_\partial(D^k,M)\lra \Omega\Emb_\partial^{(\sim)}(D^{k-1},M).
\end{equation} 
By a form of Morlet's lemma of disjunction \cite[p.\,22]{BLR} (and an improvement by one degree in some cases due to Goodwillie \cite[p.\,6]{GoodwillieThesis}), the inclusion 
\[\Emb^{(\sim)}_\partial (D^k, (M\cup H)\backslash \nu(D^k))\subset \Emb^{(\sim)}_{\partial}(D^k,M\cup H)\]
is $(2d-2k-4)$-connected, so in this case the map \eqref{equ:delooping} is $(2d-2k-4)$-connected as well. Iterating this argument, we arrive at the following (c.f.\,\cite[Lemma 2.3]{BLR}).

\begin{thm}[Morlet]\label{thm:disjunction}
Let $M$ be a compact $d$-manifold with boundary and $D^k\subset M$ a neat $k$-disc. Up to contractible choices, there is a map
\[
\Emb^{{(\sim)}}_\partial(D^k,M)\lra \Omega^k\Emb^{{(\sim)}}(*,M)
\]
which is $(2d-2k-4)$-connected if $d-k\ge3$. 
\end{thm}

From work of Haefliger, one can deduce the following (see \cite[Proposition 2.4]{BLR}).

\begin{thm}[Haefliger]\label{thm:connectivity-bound-haefliger}
For a compact $d$-manifold $M$ that is $c$-connected for $0\le c\le d-2$, the space $\Emb^{{(\sim)}}(*,M)$ is $(d+c-3)$-connected.
\end{thm}

In view of \cref{thm:disjunction}, this has the following corollary (see \cite[Proposition 2.5]{BLR}).

\begin{cor}\label{cor:connectivity-bound-emb-tilde}
Let $M$ be a compact $d$-manifold with boundary and $D^k\subset M$ a neat $k$-disc. If $M$ is $c$-connected and $d-k\ge3$, then $\Emb^{{(\sim)}}_\partial(D^k,M)$ is $\min(2d-2k-5,d-k+c-3)$-connected.
\end{cor}

Applied to the cases \eqref{equ:tohofib-embeddings-Vg} and \eqref{equ:tohofib-embeddings-Wg} of our interest, we conclude that the total homotopy fibres of the squares $\square_{V_{g-1,g}}$, $\square_{W_{g-1,g-1/2}}$, and $\square_{W_{g-1/2,g-1}}$ are $(2n-5)$-connected, so an induction on $g$ using the factorisations \eqref{equ:factorisationVg} and \eqref{equ:factorisationWg} shows that the same holds for the squares $\square_{V_{g}}$ and $\square_{W_{g}}$. Together with \cref{prop:disjunction-problem}, this proves \cref{thm:fibre-to-haut} in the weaker range $*<2n-5$.

\subsection{Secondary disjunction}
To go beyond this range, we use a relative version of \cref{thm:disjunction} which can be seen as a special case of a multi-relative disjunction lemma for $\Emb^{(\sim)}_\partial(-,-)$. The latter can be derived from Goodwillie's multi-relative disjunction lemma for concordance embeddings \cite{GoodwillieThesis} by a relatively straight-forward induction (see \cite[Lemma\,7.2]{GoodwillieKlein}).

In the situation of \cref{sect:delooping-disjunction}, suppose we are given an additional neat submanifold $N^m\subset M^d$ disjoint from the previously fixed disc $D^k$. Applying the above procedure to the pair $(M,M\backslash \nu(N))$ where $\nu(N)$ is an open tubular neighbourhood of $N\subset M$ yields a square
\[
\begin{tikzcd}
\Emb^{{(\sim)}}_\partial(D^k,M\backslash \nu(N))\dar\rar&\Omega\Emb^{{(\sim)}}_\partial(D^{k-1},M\backslash \nu(N))\dar\\
\Emb^{{(\sim)}}_\partial(D^k,M)\rar&\Omega\Emb^{{(\sim)}}_\partial(D^{k-1},M)
\end{tikzcd}
\]
which commutes up to preferred homotopy and agrees up to equivalence with the square induced by inclusion
\[
\begin{tikzcd}[column sep=0.3cm]
\Emb^{(\sim)}_\partial\big(D^k, (M\cup H)\backslash (\nu(D^k)\cup \nu(N))\big)\rar\dar& \Emb^{(\sim)}_\partial \big(D^k, (M\cup H)\backslash \nu(N)\big)\dar\\
\Emb^{(\sim)}_\partial\big(D^k, (M\cup H)\backslash \nu(D^k)\big)\rar&\Emb^{(\sim)}_{\partial}\big(D^k,M\cup H\big)
\end{tikzcd}
\]
which is $(3d-2k-m-6)$-cartesian if $d-k\ge3$ and $d-m\ge3$ by an application of \cite[Lemma\,7.2]{GoodwillieKlein}. Iterating this argument yields the following relative from of \cref{thm:disjunction}.

\begin{thm}[Goodwillie]\label{thm:secondary-disjunction}
Let $M^d$ be a compact manifold with boundary, $D^k\subset M$ a neat $k$-disc, and $N^m\subset M$ a neat submanifold disjoint from $D^k$. There is a square
\[
\begin{tikzcd}
\Emb^{{(\sim)}}_\partial(D^k,M\backslash \nu(N))\arrow[d,hook]\rar&\Omega^k\Emb^{{(\sim)}}(*,M\backslash \nu(N))\dar[d,hook]\\
\Emb^{{(\sim)}}_\partial(D^k,M)\rar& \Omega^k\Emb^{{(\sim)}}(*,M)
\end{tikzcd}
\]
which commutes up to a preferred homotopy and is $(3d-2k-m-6)$-cartesian if $d-k\ge3$ and $d-m\ge3$. Here $\nu(N)\subset M$ is an open tubular neighborhood of $N$ disjoint from $D^k$.
\end{thm}

In \cref{sec:proof-computation-by-hand}, we will use \cref{thm:secondary-disjunction} to compare the embedding spaces \eqref{equ:tohofib-embeddings-Vg} and \eqref{equ:tohofib-embeddings-Wg} to spaces of embeddings of discs. In the next section, we explain how to compute the homotopy groups of the relevant embedding spaces of discs.

\subsection{The case of discs}
As a next step, we compute the rational homotopy groups of the spaces $\Emb_\partial^{(\sim)}(D^k,D^d)$ based at the standard inclusion in a range by combining \cref{thm:disjunction} with computations for the spaces $\Emb_{\partial}(D^k,D^d)$ of \emph{long embeddings} due to Arone--Turchin and Fresse--Turchin--Willwacher \cite{AroneTurchin,FresseTurchinWillwacher}. Note that $\Emb_\partial^{(\sim)}(D^k,D^d)$ carries an $H$-space structure induced by stacking, so its fundamental group is abelian and may be rationalised.

\begin{prop}\label{prop:longembeddings}
For $k\ge2$ and $d-k\ge3$, we have
\[\pi_*(\Emb^{(\sim)}_\partial(D^k,D^d);\iota)_\bfQ\cong 
\begin{cases}
\bigoplus_{m>0,\ m\equiv0,2 (\mathrm{mod}\,4)}\bfQ[m(d-k-2)]&k\text{ odd}, d\text{ odd}\\
\bigoplus_{m>1,\ m\equiv1(\mathrm{mod}\,4)}\bfQ[m(d-k-2)]&k\text{ odd},d\text{ even}\\
\bigoplus_{m>0,\ m\equiv3(\mathrm{mod}\,4)}\bfQ[m(d-k-2)]&k\text{ even},d\text{ odd}\\
\bigoplus_{m>1,\ m\equiv1,3(\mathrm{mod}\,4)}\bfQ[m(d-k-2)]&k\text{ even},d\text{ even.}
\end{cases}
\]
in the range $0<*<2d-k-5$ where $\iota\colon D^k\hookrightarrow D^d$ is the standard inclusion.
\end{prop}

During the proof of this proposition, we write 
\[\oG(m)\coloneq \hAut(S^{m-1}) \quad\text{and}\quad G \coloneq {\hocolim}_m\oG(m)\] 
where the colimit is taking over the maps induced by unreduced suspension. The standard action of $\oO(m)$ on $S^{m-1}$ defines a map $\oO(m)\ra\oG(m)$ which is compatible with the stabilisation maps given by inclusion and unreduced suspension. Moreover, this map fits into a commutative diagram of horizontal fibre sequences
\[
\begin{tikzcd}
\oO(m-1)\rar{\subset}\dar &\oO(m)\dar\rar&S^{m-1}\arrow[d,equal]\\
\Omega^{m-1}_{\pm1}S^{m-1}=\hAut_*(S^{m-1})\rar{\subset}&G(m)\rar{\ev}&S^{m-1}.
\end{tikzcd}
\]
Using the induced map of long exact sequences of homotopy groups, one computes that
\begin{equation}\label{equ:homotopy-g(n)}
\pi_*(G(m))_\bfQ\cong \begin{cases}
\bfQ[m-1]&m\text{ even}\\
\bfQ[2m-3]&m \text{ odd}
\end{cases}\quad\text{for }*>0,\end{equation}
a fact we shall make use of in the proof of the proposition.
\begin{proof}[Proof of \cref{prop:longembeddings}]
We begin the proof by developing a commutative (up to preferred homotopy) diagram of horizontal homotopy fibre sequences
\begin{equation}\label{eq:PfLongembeddings}
\begin{tikzcd}[column sep=0.4cm]
\overline{\Emb}_\partial(D^k,D^d)\rar&\Emb_\partial(D^k,D^d)\rar\dar&\Omega^k \oO(d)/\oO(d-k)\dar&\\
&\BlockEmb_\partial(D^k,D^d)\rar &\Omega^k\oO/\oO(d-k)\rar&\Omega^k\oG/\oG(d-k).
\end{tikzcd}
\end{equation}
Taking derivatives gives a map $\Emb_\partial(D^k,D^d)\ra \Bun_\partial(\tau_{D^k},\tau_{D^d})$ to the space of bundle maps $\tau_{D^k}\ra \tau_{D^d}$ which agree with the differential of $D^k\subset D^d$ at the boundary, equipped with the compact-open topology. The standard framing of the disc provides an equivalence between this space and $\Omega^k \OO(d)/\OO(d-k)$, which explains the upper row; $\overline{\Emb}_\partial(D^k,D^d)$ is defined as the homotopy fibre of this map over the base point. This derivative map does not factor over $\BlockEmb_\partial(D^k,D^d)$, but up to preferred equivalence, there is a compatible block-derivative map from $\BlockEmb_\partial(D^k,D^d)$ to the stabilisation $\colim_{i}\Bun_\partial(\tau_{D^k}\oplus\epsilon^i,\tau_{D^d}\oplus \epsilon^i)\simeq\Omega^k\oO/\oO(d-k)$ with the trivial line bundle $\epsilon$ (see \cite[Section\,1.9]{KrannichConc1} for a point-set model of this map in the case of diffeomorphisms; this directly generalises to embeddings). This defines the bottom middle map, so we are left with justifying that this map is equivalent to the fibre inclusion of the homotopy fibre of the standard map $\Omega^k\oO/\oO(d-k)\ra\Omega^k\oG/\oG(d-k)$. Making use of the canonical map $\Omega^k\BlockEmb(*,D^{n-k})\ra \BlockEmb_\partial(D^k,D^d)$, which one checks on homotopy groups to be an equivalence, it suffices to prove the claimed identification for $k=1$, which is explained for instance in \cite[p.\,241-242]{GoodwillieKleinWeiss}.

In \cite{FresseTurchinWillwacher}, Fresse, Turchin, and Willwacher identified the rational homotopy groups of $\smash{\overline{\Emb}_\partial(D^k,D^d)}$ under the assumption that $d-k\ge3$ with a shift of the homology of a certain chain complex denoted $\mathrm{HGC}_{k},d$ of ``hairy graphs'': one has $\smash{\pi_*(\overline{\Emb}_\partial(D^k,D^d))_\bfQ\cong H_{*+k}(\mathrm{HGC}_{k},d)}$ (combine in loc.cit.\,Equation (7) and (9) with Corollary 3). The homology of this chain complex has a direct sum decomposition by ``loop order'' and the part of loop order $\le 2$ is known to be concentrated in degree $*+k\ge 2(d-3)+1$ (see the paragraph after Corollary 3 loc.cit.), so in degrees $*<2d-k-5$ it is spanned by the part of loop order $0$ and $1$. From the known computation of the latter (see Equation (1) and (2) in loc.cit.) we thus obtain that for $*<2d-k-5$, the groups $\pi_*(\overline{\Emb}_\partial(D^k,D^d))_\bfQ$ are given by
\begin{equation}\label{equ:AroneTurchin-computation}
\begin{cases}
\bfQ[d-2k-1]\oplus \bigoplus_{m>0,\ m\equiv0,2(\mathrm{mod}\,4)}\bfQ[m(d-k-2)]&k\text{ odd}, d\text{ odd}\\
\bfQ[2d-3k-3]\oplus \bigoplus_{m>0,\ m\equiv1(\mathrm{mod}\,4)}\bfQ[m(d-k-2)]&k\text{ odd},d\text{ even}\\
\bfQ[2d-3k-3]\oplus\bigoplus_{m>0,\ m\equiv3(\mathrm{mod}\,4)}\bfQ[m(d-k-2)]&k\text{ even},d\text{ odd}\\
\bfQ[d-2k-1]\oplus\bigoplus_{m>0,\ m\equiv1,3(\mathrm{mod}\,4)}\bfQ[m(d-k-2)]&k\text{ even},d\text{ even.}
\end{cases}
\end{equation}

Looping the bottom fibre sequence in \eqref{eq:PfLongembeddings} then taking vertical homotopy fibres, we obtain a homotopy fibre sequence
\begin{equation}\label{equ:framed-embeddings-fibration}
\hofib\left(\overline{\Emb}_\partial(D^k,D^d)\ra\Omega^{k+1}\oG/\oG(d-k)\right)\ra \Emb^{(\sim)}_\partial(D^k,D^d)\ra \Omega^{k+1}\oO/\oO(d)
\end{equation}
and as a next next step we show that the induced map 
\[
\pi_*(\overline{\Emb}_\partial(D^k,D^d))_\bfQ\lra \pi_{*}(\Omega^{k+1}\oG/\oG(d-k))_\bfQ\cong\begin{cases}
\bfQ[d-2k-1]&d-k\text{ even}\\
\bfQ[2d-3k-3]&d-k\text{ odd}
\end{cases}
\]
is surjective in positive degrees. 

If this were to fail for $d-k$ even, then using \eqref{equ:AroneTurchin-computation} the homotopy fibre in \eqref{equ:framed-embeddings-fibration} would be rationally nontrivial in degree $d-2k-1$ and $d-2k-2$, so using 
\begin{equation}
\label{equ:so-sod}
\pi_*(\Omega^{k+1}\oO/\oO(d))_\bfQ\cong \begin{cases}
\bfQ[d-k-1]&d \text{ even}\\
0&d\text{ odd}
\end{cases}\quad\text{for }*<2d-k-2
\end{equation}
and the fibre sequence \eqref{equ:framed-embeddings-fibration}, this would mean that $\Emb^{(\sim)}_\partial(D^k,D^d)$ is rationally nontrivial in at least one of the degrees $d-2k-1$ or $d-2k-2$. But this space is $(2d-2k-5)$-connected by \cref{cor:connectivity-bound-emb-tilde} and $d-2k-1\le 2d-2k-5$ as $d-k\ge3$ and $k\ge1$, so this cannot be true. 

For $d-k$ odd, the claimed surjectivity must hold as well, by a similar argument: if it does not, then  the homotopy fibre in \eqref{equ:framed-embeddings-fibration} would be nontrivial in degrees $2d-3k-4$ and $2d-3k-3$, so $\Emb^{(\sim)}_\partial(D^k,D^d)$ would in view of \eqref{equ:so-sod} be rationally nontrivial in at least one of these degrees, which cannot be the case since by \cref{cor:connectivity-bound-emb-tilde} this space is $(2d-2k-5)$-connected and $2d-3k-3\le 2d-2k-5$ since we assumed $k\ge2$.

As a result, in degrees $*<2d-k-4$, the rational homotopy of the homotopy fibre in \eqref{equ:framed-embeddings-fibration} is given by \eqref{equ:AroneTurchin-computation} minus the first summand. Using again that $\Emb^{(\sim)}_\partial(D^k,D^d)$ is $(2d-2k-5)$-connected, we see that in the long exact sequence of \eqref{equ:framed-embeddings-fibration}, the classes in \eqref{equ:so-sod} must cancel against classes in the homotopy fibre. This kills the classes in degree $d-k-2$ in the case where $d$ is even, so the claim follows.
\end{proof}

We now turn towards the proof of the main result of the section, \cref{thm:fibre-to-haut}.

\subsection{Proof of \cref{thm:fibre-to-haut}}\label{sec:proof-computation-by-hand}
The claim regarding the finiteness of the fundamental groups has been dealt with as part of \cref{prop:disjunction-problem}, so we are left with showing that the higher rational homotopy groups are as claimed, which we do for $V_g$ first. By \cref{prop:disjunction-problem}, these homotopy groups agree up to a shift by $1$ with the rational homotopy groups of the total homotopy fibre of the square $\square_{V_{g}}$ in \eqref{equ:diff-blockdiff-squares}, so our task is to compute
\[
\pi_*(\tohofib(\square_{V_{g}}))_\bfQ\cong \begin{cases}
\bfQ^{\frac{g(g-1)}{2}}[2n-3]&n\text{ even}\\
\bfQ^{\frac{g(g+1)}{2}}[2n-3]&n\text{ odd}
\end{cases}\quad\text{for }0<*<3n-6.
\]
Using the factorisation in \eqref{equ:factorisationVg} and the equivalence \eqref{equ:tohofib-embeddings-Vg}, these homotopy groups fits into a long exact sequence of the form
\[
\ldots\ra\pi_*(\tohofib(\square_{V_{g-1}}))\ra\pi_*(\tohofib(\square_{V_{g}}))\ra\pi_{*}(\Emb^{(\sim)}_\partial(D^{n+1},V_g))\ra \ldots,
\]
where $D^{n+1}\subset V_g$ is the cocore of the last handle. Given that $\tohofib(\square_{V_{0}})$ is evidently contractible, the claim for $V_g$ will follow by induction on $g$ once we show
\[\pi_*\Emb^{(\sim)}_\partial(D^{n+1},V_{g})_\bfQ\cong \begin{cases}\bfQ^{g-1}[2n-3]&n\text{ even}\\\bfQ^g[2n-3]&n\text{ odd}\end{cases}\quad\text{for }0<* < 3n-6.
\]
To this end, we apply \cref{thm:secondary-disjunction}  with $M=V_g$ and $N=\coprod_{i=1}^gD^{n+1}$ a choice of cocores of the $g$ handles, disjoint from the previously fixed cocore $D^{n+1}\subset V_g$ of the last handle. In this case, $M\backslash \nu(N)\cong D^{2n+1}$ and this diffeomorphism can be chosen such that $D^{n+1}\subset M\backslash \nu(N)\cong D^{2n+1}$ is the standard inclusion. We thus have a $(3n-6)$-cartesian square
\begin{equation}\label{equ:disunction-square-vg}
\begin{tikzcd}
\Emb^{{(\sim)}}_\partial(D^{n+1},D^{2n+1})\arrow[d,hook]\rar&\Omega^{n+1}\Emb^{{(\sim)}}(*,D^{2n+1})\arrow[d,hook]\\
\Emb^{{(\sim)}}_\partial(D^{n+1},V_g)\rar& \Omega^{n+1}\Emb^{{(\sim)}}(*,V_g).
\end{tikzcd}
\end{equation}
By \cref{thm:connectivity-bound-haefliger}, $\Omega^{n+1}\Emb^{{(\sim)}}(*,D^{2n+1})$ is $(3n-4)$-connected. From \cref{prop:longembeddings}, we obtain
\[
\pi_*(\Emb^{(\sim)}(D^{n+1},D^{2n+1}))_\bfQ\cong
\begin{cases}
\bfQ[2n-4]&n\text{ even}\\
0&n\text{ odd}\\
\end{cases}\quad\text{for }0<*<3n-6
\]
and in the next section we will compute $\pi_*(\Emb^{(\sim)}(*,V_g))_\bfQ\cong \bfQ^g[3n-2]$ for $0<*< 4n-5$ as part of \cref{thm:blockemb-point}. Putting these computations together, we see that in the range $0<*<3n-6$, the square \eqref{equ:disunction-square-vg} induces an exact sequence of the form
\[
\ldots\xlra{\partial} \begin{cases}
\bfQ[2n-4]&n\text{ even}\\
0&n\text{ odd}\\
\end{cases}\lra \pi_*(\Emb^{{(\sim)}}_\partial(D^{n+1},V_g))_\bfQ\lra \bfQ^g[2n-3] \lra\ldots,
\]
which leaves us with showing that the boundary map $\partial$ is surjective in degree $2n-4$ if $n$ is even. As the square \eqref{equ:disunction-square-vg} is natural (up to homotopy) with respect to the inclusion $V_g\subset V_{g+1}$, it suffices to check surjectivity for $g=1$, where we have \[\Emb_\partial^{(\sim)}(D^{n+1},V_1)\simeq \Omega\Emb_\partial^{(\sim)}(D^{n},D^{2n+1})\] by applying the delooping trick \eqref{equ:delooping-trick} for $M=D^{2n+1}$. From \cref{prop:longembeddings}, we see that this space has no rational homotopy in degree in $2n-4$, so $\partial$ must be surjective and the proof in the case of $V_g$ is finished. 

The argument for $W_{g,1}$ proceeds analogously: using \cref{prop:disjunction-problem}, the factorisation \eqref{equ:factorisationWg} and the equivalences \eqref{equ:tohofib-embeddings-Wg}, it suffices to show that the spaces
\[
\Emb^{(\sim)}_\partial(D^{n},W_{g-1,1}\natural (D^n\times S^n))\quad\text{and}\quad\Emb^{(\sim)}_\partial(D^{n},W_{g,1})
\]
are rationally $(3n-7)$-connected. We then apply \cref{thm:secondary-disjunction} in both cases to the cocores of the manifolds to reduce the claim to showing that the spaces
\[\Emb^{(\sim)}_\partial(D^{n},D^{2n}),\quad\Omega^{n}\Emb^{(\sim)}_\partial(*,W_{g-1,1}\natural (D^n\times S^n)),\quad\text{and}\quad\Omega^{n}\Emb^{(\sim)}_\partial(*,W_{g,1})\]
are rationally $(3n-7)$-connected, which follows from \cref{prop:longembeddings} and \cref{thm:blockemb-point} below.

\section{Block embeddings of a point}\label{sec:block-embeddings}
This section serves to explain a method to compute rational homotopy groups of the spaces $\Emb^{(\sim)}(*,M)$ (see \cref{section:splitting-handles} for the notation) for $1$-connected compact manifolds $M$ and to use it to provide the following missing ingredient in the proof of \cref{thm:fibre-to-haut}. Note that spaces involved are $1$-connected by \cref{thm:connectivity-bound-haefliger}, so rationalising their homotopy groups is no issue.

\begin{thm}\label{thm:blockemb-point}
For $n\ge3$ and in the range $*<4n-5$, we have
\vspace{-0.2cm}
\[
\def\arraystretch{1.5}
\begin{array}{r@{\hskip 0.2cm} c@{\hskip 0.2cm} l@{\hskip 0.2cm} }
\pi_*(\Emb^{(\sim)}(*,V_g))_\bfQ&\cong&\bfQ^g[3n-2]\\
\pi_*(\Emb^{(\sim)}(*,W_{g,1}))_\bfQ&=&0\\
\pi_*(\Emb^{(\sim)}(*,W_{g,1}\natural T_h))_\bfQ&=&0
\end{array}
\]
for $g,h\ge0$, where $T_h=\natural^h D^n\times S^n$.
\end{thm}

\begin{rem}This in particular shows that the range in \cref{thm:connectivity-bound-haefliger} is sharp, even rationally.
\end{rem}

We first translate the problem into unstable homotopy theory. To this end, we fix more generally a compact $d$-manifold $M$ together with a fixed embedded disc $D^d\subset \interior(M)$ and denote by $\hAut_{\partial\sqcup *}(M)/\hAut_\partial(M^\circ)$ the homotopy fibre of the map $\BhAut_{\partial M^\circ}(M^\circ)\ra \BhAut_{\partial M\sqcup *}(M)$ induced by extending homotopy equivalences by the identity, where $*\in \interior(M)$ is the centre of the disc and $M^\circ\coloneq M\backslash\interior(D^d)$. Note that $\partial M^\circ=\partial M\sqcup \partial D^d$.

\begin{prop}\label{prop:identification-block-point-embeddings}
If $d\ge6$ and $M$ is $1$-connected and with nonempty boundary, then the space $\Emb^{(\sim)}(*,M)$ is $1$-connected and fits into a rational fibre sequence of the form
\[\Omega_0\big(\hAut_{\partial\sqcup *}(M)/\hAut_\partial(M^\circ)\big)\lra \Emb^{(\sim)}(*,M)\lra \Omega{\big(\oO/\oO(d)\big)}.\]
\end{prop}
\begin{proof}
It follows from \cref{thm:connectivity-bound-haefliger} that $\Emb^{(\sim)}(*,M)$ is $4$-connected, so it is in particular simply connected. Specialising the equivalence \eqref{equ:no-normal-data} to $k=0$, we see that restriction to the centre induces an equivalence $\Emb^{(\sim)}(D^d,M)\simeq \Emb^{(\sim)}(*,M)$.
Similarly to the beginning of the proof of \cref{prop:longembeddings}, taking derivatives induces the first of the two squares of based spaces
\[
\begin{tikzcd}
\Emb(D^d,M)\rar{\simeq}\dar&\Bun(\tau_{D^d},\tau_{M})\dar&&\oO(d)\rar\dar{\subset}&\Bun(\tau_{D^d},\tau_M)\rar\dar&M\arrow[d,equal]\\
\BlockEmb(D^d,M)\rar&\Bun(\tau^s_{D^d},\tau^s_M)&&\oO\rar&\Bun(\tau^s_{D^d},\tau^s_{M})\rar&M.
\end{tikzcd}
\]
where $\Bun(\tau^s_{D^d},\tau^s_{M})=\colim_i\Bun(\tau_{D^d}\oplus \epsilon^i,\tau_{M}\oplus \epsilon^i)$ is the space of stable bundle maps and the right-hand diagram of horizontal fibre sequences is induced by evaluating bundle maps at the centre of the disc and stabilising with trivial bundles. Combining these diagrams with the above equivalence, we obtain a fibre sequence
\begin{equation}\label{equ:block-emb-point-fibre-sequence}\Omega \hofib\big(\BlockEmb(D^d,M)\ra\Bun(\tau^s_{D^d},\tau^s_M)\big)\lra \Emb^{(\sim)}(*,M)\lra \Omega\big(\oO/\oO(d)\big).\end{equation}
To simplify the left-hand term, we choose an embedded disc $D^{d-1}\subset \partial M$ and consider the square based at the identities (see \cref{sec:block-diffeos} for the notation)
\[
\begin{tikzcd}
\BlockDiff_{\half \partial}(M^\circ)\dar\rar&\BlockDiff_{\half \partial}(M)\dar\\
\hAut_{\half \partial}(\tau^s_{M^\circ},D^{d-1})\rar&\hAut_{\half \partial}(\tau^s_M,D^{d-1})
\end{tikzcd}
\]
whose horizontal arrows are induced by extension by the identity and whose vertical arrows are explained in \cref{sec:block-diffeos}. As $M$ is simply connected, so is $M^\circ$ and hence \cref{thm:blockdiff-identification} (using the assumption $d \geq 6$) implies that the total homotopy fibre of this square (which is a loop space) has finitely many components, each of them rationally trivial. As the homotopy fibre of the upper horizontal arrow is equivalent to $\Omega\BlockEmb(D^d,M)$, we have a loop map
\begin{equation}\label{equ:block-emb-to-tangential-haut}
\Omega \BlockEmb(D^d,M)\lra\Omega\big( \hAut_{\half \partial}(\tau^s_M,D^{d-1})/\hAut_{\half \partial}(\tau^s_{M^\circ},D^{d-1})\big)
\end{equation}
whose homotopy fibre is the total homotopy fibre of the previous square, so it has finitely many components which are rationally trivial. Now consider the based square 
\begin{equation}\label{equ:tohofibsquare}
\begin{tikzcd}
\hAut_{\half \partial}(\tau^s_{M^\circ},D^{d-1})\dar\rar& \hAut_{\half \partial}(\tau^s_M,D^{d-1})\dar\\
\hAut_{\half \partial}(M^\circ,D^{d-1})\rar& \hAut_{\half \partial}(M,D^{d-1})
\end{tikzcd}
\end{equation}
induced by forgetting bundle maps. The map on vertical homotopy fibres agrees with a map of the form $\Omega\Map_{\half \partial}(M^\circ,\BO)\ra \Omega\Map_{\half \partial}(M,\BO)$ where the loop spaces are taking at a classifier $TM^s\colon M\ra \BO$ of the stable tangent bundle and its restriction to $M^\circ$ respectively. The map is induced by extending maps along $M^\circ\subset M$ using $TM^s|_{D^d}$. This map agrees with the inclusion of the homotopy fibre at $TM^s|_{D^d}$ of the map $\smash{\Omega\Map_{\half \partial}(M,\BO)\ra \Omega\Map(D^d,\BO)\simeq\oO}$ induced by restriction along $D^d\subset M$, so it follows that the total homotopy fibre of the square \eqref{equ:tohofibsquare} is equivalent to $\Omega\oO$. This implies that the based square
\[
\begin{tikzcd}[column sep=0.3cm]
\Omega\big(\hAut_{\half\partial}(\tau^s_M,D^{d-1})/\hAut_{\half \partial}(\tau^s_{M^\circ},D^{d-1})\big) \rar\dar& \Omega\big(\hAut_{\half \partial}(M,D^{d-1})/\hAut_{\half \partial}(M^\circ,D^{d-1})\big)\dar\\
\Omega\Bun(\tau_{D^d}^s,\tau_M)\rar&\Omega M
\end{tikzcd}
\] 
induced by restriction to the disc (or its center) is homotopy cartesian since the horizontal homotopy fibres are compatibly equivalent to $\Omega\oO$. Since $\hAut_\partial(D^{d-1})$ is contractible, the inclusion $\hAut_{\partial}(M)\subset \hAut_{\half\partial}(M,D^{d-1})$ is an equivalence (and similarly for $M^\circ$), so the homotopy fibre of the right vertical map in the diagram is equivalent to $\Omega(\hAut_{\partial\sqcup *}(M)/\hAut_\partial(M^\circ))$ and hence so is that of the left vertical map. Finally, we consider the square of based spaces
 \[
\begin{tikzcd}
\Omega\BlockEmb(D^d,M)\rar\dar&\Omega\big(\hAut_{\half\partial}(\tau^s_M,D^{d-1})/\hAut_{\half \partial}(\tau^s_{M^\circ},D^{d-1})\big)\dar\\
\Omega\Bun(\tau^s_{D^d},\tau^s_M)\arrow[r,equal]&\Omega\Bun(\tau^s_{D^d},\tau^s_M)
\end{tikzcd}
\] 
whose top horizontal arrow is \eqref{equ:block-emb-to-tangential-haut}, so the total homotopy fibre of this square has finitely many components which are rationally trivial. By the discussion above, the homotopy fibre of the right-hand map agrees with $\Omega(\hAut_{\partial\sqcup *}(M)/\hAut_\partial(M^\circ))$, so by taking vertical homotopy fibres, we arrive at a map
\begin{equation}\label{equ:framed-block-emb-to-haut}
\Omega\hofib(\BlockEmb(D^d,M)\ra \Bun(\tau^s_{D^d},\tau^s_M))\lra \Omega(\hAut_{\partial\sqcup *}(M)/\hAut_\partial(M^\circ))
\end{equation}
whose homotopy fibre has finitely many components that are rationally trivial. Combining this with the fibre sequence \eqref{equ:block-emb-point-fibre-sequence}, this would show the claim if we knew that the source in \eqref{equ:framed-block-emb-to-haut} is connected and the map thus only hits the basepoint component of $\Omega(\hAut_{\partial\sqcup *}(M)/\hAut_\partial(M^\circ))$. But this holds by \eqref{equ:block-emb-point-fibre-sequence} and the fact that $\Emb^{(\sim)}(*,M)$ is $1$-connected, so the proof is finished.
\end{proof}

\subsection{Pointed Poincaré embeddings of discs}\label{section:poincare-embeddings}
In view of \cref{prop:identification-block-point-embeddings}, we are interested in the rational homotopy groups of the space
\[
\Omega \hAut_{\partial\sqcup *}(M)/\hAut_\partial(M^\circ)\simeq \hofib_{\id}\big(\Map_\partial(M^\circ,M^\circ)\lra \Map_{\partial \sqcup*}(M,M)\big).
\] 
for a $1$-connected $d$-manifold $M$ with nonempty boundary and $d\ge6$. We assume the boundary to be connected for simplicity. To explain our method, we fix a disc $D^d\subset \interior(M)$ with centre $*\in D^d$ and recall that $M^\circ=M\backslash \interior(D^d)$, so $\partial(M^\circ)=\partial M\sqcup S^{d-1}$. We denote by $\iota^{\circ}\colon S^{d-1}\hookrightarrow M^\circ$ the inclusion of the additional boundary sphere, and by $\iota_\partial \colon \partial M \hookrightarrow M^\circ$ the inclusion of the original boundary. Extension by the identity induces a map of horizontal fibre sequences \begin{equation}\label{boundaryrestriction}
\begin{tikzcd}[column sep=1.5cm]
\Map_\partial(M^{\circ},M^{\circ})\rar\dar&\Map_{S^{2n-1}}(M^{\circ},M^{\circ})\rar{(-)\circ \iota_\partial}\dar\arrow[d, phantom, shift left=1.9cm, "\circled{$1$}"]&\Map(\partial M,M^{\circ})\arrow[d,"\inc \circ (-)"]\\
\Map_{\partial \sqcup*}(M,M)\rar&\Map_{*}(M,M)\rar{(-)\circ\iota_\partial}&\Map(\partial M,M),
\end{tikzcd}
\end{equation}
which we consider as a diagram of based spaces using the standard inclusions. Our goal is to compute the rational homotopy groups of the homotopy fibre of the left vertical map, i.e.\,those of the total homotopy fibre
\begin{equation}\label{equ:hAut-as-tohofib}
\tohofib(\circled{$1$})\simeq \Omega \hAut_{\partial\sqcup *}(M)/\hAut_\partial(M^\circ).
\end{equation}
Fixing inverse homotopy equivalences
\[
\begin{tikzcd}
M\vee S^{d-1}\arrow[rr,"{j}","\simeq"']&&M^{\circ}& M^{\circ}\arrow[rr,"p","\simeq"']&&M\vee S^{d-1}\\
&S^{d-1}\arrow[ul,"\inc_2"]\arrow[ur,"\iota^{\circ}",swap]&&&S^{d-1}\arrow[ul,"\iota^{\circ}"]\arrow[ur,"\inc_2",swap]&
\end{tikzcd}
\]
over $S^{d-1}$supported in a disc in $M$, there is a diagram
\[
\begin{tikzcd}[column sep=1.5cm]
\Map_{S^{d-1}}(M^{\circ},M^{\circ})\dar{p\circ(- )\circ (j\circ \inc_{1})}[swap]{\simeq}\arrow[rr,"(-)\circ\iota_\partial"]&&\Map(\partial M,M^{\circ})\dar{p\circ(-)}[swap]{\simeq}\\
\Map_{*}(M,M\vee S^{d-1})\arrow[r,"(-)\circ \iota_\partial+ \inc_2"]&\Map_{*}(\partial M,M\vee S^{d-1})\arrow[r,"\subset"]&\Map(\partial M,M\vee S^{d-1})
\end{tikzcd}
\]
which commutes up to preferred homotopy. Here "+" denotes the $\Map_*(S^{d-1},M\vee S^{d-1})$-module structure on $\Map_*(\partial M,M\vee S^{d-1})$ induced by a pinch map $\partial M\ra \partial M\vee S^{d-1}$ and all spaces inherit base points by the image of $\id_{M^{\circ}}$; for instance the lower left corner is based at the inclusion $M\subset M\vee S^{d-1}$. Consequently, the square $\circled{$1$}$ factors up to equivalence as the composition of two squares commuting up to preferred homotopy 
\[
\begin{tikzcd}[column sep=1.5cm]
\Map_{*}(M,M\vee S^{d-1})\arrow[r,"(-)\circ \iota_\partial+ \inc_2"]\arrow[d,"\pr_1\circ(-)"]\arrow[d, phantom, shift left=2.2cm, "\circled{$2$}"]&\Map_{*}(\partial M,M\vee S^{d-1})\arrow[r,"\subset"]\arrow[d,"\pr_1\circ(-)"]\arrow[d, phantom, shift left=2.5cm, "\circled{$3$}"]&\Map(\partial M,M\vee S^{d-1})\arrow[d,"\pr_1\circ(-)"]\\
\Map_{*}(M,M)\arrow[r,"(-)\circ \iota_\partial"]&\Map_{*}(\partial M,M)\arrow[r,"\subset"]&\Map(\partial M,M),
\end{tikzcd}
\]
so we obtain a fibre sequence
\begin{equation}\label{equ:lestohofib}
\tohofib(\circled{$2$})\lra \tohofib(\circled{$1$})\lra \tohofib(\circled{$3$}).
\end{equation}
It will be sometimes convenient to replace $\tohofib(\circled{$2$})$ by the total homotopy fibre of the square $\circled{$2$}'$ obtained from $\circled{$2$}$ by replacing the upper horizontal map  $(-)\circ \iota_\partial+ \inc_2$ by $(-)\circ\iota_\partial$. These total homotopy fibres are equivalent via the map
\begin{equation}\label{equ:modified-square}
\tohofib(\circled{$2$})\xlra{\simeq}\tohofib(\circled{$2$}')
\end{equation}
induced by acting with $-\inc_2$.

The homotopy groups of $\tohofib(\circled{$3$})$ can easily be computed: the responsible square is obtained by taking horizontal fibres from the square
\[
\begin{tikzcd}
\Map(\partial M,M\vee S^{d-1})\arrow[d,"\pr_1\circ(-)",swap]\rar{\ev}&M\vee S^{d-1}\dar{\pr_1}\\
\Map(\partial M,M)\rar{\ev}&M,
\end{tikzcd}
\]
so we have an equivalence $\tohofib(\circled{$3$})\simeq \Omega\hofib(M\vee S^{d-1}\xra{\pr_1} M)$. Since the projection $M\vee S^{d-1}\ra M$ has an evident section, we conclude
\begin{equation}\label{equ:homotopy-tohofib3}
\pi_*(\tohofib(\circled{$3$}))= \ker\big(\pi_{*+1}(M\vee S^{d-1})\xra{(\pr_1)_*} \pi_{*+1}(M)\big).
\end{equation}
In all our examples, $M$ comes with an equivalence to a wedge of spheres, so employing the notation of Sections~\ref{sec:gradings} and \ref{sec:homotopy-groups-haut} we choose a splitting $s_M\colon H_M\ra \pi_{*+1}(M)_\bfQ$ of the Hurewicz map which induces an analogous splitting for $M\vee S^{d-1}$ and thus provides isomorphisms
\[\bfL(H_M)\cong \pi_{*+1}(M)_\bfQ\quad\text{and}\quad \bfL(H_M\oplus x)\cong \pi_{*+1}(M\vee S^{d-1})_\bfQ,\]
where $\bfL(-)$ is the free graded Lie algebra and $x$ is a class in degree $d-2$. Then \eqref{equ:homotopy-tohofib3} becomes
\begin{equation}\label{equ:homotopy-tohofib3-rational}
\pi_*(\tohofib(\circled{$3$}))_\bfQ\cong \widetilde{\bfL}(H_M\oplus x)\coloneq \ker\big(\bfL(H_M\oplus x)\xra{\pr_{H_M}}\bfL(H_M)\big).
\end{equation}

We now apply this setup to the three manifolds featuring in \cref{thm:blockemb-point}.

\begin{notation}\label{conv:vs-for-block-emb-comp}In addition to the notation $H_V$, $H_W$, $K_W$, etc. of \cref{sec:gradings}, we write 
\[
T_h\coloneq \natural^h D^n\times S^n,\quad U_h\coloneq \partial T_h\cong \sharp^hS^{n-1}\times S^n,\quad U_{h,1}=\sharp^hS^{n-1}\times S^n\backslash \interior(D^{2n-1})
\]
and abbreviate $H_{WT}\coloneq s^{-1}\widetilde{\oH}_*(W_{g,1}\natural  T_h;\bfQ)$. The latter decomposes  canonically $H_{WT}=H_W\oplus H_T$ into the subspaces corresponding to $W_{g,1}$ and $T_h$. Further we write $H_U\coloneq s^{-1}\widetilde{\oH}_*(U_{g,1};\bfQ)$. The map $H_U\ra H_{WT}$ induced by inclusion lands in the subspace $H_T$ and we denote the kernel by $K_U\coloneq \ker(H_U\ra H_{T})$. Note that there are canonical isomorphisms
\[
H_W/K_W\cong H_V\quad\text{and}\quad H_U/K_U\cong H_T.
\]
Poincaré duality moreover induces canonical isomorphisms of graded $\bfQ$-vector spaces \[
(H_V)^{\vee}\cong s^{-(2n-2)}K_W\quad (H_W)^\vee\cong s^{-(2n-2)}H_W\quad\text{and}\quad (H_T)^{\vee}\cong s^{-(2n-3)}K_U.
\]
In this section, we work in plain graded $\bfQ$-vector spaces instead of $\bfQ[\langle\rho\rangle]$-modules. In particular, we use that for graded $\bfQ[\langle \rho\rangle]$-modules $V$ we have $V\cong V \otimes\bfQ^-$ as graded $\bfQ$-vector spaces.
\end{notation}

\subsection{Case I: $M=W_{g,1}$}\label{section:computation-Wg}
Applying \cref{lem:boundary-restriction-on-mapping-spaces} we see that the square $\circled{$2$}'$ has on rational homotopy groups in positive degrees the form (this uses that the projection $W_{g,1}\vee S^{2n-1}\ra W_{g,1}$ is compatible with the chosen splittings of the rational Hurewicz map)
\[
\begin{tikzcd}
s^{-(2n-2)}\bfL(H_W\oplus x)\otimes H_W\rar{[-,-]}\dar{\id\otimes \pr_{H_W}}& s^{-(2n-2)}\bfL(H_W\oplus x)\dar{\pr_{H_W}}\\
s^{-(2n-2)}\bfL(H_W)\otimes H_W\rar{[-,-]}& s^{-(2n-2)}\bfL(H_W).
\end{tikzcd}
\]
Using the notation introduced in \eqref{equ:homotopy-tohofib3-rational}, the long exact sequence obtained by taking vertical homotopy fibres in $\circled{$2$}'$ is thus in positive degrees of the form
\[
\ldots\ra \pi_*\tohofib(\circled{$2$}')_\bfQ \ra s^{-(2n-2)} \widetilde{\bfL}(H_W\oplus x) \otimes H_W\xra{[-,-]}s^{-(2n-2)}\widetilde{\bfL}(H_W\oplus x) \ra\ldots,
\]
Suppressing degree shifts from the notation, we see from \cref{lem:decomposition} that in degrees $0<*< 4n-4$, the bracket has the form
\begin{equation}\label{equ:decomp-HW}
\begin{tikzcd}[column sep=-0.2cm]
{\big(x\otimes H_W\big)}&\oplus& {\big([x,H_W]\otimes H_W\big)}&\oplus\dar{[-,-]}& {\big([H_W,[x,H_W]]\otimes H_W\big)}&&\\
{[x,H_W]}&\oplus&{[[x,H_W],H_W]}&\oplus&{ [[H_W,[x,H_W]],H_W]}&\oplus&{ [x,[x,H_W]]}.
\end{tikzcd}
\end{equation}
As the first three direct summands of the target are evidently surjected upon by the corresponding summands of the source and also have the same dimension as a result of the dimension formulas in \cref{lem:decomposition}, we conclude 
\[
\pi_*(\tohofib(\circled{$2$}'))_\bfQ\cong s^{-(2n-3)}[x,[x,H_W]]\quad\text{for }*< 4n-5.
\]
 Moreover, we have $\widetilde{\bfL}(H_W\oplus x)=x\oplus [x,H_W]$ in degrees $*< 4n-4$, so using \eqref{equ:modified-square} and \eqref{equ:homotopy-tohofib3-rational}, the long exact sequence induced by \eqref{equ:lestohofib} has in degrees $*< 4n-5$ the form
 \vspace{-0.1cm}
\begin{equation}\label{equ:les-tohofib1}
\ldots \lra s(x\oplus [x,H_W])\xlra{\partial} s^{-(2n-3)}[x,[x,H_W]]\lra \pi_*(\tohofib(\circled{$1$}))_\bfQ\lra \ldots.
\end{equation}
The class $\partial(x)$ vanishes for the degree reasons, and moreover we claim that $\partial$ maps $[x,H_W]$ surjectively onto $[x,[x,H_W]]$ (and hence isomorphically, as they are equidimensional by \cref{lem:decomposition}). Chasing through the computation, we see that this map is given by the composition 
\[s[x,H_W]\subset s\widetilde{\bfL}(H_W\oplus x)\lra s^{-(2n-3)}\widetilde{\bfL}(H_W\oplus x)
\lra s^{-(2n-3)}[x,[x,H_W]],\] 
where the first arrow is the restriction of the boundary map of the evaluation fibration 
\[\hspace{-0.15cm}
s\bfL(H_W\oplus x)\cong\pi_*(W_{g,1}\vee S^{2n-1})_\bfQ\ra \pi_{*-1}(\Map_*(\partial W_{g,1},W_{g,1}\vee S^{2n-1});\iota_\partial+\inc_2)_\bfQ\cong s^{-(2n-3)}\bfL(H_W\oplus x).
\]
and the second is the projection onto $[x,[x,H_W]]$ with respect to the above decomposition of $\widetilde{\bfL}(H_W\oplus x)$ in a range. The boundary map of the evaluation fibration is up to signs given by taking bracket with the class $(\iota_\partial+\inc_2)\in \pi_{*+1}(W_{g,1}\vee S^{2n-1})_\bfQ\cong \bfL(H_W\oplus x)$ (see \cite[Theorem\,3.2]{Whitehead}). The inclusion $\inc_2$ corresponds to the class $x$ and $\iota_\partial$ corresponds to an element $\omega\in[H_W,H_W]$, so the image of $\partial$ in \eqref{equ:les-tohofib1} is given by the projection of $[x+\omega,[x,H_W]]$ onto $[x,[x,H_W]]$ with respect to the decomposition of $\widetilde{\bfL}(H_W\oplus x)$ in \eqref{equ:decomp-HW}. As the projection of $[\omega,[x,H_W]]\subset [[H_W,H_W],[x,H_W]]\subset [[H_W,[x,H_W]],H_W]$ is trivial (see \cref{lem:decomposition}), this is indeed surjective. Going back to the exact sequence \eqref{equ:les-tohofib1}, we conclude $\pi_*(\tohofib(\circled{$1$}))_\bfQ\cong \bfQ[2n-2]$ for $0<*< 4n-5$, so using the equivalence \eqref{equ:hAut-as-tohofib}, the long exact sequence resulting from \cref{prop:identification-block-point-embeddings} has in degrees $*< 4n-5$ the form
\[
\ldots\lra  \pi_{*+2}(\SO/\SO(2n))_\bfQ\xlra{\partial} \bfQ[2n-2]\lra \pi_*(\Emb^{(\sim)}(*,W_{g,1})_\bfQ\lra\ldots .
\]
In this range, we have $\pi_{*+2}(\SO/\SO(2n))_\bfQ\cong\bfQ[2n-2]$, as already noted in \eqref{equ:so-sod}, so given that $\Emb^{(\sim)}(*,W_{g,1})$ is $(3n-4)$-connected by \cref{thm:connectivity-bound-haefliger}, the boundary map $\partial$ has to be an isomorphism and we conclude that $\Emb^{(\sim)}(*,W_{g,1})$ is rationally $(4n-6)$-connected, which finishes the proof of \cref{thm:blockemb-point} for $W_{g,1}$.

\subsection{Case II: $M=V_{g}$}\label{section:computation-Vg}
To study the right vertical arrow of $\circled{$2$}$ (tentatively denoted as $\psi$) for $M=V_g$ and $\partial M=\partial V_g=W_g$, we use the map of fibre sequences
\[
\begin{tikzcd}
\Map_*(W_g,V_g\vee S^{2n})\arrow[d,"(\pr_1\circ (-))\eqcolon\psi"]\arrow[r]&\Map_*(W_{g,1},V_g\vee S^{2n})\dar{(\pr_1\circ (-))}\rar&\Map_*(S^{2n-1},V_g\vee S^{2n})\dar{(\pr_1\circ (-))}\\
\Map_*(W_g,V_g)\rar&\Map_*(W_{g,1},V_g)\rar&\Map_*(S^{2n-1},V_g).
\end{tikzcd}
\]
induced by the cofibre sequence $S^{2n-1}\ra W_{g,1}\ra W_g$. Using \cref{lem:boundary-restriction-on-mapping-spaces} for $M=W_{g,1}$ and the splittings of the vertical arrows induced by the inclusion $V_g\subset V_{g}\vee S^{2n-1}$, we see that the short exact sequence of the fibration obtained by taking vertical homotopy fibres at the inclusion $\iota_\partial\colon W_g\hookrightarrow V_g$ is of the form 
\vspace{-0.1cm}
\[
\ldots\lra \pi_*(\hofib(\psi);\iota_\partial)_\bfQ\lra s^{-(2n-2)}\widetilde{\bfL}(H_V\oplus x)\otimes H_W\xlra{[-,\iota(-)]}s^{-(2n-2)}\widetilde{\bfL}(H_V\oplus x)\lra \ldots,
\]
where $\iota\colon H_W\ra H_V$ is the canonical projection and $x$ is of degree $2n-1$. As $V_g\simeq \vee^gS^n$, the homotopy groups of the homotopy fibre of the left vertical arrow of $\circled{$2$}'$ are given by 
\[\widetilde{\bfL}(H_V\oplus x)\otimes H_{V}^\vee\cong s^{-(2n-2)}\widetilde{\bfL}(H_V\oplus x)\otimes K_W.\]
 Mapping in these groups into the previous long exact sequence and using $H_{W}/K_W\cong H_V$, we arrive at a long exact sequence 
 \vspace{-0.1cm}
\[
\ldots\lra \pi_{*-1}(\tohofib(\circled{$2$}'))_\bfQ\lra s^{-(2n-2)}\widetilde{\bfL}(H_V\oplus x)\otimes H_V\xlra{[-,-]}s^{-(2n-2)}\widetilde{\bfL}(H_V\oplus x)\lra \ldots.
\]
Omitting degree shifts, we see from \cref{lem:decomposition} that the bracket is of the form
\[
\begin{tikzcd}[column sep=-0.2cm]
{\big(x\otimes H_V\big)}&\oplus&{\big([x,H_V]\otimes H_V\big)}&\oplus&{ \big([H_V,[x,H_V]]\otimes H_V\big)}\dar{[-,-]}& \oplus &{\big([x,x]\otimes H_V\big) }&&\\
{[x,H_V]}&\oplus&{[[x,H_V],H_V]}&\oplus &{ [[H_V,[x,H_V]],H_V]}&\oplus&{ [[x,x],H_V]}&\oplus&{[x,x]}
\end{tikzcd}
\]
 in degrees $1<*< 4n-3$. In addition to \cref{lem:decomposition}, this uses $[H_V,[x,x]]=[x,[x,H_V]]$ which follows from an application of the Jacobi identity. As for $W_{g,1}$ we combine this with the dimension formulas in \cref{lem:decomposition} to conclude $\pi_*(\tohofib(\circled{$2$}'))_\bfQ\cong s^{-2n}[x,x]$ for $0<*< 4n-5$. From \eqref{equ:homotopy-tohofib3-rational}, we see $\pi_*(\tohofib(\circled{$3$}))_\bfQ\cong \widetilde{\bfL}(H_V\oplus x)=x\oplus [x,H_V]$ for $*< 4n-4$, so using \eqref{equ:modified-square} the long exact sequence of \eqref{equ:lestohofib} has in the range $0<*< 4n-5$ the form
\[
\ldots  \lra s^{-1}(x\oplus [x,H_V]) \lra s^{-2n}[x,x] \lra\pi_*(\tohofib(\circled{$1$}))_\bfQ\lra\ldots.
\]
By a similar argument as for $W_{g,1}$, one shows that the first map in this sequence maps $s^{-1}x$ isomorphically to $s^{-2n}[x,x]$, so we deduce $\pi_*(\tohofib(\circled{$1$}))_\bfQ\cong [x,H_V]$ for $0<*< 4n-5$. As $\SO/\SO(2n+1)$ is rationally $(4n+2)$-connected, we can combine this computation with  \eqref{equ:hAut-as-tohofib} and \cref{prop:identification-block-point-embeddings} to conclude that $\pi_*(\Emb^{(\sim)}(*,V_g))_\bfQ\cong [x,H_V]\cong\bfQ^g[3n-2]$ for $*<4n-5$, which proves \cref{thm:blockemb-point} for $V_g$.

\subsection{Case III: $M=W_{g,1}\natural  T_h$}\label{sec:WgTh-block-computation}To ease the notation, we abbreviate $M\coloneq W_{g,1}\natural  T_h$ and write
 \vspace{-0.1cm}
\[
\overline{\Map}_A(X,M\vee S^{2n-1})\coloneq \hofib\big(\Map_A(X,M\vee S^{2n-1})\xra{\pr \circ-} \Map_A(X,M)\big)
\] for a space $X$ together with a subspace $A\subset X$; the fixed map $A\ra M\vee S^{2n-1}$ and the basepoints will be clear from the context. Restriction along the system of inclusions
\[
\begin{tikzcd}[row sep=-0.2cm]
&U_{h,1}\arrow[dr,hook]&\\
S^{2n-2}\arrow[ur,hook]\arrow[dr,hook]&&W_{g,1}\natural  T_h=M\\
&D^{2n-1}\arrow[ur,hook]&
\end{tikzcd}, 
\]
where the left arrows are boundary inclusions, $D^{2n-1}=\partial W_{g,1}\cap \partial(W_{g,1}\natural  T_h)$, and $U_{h,1}=\partial T_h\cap \partial (W_{g,1}\natural T_h)$ induces a commutative diagram of horizontal and vertical fibre sequences
\[
\begin{tikzcd}[column sep=0.4cm,row sep=0.4cm]
\tohofib(\circled{2}')\dar\rar&\overline{\Map}_{D^{2n-1}}(M,M\vee S^{2n-1})\arrow[d]\rar&\overline{\Map}_*(U_h,M\vee S^{2n-1})\dar\\
\overline{\Map}_{U_{h,1}}(M,M\vee S^{2n-1})\dar\rar&\overline{\Map}_*(M,M\vee S^{2n-1})\dar\rar&\overline{\Map}_*(U_{h,1},M\vee S^{2n-1})\dar\\
\overline{\Map}_{S^{2n-2}}(D^{2n-1},M\vee S^{2n-1})\rar&\overline{\Map}_*(D^{2n-1},M\vee S^{2n-1})\rar&\overline{\Map}_*(S^{2n-2},M\vee S^{2n-1}).
\end{tikzcd}
\]
The long exact sequence of the middle row has the form
\[
\ldots \ra \pi_*(\overline{\Map}_{U_{h,1}}(M,M\vee S^{2n-1}))_{\bfQ}\ra \widetilde{\bfL}(H_{WT}\oplus x)\otimes H_{WT}^\vee\ra \widetilde{\bfL}(H_{WT}\oplus x)\otimes H_U^\vee\ra \ldots,
\]
where the right map is induced by the map $H_U\ra H_{WT}$ induced by inclusion. The image of the latter is $H_T\subset H_{WT}$, its kernel is $K_U$, and its cokernel is $H_W$, so the long exact sequence yields a short exact sequence of the form
\begin{equation}\label{equ:intermediate-ses-Map-Uh}
0\ra s^{-1} \widetilde{\bfL}(H_{WT}\oplus x)\otimes K_U^\vee \ra \pi_*(\overline{\Map}_{U_{h,1}}(M,M\vee S^{2n-1}))_{\bfQ}\ra  \widetilde{\bfL}(H_{WT}\oplus x)\otimes H_{W}^\vee\ra 0.
\end{equation}
Extension by the identity along the inclusion $W_{g,1}\subset W_{g,1}\natural T_h=M$ induces a map 
\[\overline{\Map}_{D^{2n-1}}(W_{g,1},M\vee S^{2n-1})\lra \overline{\Map}_{U_{h,1}}(M,M\vee S^{2n-1})\] (here $D^{2n-1}\subset \partial W_{g,1}$ is the disc at which the boundary connected sum $W_{g,1}\natural T_h$ is taken) which on homotopy groups induces a splitting of \eqref{equ:intermediate-ses-Map-Uh} which combined with the isomorphisms $K_U^\vee\cong s^{-(2n-3)}H_T$ and $H_W^\vee\cong s^{-(2n-2)}H_W$ provides an isomorphism
\[
\pi_*(\overline{\Map}_{U_{h,1}}(M,M\vee S^{2n-1}))_{\bfQ}\cong  s^{-(2n-2)}\widetilde{\bfL}(H_{WT}\oplus x)\otimes H_{WT}.
\]
The long exact sequence of the leftmost column therefore has the form
\begin{equation}\label{equ:les-tohofib2-WgTh}
\ldots\ra\pi_*(\tohofib(\circled{2}'))_\bfQ\ra s^{-(2n-2)}\widetilde{\bfL}(H_{WT}\oplus x)\otimes H_{WT}\to s^{-(2n-2)}\widetilde{\bfL}(H_{WT}\oplus x)\ra \ldots,
\end{equation}
and we claim that the right map is given by bracketing. On the subspace $s^{-(2n-2)}\widetilde{\bfL}(H_{WT}\oplus x)\otimes H_T$, this map is a quotient of the induced map on $\pi_{*+1}(-)_\bfQ$ of the lower right vertical arrow in the above diagram, which is given by $[-,\iota(-)]\colon s^{-(2n-3)}\widetilde{\bfL}(H_{WT}\oplus x)\otimes H_U\ra\widetilde{\bfL}(H_{WT}\oplus x)$ by \cref{lem:boundary-restriction-on-mapping-spaces}. Secondly, on the subspace $s^{-(2n-2)}\widetilde{\bfL}(H_{WT}\oplus x)\otimes H_W$, this map is the induced map on homotopy groups by the composition
\[
\overline{\Map}_{D^{2n-1}}(W_{g,1},M\vee S^{2n-1})\lra \overline{\Map}_{U_{h,1}}(M,M\vee S^{2n-1})\lra \overline{\Map}_{D^{2n-1}}(S^{2n-1},M\vee S^{2n-1}),
\]
which is induced by restricting along the boundary inclusion $S^{2n-1}\subset W_{g,1}$, so is given by $[-,\iota(-)]\colon s^{-(2n-2)}\widetilde{\bfL}(H_{WT}\oplus x)\otimes H_W\ra\widetilde{\bfL}(H_{WT}\oplus x)$ by another application of \cref{lem:boundary-restriction-on-mapping-spaces}.

Using the exact sequence \eqref{equ:les-tohofib2-WgTh}, the rest of the argument proceeds as for $W_{g,1}$ in 
 \cref{section:computation-Wg} and eventually shows that $\Emb^{(\sim)}(*,W_{g,1}\natural  T_h)$ is rationally $(4n-6)$-connected.

\section{Claspers}\label{sec:clasper}
By Theorem \ref{thm:fibre-to-haut}, the dimensions of the vector spaces (see \cref{sec:gradings} for the notation) \[\pi_{2n-2}(\hAut^{\cong,\ell}_\partial(V_g)/\Emb^{\sfr,\cong}_{\nicefrac{1}{2}\partial}(V_g)_\ell)_\bfQ\ \ \text{and}\ \  S_{2}(H_{V_g})_{2(n-1)}=\begin{cases}
\Lambda^2(\bar{H}_{V_g})&n\text{ even}\\
\Sym^2(\bar{H}_{V_g})&n\text{ odd}
\end{cases}\] 
agree for $n>3$. Both sides are acted upon by the semi-direct product of the group \[\check{\Lambda}_{V} \coloneq \pi_1(\BEmb^{\fr, \cong}_{\nicefrac{1}{2} D^{2n}}(V_g,W_{g,1})_\ell)\cong \pi_1(\BEmb^{\sfr, \cong}_{\nicefrac{1}{2} D^{2n}}(V_g,W_{g,1})_\ell)\] acted upon by the reflection involution $\langle \rho\rangle\cong \bfZ/2$ (see Sections \ref{sec:MCG} and \ref{sec:Reflection}); the action on $S_{2}(H_{V_g})$ is trivial since $\rho$ fixes $H_{V_g}$. This section serves to improve this identification of the underlying vector spaces to an equivariant statement. 

\begin{thm}\label{thm:ClasperMain}
For $n>3$, there is an isomorphism of $\bfQ[\check{\Lambda}_{V}\rtimes \langle\rho\rangle]$-modules
\[
\pi_{2n-2}\big(\hAut^{\cong,\ell}_\partial(V_g)/\Emb^{\sfr,\cong}_{\nicefrac{1}{2}\partial}(V_g)_\ell\big)_\bfQ\cong S_{2}(H_{V_g})_{2(n-1)}.
\]
\end{thm}

The arguments in Section \ref{sec:Disj} break the symmetry of the manifold $V_g$ in an essential way, so we will not try to trace the equivariance through that section. Rather, we shall use a variant of Watanabe's parametrised form \cite{WatanabeII} of the Goussarov--Habiro theory of claspers to produce an equivariant map from the right to the left term in \cref{thm:ClasperMain}, use Watanabe's results to deduce that this map is non-zero, and then argue that it has to be an isomorphism.

\begin{conv}\label{equ:clasper-convention}To simplify the notation, we adopt the following conventions:
\begin{enumerate} 
\item We fix an integer $n\ge3$; the assumption $n>3$ will only play a role at the end. 
\item We consider $V_g$ as a submanifold of $D^{2n+1}$ by intersecting the standard model $V_g\subset \bfR^{2n+1}$ of \cref{sec:standard-model} with an  oriented disc $D^{2n+1}\subset \bfR^{2n+1}$ with $V_g$ in its interior. 
\item As usual, we consider $V_g$ and all its submanifolds as framed via the standard framing from \cref{sec:standard-framing} which we omit from the notation.
\item We use the basis $f_1,\ldots,f_g\in H^\bfZ_{V_g}$ induced by the cores of the handles from \cref{sec:standard-model}.
\end{enumerate}
\end{conv}

We begin with a brief recollection on framed embedding spaces.

\subsection{Framed embeddings}\label{sec:framed-embeddings}
Given $d$-manifolds $M$ and $N$ with fixed framings $\ell_M$ and $\ell_N$, the space of \emph{framed embeddings} $\Emb^{\fr}((M,\ell_M),(N,\ell_N))$ is the homotopy fibre 
\[\Emb^{\fr}((M,\ell_M),(N,\ell_N))\coloneq \hofib_{\ell_M}\big(\Emb(M,N)\xra{(-)^*\ell_N}\Bun(\tau_M,\varepsilon^{d} \big)\big).\]
of the map given by taking the derivative followed by postcomposition with the framing $\ell_N$. The usual composition of embeddings extends to a composition map 
\[\Emb^{\fr}((M,\ell_M),(N,\ell_N))\times \Emb^{\fr}((N,\ell_N),(K,\ell_K)) \lra \Emb^{\fr}((M,\ell_M),(K,\ell_K))\]
which is unital and associative if one uses Moore-paths in the definition of the homotopy fibres, so it gives rise to a topologically enriched category of framed $d$-manifolds and framed embeddings. In particular $\Emb^{\fr}(M,\ell_M)\coloneq \Emb^{\fr}((M,\ell_M),(M,\ell_M))$ is a topological monoid. Replacing embeddings by diffeomorphisms in this definition yields a group-like monoid $\Diff^{\fr}(M,\ell_M)$ which is related to $\BDiff^\fr(M)$ as defined in \cref{sec:tangential-structures} by an equivalence $\Diff^{\fr}(M,\ell_M)\simeq \Omega \BDiff^\fr(M)$ where the loop space is based at the fixed framing $\ell_M$.

Basing $\BDiff^\fr_\partial(M)$ and $\BDiff^\fr_\partial(N)$ at $\ell_M$ and $\ell_N$, extension by the identity induces a map
\[\pi_0(\Emb^\fr(M,N))\lra [\BDiff^\fr_\partial(M),\BDiff^\fr_\partial(N)]_*\] where $[-,-]_*$ denotes based homotopy classes\footnote{This map is not strictly base point preserving, but the image of the basepoint of the source is connected to the basepoint by a preferred path, and this is good enough.}. This is functorial, so in particular $\pi_0(\Emb^\fr(M,M))$ acts on $\BDiff^\fr_\partial(M)$ in the based homotopy category. Restricting this action along the map \[\pi_1(\BDiff^\fr(M),\ell_M)\cong \pi_0(\Diff^\fr(M,\ell_M))\lra \pi_0(\Emb^\fr(M,\ell_M))\] agrees with the action of $\pi_1(\BDiff^\fr(M),\ell_M)$ on $\BDiff^\fr_\partial(M)$ induced by conjugation.

\subsubsection{Framed embeddings between $V_g$}
We will occasionally need to produce framed embeddings of the type $V_a \sqcup V_b \hookrightarrow V_g$, and it will be convenient to do so by indicating where (thickenings of) the cores of the handles go and ignoring the rest. Let $U_g = \sqcup^g V_1$, so that $V_g$ is obtained from $U_g$ by attaching $(g-1)$ 1-handles: more precisely, using the notation of \cref{sec:standard-model} we write $U_g \coloneq \psi(\sqcup^g V_1) \subset V_g$, and denote $\inc_i : V_1 \hookrightarrow V_3$ for the inclusion of the $i$th thickened core.

\begin{lem}\label{lem:EmbeddingsAndCores}
For $n \geq 3$ the map
\[\textstyle{\pi_0\Emb^{\fr}(V_a \sqcup V_b, V_g) \lra \pi_0\Emb(U_a \sqcup U_b, V_g)},\]
induced by precomposition with $\bigsqcup_{i=1}^a \mathrm{inc}_i \sqcup \bigsqcup_{j=1}^b \mathrm{inc}_j: U_a \sqcup U_b \hookrightarrow V_a \sqcup V_b$, is a bijection.
\end{lem}
\begin{proof}
By the parametrised isotopy extension theorem the fibre of the restriction map over an embedding $e : U_a \sqcup U_b \to \mathrm{int}(V_g)$ is equivalent to $\Emb^\fr_\partial(\sqcup^{a-1+b-1} [0,1] \times \bfR^{2n}, V_g \setminus \mathrm{int}(\im(e)))$. This is $(n-2)$-connected  by general position, so $1$-connected if $n \geq 3$.
\end{proof}

\subsection{Splitting the Weiss fibre sequence}
The proof of \cref{thm:ClasperMain} makes use of the framed Weiss fibre sequence $\BDiff^{\fr}_\partial(D^{2n+1}) \ra \BDiff^{\fr}_{\partial}(V_g) \ra \BEmb^{\fr, \cong}_{\half \partial}(V_g)$ from \cref{sec:weiss-fs}, which has the following convenient property that distinguishes it from the Weiss fibre sequence for $W_{g,1}$: the framed standard embedding $V_g \subset D^{2n+1}$ splits the inclusion of the fibre up to equivalence, so it induces an equivalence \begin{equation}\label{eq:FrWeissSplit}
\BDiff^{\fr}_{\partial}(V_g) \overset{\simeq}\lra \BDiff^{\fr}_\partial(D^{2n+1}) \times \BEmb^{\fr, \cong}_{\half \partial}(V_g).
\end{equation}
This is natural up to homotopy with respect to framed embeddings $V_g\hookrightarrow V_h$ that commute with the standard framed embedding into $D^{2n+1}$ up to isotopy of framed embeddings.

We shall be particularly interested in the induced splitting on $\pi_{n-1}(-)_\bfQ$,
\begin{equation}\label{equ:splitting-homotopy-group}
\pi_{n-1}(\BDiff^{\fr}_{\partial}(V_g))_\bfQ\cong \pi_{n-1}(\BDiff^{\fr}_\partial(D^{2n+1}))_\bfQ \oplus \pi_{n-1}(\BEmb^{\fr, \cong}_{\half \partial}(V_g))_{\bfQ},
\end{equation}
which one can further simplify by using the isomorphisms
\[\pi_{n-1}(\BEmb^{\fr, \cong}_{\half \partial}(V_g))_\bfQ\xlra{\cong} \pi_{n-1}(\BEmb^{\sfr, \cong}_{\half \partial}(V_g))_\bfQ \xlra{\cong} \pi_{n-1}(\BhAut_\partial(V_g))_\bfQ\]
resulting from \cref{lem:connectivity-framing-comparison} and \cref{thm:fibre-to-haut}. Furthermore, \cref{cor:lowest-htp-grp-haut-v} gives an isomorphism 
\begin{equation}\label{equ:iso-first-haut-group}
\kappa_V\colon \pi_{n-1}(\BhAut_\partial(V_g))_\bfQ \xlra{\cong} S_{1^3}(H_{V_g})_{3(n-1)}=\begin{cases}
 \Lambda^3(\bar{H}_{V_g})&n\text{ odd}\\
 \Sym^3(\bar{H}_{V_g})&n\text{ even}\end{cases}\end{equation}
which is $\pi_0\hAut_{D^{2n}}(V_g,W_{g,1})$-equivariant and natural with respect to orientation-preserving embeddings $V_g\hookrightarrow V_h$ by \cref{lem:lowest-htp-grp-haut-v-naturality}. The splitting \eqref{equ:splitting-homotopy-group} thus becomes
\begin{equation}\label{equ:splitting-homotopy-group2}
\pi_{n-1}(\BDiff^{\fr}_{\partial}(V_g))_\bfQ\cong \pi_{n-1}(\BDiff^{\fr}_\partial(D^{2n+1}))_\bfQ \oplus S_{1^3}(H_{V_g})_{3(n-1)}
\end{equation}
and is natural with respect to framed embeddings $V_g\hookrightarrow V_h$ that commute with the standard framed embedding into $D^{2n+1}$ up to isotopy of framed embeddings.

\subsection{Watanabe's Borromean construction}
In \cite{WatanabeII} Watanabe has constructed classes in the homotopy groups $\pi_{k(n-1)}(\BDiff^{\fr}_\partial(D^{2n+1}))$ for $k\ge2$ and $n\ge3$ indexed by trivalent graphs with $k$ vertices (and some extra data). The basic building block of his construction is a class \begin{equation}\label{equ:alpha-wat}\alpha_{\text{Wat}} \in \pi_{n-1}(\BDiff^{\fr}_\partial(V_3))\end{equation}
which he produces (see Section 4.1 loc.cit.) by first constructing an explicit $V_3$-bundle over $S^{n-1}$ using a high-dimensional version of the Borromean rings, giving rise to $\alpha_{\text{Bor}} \in \pi_{n-1}(\BDiff_\partial(V_3))$ (denoted $\delta(\alpha')$ on p.\ 638 loc.cit.), and then showing by obstruction theory that, as long as $n$ is odd, this class may be lifted along $\BDiff^{\fr}_\partial(V_3)\ra \BDiff_\partial(V_3) $ after perhaps scaling by some non-zero integer (see Proposition 4.2 loc.cit.).

\subsubsection{The Borromean property}The crucial feature of the class $\alpha_{\text{Bor}}$, which can be seen directly from Watanabe's construction (see in particular the proof of Lemma 4.3 loc.\ cit.), is the following ``Borromean property'': attaching a handle to any of the three summands of $V_3 = \natural^3 S^n \times D^{n+1}$ along $S^n \times D^n_{-} \subset S^n \times S^n$ where $D^n_-\subset D^n$ is a hemisphere gives ``handle-filling'' embeddings
\begin{equation}\label{equ:handle-filing}n_i \colon V_3 \longhookrightarrow V_2\quad\text{for } i=1,2,3,\end{equation}
and under any of the induced maps $(n_i)_* \colon \pi_{n-1}(\BDiff_\partial(V_3)) \to \pi_{n-1}(\BDiff_\partial(V_2))$ the class $\alpha_{\text{Bor}}$ is sent to zero. In the corrigendum \cite{WatanabeIIerr} Watanabe shows that there is a choice of lift \eqref{equ:alpha-wat} to the classifying space for \emph{framed} smooth bundles which still satisfies this Borromean property, and he also removes the assumption that $n$ be odd (see Lemma A and Remark 6 loc.cit.). 

For our purposes, it is convenient to work with the following axiomatic description of Watanabe's building block \eqref{equ:alpha-wat}. This alternative description is slightly surprising, as it is by definition symmetric in the three handles of $V_3$, whereas Watanabe's construction favours one.

\begin{prop}\label{prop:alphaExists}
There exists a unique class $\alpha \in \pi_{n-1}(\BDiff^{\fr}_\partial(V_3))_\bfQ$
such that
\begin{enumerate}
\item\label{it:alphaExists2} under each of the handle-filling maps
\[(n_i)_*: \pi_{n-1}(\BDiff^{\fr}_\partial(V_3))_\bfQ \lra \pi_{n-1}(\BDiff^{\fr}_\partial(V_2))_\bfQ\quad\text{for }i=1,2,3\]
the class $\alpha$ is sent to zero, and
\item\label{it:alphaExists1} under the composition
\[
 \pi_{n-1}(\BDiff^{\fr}_\partial(V_3))_\bfQ\ra \pi_{n-1}(\BhAut_\partial(V_3))_\bfQ\xra{\kappa_V} S_{1^3}(H_{V_3})_{3(n-1)}=\begin{cases}
 \Lambda^3(\bar{H}_{V_g})&n\text{ odd}\\
 \Sym^3(\bar{H}_{V_g})&n\text{ even}\end{cases},
\]
the class $\alpha$ maps to $[f_1 \otimes f_2 \otimes f_3]$.
\end{enumerate}
Moreover, $\alpha$ agrees up to a nonzero scalar with $\alpha_{\text{Wat}}\in\pi_{n-1}(\BDiff^{\fr}_\partial(V_3))_\bfQ$ from \cite[Lem A]{WatanabeIIerr}.
\end{prop}

\begin{proof}
Our candidate for $\alpha$ is the class corresponding to $(0,[f_1\otimes f_2\otimes f_3])$ under the splitting \eqref{equ:splitting-homotopy-group2}. This satisfies \ref{it:alphaExists1} by construction. Property \ref{it:alphaExists2} follows from the naturality of the splitting \eqref{equ:splitting-homotopy-group2} with respect to the $n_i$'s together with the observation that the induced maps
\begin{equation}\label{eq:contraction}
(n_i)_* \colon S_{1^3}(H_{V_3})_{3(n-1)} \lra S_{1^3}(H_{V_2})_{3(n-1)}
\end{equation}
all annihilate $[f_1 \otimes f_2 \otimes f_3] \in S_{1^3}(H_{V_3})_{3(n-1)}$. To prove uniqueness, let $\bar{\alpha}$ be another class satisfying \ref{it:alphaExists2} and \ref{it:alphaExists1}, and $\bar{\alpha}=(\bar{\alpha}_1,\bar{\alpha}_2)$ its splitting with respect to \eqref{equ:splitting-homotopy-group2}. By \ref{it:alphaExists1}, we must have $\bar{\alpha}_2=[f_1\otimes f_2\otimes f_3]$, so we only need to show that $\bar{\alpha}_1$ vanishes. This follows from \ref{it:alphaExists2} and the fact that the inclusion $V_g\subset D^{2n+1}$ factors up to isotopy over the handle-filling embeddings $n_i$. 

To show the final part of the claim, we decompose $\alpha_{\text{Wat}}$ with respect to the splitting \eqref{equ:splitting-homotopy-group2} as $\alpha_{\text{Wat}}=(\nu,\beta)$ for some $\smash{\nu\in \pi_{n-1}(\BDiff^{\fr}_\partial(D^{2n+1}))_\bfQ}$ and $\beta\in S_{1^3}(H_{V_3})_{{3(n-1)}}$. By construction of the splitting, $\nu$ is the image of $\alpha_{\text{Wat}}$ under the map $\BDiff^{\fr}_{\partial}(V_3) \ra \BDiff^{\fr}_\partial(D^{2n+1})$ induced by the framed standard embedding $V_3\subset D^{2n+1}$. Since the latter factors over any of the handle-filling embeddings \eqref{equ:handle-filing} and $\alpha_{\text{Wat}}$ satisfies the Borromean property \cite[Lemma A]{WatanabeIIerr}, it follows that $\nu$ is trivial, so $\alpha_{\text{Wat}}=(0,\beta)$. As $\alpha_{\text{Wat}}$ is non-trivial \cite[Corollary 5.4]{WatanabeII}, $\beta \in S_{1^3}(H_{V_3})_{{3(n-1)}}$ must be nontrivial, and by the Borromean property of $\alpha_{\text{Wat}}$, the nontrivial class $\beta \in S_{1^3}(H_{V_3})_{{3(n-1)}}$ is contained in the common kernel of the maps \eqref{eq:contraction}. But this common kernel is one-dimensional and generated by $[f_1\otimes f_2\otimes f_3]$, so $\beta$ must indeed be a multiple of this class.
\end{proof}

\begin{rem}
The final part of the proof of \cref{prop:alphaExists} shows that Property \ref{it:alphaExists2}  already characterises $\alpha$ up to a non-zero scalar, so Property \ref{it:alphaExists1} is merely a normalisation.
\end{rem}

\subsection{The construction}\label{equ:the-construction}We begin the construction of the isomorphism in \cref{thm:ClasperMain} by choosing once and for all a framed embedding
\[b \colon V_3 \sqcup V_3 \hookrightarrow V_2\] as depicted in \cref{fig:Clasper1}. More precisely, this is constructed as follows: inside the standard model
\[V_2 \cong \big(S^n \times D^{n+1} \big) \cup ([0,1] \times D^{2n}) \cup \big(S^n \times D^{n+1}\big)\]
of \cref{sec:standard-model} there are disjoint framed embeddings $V_1 \hookrightarrow V_2$ given by the two copies of $S^n \times \tfrac{1}{2}D^{n+1}$, and inside two disjoint discs away from these we can form Hopf linked framed embeddings $V_1 \sqcup V_1 \hookrightarrow D^{2n+1}$, where each individual embedding $V_1 \hookrightarrow D^{2n+1}$ is framed isotopic to the tautological framed embedding.  Appealing to \cref{lem:EmbeddingsAndCores}, these are then joined together as shown in \cref{fig:Clasper1}.

As further piece of terminology, we say that a framed embedding $\varphi\colon V_3\hookrightarrow V_g$ has a \emph{null handle} if it factors through $V_2$ via one of the handle-filling embeddings $n_i$ in \eqref{equ:handle-filing}, up to isotopy of framed embeddings. Note that the restrictions of $b$ to both copies of $V_3$ are of this form. By the Borromean property of \cref{prop:alphaExists} \ref{it:alphaExists2}, it follows that the class $\varphi_*(\alpha)\in \pi_{n-1}(\BDiff^\fr_\partial(V_g))$ obtained from $\alpha$ by extension along a framed embedding $\varphi$ with a null-handle vanishes.

\begin{figure}
\begin{center}
\centering{
\resizebox{10.5cm}{!}{
\begingroup%
  \makeatletter%
  \providecommand\color[2][]{%
    \errmessage{(Inkscape) Color is used for the text in Inkscape, but the package 'color.sty' is not loaded}%
    \renewcommand\color[2][]{}%
  }%
  \providecommand\transparent[1]{%
    \errmessage{(Inkscape) Transparency is used (non-zero) for the text in Inkscape, but the package 'transparent.sty' is not loaded}%
    \renewcommand\transparent[1]{}%
  }%
  \providecommand\rotatebox[2]{#2}%
  \newcommand*\fsize{\dimexpr\f@size pt\relax}%
  \newcommand*\lineheight[1]{\fontsize{\fsize}{#1\fsize}\selectfont}%
  \ifx\svgwidth\undefined%
    \setlength{\unitlength}{362.4168408bp}%
    \ifx\svgscale\undefined%
      \relax%
    \else%
      \setlength{\unitlength}{\unitlength * \real{\svgscale}}%
    \fi%
  \else%
    \setlength{\unitlength}{\svgwidth}%
  \fi%
  \global\let\svgwidth\undefined%
  \global\let\svgscale\undefined%
  \makeatother%
  \begin{picture}(1,0.39460563)%
    \lineheight{1}%
    \setlength\tabcolsep{0pt}%
    \put(0,0){\includegraphics[width=\unitlength,page=1]{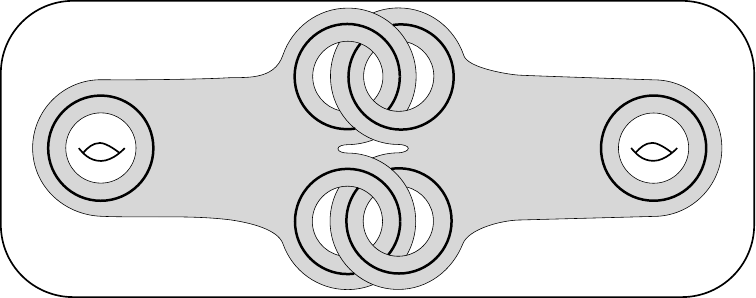}}%
    \put(0.21508299,0.19692724){\makebox(0,0)[lt]{\lineheight{1.25}\smash{\begin{tabular}[t]{l}$f_3$\end{tabular}}}}%
    \put(0.75908056,0.19645361){\makebox(0,0)[lt]{\lineheight{1.25}\smash{\begin{tabular}[t]{l}$f_3$\end{tabular}}}}%
    \put(0.35852678,0.25751827){\makebox(0,0)[lt]{\lineheight{1.25}\smash{\begin{tabular}[t]{l}$f_2$\end{tabular}}}}%
    \put(0.60804765,0.25836494){\makebox(0,0)[lt]{\lineheight{1.25}\smash{\begin{tabular}[t]{l}$f_2$\end{tabular}}}}%
    \put(0.36083678,0.11720455){\makebox(0,0)[lt]{\lineheight{1.25}\smash{\begin{tabular}[t]{l}$f_1$\end{tabular}}}}%
    \put(0.60669933,0.1178787){\makebox(0,0)[lt]{\lineheight{1.25}\smash{\begin{tabular}[t]{l}$f_1$\end{tabular}}}}%
  \end{picture}%
\endgroup%
}}
\caption{The basic pattern}\label{fig:Clasper1}
\end{center}
\end{figure}

The embedding $b$ induces by extension with the identity a map \[b \colon \BDiff^{\fr}_\partial(V_3) \times \BDiff^{\fr}_\partial(V_3) \lra \BDiff^{\fr}_\partial(V_2)\] which we combine with the class from \cref{prop:alphaExists} to define a class
\[
\alpha^{(2)} \coloneq b_*(\alpha \otimes \alpha) \in \oH_{2(n-1)}(\BDiff^{\fr}_\partial(V_2)_\ell;\bfQ).\]

\begin{lem}
$\alpha^{(2)}\in \oH_{2(n-1)}(\BDiff^{\fr}_\partial(V_2)_\ell;\bfQ)$ lies in the image of the rational Hurewicz map.
\end{lem}
\begin{proof}As the restrictions of $b$ to the two copies of $V_3$ have null handles, the restriction of \[b \circ (\alpha \times \alpha): S^{n-1} \times S^{n-1} \ra \BDiff^{\fr}_\partial(V_2)_\bfQ\] along $S^{n-1} \vee S^{n-1} \subset S^{n-1} \times S^{n-1}$ is nullhomotopic. Choosing a nullhomotopy induces a class in $\pi_{2(n-1)}(\BDiff^{\fr}_\partial(V_2))_\bfQ$ whose Hurewicz image in $\oH_{2(n-1)}(\BDiff^{\fr}_\partial(V_2)_\ell;\bfQ)$ is the class $\alpha^{(2)}$.
\end{proof}

In the following, we abbreviate the image of the rational Hurewicz map by
\[h_i(-)_\bfQ \coloneq \im\big(\pi_i(-)_\bfQ \to \oH_i(-;\bfQ) \big)\quad\text{for }i>1.\]
Given a framed embedding $\varphi \colon V_2 \hookrightarrow V_g$ we say that $(\varphi \circ b) \colon V_3 \sqcup V_3 \hookrightarrow V_g$
is the \emph{pattern associated to $e$}. Using the previous lemma, this induces a class \[\varphi_*(\alpha^{(2)}) \in h_{2(n-1)}(\BDiff^{\fr}_\partial(V_g))_\bfQ,\] and this gives rise to a function
\[c_{\Diff}\colon \pi_0(\Emb^{\fr}(V_2, V_g)) \lra h_{2(n-1)}(\BDiff^{\fr}_\partial(V_g)_\ell)_\bfQ\]
which is equivariant with respect to the action of $\pi_1(\BDiff^\fr(V_g))\cong \pi_0\Diff^\fr(V_g,\ell_{V_g})$ by postcomposition in the source and conjugation in the target (see \cref{sec:framed-embeddings}). From the discussion at the beginning of this section, it follows that $c_{\Diff}$ vanishes on framed embeddings with null handles. Restricting the action along the morphism induced by relaxing the boundary condition
 \[\check{\Lambda}_V \rtimes \langle \rho\rangle \cong \pi_1(\BDiff^\fr_{D^{2n}}(V_g)) \rtimes \langle \rho\rangle \lra \pi_1(\BDiff^\fr(V_g))\] 
 (see the proof of \cref{lem:finite-kernel} for the first isomorphism) and forgetting from diffeomorphisms to embeddings, the function $c_{\Diff}$ yields a $(\check{\Lambda}_V\rtimes\langle \rho\rangle)$-equivariant function
\begin{equation}\label{equ:c-function}c\colon \pi_0(\Emb^{\fr}(V_2, V_g)) \lra h_{2(n-1)}(\BEmb^{\fr,\cong}_{\half\partial}(V_g)_\ell)_\bfQ.\end{equation} This function is the subject of the following proposition central to the proof of \cref{thm:ClasperMain}.
\begin{prop}\label{prop:ClasperMain}\ 
\begin{enumerate}[(i)]
\item\label{it:ClasperMain1} The $(\check{\Lambda}_V\rtimes \langle \rho\rangle)$-equivariant function
\[
h\colon \pi_0(\Emb^{\fr}(V_2, V_g)) \lra \bar{H}^\bfZ_{V_g} \times \bar{H}^\bfZ_{V_g}
\]
that sends $\varphi$ to $(\varphi_*(f_1), \varphi_*(f_2))$ is surjective.
\item\label{it:ClasperMain2} The function $c$ from \eqref{equ:c-function} factors as a composition
\[\pi_0(\Emb^{\fr}(V_2, V_g)) \xlratwohead{h} \bar{H}^\bfZ_{V_g} \times \bar{H}^\bfZ_{V_g}\xlra{\bar{c}} h_{2(n-1)}(\BEmb^{\fr}_{\half\partial}(V_g)_\ell)_\bfQ.\]
 \item\label{it:ClasperMain3} The map $\bar{c}$ is bilinear and $(-1)^{n+1}$-symmetric, so induces a $\bfQ[\check{\Lambda}_V\rtimes\langle \rho\rangle]$-module map
 \[S_{2}(H_{V_g})_{2(n-1)}\lra h_{2(n-1)}(\BEmb^{\fr}_{\half\partial}(V_g)_\ell)_\bfQ.\]
\end{enumerate}
\end{prop}

Before proving \cref{prop:ClasperMain}, we explain how it can be used to accomplish the goal of this section, proving \cref{thm:ClasperMain}. In addition to \cref{prop:ClasperMain} and \cref{thm:fibre-to-haut}, this relies on Watanabe's main result of \cite{WatanabeII}.

\subsection{Proof of Theorem \ref{thm:ClasperMain}}
Assume $n>3$. Proposition \ref{prop:ClasperMain} gives $\bfQ[\check{\Lambda}_{V}\rtimes \langle \rho\rangle]$-module maps
\[\Phi: S_2(H_{V_g})_{2(n-1)} \lra h_{2n-2}(\BEmb_{\half\partial}^{\fr, \cong}(V_g)_\ell)_\bfQ \lra h_{2n-2}(\BEmb_{\half\partial}^{\sfr, \cong}(V_g)_\ell)_\bfQ.\]
By \cref{thm:fibre-to-haut} the homotopy fibre of the universal cover of
$F\coloneq \hAut^{\cong,\ell}_\partial(V_g)/\Emb^{\sfr, \cong}_{\half\partial}(V_g)_\ell$
has trivial rational homology in degrees $* < 2n-2$, so $\oH_{2n-2}(F;\bfQ)\cong [\pi_{2n-2}(F)_\bfQ]_{\pi_1(F)}$ and the Serre spectral sequence provides an exact sequence of $\bfQ[\check{\Lambda}_{V}\rtimes \langle \rho\rangle]$-modules
\vspace{-0.1cm}
\[ \cdots \overset{d^{2n-1}}\lra [\pi_{2n-2}(F)_\bfQ]_{\pi_1(F)} \lra \oH_{2n-2}(\BEmb_{\half\partial}^{\sfr, \cong}(V_g)_\ell;\bfQ) \lra \oH_{2n-2}(\BhAut_{\partial}^{\cong,\ell}(V_g);\bfQ).\]
The reflection $\rho$ acts trivially on $H_{V_g}$ and so also on the source of the composition
\vspace{-0.1cm}
\[S_2(H_{V_g})_{2(n-1)}\overset{\Phi}\lra h_{2n-2}(\BEmb_{\half\partial}^{\sfr, \cong}(V_g))_\bfQ \lra h_{2n-2}(\BhAut_{\partial}^{\cong,\ell}(V_g))_\bfQ,\]
but it follows from \cref{thm:KrannichhAut} that $\rho$ acts by $-1$ on the target, so this composition must be trivial and we obtain a factorisation 
\[\Phi: S_2(H_{V_g})_{2(n-1)} \lra [\pi_{2n-2}(F)_\bfQ]_{\pi_1(F)}/\mathrm{im}(d^{2n-1}) \longhookrightarrow h_{2n-2}(\BEmb_{\half\partial}^{\sfr, \cong}(V_g))_\bfQ.\]
By \cref{thm:fibre-to-haut}, $\pi_{2n-2}(F)_\bfQ$ has the same dimension as $S_2(H_{V_g})_{2(n-1)}$, so the dimension of the middle group in this composition is at most that of $S_2(H_{V_g})_{2(n-1)}$. As the latter is zero or irreducible as a $\check{\Lambda}_V$-module (see \cref{sec:some-rep-theory} below) we conclude that either
\begin{enumerate}[(i)]
\item\label{it:dichotomy1} the composition $\Phi$ is zero, or
\item\label{it:dichotomy2} the maps in the zig-zag of $\bfQ[\check{\Lambda}_{V}\rtimes \langle \rho\rangle]$-modules
\[S_2(H_{V_g})_{2(n-1)} \lra [\pi_{2n-2}(F)_\bfQ]_{\pi_1(F)}/\mathrm{im}(d^{2n-1})\lla\pi_{2n-2}(F)_\bfQ\] are isomorphisms and \cref{thm:ClasperMain} follows.
\end{enumerate}
We wish to show that the latter case holds. Suppose for a contradiction that the former case holds and consider the framed inclusion $\varphi \colon V_2 \hookrightarrow V_g$ as the first two handles, giving a class \[\xi\coloneq c_{\Diff}(\varphi)\in h_{2n-2}(\BDiff_\partial^{\fr}(V_g)_\ell)_\bfQ.\] Under the splitting 
\[h_{2(n-1)}(\BDiff_\partial^{\fr}(V_g)_\ell)_\bfQ\cong h_{2(n-1)}(\BDiff_\partial^{\fr}(D^{2n+1})_\ell)_\bfQ\oplus  h_{2(n-1)}(\BEmb_{\half \partial}^{\fr,\cong}(V_g)_\ell)_\bfQ
\] induced by \eqref{eq:FrWeissSplit}, the class $\xi$ maps to $(c_{\Diff}(\std),c(\varphi))$ where $\std\colon V_2\hookrightarrow D^{2n+1}$ is the standard embedding. As both of the components vanish---$c_{\Diff}(\std)$ because $\std$ has a null-handle and $c(\varphi)$ by our assumption that  \ref{it:dichotomy1} holds---the class $\xi$ is trivial. On the other hand, we will now argue that Watanabe's work \cite{WatanabeII} implies that $\xi$ is nontrivial, which yields a contradiction and thus finishes the proof. To this end, we choose a nonstandard framed embedding $\psi\colon V_g \hookrightarrow D^{2n+1}$ for which the cores of the first two handles link. Sending forwards the class $\xi$ using this embedding gives a class in $h_{2n-2}(\BDiff_\partial^\fr(D^{2n+1}))$ which agrees by definition with $\theta_*(\alpha\otimes \alpha)$ where $\theta$ is the composition of embeddings $ (\psi\circ \varphi\circ b)\colon V_3 \sqcup V_3\hookrightarrow D^{2n+1}$ visualised in \cref{fig:Clasper5}. As a result of the final part of \cref{prop:alphaExists}, the class $\theta_*(\alpha\otimes \alpha)$ agrees up to a nontrivial constant with the class $\theta_*(\alpha_{\text{Wat}}\otimes \alpha_{\text{Wat}})$ where $\alpha_{\text{Wat}}\in \pi_{n-1}(\BDiff^{\fr}_\partial(V_3))_\bfQ$ is the class from \cite[Lemma A]{WatanabeIIerr}. Using the notation from \cite{WatanabeII}, the class $\theta_*(\alpha_{\text{Wat}}\otimes \alpha_{\text{Wat}})$ is thus by definition exactly the image of the theta-graph $\Theta \in\cG_{2,3}$ under the map $\psi_2$ from Definition 1 on p.\ 641 loc.cit.. But this image is nontrivial by Theorem 3.1 (i) loc.cit., since the image of $\Theta \in\cG_{2,3}$ in $\cA_{2,3}\cong\bfQ$ is a generator (see Remark 2 on p.\ 632 loc.cit.\ and Figure 2 on p.\ 633). Note that the statement of Theorem 3.1 assumes that $n\ge3$ is odd, but it also holds for even $n\ge3$ by \cite[Remark 6]{WatanabeIIerr}.

\begin{figure}[ht]
\begin{center}
\centering{
\resizebox{10.5cm}{!}{
\begingroup%
  \makeatletter%
  \providecommand\color[2][]{%
    \errmessage{(Inkscape) Color is used for the text in Inkscape, but the package 'color.sty' is not loaded}%
    \renewcommand\color[2][]{}%
  }%
  \providecommand\transparent[1]{%
    \errmessage{(Inkscape) Transparency is used (non-zero) for the text in Inkscape, but the package 'transparent.sty' is not loaded}%
    \renewcommand\transparent[1]{}%
  }%
  \providecommand\rotatebox[2]{#2}%
  \newcommand*\fsize{\dimexpr\f@size pt\relax}%
  \newcommand*\lineheight[1]{\fontsize{\fsize}{#1\fsize}\selectfont}%
  \ifx\svgwidth\undefined%
    \setlength{\unitlength}{362.4168408bp}%
    \ifx\svgscale\undefined%
      \relax%
    \else%
      \setlength{\unitlength}{\unitlength * \real{\svgscale}}%
    \fi%
  \else%
    \setlength{\unitlength}{\svgwidth}%
  \fi%
  \global\let\svgwidth\undefined%
  \global\let\svgscale\undefined%
  \makeatother%
  \begin{picture}(1,0.40828264)%
    \lineheight{1}%
    \setlength\tabcolsep{0pt}%
    \put(0,0){\includegraphics[width=\unitlength,page=1]{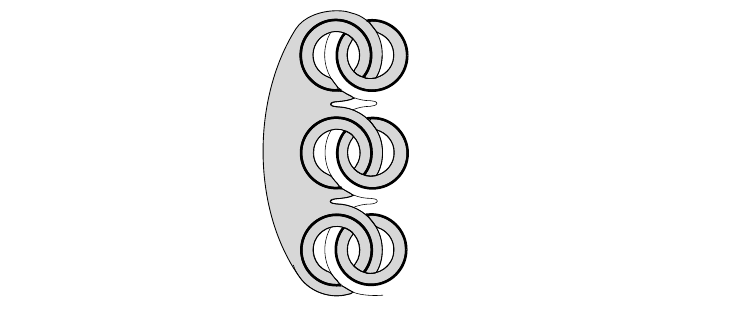}}%
    \put(0.36985376,0.28757217){\makebox(0,0)[lt]{\lineheight{1.25}\smash{\begin{tabular}[t]{l}$f_3$\end{tabular}}}}%
    \put(0,0){\includegraphics[width=\unitlength,page=2]{./figures/Clasper52.pdf}}%
    \put(0.36998936,0.10657059){\makebox(0,0)[lt]{\lineheight{1.25}\smash{\begin{tabular}[t]{l}$f_1$\end{tabular}}}}%
    \put(0.54546549,0.10535163){\makebox(0,0)[lt]{\lineheight{1.25}\smash{\begin{tabular}[t]{l}$f_1$\end{tabular}}}}%
    \put(0.54546549,0.19530997){\makebox(0,0)[lt]{\lineheight{1.25}\smash{\begin{tabular}[t]{l}$f_2$\end{tabular}}}}%
    \put(0.36917122,0.19531){\makebox(0,0)[lt]{\lineheight{1.25}\smash{\begin{tabular}[t]{l}$f_2$\end{tabular}}}}%
    \put(0.54266111,0.28643627){\makebox(0,0)[lt]{\lineheight{1.25}\smash{\begin{tabular}[t]{l}$f_3$\end{tabular}}}}%
    \put(0,0){\includegraphics[width=\unitlength,page=3]{./figures/Clasper52.pdf}}%
  \end{picture}%
\endgroup%
}}
\caption{The embedding $\theta$.}\label{fig:Clasper5}
\end{center}
\end{figure}
\subsection{Some clasper calculus}\label{sec:some-claspers}Before giving the proof of \cref{prop:ClasperMain}, we describe a linearity property of the class $\alpha$ from \cref{prop:alphaExists}. To this end, we consider framed embeddings
\begin{equation}\label{equ:linearity-embeddings}
\varphi_3, \varphi_4, \varphi_{34} \colon V_3 \longhookrightarrow V_4
\end{equation}
as depicted in \cref{fig:Clasper2}, which are such that each embedding $V_3 \overset{\varphi}\hookrightarrow V_4 \subset D^{2n+1}$ is framed isotopic to the tautological embedding, and  whose induced maps on homology fix the classes $f_1$ and $f_2$ and satisfy
\[
(\varphi_3)_*(f_3) = f_3,\quad (\varphi_4)_*(f_3) = f_4,\quad\text{and}\quad(\varphi_{34})_*(f_3) = f_3+f_4.
\]
These embeddings induce classes \[\alpha_3 \coloneq (\varphi_3)_*(\alpha),\quad \alpha_4 \coloneq (\varphi_4)_*(\alpha),\quad\text{ and }\alpha_{34} \coloneq (\varphi_{34})_*(\alpha)\] in  the group $\pi_{n-1}(\BDiff^{\fr}_\partial(V_4))_\bfQ$ which are related as follows.
\begin{figure}
\begin{center}
\centering{
\resizebox{10.5cm}{!}{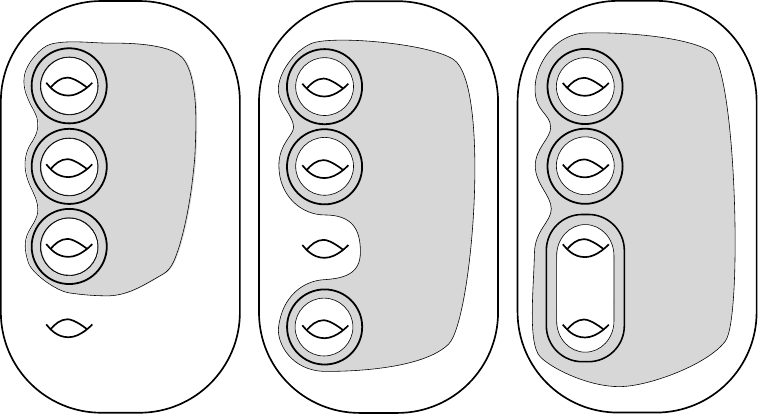}}
\caption{An example of embeddings $\varphi_3, \varphi_4, \varphi_{34}$.}\label{fig:Clasper2}\end{center}
\end{figure}
\begin{lem}\label{lem:clasperlinear}
We have $\alpha_{34} = \alpha_{3} + \alpha_4\in\pi_{n-1}(\BDiff^{\fr}_\partial(V_4))_\bfQ$.
\end{lem}
\begin{proof}
The claim is equivalent to showing that both summands of $\delta \coloneq \alpha_{34} - \alpha_3 - \alpha_4$ with respect to the decomposition \eqref{equ:splitting-homotopy-group2} are trivial. In the first summand all of $\alpha_3$, $\alpha_4$,  $\alpha_{34}$ are zero, as the embeddings \eqref{equ:linearity-embeddings} have null handles when postcomposed with the inclusion $V_4\subset D^{2n+1}$. In the second summand, by naturality of the isomorphisms \eqref{equ:iso-first-haut-group} with respect to framed embeddings we have
$\kappa_V(\alpha_3) = [f_1 \otimes f_2 \otimes f_3]$, $\kappa_V(\alpha_4) = [f_1 \otimes f_2 \otimes f_4]$, and $\kappa_V(\alpha_{34}) = [f_1 \otimes f_2 \otimes (f_3+f_4)]$
and so $\kappa_V(\delta)$ is trivial too.
\end{proof}

\begin{proof}[Proof of \cref{prop:ClasperMain}]
To show \ref{it:ClasperMain1}, we first use that $\pi_0(\Emb^{\fr}(V_2, V_g)) \to \pi_0 (\Imm^{\fr}(V_2, V_g))$ is surjective by general position. By Smale--Hirsch theory the forgetful map \[\Imm^{\fr}(V_2, V_g) \lra \Map(V_2, V_g)\] is an equivalence, so we have $\pi_0( \Imm^{\fr}(V_2, V_g)) \cong \bar{H}^\bfZ_{V_g} \times \bar{H}^\bfZ_{V_g}$ by evaluating on $f_1, f_2 \in \oH_n(V_2;\bfZ)$. 

To show \ref{it:ClasperMain2}, we thus need to prove that the class $c(\varphi)$ only depends on the image of $\varphi$ in $\pi_0(\Imm^{\fr}(V_2, V_g))\cong \bar{H}^\bfZ_{V_g} \times \bar{H}^\bfZ_{V_g}$. To do so consider the commutative diagram
\vspace{-0.1cm}
\[
\begin{tikzcd}
\Emb^{\fr}(V_2, V_g) \rar \dar& \Emb^{\fr}(U_2, V_g) \dar \rar & \Emb( S^n \sqcup  S^n, V_g) \dar\\
\Imm^{\fr}(V_2, V_g) \rar & \Imm^{\fr}(U_2, V_g) \rar& \Imm( S^n \sqcup S^n, V_g),
\end{tikzcd}
\]
where the horizontal maps are all given by restriction. The right-hand square is homotopy cartesian. By Smale--Hirsch theory, the lower left-hand map has homotopy fibre equivalent to $\Omega V_g$, which is $(n-2)$-connected. The homotopy fibre at $e : U_2 \hookrightarrow V_g$ of the upper left-hand map is discussed in the proof of \cref{lem:EmbeddingsAndCores} and is $(n-2)$-connected. Thus the left-hand square is $(n-2)$-cartesian, so the outer square is too. The path-components of the homotopy fibre of the right-hand vertical map (and thus those of the left-hand vertical map) in that square can be studied by elementary singularity theory: given two embeddings $\varphi, \varphi': S_1^n \sqcup S_2^n = S^n \sqcup S^n \hookrightarrow V_g$ which are regularly homotopic, putting such a regular homotopy in general position shows that $\varphi$ and $\varphi'$ differ by a sequence of isotopies, self-linking changes of each $S^n_i$, and linking changes between $S^n_1$ and $S^n_2$. Here by a (self-)linking change, we mean finding a chart in which $\varphi(S^n_1 \sqcup S^n_2)$ has the form
$\bfR^n \times \{0\} \times \{0\} \sqcup \{0\} \times \{2\} \times \bfR^n \subset \bfR^{n} \times \bfR \times \bfR^n$
and replacing $\bfR^n \times \{0\} \times \{0\}$ with $\Gamma \times \{0\}$ where $\Gamma \subset \bfR^n \times \bfR$ is the graph of a compactly-supported function $\bfR^n \to \bfR$ sending 0 to 3. More conceptually, we form the ambient connect-sum of $\bfR^n \times \{0\} \times \{0\}$ with the unit meridian sphere of $\{0\} \times \{2\} \times \bfR^n$, along the path $\{0\} \times [0,1] \times \{0\}$ with its natural normal framing. To prove \ref{it:ClasperMain2} we must therefore show that modifying a framed embedding $\varphi\in\pi_0(\Emb^{\fr}(V_2, V_g))\cong \pi_0(\Emb^{\fr}(U_2, V_g))$ by the analogous (self-)linking changes does not change the class $c(\varphi)$. We will use the framed embeddings indicated in \cref{fig:Clasper2}.

\medskip

\noindent \textbf{Case 1}. If $\varphi$ is modified to $\varphi'$ by a self-linking change of $S_1^n$ (or analogously of $S^n_2$), then there is an embedding
\[\phi_1: V_4 \sqcup V_3 \hookrightarrow V_2 \overset{e}\hookrightarrow V_g\]
 whose restriction along $(\varphi_3 \sqcup V_3) \colon V_3 \sqcup V_3 \hookrightarrow V_4 \sqcup V_3$ is the pattern associated to $\varphi$, whose restriction along $\varphi_{34} \sqcup V_3$ is the pattern associated to $\varphi'$, and whose restriction along $\varphi_4 \sqcup V_3$ has a null handle. Such an embedding is depicted in \cref{fig:Clasper3}. Now 
\begin{align*}c(\varphi') &= (\phi_1 \circ (\varphi_{34} \sqcup V_3))_*(\alpha \otimes \alpha)\\ 
&= (\phi_1)_*(\alpha_{34} \otimes \alpha)\\
&= (\phi_1)_*((\alpha_3 + \alpha_4) \otimes \alpha)\quad\text{ by Lemma \ref{lem:clasperlinear} }\\
&= c(\varphi) + (\phi_1 \circ (\varphi_4 \sqcup V_3))_*(\alpha \otimes \alpha).\end{align*} As $\phi_1 \circ (\varphi_4 \sqcup V_3)$ has a null handle, the last term vanishes, as required.
\begin{figure}[ht]
\begin{center}
\centering{
\resizebox{10.5cm}{!}{
\begingroup%
  \makeatletter%
  \providecommand\color[2][]{%
    \errmessage{(Inkscape) Color is used for the text in Inkscape, but the package 'color.sty' is not loaded}%
    \renewcommand\color[2][]{}%
  }%
  \providecommand\transparent[1]{%
    \errmessage{(Inkscape) Transparency is used (non-zero) for the text in Inkscape, but the package 'transparent.sty' is not loaded}%
    \renewcommand\transparent[1]{}%
  }%
  \providecommand\rotatebox[2]{#2}%
  \newcommand*\fsize{\dimexpr\f@size pt\relax}%
  \newcommand*\lineheight[1]{\fontsize{\fsize}{#1\fsize}\selectfont}%
  \ifx\svgwidth\undefined%
    \setlength{\unitlength}{362.36684003bp}%
    \ifx\svgscale\undefined%
      \relax%
    \else%
      \setlength{\unitlength}{\unitlength * \real{\svgscale}}%
    \fi%
  \else%
    \setlength{\unitlength}{\svgwidth}%
  \fi%
  \global\let\svgwidth\undefined%
  \global\let\svgscale\undefined%
  \makeatother%
  \begin{picture}(1,0.39452213)%
    \lineheight{1}%
    \setlength\tabcolsep{0pt}%
    \put(0,0){\includegraphics[width=\unitlength,page=1]{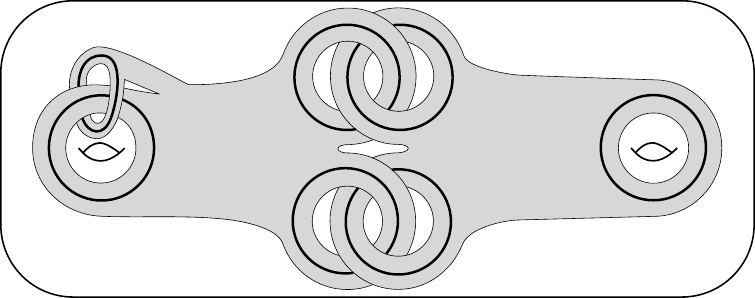}}%
    \put(0.36540612,0.13645683){\makebox(0,0)[lt]{\lineheight{1.25}\smash{\begin{tabular}[t]{l}$f_1$\end{tabular}}}}%
    \put(0.59458689,0.22847834){\makebox(0,0)[lt]{\lineheight{1.25}\smash{\begin{tabular}[t]{l}$f_2$\end{tabular}}}}%
    \put(0.59900082,0.13652204){\makebox(0,0)[lt]{\lineheight{1.25}\smash{\begin{tabular}[t]{l}$f_1$\end{tabular}}}}%
    \put(0.365064,0.22921399){\makebox(0,0)[lt]{\lineheight{1.25}\smash{\begin{tabular}[t]{l}$f_2$\end{tabular}}}}%
    \put(0.2120488,0.19022452){\makebox(0,0)[lt]{\lineheight{1.25}\smash{\begin{tabular}[t]{l}$f_3$\end{tabular}}}}%
    \put(0.76378643,0.19022452){\makebox(0,0)[lt]{\lineheight{1.25}\smash{\begin{tabular}[t]{l}$f_3$\end{tabular}}}}%
    \put(0.16349588,0.29689379){\makebox(0,0)[lt]{\lineheight{1.25}\smash{\begin{tabular}[t]{l}$f_4$\end{tabular}}}}%
  \end{picture}%
\endgroup%
}}
\caption{The embedding $\phi_1$.}\label{fig:Clasper3}
\end{center}
\end{figure}

\medskip

\noindent \textbf{Case 2}.  If $\varphi$ is modified to $\varphi'$ by a linking change of $S_1^n$ and $S^n_2$, say, then there is an embedding
\[\phi_2: V_4 \sqcup V_4 \hookrightarrow V_2 \overset{e}\hookrightarrow V_g\]
 whose restriction along $\varphi_3 \sqcup \varphi_3 \colon V_3 \sqcup V_3 \subset V_4 \sqcup V_4$ is the pattern associated to $\varphi$, whose restriction along $\varphi_{34} \sqcup \varphi_{34}$ is the pattern associated to $\varphi'$, whose restriction along $\varphi_{3} \sqcup \varphi_{4}$ or $\varphi_4 \sqcup \varphi_3$ has a null handle (the ``4''), and whose restriction along $\varphi_4 \sqcup \varphi_4$ factors up to isotopy through $D^{2n+1} \subset V_g$; see \cref{fig:Clasper4} for such an embedding. By the same reasoning as above this gives 
\[c(\varphi') = c(\varphi) + (\phi_2 \circ (\varphi_4 \sqcup \varphi_4))_*(\alpha \otimes \alpha).\] 
The last term need no longer be zero, but it is in the image of $h_{2(n-1)}(\BDiff^{\fr}_\partial(D^{2n+1})_\ell)_\bfQ$, and hence vanishes in $h_{2(n-1)}(\BEmb^{\fr}_{\half\partial}(V_g)_\ell)_\bfQ$. 
This finishes the proof of \ref{it:ClasperMain2}.
\begin{figure}[ht]
\begin{center}
\centering{
\resizebox{10.5cm}{!}{
\begingroup%
  \makeatletter%
  \providecommand\color[2][]{%
    \errmessage{(Inkscape) Color is used for the text in Inkscape, but the package 'color.sty' is not loaded}%
    \renewcommand\color[2][]{}%
  }%
  \providecommand\transparent[1]{%
    \errmessage{(Inkscape) Transparency is used (non-zero) for the text in Inkscape, but the package 'transparent.sty' is not loaded}%
    \renewcommand\transparent[1]{}%
  }%
  \providecommand\rotatebox[2]{#2}%
  \newcommand*\fsize{\dimexpr\f@size pt\relax}%
  \newcommand*\lineheight[1]{\fontsize{\fsize}{#1\fsize}\selectfont}%
  \ifx\svgwidth\undefined%
    \setlength{\unitlength}{362.4168408bp}%
    \ifx\svgscale\undefined%
      \relax%
    \else%
      \setlength{\unitlength}{\unitlength * \real{\svgscale}}%
    \fi%
  \else%
    \setlength{\unitlength}{\svgwidth}%
  \fi%
  \global\let\svgwidth\undefined%
  \global\let\svgscale\undefined%
  \makeatother%
  \begin{picture}(1,0.39460563)%
    \lineheight{1}%
    \setlength\tabcolsep{0pt}%
    \put(0,0){\includegraphics[width=\unitlength,page=1]{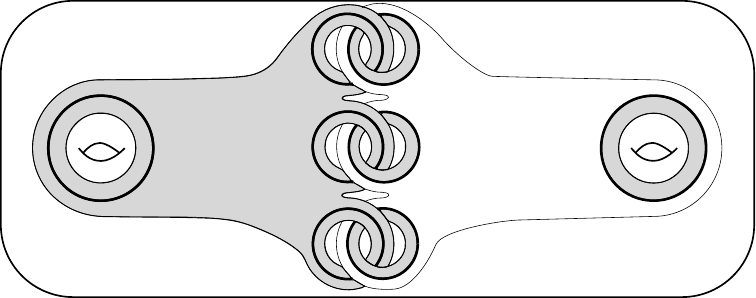}}%
    \put(0.21392094,0.19695445){\makebox(0,0)[lt]{\lineheight{1.25}\smash{\begin{tabular}[t]{l}$f_3$\end{tabular}}}}%
    \put(0.38516459,0.28196611){\makebox(0,0)[lt]{\lineheight{1.25}\smash{\begin{tabular}[t]{l}$f_4$\end{tabular}}}}%
    \put(0,0){\includegraphics[width=\unitlength,page=2]{./figures/Clasper42.pdf}}%
    \put(0.3853002,0.10096453){\makebox(0,0)[lt]{\lineheight{1.25}\smash{\begin{tabular}[t]{l}$f_1$\end{tabular}}}}%
    \put(0.56077633,0.09974557){\makebox(0,0)[lt]{\lineheight{1.25}\smash{\begin{tabular}[t]{l}$f_1$\end{tabular}}}}%
    \put(0.56077633,0.1897039){\makebox(0,0)[lt]{\lineheight{1.25}\smash{\begin{tabular}[t]{l}$f_2$\end{tabular}}}}%
    \put(0.38448205,0.18970394){\makebox(0,0)[lt]{\lineheight{1.25}\smash{\begin{tabular}[t]{l}$f_2$\end{tabular}}}}%
    \put(0.55797194,0.2808302){\makebox(0,0)[lt]{\lineheight{1.25}\smash{\begin{tabular}[t]{l}$f_4$\end{tabular}}}}%
    \put(0.76253321,0.19781403){\makebox(0,0)[lt]{\lineheight{1.25}\smash{\begin{tabular}[t]{l}$f_3$\end{tabular}}}}%
  \end{picture}%
\endgroup%
}}
\caption{The embedding $\phi_2$.}\label{fig:Clasper4}
\end{center}
\end{figure}

\vspace{1ex}

The first part of \ref{it:ClasperMain3} follows from the linearity property described in Lemma \ref{lem:clasperlinear}. More precisely, given $h_1, h_1', h_2 \in \bar{H}_{V_g}$ we can represent these classes by disjoint framed embeddings $V_1  \hookrightarrow V_g$, and hence find a framed embedding
$\phi: V_4 \sqcup V_3 \hookrightarrow V_g$
where the first and second handles of the $V_4$ and $V_3$ are linked to each other, the third and fourth handles of the $V_4$ are sent to $h_1$ and $h'_1$ respectively, and the third handle of the $V_3$ is sent to $h_2$. Now the embedding
$\phi \circ (\varphi_{34} \sqcup V_3) \colon V_3 \sqcup V_3 \hookrightarrow V_g$
may be factored as $\varphi \circ b$ for an embedding $\varphi \colon V_2 \hookrightarrow V_g$ with $\varphi_*(f_1) = h_1+h'_1$ and $\varphi_*(f_2)=h_2$, and similarly $\phi \circ (\varphi_{3} \sqcup V_3)$ may be factored as an embedding corresponding to $(h_1, h_2)$ and $\phi\circ (\varphi_{4} \sqcup V_3)$ may be factored as an embedding corresponding to $(h_1', h_2)$. We therefore have
\begin{align*}
\bar{c}(h_1+h'_1, h_2) &= (\phi \circ (\varphi_{34} \sqcup V_3))_*(\alpha \otimes \alpha) = \phi_*(\alpha_{34} \otimes \alpha)\\
&= \phi_*((\alpha_3 + \alpha_4) \otimes \alpha) \quad\text{ by Lemma \ref{lem:clasperlinear} }\\
&= (\phi \circ \varphi_3)_*(\alpha \otimes \alpha) +  (\phi \circ \varphi_4)_*(\alpha \otimes \alpha)\\
&= \bar{c}(h_1, h_2) + \bar{c}(h'_1, h_2)
\end{align*}
as required. For the second part of \ref{it:ClasperMain3}, by naturality it is enough to consider the case $V_g = V_2$. To calculate $\bar{c}(f_2, f_1)$ we choose a framed embedding $\varphi \colon V_2 \hookrightarrow V_2$ such that $\varphi_*(f_1)=f_2$ and $\varphi_*(f_2) = f_1$. This may be chosen so that $\varphi \circ b$ is given by
\[b' \colon V_3 \sqcup V_3 \xlra{\text{swap}} V_3 \sqcup V_3 \overset{b}\longhookrightarrow V_2,\]
so that we have
\[b' \colon  \BDiff^{\fr}_\partial(V_3) \times \BDiff^{\fr}_\partial(V_3) \overset{\text{swap}}\lra \BDiff^{\fr}_\partial(V_3) \times \BDiff^{\fr}_\partial(V_3) \overset{b_*}\lra \BDiff^{\fr}_\partial(V_2).\]
As $\text{swap}_*(\alpha \otimes \alpha) = (-1)^{(n-1)^2} \alpha \otimes \alpha = (-1)^{n-1} \alpha \otimes \alpha$, and $b_*(\alpha \otimes \alpha) = \bar{c}(\varphi_1, \varphi_2)$, it follows that $\bar{c}(f_2, f_1) = (-1)^{n-1}\bar{c}(f_1, f_2)$ as claimed. 
\end{proof}

\section{Homology of self-embedding spaces}\label{sec:HEmbSpace}
Based on Theorems \ref{thm:KrannichhAut}, \ref{thm:fibre-to-haut}, and \ref{thm:ClasperMain}, we determine in a range the rational homology  of the self-embedding spaces appearing in \cref{cor:reduction-to-self-embeddings}. To state the result, note that these spaces are related to certain stabilised arithmetic groups by compositions 
\[
\BEmb^{\fr,\cong}_{\nicefrac{1}{2}\partial}(W_{g,1})_\ell\ \ra B G_W\ra \BOSp_\infty(\bfZ)\ \ \ \text{and}\ \ \ 
\BEmb^{\fr,\cong}_{\nicefrac{1}{2}D^{2n}}(V_g, W_{g,1})_\ell\ \ra \BGL(\bar{H}_V^\bfZ)\ra \BGL_\infty(\bfZ)
\]
where the first maps are given by the action on homology (see \cref{sec:MCG}) and the second maps use the standard bases from \cref{sec:standard-model} and the stabilisation maps induced by block-inclusion.

\begin{thm}\label{thm:homology-framed-selfemb}
For each $n\ge3$ and for all large enough $g$ the maps
\[\BEmb^{\fr,\cong}_{\nicefrac{1}{2}\partial}(W_{g,1})_\ell\lra\BOSp_\infty(\bfZ)\quad\text{and}\quad \BEmb^{\fr}_{\nicefrac{1}{2}D^{2n}}(V_g,W_{g,1})_\ell\lra \BGL_{\infty}(\bfZ)\]
induce isomorphisms on rational homology in degrees $*<3n-5$ and epimorphisms for $*=3n-5$.
\end{thm}
\noindent The plan to prove \cref{thm:homology-framed-selfemb} is:

\begin{enumerate}[label=$\circled{\arabic*}$, ref={$\circled{\arabic*}$} ]
\item\label{enum:homotopy-selfemb} Compute (in a range) the homotopy Lie algebras \[\pi_*(\Omega_0\BEmb^{\fr,\cong}_{\nicefrac{1}{2}\partial}(W_{g,1})_\ell)_\bfQ\quad\text{and}\quad\pi_*(\Omega_0\BEmb^{\fr}_{\nicefrac{1}{2}D^{2n}}(V_g,W_{g,1})_\ell)_\bfQ.\]
\item\label{enum:homology-cover-selfemb} Use \ref{enum:homotopy-selfemb} to compute (in a range and up to some ambiguity) the rational homology of the universal covers 
\vspace{-0.2cm}
\[
\def\arraystretch{1.5}
\begin{array}{r@{\hskip 0.2cm} c@{\hskip 0.2cm} l@{\hskip 0.2cm} }
\BEmb^{\fr,\cong}_{\nicefrac{1}{2}\partial}(W_{g,1})_\ell^\sim&\lra&\BEmb^{\fr,\cong}_{\nicefrac{1}{2}\partial}(W_{g,1})_\ell\\
\BEmb^{\fr}_{\nicefrac{1}{2}D^{2n}}(V_g,W_{g,1})_\ell^\sim&\lra&\BEmb^{\fr}_{\nicefrac{1}{2}D^{2n}}(V_g,W_{g,1})_\ell.
\end{array}
\]
\item\label{enum:homology-selfemb}Run the universal cover spectral sequence to prove \cref{thm:homology-framed-selfemb}, based on \ref{enum:homology-cover-selfemb}. In doing so, we show that the ambiguity in \ref{enum:homology-cover-selfemb} does not affect the end result.
\end{enumerate}

\begin{conv}
Throughout this section, we fix $n\ge3$ and assume $g\ge2$. 
\end{conv} 

\subsection{Representations of $G_W$ and $G_V$ and their cohomology}\label{sec:irreducibility}
We begin this section with some representation-theoretic and group-cohomological preliminaries on the groups $\check{\Lambda}_V$ and $\check{\Lambda}_W$ mapping to $G_V \cong M_V^\bfZ\rtimes \GL(\bar{H}^\bfZ_V)$ and $G_W$ respectively. For this, it is worth recalling the explicit descriptions of $G_W$ and $G_V$ as matrix groups from \cref{sec:MCG} in terms of $\Sp^q_{2g}(\bfZ)$, $\oO_{g,g}(\bfZ)$, and $\GL_g(\bfZ)$. Furthermore, recall from \cref{sec:involution-on-mcg} that these groups are acted upon by the reflection involution $\rho$ (though $\rho$ acts trivially on the subgroup $\GL(\bar{H}^\bfZ_V)\le G_V$) and that the intersection form induces isomorphisms (see \cref{sec:involution-and-intersection}) $H_W\cong s^{2n-2}H_W^\vee \otimes \bfQ^-$ as a $(G_W\rtimes\langle \rho\rangle)$-module and an isomorphism $H_W\cong H_V\oplus s^{2n-2}H_V^\vee\otimes\bfQ^-$ as a module over the subgroup $\GL(\bar{H}_V^\bfZ)\times\langle \rho\rangle\le G_V\rtimes \langle \rho\rangle$. 

\subsubsection{Some representation theory}\label{sec:some-rep-theory}
Recall from classical representation theory (see e.g.\,\cite[Sections 6, 17, 19]{FultonHarris}) that the $\bfC[\GL_g(\bfC)]$-module $S_\mu(\bfC^g)=S_\mu(\bfC^g[0])$ for a partition $\mu$ (see \cref{sec-gradings-and-schur}) is nonzero if and only if $\mu$ has at most $g$ rows (when thought of as a Young diagram), and in this case it is irreducible. Moreover, for partitions $\mu\neq \lambda$, the $\bfC[\GL_g(\bfC)]$-modules $S_\mu(\bfC^g)$ and $S_{\lambda}(\bfC^g)$ are distinct if they are nonzero. By dualising, the same discussion applies when replacing the $S_\mu(\bfC^g)$'s with their duals $S_\mu(\bfC^g)^\vee\cong S_\mu((\bfC^g)^\vee)$. Similarly, the $\bfC[\OSp_{2g}(\bfC)]$-module $V_\mu(\bfC^g)$ is nonzero for $n$ odd if and only if $\mu$ has at most $g$ rows, and for $n$ even if and only if the first two columns of $\mu$ have in sum at most $g$ boxes, and in these cases $V_\mu(\bfC^{2g})$ is an irreducible $\bfC[\OSp_{2g}(\bfC)]$-module. As in the $\GL_g(\bfC)$-case, for $\mu\neq\lambda$ the $\bfC[\OSp_{2g}(\bfC)]$-modules $V_\mu(\bfC^{2g})$ and $V_{\lambda}(\bfC^{2g})$ are distinct if they are nonzero. 

Using that the representations discussed above are algebraic (i.e.\ extend to representations of the ambient algebraic groups $\GL_g(-)$ or $\OSp_{g}(-)$) and the subgroups $\OO_{g,g}(\bfZ)\subset \OO_{g,g}(\bfC)$ and $\Sp_{2g}^q(\bfZ)\subset\Sp_{2g}(\bfZ)\subset \Sp_{2g}(\bfC)$ are Zariski dense, the analogous properties hold for the $\bfQ[G_W]$-modules $V_\mu(\bar{H}_{W})$ and thus also for their graded versions $V_\mu(H_W)$; this remains true after pulling back along any epimorphism onto $G_W$, such as $\check{\Lambda}_W\twoheadrightarrow G_W$.

Unfortunately $\GL_g(\bfZ) \subset \GL_g(\bfC)$ is not Zariski dense (its Zariski closure consists of the complex matrices of determinant $\pm 1$) so more care is needed in this setting. It is still true that the $\bfQ[\GL_g(\bfZ)]$-modules $S_\mu(\bfQ^g)$ and $S_\mu(\bfQ^g)^\vee$ are irreducible when they are non-zero, because $S_\mu(\bfC^g)$ remains irreducible as a $\SL_g(\bfC)$-module (as every element of $\GL_g(\bfC)$ is a scalar multiple of an element of $\SL_g(\bfC)$) and $\SL_g(\bfZ) \subset \SL_g(\bfC)$ is Zariski dense. Thus the $\bfQ[\GL(\bar{H}^\bfZ_V)]$-modules $S_\mu(\bar{H}_V)$ and $S_\mu(\bar{H}_V)^\vee$ are still zero or irreducible, as are their graded versions $S_\mu(H_V)$ and $S_\mu(H_V)^\vee$; this remains true after pulling back these modules along any epimorphism onto $\GL(\bar{H}^\bfZ_V)$, such as $\check{\Lambda}_V\twoheadrightarrow\GL(\bar{H}^\bfZ_V)$, or $G_V\twoheadrightarrow \GL(\bar{H}^\bfZ_V)$. What is no longer true is that these are \emph{distinct} irreducible modules for distinct partitions: the square of the determinant representation of $\GL_g(\bfC)$ is trivial when restricted to $\GL_g(\bfZ)$, so e.g.\ $S_{(2^g)}(\bfQ^g) \cong S_0(\bfQ^g)$ as $\bfQ[\GL_g(\bfZ)]$-modules. But we may still decompose algebraic $\GL_g(\bfZ)$-representations into irreducibles by decomposing some $\GL_g(\bfC)$-representation which induces it into irreducibles.

\subsubsection{Manipulating Schur functors}\label{sec:schur-manipulations}At various points in this section, we have to calculate tensor products, plethysm, or restrictions of (symplectic or orthogonal) Schur functors, such as $s^{-(2n-2)}S_{1^3}(H_W)\cong s^{-(2n-2)}V_{1^3}(H_W)\oplus H_W$ as $G_W$-modules or $S_2(H_V)\otimes H_V=S_{2,1}(H_V)\oplus S_3(H_V)$ as $\GL(\bar{H}^\bfZ_V)$-modules (and thus also as $G_V$ or $\check{\Lambda}_V$-modules). In all cases we performed these calculations using the {\tt SymmetricFunctions} package provided by {\tt SageMath} \cite{sagemath}.

\subsubsection{Stable cohomology of $G_W$ and $\GL(\bar{H}^\bfZ_V)$}\label{sec:stable-cohomology-borel}

Recall from \cite{Borel1} that $\oH^*(\BGL_{\infty}(\bfZ);\bfQ)$ is an exterior algebra with generators in degrees $4i+1$ for $i\ge1$, and that $\oH^*(\BOSp_\infty(\bfZ);\bfQ)$ is a polynomial algebra with generators in degrees $4i$ for $i\ge1$ if $n$ is even and in degrees $4i-2$ for $i\ge1$ if $n$ is odd. For $\oH^*(\BOSp_\infty(\bfZ);\bfQ)$, it will be convenient for us to use explicit generators
\[\oH^*(\BOSp_\infty(\bfZ);\bfQ)\cong \begin{cases} \bfQ[\sigma_{4},\sigma_8,\ldots]&n\text{ even}\\ \bfQ[\sigma_{2},\sigma_6,\ldots]&n\text{ odd}\end{cases}
\]
defined in the terms of the signature as explained e.g.\ in \cite[Lemma 4.1]{KR-WDisc}. The definition in terms of the signature makes apparent that the reflection $\rho$ of \cref{sec:involution-on-mcg} acts on the $\sigma_i$'s as $-1$ whereas it acts trivially on $\oH^*(\GL({\bar{H}_{V_g}}^\bfZ);\bfQ)$ for all $g$ and thus also on $\oH^*(\BGL_\infty(\bfZ);\bfQ)$.

Combined with further work of Borel \cite{Borel2}, the calculation of $\oH^*(\BGL_{\infty}(\bfZ);\bfQ)$ and  $\oH^*(\BOSp_\infty(\bfZ);\bfQ)$ allows for a computation of the $G_W$- and $\GL(\bar{H}^\bfZ_V)$-homology with coefficients in large class of $\bfQ[G_{W}]$-modules $M_W$ and $\bfQ[\GL(\bar{H}^\bfZ_V)]$-modules $M_V$ in a range of degrees increasing with $g$ in terms of the invariants. Concretely, the maps
\begin{equation}\label{equ:borel-vanishing-maps}\begin{gathered}\oH^*( \OSp_\infty(\bfZ);\bfQ)\otimes M_W^{G_{W}}\xlra{\mathrm{res}\otimes\id}\oH^*(G_W;\bfQ)\otimes M_W^{G_{W}}\xlra{\mathrm{cup}} \oH^*(G_{W};M_W)\quad\text{and}\\
\oH^*( \GL_\infty(\bfZ);\bfQ)\otimes M_V^{\GL(\bar{H}_{V})}\xlra{\mathrm{res}\otimes\id}\oH^*(\GL(\bar{H}_{V});\bfQ)\otimes M_V^{\GL(\bar{H}_{V})}\xlra{\mathrm{cup}} \oH^*(\GL(\bar{H}_{V});M_V).\end{gathered}\end{equation}
are often isomorphisms in a range of degrees. The following instance will be sufficient for us:

\begin{thm}\label{thm:Borel-vanishing}Fix $p,q\ge0$. There is a range of degrees tending to $\infty$ with $g$ in which the maps \eqref{equ:borel-vanishing-maps} are isomorphisms for all summands $M_W$ and $M_V$ of $\bar{H}_{W}^{\otimes p}$ and $\bar{H}_{V}^{\otimes p}\otimes (\bar{H}^\vee_{V})^{\otimes q}$ respectively.
\end{thm}
\begin{proof}
As the maps in question are functorial in $M_V$, it suffices to show the statement for $M_W=\bar{H}^{\otimes p}_W$ and $M_V=\bar{H}_{V}^{\otimes p}\otimes (\bar{H}^\vee_{V})^{\otimes q}$. Using that $G_W\cong \Sp_{2g}^q(\bfZ)$ for $n$ odd and $G_W\cong \oO_{g,g}(\bfZ)$ for $n$ even, the claim about the maps forming the first row in \eqref{equ:borel-vanishing-maps} follows from the consequence of Borel's work explained in \cite[Theorem 2.3]{KR-WTorelli}, so we may focus on the second row. To make the statement fit closer to Borel's work, we may extend coefficient from $\bfQ$ to $\bfC$ and replace $\GL_\infty(\bfZ)$ by $\SL_\infty(\bfZ)$ and $\GL(\bar{H}_{V})$ by $\SL(\bar{H}_{V})$; the $\GL$-case follows from this by taking invariants with respect to the action of $\bfZ^\times=\GL(\bar{H}_{V})/\SL(\bar{H}_{V})=\GL_\infty(\bfZ)/\SL_\infty(\bfZ)$. Denoting the semi-simple algebraic $\SL_g(\bfC)$-representation $(\bfC^g)^{\otimes p}\otimes ((\bfC^g)^\vee)^{\otimes q}$ by $\tau_{p,q}$, decomposing this module into irreducibles and using that (i) the statement is vacuously true for sums of trivial representations, (ii) $V^{\SL_g(\bfC)}=V^{\SL_g(\bfZ)}$ for an algebraic $\SL_g(\bfC)$-representation $V$ by Zariski density, and (iii) $\SL_g$ is (for $g \geq 2$) simple and connected as an algebraic group over $\bfQ$, we see from \cite[Theorem 4.4]{Borel2} that the statement for $\tau_{p,q}$ holds in degrees $*\le \min(M(\SL_g(\bfR),\tau_{p,q}), C(\SL_g,\tau_{p,q}))$ for certain non-negative constants $M(\SL_g(\bfR),\tau_{p,q})$ and $C(\SL_g,\tau_{p,q})$ defined in Section 4.1 loc.\,cit.\,(see also the introduction of \cite{TshishikuBorel} for an efficient description of these constants). According to \cite[Section 4.1]{Borel2}, we have $M(\SL_g(\bfR),\tau_{p,q})\ge \rk_\bfR(\SL_g(\bfR))-1=g-2$, so we are left to show that the constant $C(\SL_g,\tau_{p,q})$ tends to $\infty$ with $g$. To this end, we first spell out the definition of $C(\SL_g,\tau_{p,q})$ for which we use the Cartan subalgebra $\mathfrak{a}\subset \mathfrak{sl}_g$ of trace-free diagonal matrices whose dual $\mathfrak{a}^\vee$ is spanned by the standard coordinate functions $\alpha_i$, subject to the relation $\sum_{i=1}^g\alpha_i=0$. The standard  positive roots are $\alpha_i-\alpha_j$ for $i<j$ and the $g-1$ positive simple roots (which form a basis of $\mathfrak{a}^\vee$) are $\alpha_i-\alpha_{i+1}$ for $i\le g-1$. We abbreviate half the sum of the positive roots (sometimes called the \emph{Weyl vector}) by \[\textstyle{\rho\coloneq\tfrac{1}{2}\sum_{i<j}(\alpha_i-\alpha_j)},\] write $\phi>0$ for $\phi\in \mathfrak{a}^\vee$ if $\phi=\sum_{i=1}^{g-1}c_i\cdot(\alpha_i-\alpha_{i+1})$ with $c_i>0$ for all $i$, and call $c_i$ the $i$th \emph{coefficient} of $\phi$, so $\phi>0$ if all coefficients of $\phi$ are positive. In these terms, the constant $C(\SL_g,\tau_{p,q})$ is given as the largest $k\ge0$ such that $\rho+\mu>0$ for every weight $\mu$ of $\Lambda^k\mathfrak{u}^\vee \otimes\tau_{p,q}$, where $\mathfrak{u}\subset \mathfrak{sl}_g$ is the sum of the positive root spaces. The weights of $\mathfrak{u}^\vee$ are the negatives of the positive roots, so the weights of $\Lambda^k\mathfrak{u}^\vee$ are of the form $-\sum\{k\text{ distinct positive roots}\}$. Moreover, since the weights of $\bfC^g$ are the $\alpha_i$'s and those of $(\bfC^g)^\vee$ the $-\alpha_i$'s, the weights of $\tau_{p,q}$ are of the form $\sum\{p\text{ }\alpha_i\text{'s}\}-\sum\{q\text{ }\alpha_i\text{'s}\}$, so the weights $\mu$ of $\Lambda^k\mathfrak{u}^\vee \otimes\tau_{p,q}$ have the form $\textstyle{-\sum\{k\text{ distinct positive roots}\}+\sum\{p\text{ }\alpha_i\text{'s}\}-\sum\{q\text{ }\alpha_i\text{'s}\}}$, and the task becomes to estimate the coefficients of weights of the form
\begin{equation}\label{equ:root-to-be-estimated}\textstyle{\rho+\mu=\rho-\sum\{k\text{ distinct positive roots}\}+\sum\{p\text{ }\alpha_i\text{'s}\}-\sum\{q\text{ }\alpha_i\text{'s}\}}.\end{equation} The $\alpha_i$'s can be expressed in terms of simple roots as \[\textstyle{\alpha_i=-\sum_{j=1}^{i-1}\tfrac{j}{g}(\alpha_j-\alpha_{j+1})+\sum_{l=i}^{g-1}\tfrac{g-l}{g}(\alpha_l-\alpha_{l+1})}.\] In particular, the absolute value of their coefficients is $\le1$, so the absolute value of the coefficients of weights of the form $\sum\{p\text{ }\alpha_i\text{'s}\}-\sum\{q\text{ }\alpha_i\text{'s}\}$ is $\le p+q$. Moreover, writing $\alpha_i-\alpha_j=(\alpha_i-\alpha_{i+1})+\ldots +(\alpha_{j-1}-\alpha_j)$, we see that the coefficients of any positive root are $0$ or $1$, so the coefficients of weights of the form $-\sum\{k\text{ distinct positive roots}\}$ is $\ge -k$. Moreover, it is easy to see that the $i$th coefficient $f_i(g)$ of $\rho=\tfrac{1}{2}\sum_{i<j}(\alpha_i-\alpha_j)$ is $\tfrac{1}{2} i (g-i)$, and hence is $\geq \tfrac{1}{2}(g-1)$ using that $1 \leq i \leq g-1$. Thus each coefficient of \eqref{equ:root-to-be-estimated} is $\smash{\ge \tfrac{1}{2}(g-1) -k-(p+q)}$. By definition of $C(\SL_g,\tau_{p,q})$, this gives $\smash{C(\SL_g,\tau_{p,q})\ge \tfrac{1}{2}(g-3)-(p+q)}$ which indeed tends to $\infty$ with $g$.
\end{proof}

\begin{ex}\label{ex:Borel-vanishing-for-irred}It follows from \cref{thm:Borel-vanishing} that for partitions $\mu\vdash k>0$ there is a range of degrees tending to $\infty$ with $g$ in which the groups $\oH^*(G_{W};V_\mu(H_W))$, $\oH^*(\GL(\bar{H}^\bfZ_V);S_\mu(H_V))$, and $\oH^*(\GL(\bar{H}^\bfZ_V);S_\mu(H_V)^\vee)$ vanish, since $V_\mu(H_W)$, $S_\mu(H_V)$, and $S_\mu(H_V)^\vee$ have no invariants and are (via the norm and up to regrading) summands of  $\bar{H}_{W}^{\otimes k}$, $\bar{H}_{V}^{\otimes k}$, and $(\bar{H}_{V}^{\otimes k})^\vee$ respectively.
\end{ex}

\subsubsection{Stable homology of $G_V$}\label{sec:stable-cohomolog-gv}
In \cref{cor:GradedGpCoh}, as a replacement for \cref{thm:Borel-vanishing} for the group $G_V$, we establish an isomorphism of bigraded groups (the first grading is the homological degree and the second the degree of the coefficients)
\begin{equation}\label{equ:stable-gv-homology}\oH_*\big(G_V; S_\lambda({H}_{V}) \otimes S_\mu({H}_{V}^\vee)\big)\cong\oH_*(\mathrm{GL}_\infty(\bfZ);\bfQ[0]) \otimes \chi(\lambda, \mu)\end{equation} in a range of homological degrees increasing with $g$ for partitions $\lambda$ and $\mu$. Here $\chi(\lambda, \mu)$ vanishes if $|\lambda|-|\mu|$ is odd or negative, and if $|\lambda|-|\mu|=2t \geq 0$ then
\[\chi(\lambda, \mu) \coloneq \left[ \Ind_{\Sigma_{|\mu|} \times \Sigma_{2t}}^{\Sigma_{|\lambda|}} \left( (\mu) \boxtimes \Ind^{\Sigma_{2t}}_{\Sigma_t \wr \Sigma_2} (1^t) \wr (2) \right) \otimes (\lambda) \right]^{\Sigma_{|\lambda|}} [t, (|\lambda|-|\mu|)(n-1)]. \]
We use this to prove the following lemma, which collects all computations of $G_V$-homology groups necessary for this section. In the statement (and the subsequent sections), 
in addition to Schur functors applied to $H_V$ and $H_V^\vee$ we will make use of the graded $\bfQ[G_V\rtimes \langle \rho\rangle ]$-modules 
\[K_\mu \coloneq \ker\big(S_{\mu}(H_{W}) \twoheadrightarrow S_{\mu}(H_{V})\big)\] associated to the canonical projection $H_W\ra H_V$. 

\begin{lem}\label{lem:GV-homology}
In a range of homological degrees increasing with $g$, there are isomorphisms of bigraded $\oH_*(\GL_\infty(\bfZ);\bfQ)$-comodules
\begin{enumerate}
\item\label{enum:gv-homology-i}  $\oH_*( G_V ; \bfQ[0]) \cong \oH_*( \GL_{\infty}(\bfZ) ; \bfQ[0])$,
\item\label{enum:gv-homology-ii} $\oH_*(G_V;H_V^{\otimes p} \otimes (H_V^\vee)^{\otimes q})=0$ if $p-q<0$,
\item\label{enum:gv-homology-iii}  $\oH_*(G_V;S_2(H_V^\vee))=0$,
\item\label{enum:gv-homology-iv} $\oH_*( G_V ; K_{1^3})=0$,
\item\label{enum:gv-homology-v}  $\oH_*(G_V;K_{1^3}\otimes S_2(H_V)^\vee) =0$,
\item\label{enum:gv-homology-vi} $\oH_*(G_V ; S_{1^2}(H_V)\otimes H_V \otimes S_{1^2}(H_V^\vee) \otimes H_V^\vee) \cong \oH_*(\GL_\infty(\bfZ);\bfQ[0])^{\oplus 2}$,
\item\label{enum:gv-homology-vii} $\oH_*(G_V ; S_2(H_V) \otimes S_2(H_V^\vee)) \cong \oH_*(\GL_\infty(\bfZ);\bfQ[0])$,
\item\label{enum:gv-homology-viii} $\oH_*(G_V ;  S_{1^2}(H_V) \otimes S_{1^2}(H_V^\vee)) \cong \oH_*(\GL_\infty(\bfZ);\bfQ[0])$,
\item\label{enum:gv-homology-ix} $\oH_*(G_V;S_{1^2}(S_{1^2}(H_V) \otimes H_V^\vee))=0$,
\item\label{enum:gv-homology-x} $\oH_*(G_V;S_2(H_V)) \cong \oH_*(\GL_\infty(\bfZ);\bfQ[0]) \otimes \bfQ[1,2(n-1)]$,
\item\label{enum:gv-homology-xi} $\oH_*(G_V; S_{2,1}(H_V) \otimes H_V^\vee) \cong \oH_*(\GL_\infty(\bfZ);\bfQ[0]) \otimes \bfQ[1,2(n-1)]$.
\end{enumerate}
\end{lem}
\begin{rem}We could state this lemma for $\check{\Lambda}_V$ instead of $G_V$, because the surjection $\check{\Lambda}_W\twoheadrightarrow G_V$ has finite kernel by \cref{lem:finite-kernel}, so the $G_V$- and $\check{\Lambda}_W$-homology of any $\bfQ[G_V]$-module agree.
\end{rem}

\begin{proof}[Proof of \cref{lem:GV-homology}]
Items \ref{enum:gv-homology-iv} and \ref{enum:gv-homology-v} follow from the ``centre kills'' trick, as $G_V\subset \GL(\bar{H}_W)$ contains the central element $-\mathrm{id} \in \GL(\bar{H}_W)$ (this is clear from the description in \eqref{equ:matrix-description-Gg}) which acts as $(-1)$ on both $K_{1^3}$ and $K_{1^3}\otimes S_2(H_V)^\vee$. For the remainder we will apply \eqref{equ:stable-gv-homology}. The dimensions of the bigraded vector space $\chi(\lambda, \mu)$ can be calculated in terms of products and plethysm of Schur functions. In the following we did all such calculations using {\tt SageMath} (in fact only cases $t \leq 1$ arise and no plethysm is necessary).

\begin{enumerate}
\item Taking $(\lambda)=(\mu)=(0)$, we have $\chi((0), (0)) = \bfQ[0,0]$.
\item $\oH_*(G_V;H_V^{\otimes p} \otimes (H_V^\vee)^{\otimes q}) = \bigoplus_{\lambda \vdash p} \bigoplus_{\mu \vdash q} \oH_*(G_V; S_\lambda({H}_{V}) \otimes S_\mu({H}_{V}^\vee))=0$ for $p-q<0$.
\item This follows from \ref{enum:gv-homology-ii}.
\setcounter{enumi}{5}
\item We have $S_{1^2}(H_V)\otimes H_V \cong S_{1^3}(H_V) \oplus S_{2,1}(H_V)$, so we must calculate $\oH_*(G_V; S_\lambda({H}_{V}) \otimes S_\mu({H}_{V}^\vee))$ for $((\lambda), (\mu))$ being $((1^3), (1^3))$, $((1^3), (2,1))$, $((2,1), (1^3))$, and $((2,1),(2,1))$. We find that
$\chi((1^3), (1^3)) = \chi((2,1),(2,1)) = \bfQ[0,0]$ and $\chi((1^3), (2,1)) = \chi((2,1), (1^3))=0.$
\item We compute $\chi((2),(2)) = \bfQ[0,0]$.
\item We compute $\chi((1^2),(1^2)) = \bfQ[0,0]$.
\item We must first expand out $S_{1^2}(S_{1^2}(H_V) \otimes H_V^\vee)$. In general we have $S_{1^2}(A \otimes B) \cong S_{1^2}(A) \otimes S_2(B) \oplus S_2(A) \otimes S_{1^2}(B)$, so 
\begin{align*}
S_{1^2}(S_{1^2}(H_V) \otimes H_V^\vee) &\cong S_{1^2}(S_{1^2}(H_V)) \otimes S_2(H_V^\vee) \oplus S_{2}(S_{1^2}(H_V)) \otimes S_{1^2}(H_V^\vee)\\
&\cong S_{2,1,1}(H_V) \otimes S_2(H_V^\vee) \oplus S_{1^4}(H_V) \otimes S_{1^2}(H_V^\vee) \oplus S_{2^2}(H_V) \otimes S_{1^2}(H_V^\vee).
\end{align*}
Then $\chi((2,1,1), (2))=0$, $\chi((1^4), (1^2)) = 0$, and $\chi((2^2), (1^2))=0$.
\item We compute $\chi((2),(0)) = \bfQ[1,2(n-1)]$.
\item We compute $\chi((2,1),(1)) = \bfQ[1, 2(n-1)]$.\qedhere
\end{enumerate}
\end{proof}

We now begin executing the outlined plan to prove \cref{thm:homology-framed-selfemb}.

\subsection{Step \ref{enum:homotopy-selfemb}}\label{sec:universal-covers-embedding-spaces}We first settle the $W_{g,1}$-case.
\begin{prop}\label{prop:HtpyEmbWg}
There is an isomorphism of graded $\bfQ[\check{\Lambda}_{W} \rtimes \langle \rho \rangle]$-modules
\[\pi_*(\Omega_0\BEmb^{\fr,\cong}_{\nicefrac{1}{2}\partial}(W_{g,1})_\ell)_\bfQ \cong s^{-(2n-2)}\big(V_{1^3}(H_{W}) \oplus S_{2^2}(H_{W})\big)\otimes \bfQ^-\quad\text{for }*<3n-6.\]
Furthermore if $2n-2$ lies in this range (i.e. for $n>4$), then the image of the Lie bracket
\[[-,-] \colon \pi_{*}(\Omega_0\BEmb^{\fr,\cong}_{\nicefrac{1}{2}\partial}(W_{g,1})_\ell)_\bfQ^{\otimes 2} \lra \pi_{*}(\Omega_0\BEmb^{\fr,\cong}_{\nicefrac{1}{2}\partial}(W_{g,1})_\ell)_\bfQ\]
contains a trivial $\bfQ[2n-2]$-summand as a $\bfQ[\check{\Lambda}_{W}]$-module.
\end{prop}
\begin{proof}
By \cref{thm:fibre-to-haut} the homotopy fibre of the map $\BEmb^{\sfr}_{\nicefrac{1}{2}\partial}(W_{g,1})_{\ell} \ra \BhAut_{\partial}^{\cong,\ell}(W_{g,1})$  from \eqref{sec:recall-maps-to-haut} has trivial rational homotopy groups in degrees $1<*< 3n-5$. \cref{thm:haut-of-wg} provides an isomorphism $\pi_*(\Omega_0 \BhAut_{\partial}^{\cong,\ell}(W_{g,1}))_\bfQ \cong \cL ie\dl H_{W},2n\dr\otimes\bfQ^-$ in positive degrees which combined with \cref{rem:LieChar} gives
\[\pi_*(\Omega_0 \BEmb^{\sfr,\cong}_{\nicefrac{1}{2}\partial}(W_{g,1})_{\ell^s})_\bfQ \cong s^{-(2n-2)}\big(S_{1^3}(H_{W}) \oplus S_{2^2}(H_{W})\big)\otimes \bfQ^-\quad\text{for }*<3n-6.\]
Now consider the homotopy fibre sequence
\begin{equation}\label{eq:FrToSfr}
\Map_{\nicefrac{1}{2}\partial}(W_{g,1}, \Omega\SO/\SO(2n))\lra \BEmb^{\fr,\cong}_{\nicefrac{1}{2}\partial}(W_{g,1})_\ell \lra \BEmb^{\sfr,\cong}_{\nicefrac{1}{2}\partial}(W_{g,1})_{\ell^s}
\end{equation}
induced by passing from framings to stable framings. The restriction $\bar{e}$ of the Euler class $e \in \oH^{2n}(\BSO(2n);\bfQ)$ represents a rationally $(4n-1)$-connected map
$\bar{e} \colon \SO/\SO(2n) \to K(\bfQ,2n)$, so using the equivalence $(W_{g,1},\half \partial W_{g,1})\simeq (\vee^{2g}S^n,*)$
we get
\[\pi_*(\Map_{\nicefrac{1}{2}\partial}(W_{g,1}, \Omega\SO/\SO(2n)))_\bfQ \cong s^{2n-2}H_W^\vee\cong H_{W} \otimes \bfQ^-\quad\text{for }  *< 3n-2.\] In degrees $*<3n-5$, the $(\check{\Lambda}_{W} \rtimes \langle \rho \rangle)$-equivariant long exact sequence of rational homotopy groups induced by \eqref{eq:FrToSfr} thus has only one potentially nontrivial map: the connecting map
\[ \partial \colon \pi_n(\BEmb^{\sfr,\cong}_{\nicefrac{1}{2}\partial}(W_{g,1})_\ell)_\bfQ \lra  \pi_{n-1}(\Map_{\nicefrac{1}{2}\partial}(W_{g,1}, \Omega\SO/\SO(2n)))_\bfQ\cong (H_{W})_{n-1} \otimes \bfQ^-.\] But by Lemma \ref{lem:ChiVanishesOnEmbFr} the composition
\[ \pi_n(\BEmb^{\fr,\cong}_{\nicefrac{1}{2}\partial}(W_{g,1})_\ell)_\bfQ \ra \pi_n(\BEmb^{\sfr,\cong}_{\nicefrac{1}{2}\partial}(W_{g,1})_\ell)_\bfQ \xra{\cong} \pi_n(\BhAut_{\partial}^{\cong,\ell}(W_{g,1}))_\bfQ\]
is not surjective, so $\partial$ must be non-zero and, as $(H_{W})_{n-1}$ is irreducible, it must in fact be surjective. Thus, the long exact sequence induced by \eqref{eq:FrToSfr} yields in degrees $*<3n-6$ a short exact sequence of graded $\bfQ[\check{\Lambda}_{W} \rtimes \langle \rho \rangle]$-modules of the form
\[0\ra \pi_*(\Omega_0\BEmb^{\fr,\cong}_{\nicefrac{1}{2}\partial}(W_{g,1})_\ell)_\bfQ\ra s^{-(2n-2)}\big(S_{1^3}(H_{W}) \oplus S_{2^2}(H_{W})\big)\otimes \bfQ^- \ra H_W\otimes\bfQ^-\lra 0.\]
Using the decomposition $s^{-(2n-2)}S_{1^3}(H_W)=s^{-(2n-2)}V_{1^3}(H_W)\oplus H_W$ into irreducibles (see \cref{sec:schur-manipulations}), this gives the claimed description of the rational homotopy groups in the statement. Moreover, it shows that if $2n-2<3n-6$, then the target of the bracket map in the statement
\[[-,-] \colon \pi_{*}(\Omega_0\BEmb^{\fr,\cong}_{\nicefrac{1}{2}\partial}(W_{g,1})_\ell)_\bfQ^{\otimes 2} \lra \pi_{*}(\Omega_0\BEmb^{\fr,\cong}_{\nicefrac{1}{2}\partial}(W_{g,1})_\ell)_\bfQ\]
is $S_{2^2}(H_W)_{4(n-1)}\otimes\bfQ^-$, which is semi-simple. To prove the second part of the claim, it thus suffices to exhibit a nontrivial $\check{\Lambda}_W$-invariant vector of degree $2n-2$ in the image. This follows from Corollary 5.15 of \cite{KR-WDisc}, since the map $\Omega_0X_1(g)\ra \Omega_0\BhAut_\partial(W_{g,1})$ in that corollary factors through $\Omega_0\BDiff^\fr_\partial(W_{g,1})_\ell$ by construction, and thus also through $\Omega_0\BEmb^{\fr,\cong}_{\nicefrac{1}{2}\partial}(W_{g,1})_\ell$.
\end{proof}

The $V_g$-case of Step~\ref{enum:homotopy-selfemb} is harder. We first establish two preparatory results.

\begin{prop}\label{prop:HhtyhAutVgWg}
There is an isomorphism of graded $\bfQ[\check{\Lambda}_{V} \rtimes \langle \rho \rangle]$-modules
\[\pi_*(\Omega_0\BhAut^{\cong,\ell}_{D^{2n}}(V_g, W_{g,1}))_\bfQ \cong s^{-(2n-2)}(K_{1^3}\oplus K_{2^2})\otimes\bfQ^-\quad\text{for }*<3n-3.\] In particular, the action factors in this range over the surjection $\check{\Lambda}_{V} \rtimes \langle \rho \rangle \twoheadrightarrow G_V\rtimes \langle\rho \rangle$. Moreover, as a module over the subgroup $\GL(\bar{H}_V^\bfZ)\le  G_V\rtimes\langle\rho\rangle$, the image of the Lie bracket map
\[[-,-] \colon \pi_{*}(\Omega_0\BhAut^{\cong,\ell}_{D^{2n}}(V_g, W_{g,1}))_\bfQ^{\otimes 2} \lra \pi_{*}(\Omega_0\BhAut^{\cong,\ell}_{D^{2n}}(V_g, W_{g,1}))_\bfQ\]
contains a $\bfQ[2n-2]^{\oplus 2}$-summand.
\end{prop}
\begin{proof}
\cref{thm:KrannichhAut} provides an isomorphism 
\[\pi_*(\Omega_0 \BhAut^{\cong,\ell}_{D^{2n}}(V_g, W_{g,1}))_\bfQ \cong \mathrm{ker}\big(\cL ie\dl H_W,2n\dr \to \cL ie\dl H_V,2n\dr \big) \otimes \bfQ^-\subset \Der_\omega(\bfL(H_W))\] of graded $\bfQ[G_V\rtimes \langle\rho\rangle]$-modules, so the first claim follows from \cref{rem:LieChar}. To prove the claim about the Lie bracket, note that \cref{thm:KrannichhAut} in particular implies that the restriction map $\Omega_0\BhAut_{D^{2n}}(V_g, W_{g,1})\ra \Omega_0\BhAut_\partial(W_{g,1})$ is injective on rational homotopy groups, so it follows from \cref{thm:haut-of-wg} that the Lie algebra structure on $\pi_*(\Omega_0 \BhAut^{\cong,\ell}_{D^{2n}}(V_g, W_{g,1}))_\bfQ$ is induced by the restriction of the commutator bracket on the derivation space $\Der_\omega(\bfL(H_W))$. 

The isomorphism 
\[s^{-(2n-2)}(1^3) \otimes_{\Sigma_3} H_W^{\otimes 3} = \cL ie((H_W, 2n))_{n-1} \otimes \bfQ^- \cong \Der_\omega(\bfL(H_W))_{n-1}\]
of Theorem \ref{thm:haut-of-wg} is given explicitly, in \cite[Proposition 6.6]{BerglundMadsen}, by
\[
\map{t}{s^{-(2n-2)}(1^3) \otimes_{\Sigma_3} H_W^{\otimes 3}}{\Der_\omega(\bfL(H_W))_{n-1}}{[1 \otimes v_1 \otimes v_2 \otimes v_3]}{\lambda(v_3, -) [v_1, v_2] + \lambda(v_2, -)[v_3, v_1] + \lambda(v_1, -)[v_2, v_3].}
\]
Moreover, for $[1 \otimes v_1 \otimes v_2 \otimes v_3] \in (1^3) \otimes_{\Sigma_3} H_W^{\otimes 3}$ to represent an element of the subspace $K_{1^3} \subset S_{1^3}(H_W)\cong (1^3) \otimes_{\Sigma_3} H_W^{\otimes 3}$ it suffices for some $v_i$ to lie in $K_W=\ker(H_W\ra H_V)$. Recall from \cref{sec:standard-model} the basis $(e_i,f_i)_{1\le i\le g}$ such that the $e_i$'s form a basis of $K_W$ and the $f_i$'s map injectively to a basis of $H_V$. Using this basis, we consider the commutator of derivations
\[\textstyle{\Delta_1 \coloneq \sum_{i,j,k=1}^g [t(e_i, f_j, e_k), t(f_i, f_k, e_j)] \in \Der_\omega(\bfL(H_W))_{2(n-1)}}.\]
This is a $\GL(\bar{H}_V^\bfZ)$-invariant vector, as it is obtained from the invariant vector $\nu^{\otimes 3} \in (H_W^{\otimes 2})^{\otimes 3}$ with $\nu \coloneq \sum_{i=1}^g e_i \otimes f_i \in H_W^{\otimes 2}$ by permuting factors, applying $t \otimes t$ to get an element of $\Der_\omega(\bfL(H_W))_{n-1}^{\otimes 2}$, then taking the commutator. Furthermore, both $e_i\otimes f_j\otimes e_i$ and $f_i \otimes e_j \otimes f_k$ lie in $K_{1^3}$, as they contain $e$'s. Thus $\Delta_1$ is a $\GL(\bar{H}_V^\bfZ)$-invariant vector in the image of the Lie bracket map in question. Similarly, 
we consider the commutator
\[\textstyle{\Delta_2 \coloneq \sum_{i,j,k=1}^g [t(e_i, f_i, e_j), t(f_j, e_k, f_k)] \in \Der_\omega(\bfL(H_W))_{2(n-1)}},\]
which is also $\GL(\bar{H}_V^\bfZ)$-invariant and in the image of the Lie bracket map in question, by a similar argument. Since $K_{2^2}$ is semi-simple as a $\GL(\bar{H}_V^\bfZ)$-module, to show that the image of this Lie bracket contains two copies of the trivial representation it suffices to show that $\Delta_1$ and $\Delta_2$ are linearly-independent. By explicitly writing out the definition of these derivations $\Delta_1$ and $\Delta_2$ using the above map $t$, expanding everything out, and using the Jacobi identity and the Leibniz rule repeatedly, one may painstakingly evaluate these derivations on $f_1\in H_W$ to be
\begin{align*}
\Delta_1(f_1) &= (4g + 4(-1)^n) \cdot \textstyle{\sum_{i=1}^g [f_1, [e_i, f_i]] + (-(-1)^n 2g-4) \cdot \sum_{i=1}^g[[f_1, f_i], e_i]},\\
\Delta_2(f_1) &= (-(-1)^{n}2g -2) \cdot \textstyle{\sum_{i=1}^g[f_1, [e_i, f_i]]}, 
\end{align*}
which are indeed linearly independent in the free Lie algebra $\bfL(H_W)$ since we assumed $g\ge2$.
\end{proof}
\begin{rem}\vspace{-0.1cm}
We produced $\Delta_1$ and $\Delta_2$ by contemplating the trivalent graphs
$\begin{tikzpicture}
 [baseline=-.65ex]
 \node[circle,draw,fill,inner sep=1pt] (v) at (0,0) {};
 \node[circle,draw,fill,inner sep=1pt] (w) at (0.5,0) {};
\draw (v) edge[bend left] (w) edge[bend right] (w) edge (w);
\end{tikzpicture}$ and $\begin{tikzpicture}
 [baseline=-.65ex]
 \node[circle,draw,fill,inner sep=1pt] (v) at (0,0) {};
 \node[circle,draw,fill,inner sep=1pt] (w) at (0.5,0) {};
\draw (v) edge (w);
\draw (v) to[in=-140,out=-220,loop] (v);
\draw (w) to[in=40,out=-40,loop] (w);
\end{tikzpicture}$. 
\end{rem}

The second ingredient in the proof of the $V_g$-case of Step~\ref{enum:homotopy-selfemb} is the following:

\begin{lem}\label{lem:stable-twisted-homology}Fix $k>1$ and $n\ge4$. The twisted homology $\oH_*(\BEmb^{\fr,\cong}_{\nicefrac{1}{2}D^{2n}}(V_g,W_{g,1})_\ell;S_k(H^\vee_V))$ vanishes in a range of homological degrees increasing with $g$. 
\end{lem}
\begin{proof}From the Serre spectral sequence of the second delooped Weiss fibre sequence in \eqref{equ:delooped-framed-Weiss-fs} with coefficients in $S_k(H^\vee_V)$, we see that it suffices to show vanishing of $\oH_*(\BDiff^{\fr}_{D^{2n}}(V_g)_\ell;S_k(H^\vee_V))$ in a range of homological degrees increasing with $g$. Fixing an auxiliary 1-dimensional $\bfQ$-vector space $U$, an odd number $m>n$, and a functorial model of Eilenberg-MacLane spaces $K(-,m)$, we consider the tangential structure $\theta_m\colon K(U,m)\times \EO(2n+1)\ra \BO(2n+1)$ given by the projection to $\EO(2n+1)$ followed by the canonical map $\EO(2n+1)\ra \BO(2n+1)$. A $\theta_m$-structure on $V_g$ has an underlying framing, and this induces a fibration sequence
\begin{equation}\label{equ:auxiliary-structure}\Map_{D^{2n}}(V_g,K(U,m))\lra \BDiff^{\theta_m}(V_g)_\ell\lra \BDiff^{\fr}(V_g)_\ell.\end{equation}
By \cite[Theorem A*, Corollary B*]{BotvinnikPerlmutter}, there is a map
$\BDiff^{\theta_m}(V_g)_\ell\ra \Omega^\infty_0\Sigma^\infty_+K(U,m)$ which induces an isomorphism in homology in a range of degrees increasing with $g$. This map is functorial in $U$ up to homotopy, so denoting the free graded commutative algebra by $S^*(-)$ we obtain an isomorphism of graded $\bfQ[\GL(U)]$-modules
\[\oH_*(\BDiff^{\theta_m}(V_g)_\ell;\bfQ)\cong \oH_*(\Omega^\infty_0\Sigma^\infty_+K(U,m);\bfQ)\cong S^*(\widetilde{\oH}_{*}(K(U,m);\bfQ))\cong S^*(U[m])\] 
in a range of degrees increasing with $g$. Combining this with the isomorphism \[\oH_*(\Map_{D^{2n}}(V_g,K(U,m));\bfQ)\cong S^*(\bar{H}_V^\vee\otimes U[m-n])\] of graded $\bfQ[\GL(U)]$-modules, we see that the Serre spectral sequence of \eqref{equ:auxiliary-structure} has in range of total degrees increasing with $g$ the form
\begin{equation}\label{equ:e2-twisted-ss}
E_{p,q}^2=\oH_p\big(\BDiff^{\fr}(V_g)_\ell;S^*(\bar{H}_V^\vee\otimes U[m-n])_q\big)\implies S^*(U[m])_{p+q}.
\end{equation}
This is a spectral sequence of $\bfQ[\GL(U)]$-modules, and $\GL(U) = \bfQ^\times$ acts on the row $q=(m-n) \cdot k$ with weight $k$, i.e.\ as $(-)^k$. From \eqref{equ:e2-twisted-ss}, we see that different rows of the spectral sequence have different weights, so the spectral sequence collapses at the $E_2$-page and there are no extension issues. Similarly, the weight $k$ part of $S^*(U[m])$ is $S^k(U[m])$, so there in an isomorphism
\[\oH_p\big(\BDiff^{\fr}(V_g)_\ell;S^k(\bar{H}_V^\vee\otimes U[m-n])\big)\cong S^k(U[m])_{p+k(m-n)},\]
and this vanishes for $k \geq 2$, as $U[m]$ is $1$-dimensional and in odd degree. Finally, as $U \cong \bfQ$ and $m$ is odd we have $S^k(\bar{H}_V^\vee\otimes U[m-n]) \cong S^k(s^{m-1}{H}_V^\vee) \cong s^{k(m-1)}S^k(H_V^\vee)$, so the result follows.
\end{proof}

Equipped with \cref{prop:HhtyhAutVgWg} and Lemma \ref{lem:stable-twisted-homology}, we establish the $V_g$-part of Step~\ref{enum:homotopy-selfemb}.

\begin{prop}\label{prop:HtpyEmbVg}\ 
\begin{enumerate}
 \item For $n> 4$ and all large enough $g$, there is an epimorphism of graded $\bfQ[\check{\Lambda}_{V} \rtimes \langle \rho \rangle]$-modules $\partial_V\colon s^{-(2n-2)}K_{2^2}\otimes\bfQ^-\ra S_2(H_V)$ and an isomorphism of graded $\bfQ[\check{\Lambda}_{V} \rtimes \langle \rho \rangle]$-modules
\[\pi_*(\Omega_0\BEmb^{\fr,\cong}_{\nicefrac{1}{2}D^{2n}}(V_g,W_{g,1}))_\bfQ \cong (
s^{-(2n-2)}K_{1^3}\otimes \bfQ^-)\oplus \ker(\partial_V)\quad\text{for }*<3n-6.
 \] In particular, the action factors over $\check{\Lambda}_{V} \rtimes \langle \rho \rangle \twoheadrightarrow G_V\rtimes \langle\rho \rangle$ in this range. Moreover, when considered as a module over the subgroup $\GL(\bar{H}_V^\bfZ)\le G_V\rtimes\langle \rho\rangle$, the image of the bracket 
 \[[-,-] \colon \pi_{*}(\Omega_0\BEmb^{\fr,\cong}_{\nicefrac{1}{2}D^{2n}}(V_g,W_{g,1}))_\bfQ^{\otimes 2} \lra \pi_{*}(\Omega_0\BEmb^{\fr,\cong}_{\nicefrac{1}{2}D^{2n}}(V_g,W_{g,1}))_\bfQ\]
contains a $\bfQ[2n-2]^{\oplus 2}$-summand.
\item For $n=3,4$ and all large enough $g$, there is an isomorphism of graded $\bfQ[\check{\Lambda}_{V} \rtimes \langle \rho \rangle]$-modules
\[\pi_*(\Omega_0\BEmb^{\fr,\cong}_{\nicefrac{1}{2}D^{2n}}(V_g,W_{g,1}))_\bfQ \cong (
s^{-(2n-2)}K_{1^3}\otimes \bfQ^-)\quad\text{for }*<3n-6.
 \]
\end{enumerate}
\end{prop}
\begin{proof}
By \cref{lem:connectivity-framing-comparison}, in the range of degrees in question we may exchange framings with stable framings. We consider the map of horizontal fibre sequences (see Sections~\ref{sec:diff-and-self-emb} and \ref{sec:weiss-fs})
\[
\begin{tikzcd}
\BEmb^{\sfr,\cong}_{\nicefrac{1}{2}\partial}(V_g)_C \rar \dar{\varphi_0}&\BEmb^{\sfr}_{\nicefrac{1}{2}D^{2n}}(V_g,W_{g,1})_\ell \rar \dar{\varphi_1}& \BEmb^{\sfr,\cong,\ext}_{\nicefrac{1}{2}\partial}(W_{g,1})_\ell \dar{\varphi_2}\\
\overline{\BhAut_\partial(V_g)} \rar& \BhAut^{\cong,\ell}_{D^{2n}}(V_g, W_{g,1})  \rar& \BhAut_{\partial}^{\cong,\ell,\ext}(W_{g,1})
\end{tikzcd}
\]
where $(-)_C$ stands for a certain collection of path-components and $\overline{(-)}$ for a certain covering space. The homotopy fibre $\hofib(\varphi_2)$ is a covering space of the homotopy fibre of the analogous map without the $\ext$-superscripts and as the latter has finite fundamental group and trivial rational homotopy groups in the range $1<*< 3n-5$, by \cref{thm:fibre-to-haut}, the same holds for $\hofib(\varphi_2)$. By a similar argument, using Theorems \ref{thm:fibre-to-haut} and \ref{thm:ClasperMain} the homotopy fibre $\hofib(\varphi_0)$ has finite fundamental group and its  higher rational homotopy groups agree with $S_2(H_{V})$ in degrees $1< * < 3n-5$, which is concentrated in degree $2n-2$. Thus $\pi_*(\hofib(\varphi_1))_\bfQ$ vanishes in degrees $*<3n-5$ for $*\neq 2n-2$, and in degree $* = 2n-2$ it agrees with $S_2(H_V)_{2(n-1)}/L$ where $L\coloneq\im(\pi_{2n-1}(\hofib(\varphi_2))_\bfQ\ra \pi_{2n-2}(\hofib(\varphi_0)_\bfQ=S_2(H_V)_{2(n-1)}))$, which is trivial if $2n-1<3n-5$, i.e.\ $n>4$. The long exact sequence for the fibration induced by $\varphi_1$, combined with Corollary \ref{prop:HhtyhAutVgWg}, thus gives an isomorphism of $\bfQ[\check{\Lambda}_{V} \rtimes \langle \rho \rangle]$-modules
\begin{equation}\label{equ:homotopy-lower-degree}\pi_*(\varphi_1) \colon \pi_{*}(\Omega_0\BEmb^{\fr}_{\nicefrac{1}{2}D^{2n}}(V_g,W_{g,1})_\ell)_\bfQ \xlra{\cong} s^{-(2n-2)}(K_{1^3} \oplus K_{2^2})\otimes \bfQ^-\end{equation}
in degrees $*<3n-6, *\neq 2n-2,2n-3$. Under the assumption $n>4$ of the first claim, we thus have an exact sequence of $\bfQ[\check{\Lambda}_{V} \rtimes \langle \rho \rangle]$-modules
\begin{equation*}
\begin{tikzcd}[column sep=0.5cm, row sep=0.2cm]
   0 \rar & \pi_{2n-2}(\Omega_0\BEmb^{\fr}_{\nicefrac{1}{2}D^{2n}}(V_g,W_{g,1})_\ell)_\bfQ\rar\ar[draw=none]{d}[name=Y, anchor=center]{}
             & (K_{2^2} \otimes \bfQ^-)_{4(n-1)}  \rar{\partial_V} & S_2(H_V)_{2(n-1)} \ar[rounded corners,
                to path={ -- ([xshift=2ex]\tikztostart.east)
                	|- (Y.center) \tikztonodes
                	-| ([xshift=-2ex]\tikztotarget.west)
                	-- (\tikztotarget)}]{dll}\\
          & \pi_{2n-3}(\Omega_0\BEmb^{\fr}_{\nicefrac{1}{2}D^{2n}}(V_g,W_{g,1})_\ell)_\bfQ\rar
               & 0, & 
\end{tikzcd}
\end{equation*}
so to finish the proof of the first part it suffices to show that $\partial_V$ is an epimorphism, or equivalently (as $S_2(H_V)$ is an irreducible $\bfQ[\check{\Lambda}_{V}]$-module) that $\partial_V$ is nontrivial. 

Suppose for a contradiction that $\partial_V$ is trivial. Using \eqref{equ:homotopy-lower-degree}, we see that the rational Hurewicz homomorphism of the universal cover $\BEmb^{\fr}_{\nicefrac{1}{2}D^{2n}}(V_g,W_{g,1})^\sim_\ell$ is an isomorphism in degrees $\le 2n-2$, so we conclude that
\[\widetilde{\oH}_{*}(\BEmb^{\fr}_{\nicefrac{1}{2}D^{2n}}(V_g,W_{g,1})^\sim_\ell;\bfQ)\cong s^{-(2n-2)}K_{1^3}\otimes \bfQ^-\oplus S_2(H_V)\quad\text{for }*\le 2n-2.\]
Using that $\oH_*(\check{\Lambda}_V;S_2(H_V^\vee))$ and $\oH_*(\check{\Lambda}_V;(K_{1^3})\otimes S_2(H_V^\vee))$ vanish in a range of homological degrees increasing with $g$ by \cref{lem:GV-homology} \ref{enum:gv-homology-iii} and \ref{enum:gv-homology-iv}, we see from the universal cover spectral sequence with $S_2(H_V^\vee)$-coefficients that \begin{equation}\label{equ:contradicting-coinvariants}\oH_{2n-2}(\BEmb^{\fr}_{\nicefrac{1}{2}D^{2n}}(V_g,W_{g,1})_\ell;S_2(H_V^\vee))\cong \left[S_2(H_V) \otimes S_2(H_V^\vee)\right]_{\check{\Lambda}_V}\end{equation}
for large enough $g$. But evaluation induces a nontrivial $\check{\Lambda}_V$-equivariant map $S_2(H_V) \otimes S_2(H_V^\vee) \to \bfQ[0]$, so \eqref{equ:contradicting-coinvariants} is nontrivial and this contradicts \cref{lem:stable-twisted-homology}. Thus $\partial_V$ must be nontrivial and hence an epimorphism, which finishes the proof of the isomorphism statement of the first claim. To show the claim about the bracket, note that since $2n-2$ lies in the isomorphism range the target of the bracket map in degree $2n-2$ is semi-simple as a $\GL(\bar{H}^\bfZ_V)$-module. Using this, the claim follows from the addendum of Corollary \ref{prop:HhtyhAutVgWg} and naturality of the bracket under $\pi_*(\varphi_1)_\bfQ$.

For $n=3$ the claim follows directly from \eqref{equ:homotopy-lower-degree}, since $2n-3\ge 3n-6$ for $n=3$. For $n=4$, we have $2n-2\ge 3n-6$, so given  \eqref{equ:homotopy-lower-degree} we only need to show that $\pi_{2n-3}(\Omega_0\BEmb^{\fr}_{\nicefrac{1}{2}D^{2n}}(V_g,W_{g,1})_\ell)_\bfQ$ vanishes. If $L=0$, then the same argument as in the case $n>4$ applies. If $L\neq 0$, then $S_2(H_V)_{2n-2}/L=0$ since $S_2(H_V)_{2n-2}$ is irreducible, so the required vanishing follows from the analogue of the above long exact sequence for $n=4$.
\end{proof}

\subsection{Step \ref{enum:homology-cover-selfemb}} Propositions~\ref{prop:HtpyEmbWg} and \ref{prop:HtpyEmbVg} allow us to carry out Step \ref{enum:homology-cover-selfemb}.
\begin{cor}\label{cor:HomologyUnivCovEmbWg} 
There is an isomorphism of graded $\bfQ[\check{\Lambda}_{W} \rtimes \langle \rho \rangle]$-modules
\[\widetilde{\oH}_*(\BEmb^{\fr,\cong}_{\nicefrac{1}{2}\partial}(W_{g,1})_\ell^{\sim};\bfQ)\cong \big(s^{-(2n-3)}V_{1^3}(H_{W})\otimes \bfQ^-\big) \oplus s(\coker^{\Wh}_{W})\oplus s^{2}(\ker^{\Wh}_{W})\quad\text{for }*<3n-5\]
where $(\mathrm{co})\ker^{\Wh}_{W}$ is defined to be trivial for $n\le 4$ and otherwise as the (co)kernel of the bracket
\[[-,-] \colon S_{1^2}(s^{-(2n-2)}V_{1^3}(H_{W})) \lra s^{-(2n-2)}S_{2^2}(H_{W})\otimes \bfQ^-\]
resulting from Proposition \ref{prop:HtpyEmbWg}. In particular, the $(\check{\Lambda}_{W} \rtimes \langle \rho \rangle)$- action factors in this range over the surjection $\check{\Lambda}_{W} \rtimes \langle \rho \rangle \twoheadrightarrow G_W\rtimes \langle\rho \rangle$.
\end{cor}

\begin{proof}For $n\le 4$, the claim follows from \cref{prop:HtpyEmbWg} and the rational Hurewicz theorem. For $n\ge5$, it follows from \cref{prop:HtpyEmbWg} and an application of the Serre spectral sequence to the  map
$\BEmb^{\fr,\cong}_{\nicefrac{1}{2}\partial}(W_{g,1})_\ell^{\sim}\ra K\big(V_{1^3}(H_{W})_{3(n-1)},n\big)$
induced by rationalising and truncating.
\end{proof}

\begin{cor}\label{cor:HomologyUnivCovEmbVg}
For all large enough $g$, there is an isomorphism of graded $\bfQ[\check{\Lambda}_{V} \rtimes \langle \rho \rangle]$-modules
\[\widetilde{\oH}_*(\BEmb^{\fr,\cong}_{\nicefrac{1}{2}D^{2n}}(V_g,W_{g,1})^\sim;\bfQ)\cong \big(s^{-(2n-3)}K_{1^3}\otimes\bfQ^-\big)\oplus s(\coker^{\Wh}_{V})\oplus s^2(\ker^{\Wh}_{V})\quad\text{for }*<3n-5\]
where $(\mathrm{co})\ker^{\Wh}_{V}$ is defined to be trivial for $n\le4$ and otherwise as the (co)kernel of the bracket 
\[[-,-] \colon S_{1^2}(s^{-2(n-1)}K_{1^3} \otimes \bfQ^{-}) \lra \ker(\partial_V)\]
resulting from Proposition \ref{prop:HtpyEmbVg}. In particular, the $(\check{\Lambda}_{V} \rtimes \langle \rho \rangle)$- action factors in this range over the surjection $\check{\Lambda}_{V} \rtimes \langle \rho \rangle \twoheadrightarrow G_V\rtimes \langle\rho \rangle$.
\end{cor}
\begin{proof}
Using the data from \cref{prop:HtpyEmbVg}, this is proved in the same way as \cref{cor:HomologyUnivCovEmbWg}.
\end{proof}

\subsection{Step \ref{enum:homology-selfemb}} 
For Step \ref{enum:homology-selfemb}, we first compute the homology of $ G_W$ and $G_V$ with coefficients in the modules featuring in Corollaries \ref{cor:HomologyUnivCovEmbWg} and \ref{cor:HomologyUnivCovEmbVg}. As before, $W_{g,1}$ is easier and comes first.

\begin{prop}\label{prop:WgSSCohCalc}
In a range of homological degrees increasing with $g$, the bigraded groups $\oH_*( G_W ; s^{-(2n-3)}V_{1^3}(H_W))$, $\oH_*( G_W ; \coker^{\Wh}_{W})$, and $\oH_*(G_W; \ker^{\Wh}_{W})$ vanish.
\end{prop}
\begin{proof}
From \eqref{equ:matrix-description-Gg}, we see that the central element $-\id\in\GL(\bar{H}_W)$ is contained in $G_W$. As it acts as $-1$ on $V_{1^3}(H_W)$, the ``centre kills'' trick implies that the $G_W$-homology vanishes. To deal with the second and third groups, we may assume $n>4$; otherwise $\ker^{\Wh}_{W}$ and $\coker^{\Wh}_{W}$ are zero by definition. For $n>4$, these modules are defined as the (co)kernel of the bracket
\[[-,-] \colon S_{1^2}\big(s^{-(2n-2)}V_{1^3}(H_{W})\big) \lra s^{-(2n-2)}S_{2^2}(H_{W})\]
resulting from \cref{prop:HtpyEmbWg}. We may decompose the source and target of this map as $\bfQ[G_{W}]$-modules (see \cref{sec:schur-manipulations})
\begin{align*}
s^{-(2n-2)}S_{2^2}(H_W) &= s^{-(2n-2)}V_{2^2}(H_W) \oplus V_{1^2}(H_W) \oplus \bfQ[2n-2]\\
S_{1^2}(s^{-(2n-2)}V_{1^3}(H_W)) &= s^{-(4n-4)}V_{1^6}(H_W) \oplus s^{-(2n-2)}V_{1^4}(H_W)
\oplus V_{1^2}(H_W) \\ &\phantom{=\ }\oplus  
 s^{-(2n-2)}V_{2^2}(H_W)
\oplus s^{-(4n-4)}V_{2^2, 1^2}(H_W) \oplus \bfQ[2n-2]
\end{align*} 
In both equations, the summands forming the right hand side are all either zero or irreducible, and any two nonzero summands are mutually distinct. Consequently, in view of the final part of \cref{prop:HtpyEmbWg}, the copies of $\bfQ[2n-2]$ have to be sent isomorphically to each other by the bracketing map. It follows that $\ker^{\Wh}_{W}$ and $\coker^{\Wh}_{W}$ are up to regrading both sums of $V_\lambda(\bar{H}_W)$'s for $\lambda\neq \emptyset$, so the claim follows from \cref{thm:Borel-vanishing}, more specifically from \cref{ex:Borel-vanishing-for-irred}.
\end{proof}

\begin{prop}\label{prop:VgSSCohCalc}
In a range of homological degrees increasing with $g$, the bigraded groups $\oH_*(G_V; \coker^{\Wh}_{V})$ and $\oH_*( G_V; \ker^{\Wh}_{V})$
vanish.
\end{prop}

\begin{proof}
Before starting on the proof let us introduce a general concept. We say that a finite-dimensional (rational) $G_V$-representation $U$ is \emph{algebraic} if it admits a filtration $F_\bullet U$ by $G_V$-subrepresentations such that each representation $F_k U / F_{k-1} U$ is obtained up to isomorphism by restriction along $G_V \twoheadrightarrow \GL(\bar{H}_V^\bfZ)$ of a representation which is a sum of summands of $\{\bar{H}_V^{\otimes p} \otimes (\bar{H}_V^\vee)^{\otimes q}\}_{p-q=k}$ (i.e.\ is an algebraic representation of weight $k$). If $U$ has such a filtration then it has a unique one: restricted to $\GL(\bar{H}_V^\bfZ) \leq G_V$ the representation $U$ is algebraic, and $F_k U$ is the sum of all algebraic $\GL(\bar{H}_V^\bfZ)$-subrepresentations of weight $ \leq k$. We call it the \emph{weight filtration} of $U$, and write $\gr_\bullet(U)$ for its associated graded. Using the latter description, it is clear that maps of algebraic $G_V$-representations preserve the weight filtration, that the property of being algebraic is closed under forming sub- and quotient representations, and under tensoring (so also under applying Schur functors), and that $\gr_\bullet(-)$ is exact and symmetric monoidal. As for any filtered $G_V$-representation, there is a spectral sequence
\begin{equation}\label{eq:FiltModSS}
E^1_{s,t} =  \oH_t(G_V ; \gr_s(U))\implies \oH_t(G_V ; U).
\end{equation}

To start the proof proper, recall that $\ker^{\Wh}_{V}$ and $\coker^{\Wh}_{V}$ are trivial for $n \leq 4$ so there is nothing to show, and for $n>4$ they are defined the (co)kernel of the Lie bracket map 
\[[-,-] \colon S_{1^2}(s^{-2(n-1)}K_{1^3}) \lra \mathrm{ker}(\partial_V) = \mathrm{ker}(\partial_V : s^{-(2n-2)}K_{2^2}\otimes \bfQ^- \twoheadrightarrow S_2(H_V))\]
from \cref{cor:HomologyUnivCovEmbVg}. Denote the image of this map by $\im^{\Wh}_{V}$. We first claim that the source and target of this bracket map are algebraic $G_V$-representations, and so by  the inheritance properties of algebraic representations described above $\ker^{\Wh}_{V}$, $\coker^{\Wh}_{V}$, and $\im^{\Wh}_{V}$ are again algebraic. Using these inheritance properties it suffices to see that $\bar{H}_W$ is an algebraic $G_V$-representation, but the extension $0 \ra \bar{H}_{V}^\vee \ra \bar{H}_{W} \ra \bar{H}_{V} \ra 0$ shows that indeed it is, with $F_1 \bar{H}_W = \bar{H}_W$, $F_0 \bar{H}_W = F_{-1} \bar{H}_W = \bar{H}_V^\vee$, and $F_{-2} \bar{H}_W = 0$, and so $\gr_\bullet(\bar{H}_W) = \bar{H}_V[1] \oplus \bar{H}_V^\vee[-1]$. In the graded setting, we have $\gr_\bullet(H_W) = {H}_V[1] \oplus s^{2(n-1)}{H}_V^\vee[-1]$.

Our strategy will now be as follows. To show that $\oH_*(G_V; \coker^{\Wh}_{V})$ and $\oH_*( G_V; \ker^{\Wh}_{V})$ vanish in a range of homological degrees increasing with $g$, using \eqref{eq:FiltModSS} it suffices to show that $\oH_*(G_V; \gr_s(\coker^{\Wh}_{V}))$ and $\oH_*( G_V; \gr_s(\ker^{\Wh}_{V}))$ vanish in such a range for each $s \in \bfZ$. Using the short exact sequences
\[\begin{gathered}0 \lra \gr_s(\ker^{\Wh}_{V}) \lra  \gr_s(S_{1^2}(s^{-2(n-1)}K_{1^3})) \lra \gr_s(\im^{\Wh}_{V}) \lra 0\\
0 \lra \gr_s(\im^{\Wh}_{V}) \lra  \gr_s(\ker(\partial_V)) \lra \gr_s(\coker^{\Wh}_{V}) \lra 0\end{gathered}\]
this is equivalent to showing that both of the maps
\begin{equation}\label{eq:GrMaps}
\gr_s(S_{1^2}(s^{-2(n-1)}K_{1^3})) \lra \gr_s(\im^{\Wh}_{V}) \lra \gr_s(\ker(\partial_V))
\end{equation}
induce isomorphisms on $G_V$-homology in such a range for each $s \in \bfZ$. This is what we shall do.

We first calculate the left- and right-hand terms in \eqref{eq:GrMaps}. The decomposition of $S_\lambda(A \oplus B)$ as a sum of $S_\mu(A) \otimes S_\nu(B)$'s is given by the coproduct of the Schur function $s_\lambda$ in the ring of symmetric functions (whenever necessary, we compute these using {\tt SageMath}). We have 
\begin{align*}
\gr_\bullet(S_{2^2}(H_W)) &= S_{2^2}(\gr_\bullet(H_W)) = S_{2^2}(H_V[1] \oplus s^{2(n-1)}H_V^\vee[-1])\\
&= S_{2^2}(H_V)[4] \oplus \big(S_{2,1}(H_V) \otimes s^{2n-2}H_V^\vee\big)[2]\\
&\ \quad \oplus\big(S_2(H_V) \otimes S_2(s^{2(n-1)}H_V^\vee) \big)[0] \oplus \big( S_{1^2}(H_V) \otimes S_{1^2}(s^{2(n-1)}H_V^\vee)\big)[0]\\
&\ \quad \oplus \big(H_V \otimes S_{2,1}(s^{2(n-1)}H_V^\vee)\big)[-2] \oplus S_{2^2}(s^{2n-2}H_V^\vee)[-4].
\end{align*}
Neglecting terms of negative weight we then have
\begin{align*}
\gr_\bullet(K_{2^2}) &= \big(S_{2,1}(H_V) \otimes s^{2(n-1)}H_V^\vee\big)[2] \oplus\big(S_2(H_V) \otimes S_2(s^{2(n-1)}H_V^\vee) \big)[0]\\
&\quad \oplus \big( S_{1^2}(H_V) \otimes S_{1^2}(s^{2(n-1)}H_V^\vee)\big)[0] \oplus \cdots,
\end{align*}
and $\gr_\bullet(\ker(\partial_V))$ is the same but desuspended $2(n-1)$ times and with the weight 2 term replaced by $\ker(\gr_2(\partial_V) : S_{2,1}(H_V) \otimes H_V^\vee \twoheadrightarrow S_2(H_V))$. Similarly
\begin{align*}
\gr_\bullet(S_{1^3}(H_W)) &= S_{1^3}(\gr_\bullet(H_W)) = S_{1^3}(H_V[1] \oplus s^{2(n-1)}H_V^\vee[-1])\\
&= S_{1^3}(H_V)[3] \oplus \big(S_{1^2}(H_V) \otimes s^{2(n-1)}H_V^\vee\big)[1]\\
&\quad \oplus \big(H_V \otimes S_{1^2}(s^{2(n-1)}H_V^\vee)\big)[-1] \oplus S_{1^3}(s^{2(n-1)}H_V^\vee)[-3]
\end{align*}
and so $\gr_\bullet(K_{1^3}) = \big(S_{1^2}(H_V) \otimes s^{2(n-1)}H_V^\vee\big)[1] \oplus \big(H_V \otimes S_{1^2}(s^{2(n-1)}H_V^\vee)\big)[-1] \oplus S_{1^3}(s^{2(n-1)}H_V^\vee)[-3]$. Neglecting terms of negative weight, we then have
\begin{align*}
\gr_\bullet(S_{1^2}(s^{-2(n-1)} K_{1^3})) &= S_{1^2}(s^{-2(n-1)} \gr_\bullet(K_{1^3}))\\
&= S_{1^2}(S_{1^2}(H_V) \otimes H_V^\vee)[2] \oplus s^{2(n-1)}\big( S_{1^2}(H_V)\otimes H_V \otimes S_{1^2}(H_V^\vee) \otimes H_V^\vee  \big)[0] \oplus \cdots.
\end{align*}

\subsubsection*{Weight $< 0$}
The first observation is that if $U$ is \emph{any} algebraic $G_V$-representation then for each $s <0$ we have $\oH_*(G_V; \gr_s(U))=0$ in a stable range of degrees, by \cref{lem:GV-homology} \ref{enum:gv-homology-ii}. Thus for $s< 0$ the maps \eqref{eq:GrMaps} all induce isomorphisms on $G_V$-homology in a stable range of degrees.

\subsubsection*{Weight $0$}We must show that the maps
\[S_{1^2}(H_V)\otimes H_V \otimes S_{1^2}(H_V^\vee) \otimes H_V^\vee \ra s^{-2(n-1)}\gr_0(\im^{\Wh}_{V}) \ra \big(S_2(H_V) \otimes S_2(H_V^\vee) \big) \oplus \big( S_{1^2}(H_V) \otimes S_{1^2}(H_V^\vee)\big)\]
are both isomorphisms on $G_V$-homology in a stable range. By \cref{lem:GV-homology} \ref{enum:gv-homology-vi} the left-hand terms has $G_V$-homology $\oH_*(\GL_\infty(\bfZ);\bfQ[0])^{\oplus 2}$ in a stable range, and by \cref{lem:GV-homology} \ref{enum:gv-homology-vii} and \ref{enum:gv-homology-viii} the right-hand term does too.  From the last part of \cref{prop:HtpyEmbVg} it follows that $s^{-2(n-1)}\gr_0(\im^{\Wh}_{V})$ contains two copies of the trivial $G_V$-representation $\bfQ$, so by \cref{lem:GV-homology} \ref{enum:gv-homology-i} in a stable range the $G_V$-homology of the middle term contains $\oH_*(\GL_\infty(\bfZ);\bfQ[2])^{\oplus 2}$ as a summand. As the middle term is a summand of the outer ones, it follows that both maps are isomorphisms on $G_V$-homology in a stable range as required.

\subsubsection*{Weight $2$}We must show that the maps
\[S_{1^2}(S_{1^2}(H_V) \otimes H_V^\vee) \ra \gr_2(\im^{\Wh}_{V}) \ra \ker(\gr_2(\partial_V) : S_{2,1}(H_V) \otimes H_V^\vee \twoheadrightarrow S_2(H_V))\]
are both isomorphisms on $G_V$-homology in a stable range, and we will do so by showing that all three terms have trivial $G_V$-homology in a stable range. By \cref{lem:GV-homology} \ref{enum:gv-homology-ix} the left-hand term has trivial $G_V$-homology in a stable range, and as the middle term is a summand of the left-hand term it does too. For the right-hand term we have a decomposition
\[S_{2,1}(H_V) \otimes H_V^\vee \cong S_2(H_V) \oplus \ker(\gr_2(\partial_V))\]
and by \cref{lem:GV-homology} \ref{enum:gv-homology-x} and \ref{enum:gv-homology-xi} the representations $S_2(H_V)$ and $S_{2,1}(H_V) \otimes H_V^\vee$ both have $G_V$-homology  $\oH_*(\GL_\infty(\bfZ);\bfQ[0]) \otimes \bfQ[1,2(n-1)]$ in a stable range, so $\ker(\gr_2(\partial_V))$ has trivial $G_V$-homology in a stable range.
\end{proof}

Putting together the pieces, we finish the proof of this section's main result, \cref{thm:homology-framed-selfemb}.

\begin{proof}[Proof of \cref{thm:homology-framed-selfemb}]
As the compositions $\check{\Lambda}_{W_{g,1}}\twoheadrightarrow G_W\ra\OSp_\infty(\bfZ)$ and $\check{\Lambda}_{V_{g}}\twoheadrightarrow G_V\ra\GL_\infty(\bfZ)$ consist homology isomorphism in a range of degrees increasing with $g$ as a result of \cref{lem:finite-kernel},  \cref{thm:Borel-vanishing}, and \cref{lem:GV-homology} \ref{enum:gv-homology-i}, it suffices to show that the $E_2$-pages $E_{p,q}^2$ of the Serre spectral sequences of the universal cover fibrations 
\[\begin{gathered}
\BEmb^{\fr,\cong}_{\nicefrac{1}{2}\partial}(W_{g,1})^\sim_\ell \lra \BEmb^{\fr,\cong}_{\nicefrac{1}{2}\partial}(W_{g,1})_\ell \lra B
\check{\Lambda}_W\\
\BEmb^{\fr,\cong}_{\nicefrac{1}{2}D^{2n}}(V_{g},W_{g,1})^\sim_\ell  \lra \BEmb^{\fr,\cong}_{\nicefrac{1}{2}D^{2n}}(V_{g},W_{g,1})_\ell \lra B
\check{\Lambda}_V
\end{gathered}
\]
are concentrated along the bottom row for $q<3n-5$ for large enough $g$ with respect to $p$. We have computed the reduced homology of the fibres in this range in  \cref{cor:HomologyUnivCovEmbWg} and \cref{cor:HomologyUnivCovEmbVg} which in particular show that the $\check{\Lambda}_{W}$- respectively $\check{\Lambda}_{V}$-action factors through $G_W$ respectively $G_V$. As the maps between these groups have finite kernel by \cref{lem:finite-kernel}, the required vanishing follows from \cref{prop:WgSSCohCalc}, \cref{prop:VgSSCohCalc}, and \cref{lem:GV-homology} \ref{enum:gv-homology-iv}.
\end{proof}

\section{Rational homotopy groups of diffeomorphism groups of discs}\label{sec:bdiff-and-btop}
Based on \cref{thm:homology-framed-selfemb}, we compute the rational homotopy groups of $\BDiff_\partial(D^d)$ and $\BDiff_{D^d}(D^{d+1})$ in degrees $*\lesssim \tfrac{3}{2} d$ and deduce a computation of the rational homotopy groups of various related spaces in a similar range. In particular, we prove  \cref{bigthm:diff-discs} and \cref{bigcor:relative-concordance-groups}.

\subsection{Transgression and Pontryagin classes}\label{sec:transgression}
Given a sequence of based spaces $F\ra E\ra B$ whose composition is constant and an abelian group $\bfk$, consider the maps
\[
\oH^k(B;\bfk)\cong \oH^k(B,*;\bfk)\lra \oH^k(E,F;\bfk)\longleftarrow \oH^{k-1}(F;\bfk)\quad\text{for }k>0.
\]
By the long exact sequence of a pair, the image in $\oH^k(E,F;\bfk)$ of any class $c\in \oH^k(B;\bfk)$ that vanishes in $\oH^k(E;\bfk)$ lifts to a (usually non-unique) class $\tilde{c}\in \oH^{k-1}(F;\bfk)$. This has an obvious naturality property: if the sequence $F\ra E\ra B$ receives compatible maps from another sequence $F'\ra E'\ra B'$ of the same type, then the pullback of a lift $\tilde{c}$ to $\oH^{k-1}(F;\bfk)$ is a lift of the pullback of $c$ to $\oH^k(B;\bfk)$. Moreover, note that if $\oH^{k-1}(E;\bfk)=0$, then the rightmost map in the above zig-zag is injective, so the lift $\tilde{c}$ is unique; in this case we write $c^\tau\coloneq \tilde{c}\in \oH^{k-1}(F;\bfk)$. If $F\ra E\ra B$ is a fibration sequence, then $c$ is the transgression of $c^\tau$ in the usual sense, so in this case it enjoys the following the well-known properties:
\begin{enumerate}
\item The pullback of $c^\tau$ along the induced map $\Omega B\ra F$ agrees with $\Omega c\in \oH^{k-1}(\Omega B;\bfk)$.
\item For $\bfk=\bfZ$ and $x \in \pi_k(B)$ we have $\langle c, x \rangle = \langle c^\tau, \partial(x)\rangle$ for the boundary map $\partial\colon \pi_k(B) \to \pi_{k-1}(F)$, where $\langle -, - \rangle\colon \pi_k(X)\otimes \oH^k(X;\bfZ)\ra \bfZ$ is the canonical pairing.\end{enumerate}

\subsubsection{Transgressed Pontryagin and Euler classes}\label{sec:transgressed-pontryagin}We will use this to define classes
\begin{align}\label{equ:transgressed-classes}
\begin{split}
p_i^\tau&\in \oH^{4i-1}(\TOP(d)/\oO(d);\bfQ)\text{ for }i>\lfloor \tfrac{d}{2}\rfloor,\\
(p_n-e^2)^\tau&\in \oH^{4n-1}(\TOP(2n)/\oO(2n);\bfQ),\text{ and}\\
(p_n-E)^\tau&\in \oH^{4n-1}(\TOP(2n+1)/\oO(2n+1);\bfQ)
\end{split}
\end{align}
such that the $p_i^\tau$ are preserved by pullback along the stabilisation map $\TOP(d)/\oO(d)\ra \TOP(d+1)/\oO(d+1)$, and $(p_n-E)^\tau$ pulls back to $(p_n-e^2)^\tau$ for $d=2n$. First, recall that the map $\BO\ra \BTOP$ is a rational homotopy equivalence \cite[p.\,246, 5.0 (5)]{KS}, so there are well-defined Pontryagin classes $p_i\in\oH^{4i}(\BTOP;\bfQ)$ which we may pull back to classes $p_i\in\oH^{4i}(\BTOP(d);\bfQ)$. The first classes in \eqref{equ:transgressed-classes} result from these by applying the above principle to 
\begin{equation}\label{equ:o-to-top}
\TOP(d)/\oO(d)\lra \BO(d)\lra \BTOP(d),
\end{equation}
using the fact that $p_i\in \oH^{4i}(\BO(d);\bfQ)$ vanishes in $\oH^{4i}(\BO(d);\bfQ)$ for $i>\lfloor \tfrac{d}{2}\rfloor$. To construct the other two classes, we write $\G(d) = \hAut(S^{d-1})$ whose delooping receives a map from $\BTOP(d)$ classifying the underlying spherical fibration of an $\bfR^d$-bundle. Passing to the orientation preserving submonoid $\SG(d)\subset\G(d)$ we see from the computation of the rational homotopy groups in \eqref{equ:homotopy-g(n)} that there are rational homotopy equivalences
\begin{equation}\label{equ:homotopy-type-g(n)}
\BSG(2n) \simeq_\bfQ K(\bfQ^-, 2n) \quad\quad\quad \BSG(2n+1) \simeq_\bfQ K(\bfQ^+, 4n).
\end{equation}
The first of these is induced by the Euler class $e \in\oH^{2n}(\BSG(2n);\bfQ)$ and the second is induced by the unique rational cohomology class \[E \in\oH^{4n}(\BSG(2n+1);\bfQ)\] that pulls back to $e^2 \in \oH^{4n}(\BSG(2n);\bfQ)$ under the stabilisation map $\BSG(2n)\ra \BSG(2n+1)$. The $\pm$-superscripts indicate the effect of the involution on $\BSG(d)$ induced by conjugation with a reflection in $\bfR^d$. Using the map $\BSTOP(2n+1) \to \BSG(2n+1)$ we form the class $p_n - E \in\oH^{4n}(\BSTOP(2n+1);\bfQ)$, which vanishes in $\BSO(2n)$ as $E$ becomes $e^2=p_n$. As the map $\BSO(2n)\ra \BSO(2n+1)$ is injective on rational cohomology, it also vanishes on $\BSO(2n+1)$, so the orientation preserving analogue of the fibration \eqref{equ:o-to-top} induces the third class in \eqref{equ:transgressed-classes}. Its pullback to $\TOP(2n)/\oO(2n)$ agrees with the second class in \eqref{equ:transgressed-classes}, which is obtained by performing the same construction to $p_n-e^2\in\oH^{4n}(\BSTOP(2n);\bfQ)$.

\subsection{Smoothing theory, Morlet style}\label{sec:smoothing theory}
By smoothing theory and the Alexander trick there is a homotopy commutative diagram \begin{equation}\label{equ:smoothing-theory}
\begin{tikzcd}
\Omega^{d}\OO(d)\arrow[d,equal]\rar&\BDiff^\fr_\partial(D^d)\rar\dar&\BDiff_\partial(D^d)\dar\\
\Omega^{d}\OO(d)\rar&\Omega^{d}\TOP(d)\rar&\Omega^d\TOP(d)/\oO(d)
\end{tikzcd}
\end{equation}
whose vertical maps are $0$-coconnected as long as $d\neq 4$ by a result due to Morlet \cite[p.\ 241]{KS}. The induced isomorphisms
\begin{equation}\label{equ:smoothing-theory-component}
\pi_*(\BDiff^\fr_\partial(D^d)_\ell)\cong \pi_{*+d}(\TOP(d))\quad\text{and}\quad \pi_*(\BDiff_\partial(D^d))\cong \pi_{*+d}(\TOP(d)/\oO(d))
\end{equation}
in positive degrees are \emph{anti}-equivariant with respect to the reflection involution on the source and the involution on the target that is induced by conjugation with a reflection in $\bfR^d$. For the first isomorphism (the second one is similar), this can be seen by chasing through the smoothing theory equivalence and observing that: (i) the induced isomorphism $\pi_*(\BDiff^\fr_\partial(D^d)_\ell)\cong \pi_{*}(\Omega^d_0\TOP(d))$ is equivariant with respect to the reflection involution on the source and the involution on the target induced by reflection in $\TOP(d)$ \emph{and} a reflection in the loop coordinate, and (ii) reflection in the loop coordinate induces multiplication by $-1$ on homotopy groups. More precisely, the effect of the involution on a class represented by a map $(D^d,\partial D^d)\ra (\TOP(d),*)$ is by precomposition with the reflection $\rho\colon (D^d,S^{d-1})\ra (D^d,S^{d-1})$ and postcomposition with the selfmap of $\TOP(d)$ induced by conjugation with a reflection of $\bfR^d$ in one of the coordinates. 

Similarly, there is an equivalence 
\begin{equation}\label{equ:smoothing-theory-concordances}\BDiff_{D^d}(D^{d+1})\xlra{\simeq} \Omega^d_0\hofib\big(\TOP(d)/\oO(d)\ra\TOP(d+1)/\oO(d+1)\big)\end{equation} for $d\neq 4,5$ involving the stabilisation map $\TOP(d)/\oO(d)\ra\TOP(d+1)/\oO(d+1)$ induced by $(-)\times\bfR$, and this equivalence is compatible up to homotopy with the second equivalence in \eqref{equ:smoothing-theory-component} for $d$ and $d+1$ via the fibre sequence
\[\BDiff_{\partial}(D^{d+1})\lra \BDiff_{D^d}(D^{d+1})\lra \BDiff^\id_{\partial}(D^{d})\]
induced by restricting diffeomorphisms of $D^{d+1}$ to the moving part of the boundary. The induced isomorphism in positive degrees
\begin{equation}\label{equ:smoothing-theory-on-homotopy-conc}\pi_*(\BDiff_{D^d}(D^{d+1}))\cong \pi_{*+d}(\hofib\big(\TOP(d)/\oO(d)\ra\TOP(d+1)/\oO(d+1)\big)\end{equation} is \emph{anti}-equivariant with respect to the reflection involution on the source and the involution on the target induced by conjugation with reflections of $\bfR^d$ and $\bfR^{d+1}$.

\subsubsection{Pontryagin classes on diffeomorphism of discs}\label{sec:pon-classes-on-diff}Combining \eqref{equ:smoothing-theory} with the canonical equivalence $\Omega^{d}\TOP(d)\simeq \Omega^{d+1}\BSTOP(d)$, we may pull back the appropriate loopings of the classes considered in \cref{sec:transgressed-pontryagin} along the vertical maps to obtain classes
\[
\Omega^{d+1}p_i\in \oH^{4i-d-1}(\BDiff^\fr_\partial(D^d);\bfQ)\ \ \text{ and }\ \ \Omega^{d}p_i^\tau\in \oH^{4i-d-1}(\BDiff_\partial(D^d);\bfQ)\text{ for }i>\lfloor d/2\rfloor
\]
and classes
\[
\Omega^{2n}(p_n-e^2)^\tau\in \oH^{2n-1}(\BDiff_\partial(D^{2n});\bfQ)\ \ \text{ and }\ \ \Omega^{2n+1}(p_n-E)^\tau\in \oH^{2n-2}(\BDiff_\partial(D^{2n+1});\bfQ).
\]

\subsection{Diffeomorphisms of discs}\label{section:homotopy-even-discs} 
Recall that for a path-connected homotopy commutative $H$-space $X$, there are canonical cohomology classes in $\oH^i(X;\pi_i(X)_\bfQ)$ for $i\ge1$: they are given by  the composition $\oH_i(X;\bfQ)\ra Q(\oH_*(X;\bfQ))_i\cong P(\oH_*(X;\bfQ))_i\cong \pi_i(X)_\bfQ$ where the second isomorphism is induced by the Hurewicz map and the first is the usual isomorphism between the primitives and indecomposables of the Hopf algebra $\oH_*(X;\bfQ)$ (see \cite[Corollary 4.18, Appendix]{MilnorMoore}). Applying this to $\BGL_{\infty}(\bfZ)^+$ and $\BOSp_{\infty}(\bfZ)^+$ and writing 
\[K_i(\bfZ)\coloneq \pi_i(\BGL_{\infty}(\bfZ)^+)\quad\text{and}\quad K\OSp_i(\bfZ)\coloneq \pi_i(\BOSp_{\infty}(\bfZ)^+)\quad\text{ for }i>0,\]
 we obtain classes
\[
 \xi_i\in \oH^i(\BGL_{\infty}(\bfZ);K_i(\bfZ)_\bfQ)\quad\text{and}\quad \xi^{\OSp}_i\in \oH^i(\BOSp_{\infty}(\bfZ);K\OSp_i(\bfZ)_\bfQ)
\]
for $i>0$. We now consider the framed Weiss fibre sequences \eqref{equ:delooped-framed-Weiss-fs}
\begin{equation}\label{equ:framed-weiss-fs-ready-for-trangression}
\begin{gathered}
\BDiff_{D^{2n}}(D^{2n+1})\lra \BDiff^{\fr}_{D^{2n}}(V_{g})_{\ell}\lra \BEmb^{\fr}_{\nicefrac{1}{2}D^{2n}}(V_g,W_{g,1})_{\ell}\\
\BDiff^{\fr}_\partial(D^{2n})_B\lra \BDiff^{\fr}_\partial(W_{g,1})_{\ell}\lra \BEmb^{\fr, \cong}_{\half\partial}(W_{g,1})_{\ell}
\end{gathered}
\end{equation}
whose total spaces have trivial rational homology in a range of degrees with $g$ by \cref{thm:stable-homology} (assuming $n\ge4$ for the first fibration and $n\ge3$ for the second fibration). Pulling back the classes $\xi_i$ and $\xi^{\OSp}$ along the compositions
\[
\BEmb^{\fr}_{\nicefrac{1}{2}D^{2n}}(V_g,W_{g,1})_{\ell}\ra B G_{V}\ra \BGL_\infty(\bfZ)\quad\text{and}\quad
\BEmb^{\fr, \cong}_{\half\partial}(W_{g,1})_{\ell}\ra B G_{W}\ra \BOSp_\infty(\bfZ)
\]
induced by the homology actions and choosing $g\gg i$ large enough, we thus obtain classes
\begin{equation}\label{equ:xis}
\xi_i^\tau\in \oH^{i-1}(\BDiff_{D^{2n}}(D^{2n+1});K_i(\bfZ)_\bfQ)\quad\text{and}\quad (\xi_i^{\OSp})^\tau\in \oH^{i-1}(\BDiff^{\fr}_\partial(D^{2n})_\ell;K\OSp_i(\bfZ)_\bfQ).
\end{equation}
Before stating the first result of this section, let us recall that all spaces and groups above are compatibly acted upon by the reflection involution $\rho$. As pointed out in \cref{sec:involution-on-mcg} this action is trivial on the group $\GL_\infty(\bfZ)$ and hence also on $K_*(\bfZ)_\bfQ$. On $K\OSp_*(\bfZ)_\bfQ$ however $\rho$ acts by $-1$, because it does so on the indecomposables of $\oH^*(\BOSp_\infty(\bfZ);\bfQ)$, as explained in \cref{sec:stable-cohomology-borel}.

\begin{thm}\label{thm:homotopy-framed-even-disc}
The maps of graded $\bfQ[\langle \rho\rangle]$-modules
\vspace{-0.1cm}
\[
\def\arraystretch{1.5}
\begin{array}{c@{\hskip 0cm} l@{\hskip 0.1cm} c@{\hskip 0.1cm} l@{\hskip 0.1cm} l@{\hskip 0.1cm} l@{\hskip 0cm} }
\textstyle{\bigoplus_{i\ge1} \xi_i^\tau}&\colon&\pi_*(\BDiff_{D^{2n}}(D^{2n+1}))_\bfQ&\lra&s^{-1}K_{*>0}(\bfZ)_\bfQ&\text{ for }n\ge4\\
\textstyle{\bigoplus_{i\ge1} (\xi_i^{\OSp})^\tau}&\colon&\pi_*(\BDiff^{\fr}_\partial(D^{2n})_\ell)_\bfQ&\lra&s^{-1}K\OSp_{*>0}(\bfZ)_\bfQ&\text{ for }n\ge3
\end{array}
\]
are isomorphisms in degrees $*<3n-6$ and epimorphisms in degree $3n-6$.
\end{thm}

\begin{proof}The proof of the two cases are analogous; we carry out that of the first. Consider the framed Weiss fibre sequence from \eqref{equ:delooped-framed-Weiss-fs} \[\BDiff^{\fr}_{D^{2n}}(V_{g})_{\ell} \lra \BEmb^{\fr, \cong}_{\nicefrac{1}{2}D^{2n}}(V_g,W_{g,1})_{\ell} \lra B(\BDiff_{D^{2n}}(D^{2n+1}))\] whose right-hand map is an isomorphism on rational cohomology in a range increasing with $g$ by \cref{thm:stable-homology}. From \cref{thm:homology-framed-selfemb}, we know that for large $g$ the map
$\BEmb^{\fr, \cong}_{\nicefrac{1}{2}D^{2n}}(V_g,W_{g,1})_{\ell}\ra \BGL_\infty(\bfZ) $ induces a rational homology isomorphisms in degrees $*< 3n-5$ and a surjection in degree $3n-5$, so we conclude, firstly, that for large enough the $\xi_i$ give classes $\xi'_{i}\in \oH^{i}(B(\BDiff_{D^{2n}}(D^{2n+1}));K_i(\bfZ)_\bfQ)$, and, secondly, that the map
\[
\textstyle{\prod_{i>0} \xi'_{i}\colon B(\BDiff_{D^{2n}}(D^{2n+1})) \lra \prod_{i>0} K(K_i(\bfZ)_\bfQ,\bfQ)}
\]
is an isomorphism on rational homology in degrees $*< 3n-5$ and an epimorphism in degree $3n-5$. As the left-hand side is a loop space (see \cref{sec:smoothing theory}), the same holds on rational homotopy groups, so the claim follows by looping this map and using $\Omega \xi'_{i}=\xi^\tau_{i}$.
\end{proof}

For even discs \cref{thm:homotopy-framed-even-disc} can be reinterpreted in terms of Pontryagin classes in the sense of \cref{sec:pon-classes-on-diff}. To do so, recall from \cref{sec:stable-cohomology-borel} that for $i>0$ the group $K\OSp_i(\bfZ)_\bfQ$ is trivial unless $i=4j-2n$ for some $j$ and in this case we have the canonical isomorphism $\sigma_{4j-2n}\colon K\OSp_{4j-2n}(\bfZ)_\bfQ\ra\bfQ^-$ respecting the action of the reflection involution, given by evaluating the classes $\sigma_{4j-2n}\in\oH^{4j-2n}(\BOSp_\infty(\bfZ);\bfQ)$ described in \cref{sec:stable-cohomology-borel}. 

\begin{lem}\label{lem:kosp=pi}For $4j-2n>0$, the composition of $\bfQ[\langle\rho\rangle]$-modules
\[(\sigma_{4j-2n}\circ (\xi_{4j-2n}^\OSp)^\tau)\colon \pi_{4j-2n-1}(\BDiff^{\fr}_\partial(D^{2n})_\ell)_\bfQ\lra \bfQ^-\] agrees up to a nontrivial scalar with the evaluation of $\Omega^{2n+1}p_{j}^\tau\in\oH^{4j-2n-1}(\BDiff^{\fr}_\partial(D^{2n})_\ell;\bfQ)$.
\end{lem}
\begin{proof}
As $p_j$ agrees with a nonzero multiple of the Hirzebruch L-class $\cL_j$ plus decomposable cohomology classes, and looping annihilates decomposable classes, the class $\Omega^{2n+1}p_j$ is a non-zero multiple of $\Omega^{2n+1}\cL_{j}$. The latter agrees with $(\sigma_{4j-2n}\circ (\xi_{4j-2n}^\OSp)^\tau)$ as a result of the family signature theorem (see the proof of \cite[Theorem 4.23]{KR-WDisc}).
\end{proof}

\begin{cor}\label{cor:homotopy-diff-even-disc}The maps of graded $\bfQ[\langle \rho\rangle]$-modules
\vspace{-0.1cm}
\[
\def\arraystretch{1.5}
\begin{array}{c@{\hskip 0cm} l@{\hskip 0.1cm} c@{\hskip 0.1cm} l@{\hskip 0.1cm} l@{\hskip 0cm} l@{\hskip 0cm}}
\textstyle{\big(\bigoplus_{4i>2n+1} \Omega^{2n+1}p_i\big)}&\colon&\pi_*(\BDiff^{\fr}_\partial(D^{2n})_\ell)_\bfQ&\lra&\textstyle{\bigoplus_{4i>2n+1}\bfQ^-[4i-2n-1]}\\
\textstyle{\big(\Omega^{2n}(p_n-e^2)^\tau, \bigoplus_{i>n} \Omega^{2n}p^\tau_i\big)}&\colon&\pi_*(\BDiff_\partial(D^{2n}))_\bfQ&\lra&\textstyle{ \bigoplus_{i\ge n}\bfQ^-[4i-2n-1]}
\end{array}
\]
are isomorphisms in degrees $*<3n-6$ and an epimorphism in degree $3n-6$.
\end{cor}
\begin{proof}The claim about the first map is vacuous for $2n\le 4$ and follows otherwise directly by combining \cref{thm:homotopy-framed-even-disc} with \cref{lem:kosp=pi} and the discussion before the statement of that lemma. To establish the second claim we extend the diagram \eqref{equ:smoothing-theory} downwards by suitably looping the map of horizontal fibration sequences
\begin{equation}\label{equ:rational-fibration-o-top-even}
\begin{tikzcd}[column sep=0.2cm]
\TOP(2n)/\oO(2n)\rar\arrow["{((p_n-e^2)^\tau, p_i^\tau)}",d,swap]&\BSO(2n)\dar{{(e , p_i)}}\rar&\BSTOP(2n)\dar{{(e,  p_i)}}\\
 \prod_{i\ge n}K(\bfQ,4i-1)\rar&K(\bfQ,2n)\times \prod_{i=1}^{n-1}K(\bfQ,4i)\rar&K(\bfQ, 2n)\times \prod_{i=1}^\infty K(\bfQ, 4i),
\end{tikzcd}
\end{equation}
whose bottom row is the product of the fibration
\vspace{-0.1cm}
\begin{equation}\label{equ:fibre-sequence-low-pis}
\textstyle{\prod_{i>n}K(\bfQ,4i-1)\lra \prod_{i=1}^{n-1}K(\bfQ,4i)\xlra{\subset} \prod_{i\neq n}K(\bfQ,4i)}\end{equation}
with the fibration involving the square of the fundamental class $\iota\in \oH^{2n}(K(\bfQ,2n);\bfQ)$
\vspace{-0.1cm}
\[K(\bfQ,4i-1)\lra K(\bfQ,2n)\xlra{(\id , \iota^2)} K(\bfQ,2n)\times K(\bfQ,4n).\] Using the extended diagram of fibre sequences, the claim follows from \cref{thm:homotopy-framed-even-disc} and the well-known fact that the middle column in \eqref{equ:rational-fibration-o-top-even} is a rational equivalence.
\end{proof}

To compute $\pi_*(-)_\bfQ$ of $\BDiff_\partial(D^{2n+1})$ and $\BDiff^\fr_\partial(D^{2n+1})_\ell$, we consider the fibre sequence
\begin{equation}\label{equ:fibration-diff-vs-concordances}
\BDiff_\partial(D^{d+1})\lra \BDiff_{D^{d}}(D^{d+1})\lra \BDiff^{\id}_\partial(D^{d}),
\end{equation}
for $d=2n$, acted upon by the reflection involution. The answer involves the classes $\Omega^{2n+1}p_i^\tau\in\oH^{4i-2n-2}(\BDiff_\partial(D^{d+1});\bfQ)$ for $i>n$ and $\Omega^{2n+1}(p_n-E)^\tau\in\oH^{2n-2}(\BDiff_\partial(D^{d+1});\bfQ)$ from \cref{sec:smoothing theory}, and  the pullbacks of the $(2n+2)$th loopings of $E$ and $p_i$ on $\BTOP(2n+1)$ along the map $\BDiff^\fr_\partial(D^{2n+1})\ra \Omega^{2n+2}\BSTOP(2n+2)$. Together with \cref{cor:homotopy-diff-even-disc} this proves \cref{bigthm:diff-discs}.

\begin{cor}\label{cor:homotopy-diff-odd-disc}\label{cor:homotopy-framed-odd-disc}
The maps of graded $\bfQ[\langle\rho\rangle]$-modules
\vspace{-0.1cm}
\[\begin{tikzcd}[column sep=0.2cm]
\pi_*(\BDiff^\fr_\partial(D^{2n+1})_\ell)_\bfQ\arrow[d,"{\big(\Omega^{2n+2}E, \bigoplus_{i\ge n} \Omega^{2n+2}p_i , \bigoplus_{i\ge 1}\xi_{i}^\tau\big)}",swap]&
\pi_*(\BDiff_\partial(D^{2n+1}))_\bfQ\arrow[d,"{\big(\big(\Omega^{2n+1}(p_n-E)^\tau , \bigoplus_{i>n} \Omega^{2n+1}p_i^\tau\big), \bigoplus_{i\ge 1} \xi_{i}^\tau\big)}",swap]\\
\bfQ^-[2n-2]\oplus \bigoplus_{i\ge 1}\bfQ^-[4i-2n-2]\oplus s^{-1}K_{*>0}(\bfZ)_\bfQ&\bigoplus_{i\ge n}\bfQ^-[4i-2n-2]\oplus s^{-1}K_{*>0}(\bfZ)_\bfQ
\end{tikzcd}
\]
are isomorphisms in degrees $*< 3n-7$ and an epimorphism in degree $3n-7$.
\end{cor}
\begin{proof}
For $n\le 3$ the claim about the second map is vacuous since $\pi_1(\BDiff_\partial(D^{2n+1}))$ and $K_1(\bfZ)$ vanish rationally. For $n\ge4$ it follows  from \cref{thm:homotopy-framed-even-disc} and \cref{cor:homotopy-diff-even-disc} by considering the long exact sequence induced by the fibration \eqref{equ:fibration-diff-vs-concordances} for $d=2n$. The first part follows from the second, similarly to the proof \cref{cor:homotopy-diff-even-disc} by using the top row in \eqref{equ:smoothing-theory}.
\end{proof}

Finally, denoting the pullback of the cohomology classes appearing in the previous corollary along $\BDiff_{D^{2n+1}}(D^{2n+2})\ra \BDiff_\partial(D^{2n+1})$ by the same name, we conclude the following by combining \cref{cor:homotopy-diff-even-disc} and \cref{cor:homotopy-diff-odd-disc} with the fibration \eqref{equ:fibration-diff-vs-concordances} for $d=2n+1$.

\begin{cor}\label{cor:homotopy-concordances-odd-disc}
The map of graded $\bfQ[\langle\rho\rangle]$-modules
\[
\textstyle{\big(\Omega^{2n+1}(p_n-E)^\tau,\bigoplus_{i\ge 1} \xi_{i}^\tau\big)\colon \pi_*(\BDiff_{D^{2n+1}}(D^{2n+2}))_\bfQ\lra \bfQ^-[2n-2]\oplus s^{-1}K_{*>0}(\bfZ)_\bfQ }
\]
is an isomorphism in degrees $*< 3n-7$ and an epimorphism in degree $3n-7$.
\end{cor}

There is a different description of the class $\Omega^{2n+1}(p_n-E)^\tau\in\oH^{2n-2}(\BDiff_{D^{2n+2}}(D^{2n+1});\bfQ)$ which is sometimes more convenient. Recall that smoothing theory provides a fibre sequence
\[
\BDiff_{D^{2n+1}}(D^{2n+2})\lra \Omega^{2n+1}\OO(2n+2)/\OO(2n+1)\lra \Omega^{2n+1}\TOP(2n+2)/\TOP(2n+1).
\]
Pulling back $E\in \oH^{4n}(\BTOP(2n+1);\bfQ)$ to $\TOP(2n+2)/\TOP(2n+1)$ gives a class which vanishes in $\OO(2n+2)/\OO(2n+1)$ (as $E=p_n$ in $\BO(2n+1)$), so we obtain a class $E^{\tau}$ of degree $4n-1$ in the cohomology of the fibre of $\OO(2n+2)/\OO(2n+1)\ra \TOP(2n+2)/\TOP(2n+1)$. This gives a class $\Omega^{2n+1}E^\tau\in\oH^{2n-2}(\BDiff_{D^{2n+1}}(D^{2n+2});\bfQ)$, which we now show agrees with the negative of the class $\Omega^{2n+1}(p_n-E)^\tau$ featuring in \cref{cor:homotopy-concordances-odd-disc}:
\begin{lem}We have $-\Omega^{2n+1}E^\tau=\Omega^{2n+1}(p_n-E)^\tau$ in $\oH^{2n-2}(\BDiff_{D^{2n+2}}(D^{2n+1});\bfQ)$.
\end{lem}
\begin{proof}
It suffices to show the statement before looping, i.e.\ that we have $(-E)^\tau=(p_n-E)^\tau$ in the cohomology of $F\coloneq\mathrm{fib}(\OO(2n+2)/\OO(2n+1)\ra \TOP(2n+2)/\TOP(2n+1))$. For this we consider the commutative diagram of horizontal and vertical fibre sequences
\[
\begin{tikzcd}[column sep=0.2cm, row sep=0.2cm]
F\rar\dar&\oO(2n+2)/\oO(2n+1)\rar\dar&\TOP(2n+2)/\TOP(2n+1)\dar\\
\TOP(2n+1)/\oO(2n+1)\rar\dar&\BO(2n+1)\rar\dar&\BTOP(2n+1)\dar\\
\TOP(2n+2)/\oO(2n+2)\rar&\BO(2n+2)\rar&\BTOP(2n+2)
\end{tikzcd}
\]
By definition, the class $(p_n-E)^\tau$ in the cohomology of $F$ is the pullback of the same-named class in $\TOP(2n+1)/\oO(2n+1)$ resulting from the $(-)^\tau$-construction for the middle horizontal fibre sequence applied to the to the class $(p_n-E)\in\oH^{4n}(\BTOP(2n+1);\bfQ)$ which vanishes in $\BO(2n+1)$. By the naturality of $(-)^\tau$ (see \cref{sec:transgression}), the class $(p_n-E)^\tau$ in $F$ thus agrees with applying $(-)^\tau$ to the top horizontal fibre sequence applied to the pullback of $(p_n-E)$ to $\TOP(2n+2)/\TOP(2n+1)$. But the latter agrees with $-E$, since $p_n$ in $\BTOP(2n+1)$ is pulled back from $\BTOP(2n+2)$, so we have $(p_n-E)^\tau=(-E)^\tau$ as claimed. 
\end{proof}

\subsection{Concordances of discs}\label{sec:concordances}
Smoothing corners induces an identification
\begin{equation}\label{equ:smoothing-corners}\Diff_{D^d}(D^{d+1})\cong \oC(D^d)\end{equation}
of $\Diff_{D^d}(D^{d+1})$ with the group of \emph{concordances} $\oC(M)$ for $M=D^d$, where 
\[\oC(M)\coloneq \Diff_{M\times \{0\}\cup \partial M\times [0,1]}(M\times [0,1])\]
for a compact $d$-manifold $M$. Restricting a concordance to $M\times \{1\}$ induces a fibre sequence
\begin{equation}\label{equ:concordance-fibration}
\Diff_\partial(M\times [0,1])\lra \oC(M)\lra \Diff_\partial(M)\end{equation}
which for $M=D^d$ agrees with the looping of \eqref{equ:fibration-diff-vs-concordances} under the appropriate identifications. 

\subsubsection{Two involutions}\label{sec:two-involutions-concordances}For any manifold $M$, there is a \emph{concordance involution} on $\oC(M)$ given by sending a concordance $\phi$ to $(\id_M\times r)\circ (\phi|_{M\times\{1\}}^{-1}\times\id_{[0,1]})\circ \phi\circ (\id_M\times r)$
where $r(t)=(1-t)$, which makes \eqref{equ:concordance-fibration} an equivariant fibre sequence when equipping the base space with the involution given by inversion in the group $\Diff_\partial(M)$ and the fibre with the involution induced by conjugation with $(\id_M\times r)$. If $M=M'\times [0,1]$, then there is another involution on $\oC(M'\times [0,1])$ given by conjugation with $(\id_{M'} \times r\times \id_{[0,1]})$, to which we refer to as the \emph{reflection involution}. On the subspace $\Diff_\partial(M'\times[0,1]\times [0,1])\subset \oC(M'\times [0,1])$ the two involutions agree up to homotopy, since $\pi_0(\Diff([0,1]^2))=\bfZ/2$. We call the second involution the reflection involution, because it agrees for $M=D^d\cong D^{d-1}\times [0,1]$ via the identification \eqref{equ:smoothing-corners} with the reflection involution on $\Diff_{D^d}(D^{d+1})$ we considered so far, given by conjugation with a diffeomorphism $\rho\colon D^{d+1}\ra D^{d+1}$ given by reflection in one coordinate that also induces a reflection on the chosen half-disc $D^d\subset \partial D^{d+1}$ (note that there are several choices for $\rho$, but they are all homotopic as $\pi_0\Diff(D^{d+1},D^d)\cong \bfZ/2$ for $d\ge5$).

\medskip

\cref{thm:homotopy-framed-even-disc} and \cref{cor:homotopy-concordances-odd-disc} determine $\pi_*(\oC(D^d))_\bfQ$ in a range with respect to the reflection involution. This can be upgraded to incorporate the concordance involution:

\begin{thm}\label{thm:concordances-both-involutions}
The maps resulting from \eqref{equ:smoothing-corners}, \cref{thm:homotopy-framed-even-disc}, and \cref{cor:homotopy-concordances-odd-disc} induce
\vspace{-0.1cm}
\[
\def\arraystretch{1.5}
\begin{array}{r@{\hskip 0cm} c@{\hskip 0.1cm} l@{\hskip 0.2cm} l@{\hskip 0.1cm} }
\pi_*(\oC(D^{2n}))_\bfQ&\cong &\big(s^{-2}K_*(\bfZ)_\bfQ\otimes\bfQ^{+,+}[0]\big)&\text{ for }*<3n-7\text{ and }n\ge4\\
\pi_*(\oC(D^{2n+1}))_\bfQ&\cong & \big(s^{-2}K_*(\bfZ)_\bfQ\otimes\bfQ^{+,-}[0]\big)\oplus \bfQ^{-,-}[2n-3] &\text{ for }*<3n-8\text{ and }n\ge3
\end{array}
\]
where the $(\pm,\pm)$-superscripts indicate the (reflection, concordance)-involution. 
\end{thm}
\begin{proof}By \cref{thm:homotopy-framed-even-disc} and \cref{cor:homotopy-diff-odd-disc}, the map $\pi_{*}(\Diff_\partial(D^{2n+1}))_\bfQ\ra \pi_*(\oC(D^{2n}))_\bfQ$ is an epimorphism for $*<3n-7$. On $\Diff_\partial(D^{2n+1})$ the two involutions agree, so the first claim follows from \cref{thm:homotopy-framed-even-disc}. From Corollaries  \ref{cor:homotopy-diff-odd-disc} and \ref{cor:homotopy-concordances-odd-disc}, we see that the map $ \pi_*(\oC(D^{2n+1}))_\bfQ\ra \pi_*(\Diff_\partial(D^{2n+1}))_\bfQ$ is injective for $*<3n-8$. As inversion in a topological group induces multiplication by $-1$ on homotopy groups, the concordance involution acts on $\pi_*(\oC(D^{2n+1}))_\bfQ$ in this range by $-1$. \cref{cor:homotopy-concordances-odd-disc} thus implies the second claim.
\end{proof}

\subsubsection{Stabilisation of concordances}\label{sec:concordance-stability}Concordance spaces can be stabilised: taking products with $[0,1]$ and ``bending'' the result suitably to make it satisfy the boundary condition induces a map
\begin{equation}\label{equ:conc-stabilisation}
\oC(M)\lra \oC(M\times [0,1]),\end{equation}
often called the \emph{Hatcher suspension map} (see e.g.\,\cite[Section 6.2]{IgusaBook}, where this map is called the \emph{upper suspension map}). With respect to the equivalence \eqref{equ:smoothing-corners}, this map agrees for $M=D^d$ with the map induced by $(-)\times \id_{[-1,1]}$ and identifying $D^{d+1}\times[-1,1]\cong D^{d+2}$. This is equivariant with respect to the reflection involution and \emph{anti}-equivariant with respect to the concordance involution (see e.g.\,\cite[Lemma 6.5.1.2]{IgusaBook}). Choosing the reflection involution in the target of \eqref{equ:conc-stabilisation} to be conjugation with the reflection in the stabilised coordinate, we see that the image of this map on homotopy groups lands in the $(+1)$-eigenspace of $\oC(D^{d+1})$ with respect to the reflection involution, so \eqref{equ:conc-stabilisation} is trivial on $(-1)$-eigenspaces with respect to the reflection involution. Its effect on $(+1)$-eigenspaces is the content of the following proposition.

\begin{prop}\label{thm:concordance-stability-eigenspaces}
Let $k\ge1$ and $d>6$. The stabilisation map
\[\pi_*(\oC(D^{d}))_\bfQ\lra \pi_*(\oC(D^{d+k}))_\bfQ\]
is trivial on $(-1)$-eigenspaces with respect to the reflection involution. On $(+1)$-eigenspaces it is
\begin{enumerate}
\item for $d=2n$ an isomorphism for $*<3n-7$ and an epimorphism for $*=3n-7$, and
\item for $d=2n+1$ an isomorphism for $*<3n-8$ and an epimorphism for $*=3n-8$.
\end{enumerate}
\end{prop}

\begin{proof}
We use the identification \eqref{equ:smoothing-corners} throughout the proof. The claim about the $(-1)$-eigenspaces was already observed in the discussion preceding the statement. To see the desired effect on $(+1)$-eigenspaces we use that in the considered range of degrees the $(+1)$-eigenspace of $\pi_*(\oC(D^{d+k}))_\bfQ$ agrees with $s^{-2}K_{*>0}(\bfZ)_\bfQ$ by \cref{thm:homotopy-framed-even-disc} and \cref{cor:homotopy-concordances-odd-disc}, including in the degree of the surjectivity claim (as we assumed $k\ge1$). Moreover, using that $\pi_*(\oC(D^{2n+1}))\cong s^{-2}K_{*>0}(\bfZ)_\bfQ$ for $*<3n-8$ by \cref{cor:homotopy-concordances-odd-disc}, we see that it is enough to show the claim for $d=2n$ even. In this case, using \cref{thm:homotopy-framed-even-disc}, it suffices to show that the classes  $\xi_i^\tau\in\oH^{i-1}(\BDiff_{D^{2n}}(D^{2n+1});K_i(\bfZ)_\bfQ)$ from \cref{section:homotopy-even-discs} pull back from $\BDiff_{D^{2n+k}}(D^{2n+1+k})$ along the iterated stabilisation map. Let us recall how the $\xi_i^\tau$ arose, by considering the composition
\[\BDiff_{D^{2n}}(D^{2n+1})\lra \BDiff^{\fr}_{D^{2n}}(V_g)_\ell\lra \BGL_\infty(\bfZ)^+\]
for large $g$ whose second map is the homology action, followed by the plus construction. The composition is clearly constant and the total space is rationally acyclic in a range by \cref{thm:stable-homology}, so the classes $\xi_i\in \oH^{i}(\BGL_\infty(\bfZ);K_i(\bfZ)_\bfQ)$ give rise to  $\xi_i^\tau\in\oH^{i-1}(\BDiff_{D^{2n}}(D^{2n+1});K_i(\bfZ)_\bfQ)$ via the procedure of \cref{sec:transgression}. Forgetting framings, taking products $(-)\times D^k$, and smoothing corners, this sequence maps compatibly to the analogous maps 
\[\BDiff_{D^{2n+k}}(D^{2n+1+k})\lra \BDiff_{D^{2n+k}}(V_g\times D^k)\lra \BGL_\infty(\bfZ)^+,\]
so by the naturality property explained in \cref{sec:transgression}, it suffices to show that the pullback of the classes $\xi_i\in \oH^{i}(\BGL_\infty(\bfZ);K_i(\bfZ)_\bfQ)$ along the map 
\begin{equation}\label{equ:map-to-K-theory}\BDiff_{D^{2n+k}}(V_g\times D^k)\lra \BGL_\infty(\bfZ)^+\simeq \Omega^\infty_0K(\bfZ)\end{equation}
vanishes. But by Dwyer--Weiss--Williams' improved Riemann--Roch theorem \cite[p.\ 3]{DWW}, this map factors (up to homotopy and multiplication by $(-1)^n$) through the composition 
\[\BDiff_{D^{2n+k}}(V_g\times D^k)\xra{\mathrm{trf}} \Omega_{\chi(V_g\times D^k)}^\infty\bfS\xra{\simeq }\Omega_{0}^\infty\bfS\]
of the Becker--Gottlieb transfer followed by translation by the negative of the Euler characteristic. As $\Omega_{0}^\infty\bfS$ is rationally trivial, the map \eqref{equ:map-to-K-theory} is trivial on cohomology, so the claim follows.
\end{proof}

Combining \cref{thm:concordance-stability-eigenspaces} with the calculations in Corollaries \ref{thm:homotopy-framed-even-disc}, \ref{cor:homotopy-diff-even-disc}, and \ref{cor:homotopy-concordances-odd-disc} with the long exact sequence for the rational homotopy groups of a pair, one deduces the following corollary. It involves the composition $\pi_{2n-2}(\oC(D^\infty),\oC(D^{2n+1}))\ra \bfQ^-$ of the boundary map $\partial$ to $\pi_{2n-3}(\oC(D^{2n+1}))$ with the evaluation against the class $\Omega^{2n+2}E^\tau\in\oH^{2n-2}(\oC(D^{2n+1});\bfQ)^{(-)}$ resulting from \eqref{equ:smoothing-corners} and the discussion below \cref{cor:homotopy-concordances-odd-disc}. We write $\oC(D^\infty)\coloneq \hocolim_k\oC(D^k)$ for the space of \emph{stable concordances} whose homotopy groups are isomorphic to $K_{*+2}(\bfZ)_\bfQ$ as a result of \cref{thm:concordances-both-involutions} and \cref{thm:concordance-stability-eigenspaces}. 

\begin{cor}\label{cor:stable-relative-concordance-spaces}
We have $\pi_*(\oC(D^{\infty}),\oC(D^{2n}))_\bfQ=0$ for $n\ge4$ in degrees $*< 3n-6$. The map \[(\Omega^{2n+2}E^\tau\circ \partial)\colon \pi_*(\oC(D^\infty),\oC(D^{2n+1}))_\bfQ\lra \bfQ^-[2n-2]\] is for $n\ge3$ an isomorphism in degrees $*< 3n-7$ and an epimorphism in degree $*=3n-7$.
\end{cor}

Using the long exact sequence of the triple $(\oC(D^\infty),\oC(D^{d+1}),\oC(D^d))$, this yields \cref{bigcor:relative-concordance-groups}.

\begin{cor}\label{cor:relative-concordance-spaces}
There is a map compatible with the reflection involution 
\[\pi_*(\oC(D^{d+1}),\oC(D^{d}))_\bfQ\lra \bfQ^{-}[d-3]\]
which is an isomorphism in degrees $*<\lfloor \tfrac{3}{2} d\rfloor -8$ and an epimorphism in degree $*=\lfloor \tfrac{3}{2} d\rfloor -8$.
\end{cor}

\begin{rem}In particular, we can read off the optimal rational concordance stable range for discs in large dimensions from \cref{cor:relative-concordance-spaces}: the stabilisation map $\oC(D^d)\ra \oC(D^{d+1})$ is rationally $(d-4)$-connected for $d>9$ and there is a map $\pi_{d-3}(\oC(D^{d+1}), \oC(D^{d}))_\bfQ\ra \bfQ^-$ which is an epimorphism for $d> 9$ and an isomorphism for $d>11$.
\end{rem}

\begin{rem} 
The exclusion of the case $n=3$ in the claim about the first maps in \cref{thm:homotopy-framed-even-disc}, \cref{thm:concordances-both-involutions}, and \cref{cor:stable-relative-concordance-spaces} as well as the case $d=6$ in \cref{thm:concordance-stability-eigenspaces} and in \cref{thm:square-top-to-g} below are due to the fact that the analogue of \cref{thm:stable-homology} in dimension $2n+1=7$ is not yet available.
If this were to be established, then these cases need not be excluded.
\end{rem}

\section{The rational homotopy type of $\BSTOP(d)$ and orthogonal calculus}\label{sec:htp-type-BSTOP}
We conclude this work by determining the rational homotopy type of $\BSTOP(d)$ in a range and by spelling out the consequences of our results to orthogonal calculus.

\subsection{The rational homotopy type of $\BSTOP(d)$}
Our formula for the rational homotopy type of $\BSTOP(d)$ in a range involves the forgetful map to $\BSG(d)$ and the stabilisation map to $\BSTOP$
\[g\colon \BSTOP(d)\lra \BSG(d)\quad\text{and}\quad s\colon \BSTOP(d)\lra \BSTOP.\]
In odd dimensions it also involves a certain spectrum $\Theta\fbt^{(1)}\simeq_\bfQ K(\bfZ)$. 

\subsubsection{The rational homotopy type of $\BTOP(2n)$}
Even dimensions are simpler, so they come first.

\begin{thm}\label{thm:BSTOP2nHty}
For $2n\ge6$, the map \[(g,s)\colon \BSTOP(2n) \lra \BSG(2n)\times\BSTOP\] is rationally $(5n-5)$-connected.
\end{thm}
\begin{proof}
We may show the claim after post-composing the target with the rational equivalence \[\textstyle{(e,\prod_{i\ge 1}p_i)\colon \BSG(2n)\times\BSTOP\xlra{\simeq_\bfQ} K(\bfQ,2n)\times \prod_{i\ge}K(\bfQ,4i)}.\] To prove the modified claim, we consider the map of fibration sequences \eqref{equ:rational-fibration-o-top-even}.
The map $\TOP(2n)/\oO(2n)\ra \TOP/\oO$ given by stabilisation is $(2n+2)$-connected and $\TOP/\oO$ is rationally contractible \cite[p.\,246]{KS}, so the leftmost vertical map in this diagram is rationally $(2n+2)$-connected. As the middle vertical map is a rational equivalence, the righmost vertical map is rationally $(2n+3)$-connected, so we may test its connectivity after looping $(2n+2)$ times. Combining this with $\Omega \BDiff^{\fr}_\partial(D^{2n})\simeq \Omega^{2n+2}\BSTOP(2n)$ (see \cref{sec:smoothing theory}), the claim follows from the first part of \cref{cor:homotopy-diff-even-disc}.
\end{proof}

\subsubsection{$\Theta\fbt^{(1)}$ and the rational homotopy type of\ \ $\TOP(d+1)/\TOP(d)$}\label{sec:FirstDerTop}
Before turning to $\BTOP(2n+1)$, we study the homotopy fibres $\TOP(d+1)/\TOP(d)$ of the stabilisation map $\BTOP(d)\ra \BTOP(d+1)$ induced by $(-)\times \bfR$ by comparison to similarly defined fibres $\oG(d+1a)/\oG(d)$. So far we considered $\oG(d)$ as the space of homotopy automorphisms of $S^{d-1}$ and the stabilisation map $\oG(d)\ra \oG(d+1)$ to be induced by unreduced suspension, but for the remainder of this section it is convenient to work with a different model with better functoriality, namely as the space of \emph{proper} homotopy equivalences of $\bfR^d$, related to $\hAut(S^{d-1})$ by an equivalence
\[\hAut(S^{d-1}) \lra \hAut(\bfR^d)^{\mathrm{prop}}\]
given by taking open cones. With this model for $\oG(d)$ the comparison map $\TOP(d)\ra \oG(d)$ is simply given by inclusion, the stabilisation map $\oG(d)\ra \oG(d+1)$ by taking product with $\bfR$, and we have compatible analogues of the loop-suspension map
\begin{equation}
\label{equ:loop-suspension}S^d=\OO(d+1)/\OO(d)\lra \Omega\OO(d+2)/\OO(d+1)=S^{d+1},\end{equation}
for $\TOP(d)$ and $\oG(d)$ in the form of stabilisation maps
\vspace{-0.1cm}
\begin{equation}\label{equ:loop-suspension-gtop}
\def\arraystretch{1.5}
\begin{array}{c@{\hskip 0.1cm} l@{\hskip 0.1cm} c@{\hskip 0.1cm} }
\TOP(d+1)/\TOP(d)&\lra &\Omega(\TOP(d+2)/\TOP(d+1))\\
\oG(d+1)/\oG(d)&\lra &\Omega(\oG(d+2)/\oG(d+1))
\end{array}
\end{equation}
induced by conjugation with rotations in a $2$-plane (see \cite[Section 1]{BurgheleaLashofStability}). These maps define sequential spectra with $d$th spaces $\TOP(d+1)/\TOP(d)$ and $\oG(d+1)/\oG(d)$, and for reasons that will be become clear later when we discuss orthogonal calculus we denote these spectra by $\Theta\fbt^{(1)}$ and  $\Theta\fbg^{(1)}$, respectively. To have a uniform notation, we denote the sphere spectrum thought of as a sequential spectrum using \eqref{equ:loop-suspension} by $\Theta\fbo^{(1)}$. By construction, the inclusions $\OO(d)\subset\TOP(d)\subset\oG(d)$ induce maps of spectra
\[
\Theta\fbo^{(1)} \lra \Theta\fbt^{(1)}\lra \Theta\fbg^{(1)}
\]
whose composition turns out to be an equivalence; this is a consequence of the fact that the map $\OO(d+1)/\OO(d)\ra \oG(d+1)/\oG(d)$ is $(2d-3)$-connected for $d>2$ by \cite[Main Theorem 3.4, Corollary 6.6]{Haefliger}. We thus have preferred equivalences
\begin{equation}\label{equ:first-derivatives}
\bfS= \Theta\fbo^{(1)}\simeq  \Theta\fbg^{(1)}\quad\text{and}\quad \Theta\fbt^{(1)}\simeq \bfS\vee (\Theta\fbt^{(1)}/\bfS).\end{equation} The homotopy type of $\Theta\fbt^{(1)}$ was determined by Waldhausen as his $A(*)$ \cite[Theorem 2 Addendum (4)]{WaldhausenManifold}, but instead of using his result we reproduce it rationally from our point view.

\begin{prop}\label{prop:first-derivative}
We have $\pi_*(\Theta\fbt^{(1)})_\bfQ\cong K_*(\bfZ)_\bfQ$.
\end{prop}
\begin{proof}
Consider the map of fibre sequences provided by smoothing theory (see \cite[Thm A]{BurgheleaLashofStability})
\begin{equation}
\begin{tikzcd}\label{equ:concordance-stability-smoothing-theory}
\BC(D^d)\dar  \rar & \Omega^{d}\OO(d+1)/\OO(d) \dar  \rar& \Omega^{d}\TOP({d+1})/\TOP(d)\dar\ \\
\BC(D^{d+1}) \rar & \Omega^{d+1} \OO(d+2)/\OO(d+1) \rar & \Omega^{d+1} \TOP({d+2})/\TOP({d+1}),
\end{tikzcd}
\end{equation}
where the map between fibres is the stabilisation map from \cref{sec:concordance-stability}. Taking vertical homotopy colimits and rational homotopy groups, we conclude \[\pi_*(\Theta\fbt^{(1)})_\bfQ\cong \pi_*(\bfS)_\bfQ\oplus \pi_{*-1}({\hocolim}_d\BC(D^d))_\bfQ.\] But $\pi_{*-1}(\hocolim_d\BC(D^d))_\bfQ\cong K_{*>0}(\bfZ)_\bfQ$ by the discussion in \cref{sec:concordance-stability}, and this implies the claim since $\pi_0(\bfS)_\bfQ\cong K_0(\bfZ)_\bfQ$.
\end{proof}

Having determined $\Theta\fbt^{(1)}$, we express the rational homotopy type of $\TOP(d+1)/\TOP(d)$ in a range in terms of $\Theta\fbt^{(1)}$ and  $\oG(d+1)/\oG(d)$. Being the $d$th space in the spectrum  $\Theta\fbt^{(1)}$, the space $\TOP(d+1)/\TOP(d)$ has a canonical map to $\Omega^\infty( S^d\wedge \Theta\fbt^{(1)})$, and similarly for $\oG(d+1)/\oG(d)$.

\begin{thm}\label{thm:square-top-to-g}
For $d> 6$, the square
\[
\begin{tikzcd}
\TOP(d+1)/\TOP(d)\dar\rar& \oG(d+1)/\oG(d)\dar\\
\Omega^\infty(S^d\wedge\Theta\fbt^{(1)})\rar &\Omega^\infty(S^d\wedge\Theta\fbg^{(1)})
 \end{tikzcd}
\]
is rationally $5\lfloor \tfrac{d}{2}\rfloor-5$-cartesian.
\end{thm}

\begin{proof}
Combining the facts that $\OO(d+1)/\OO(d)\ra \oG(d+1)/\oG(d)$ is $(2d-3)$-connected for $d\ge 3$ as mentioned above and that $\OO(d+1)/\OO(d)\ra \TOP(d+1)/\TOP(d)$ is $(d+2)$-connected for $d\ge5$ \cite[p.\,246]{KS}, it follows that the top map is $(d+2)$-connected for all $d\ge5$, so the same holds for the bottom map. We may thus test cartesianess of the square in question after looping $d$ times, in which case we consider the commutative diagram 
\[
\begin{tikzcd}[column sep=0.2cm]
\BC(D^d)\dar{\circled{1}}  \rar & \Omega^{d}\OO(d+1)/\OO(d) \dar{\circled{2}}  \rar& \Omega^{d}\TOP({d+1})/\TOP(d)\dar{\circled{3}}\rar&\Omega^{d}\oG({d+1})/\oG(d)\dar\dar{\circled{4}} \\
\hocolim_d\BC(D^{d}) \rar & \Omega^{\infty} (\Theta\fbo^{(1)}\wedge S^d) \rar & \Omega^{\infty} (S^d\wedge\Theta\fbt^{(1)})\rar&\Omega^{\infty} (S^d\wedge\Theta\fbg^{(1)}),
\end{tikzcd}
\]
obtained from extending the diagram \eqref{equ:concordance-stability-smoothing-theory} to the right using the inclusion $\TOP(-)\subset\oG(-)$ and stabilising the bottom row. The task is now to show that the map $\hofib(\circled{3})_\bfQ\ra \hofib(\circled{4})_\bfQ$ has the claimed connectivity. 

If $d=2n>6$ is even, then the map $\OO(d+1)/\OO(d)\ra \oG({d+1})/\oG(d)$ is easily seen to be a rational equivalence using \eqref{equ:homotopy-type-g(n)} so it follows that the map $\hofib(\circled{2})_\bfQ\ra \hofib(\circled{4})_\bfQ$ is an equivalence. It thus suffices to show that the map $\hofib(\circled{2})\ra \hofib(\circled{3})$ is rationally $(3n-5)$-connected or equivalently that $\hofib(\circled{1})$ is rationally $(3n-6)$-connected. But the latter holds by \cref{cor:stable-relative-concordance-spaces}, so we are done with this case. 

If on the other hand $d=2n+1>6$ is odd, then one can deduce from \eqref{equ:homotopy-type-g(n)} that the sequence
\begin{equation}\label{equ:homotopy-type-g-mod-g}\oO(2n+2)/\oO(2n+1)\lra \oG(2n+2)/\oG(2n+1)\lra \BG(2n+1)\underset{E}{\simeq}K(\bfQ,4n)\vspace{-0.1cm}\end{equation}
is a rational fibre sequence, so we obtain a rational fibre sequence\vspace{-0.1cm}
\[\hofib(\circled{2})\lra \hofib(\circled{4})\xlra{\Omega^{2n+1}E}K(\bfQ,2n-1).\] From the diagram of horizontal rational fibre sequences
\[\begin{tikzcd}[row sep=0.3cm, column sep=0.8cm]
\hofib(\circled{1})\rar&\hofib(\circled{2})\arrow[d,equal]\rar&\hofib(\circled{3})_\bfQ\dar&\\
&\hofib(\circled{2})\rar&\hofib(\circled{4})\rar{\Omega^{2n+1}E}&K(\bfQ,2n-1)
\end{tikzcd}
\]
we see that to finish the proof that $\hofib(\circled{3})\ra \hofib(\circled{4})$ is rationally $(5n-5)-(2n+1)=3n-6$-connected, it suffices to show that the induced map $\hofib(\circled{1})\ra K(\bfQ,2n-2)$ is rationally $(3n-7)$-connected. By construction, this map factors up to homotopy as the composition $\hofib(\circled{1})\ra\BC(D^{2n+1})_\bfQ\ra K(\bfQ,2n-2)$ where the final map is represented by the class $\Omega^{2n+1}E^\tau\in\oH^{2n-2}(\BC(D^{2n+1});\bfQ)$, so the claim follows from \cref{cor:stable-relative-concordance-spaces}.
\end{proof}

\subsubsection{The rational homotopy type of $\BTOP(2n+1)$}In view of \eqref{equ:first-derivatives} \cref{thm:square-top-to-g} equivalently says that the map 
\begin{equation}\label{equ:description-btop-odd-single-map}\TOP(d+1)/\TOP(d)\lra \oG(d+1)/\oG(d)\times \Omega^\infty (S^d\wedge \Theta\fbt^{(1)}/\bfS)\end{equation} is rationally highly connected. In order to use this to provide a description of $\BSTOP(2n+1)$, we rely on the following lemma. We postpone its proof to \cref{sec:pf-factorisation}, as it will involve some orthogonal calculus which we have not yet discussed (but see \cref{rem:alternative-factorisation}).

\begin{lem}\label{lem:factorisation}
The composition
\[
\TOP(2n+2)/\TOP(2n+1)\lra \Omega^\infty (S^{2n+1}\wedge\Theta\fbt^{(1)}/\bfS)\lra \Omega^\infty (S^{2n+1}\wedge \Theta\fbt^{(1)}/\bfS)_\bfQ,\]
where $(-)_\bfQ$ denotes rationalisation, admits a factorisation up to homotopy over a map 
\[\tilde{k} : \BSTOP(2n+1) \lra \Omega^\infty (S^{2n+1}\wedge \Theta\fbt^{(1)}/\bfS)_\bfQ.\]
\end{lem}
\begin{rem}\label{rem:alternative-factorisation}There is more than one way to prove this lemma. Using more of the theory of orthogonal calculus than in the proof we shall give in \cref{sec:pf-factorisation}, we explain in \cref{sec:KThyEuler} that there is in fact a \emph{preferred} factorisation $\tilde{k}$. On the other hand, if one is willing to replace the target of $\tilde{k}$ by its $(5n-5)$-truncation then also a more elementary calculus-free argument can be given, by studying the Serre spectral sequence for the fibre sequence $\TOP(2n+2)/\TOP(2n+1) \to \BSTOP(2n+1) \to \BSTOP(2n+2)$ and exploiting the reflection involution.
\end{rem}

Fixing any factorisation $\tilde{k}$ promised by \cref{lem:factorisation} the odd-dimensional analogue of \cref{thm:BSTOP2nHty} reads as follows. Together with \cref{thm:BSTOP2nHty} it proves \cref{bigthm:top-d}, using \cref{prop:first-derivative} and rational the equivalences \eqref{equ:homotopy-type-g(n)}.

\begin{thm}\label{thm:BSTOP2nPlus1Hty}
For $n\ge3$, the map
\[(g,s,\tilde{k})\colon \BSTOP(2n+1)\lra \BSG(2n+1)\times \BSTOP \times \Omega^\infty (S^{2n+1} \wedge \Theta\fbt^{(1)}/\bfS)_\bfQ\]
is rationally $(5n-5)$-connected.
\end{thm}
\begin{proof}
Consider the commutative diagram of vertical fibration sequences
\[\begin{tikzcd}[row sep=0.4cm, column sep=0.2cm]
\TOP(2n+2)/\TOP(2n+1)\dar\rar &\oG(2n+2)/\oG(2n+1)\times \Omega^\infty ( S^{2n+1} \wedge\Theta\fbt^{(1)}/\bfS)_\bfQ\dar\\
\BSTOP(2n+1)\rar\dar&\BSG(2n+1)\times\BSTOP\times  \Omega^\infty( S^{2n+1} \wedge \Theta\fbt^{(1)}/\bfS )_\bfQ\dar \\
\BSTOP(2n+2)\rar &\BSG(2n+2)\times\BSTOP.
\end{tikzcd}
\]
The top and bottom horizontal arrows are rationally $(5n-5)$-connected by \cref{thm:BSTOP2nHty} and \cref{thm:square-top-to-g}, so the same holds for the middle arrow.
\end{proof}

\begin{rem}As $\BSG(2n+1)\simeq K(\bfQ,4n)$ and $\BSTOP\simeq \prod_{i\ge1}K(\bfQ,4i)$, \cref{lem:factorisation} in particular shows that $\BSTOP(2n+1)$ is rationally a product of Eilenberg--McLane spaces in the $\approx5n$-range.
\end{rem}

\subsection{Orthogonal calculus and the second derivative of $\BTOP(-)$}\label{sec:orthogonal-calc}

The assignment $V \mapsto \BTOP(V)$ defines an \emph{orthogonal functor} $\fbt$, i.e.\ a continuous functor from the topologically enriched category of finite-dimensional inner product spaces to the category of pointed spaces. Similarly $\BO(-)$ and $\BG(-)$ define orthogonal functors $\fbo$ and $\fbt$ (using the model $\oG(V)=\hAut(V)^{\mathrm{prop}}$ discussed in \cref{sec:FirstDerTop}), and there are maps of orthogonal functors \begin{equation}\label{equ:the-orthogonal-stars}\fbo\lra \fbt\lra \fbg\end{equation} induced by the inclusions $\OO(-)\subset\TOP(-)\subset\oG(-)$. By taking homotopy fibres we can also form the orthogonal functor $\fto$ having $\fto(V) = \TOP(V)/\OO(V)$. Such functors are the input for Weiss' theory of orthogonal calculus \cite{WeissOrthogonal} as recalled in the introduction. Before studying the derivatives of \eqref{equ:the-orthogonal-stars}, we add a little more to this recollection. We refer to \cite{WeissOrthogonal} for details.

\subsubsection{Orthogonal derivatives}\label{sec:recollection-derivatives}
Given an orthogonal functor $\fe$, one can use the inclusion $V\subset \bfR\oplus V$ to define a new functor $\fe^{(1)}$ by $\fe^{(1)}(V)\coloneq \hofib(\fe(V)\ra \fe(\bfR\oplus V))$. This new functor comes with additional structure in the form of natural maps of the form
\begin{equation}\label{equ:structure-maps-first-derivative}
\sigma\colon \fe^{(1)}(V)\lra \Omega \fe^{(1)}(\bfR\oplus V),
\end{equation} 
so gives rise to a spectrum $\Theta \fe^{(1)}$---the first \emph{derivative} of $\fe$---whose $d$th space is $\fe^{(1)}(\bfR^d)$. Moreover, the reflection along $\{0\}$ in $\bfR$ induces a $\oO(1)$-action on $\fe^{(1)}(V)$ with respect to which the maps \eqref{equ:structure-maps-first-derivative} are equivariant if $\Omega(-)$ is formed with respect to the standard $\oO(1)$-representation. By a standard untwisting argument (see \cite[p.\, 3758]{WeissOrthogonal}), this makes $\Theta \fe^{(1)}$ into a (na{\"i}ve) $\oO(1)$-spectrum, with $\Omega^\infty \Theta \fe^{(1)} \simeq \hocolim_{d} \Omega^d \fe^{(1)}(\bfR^d)$ on which $\OO(1)$ acts as described above on $\fe^{(1)}(\bfR^d)$ and by reflection in each of the $d$ coordinates of $\Omega^d$. Proceeding as above, one sets $\fe^{(2)}(V)\coloneq \hofib(\sigma\colon \fe^{(1)}(V)\ra \Omega \fe^{(1)})$, which again admits natural maps, this time of the form \[\sigma\colon \fe^{(2)}(V)\lra \Omega^2 \fe^{(2)}(\bfR\oplus V),\]
and these give rise to the second derivative $\Theta \fe^{(2)}$; the $2d$th space is $\fe^{(2)}(\bfR^{2d})$. There is an alternative description of $\fe^{(2)}(V)$ which endows it with an $\OO(2)$-action, which by the untwisting trick mentioned above leads to an $\OO(2)$-spectrum structure on $\Theta \fe^{(2)}$ (see \cite[Section\,3]{WeissOrthogonal}). This procedure can be continued and yields spectra $\Theta \fe^{(k)}$ with $\oO(k)$-action, one for each $k\ge1$.

\subsubsection{The first two rational derivatives of $\BO(-)$, $\BTOP(-)$, and $\BG(-)$}
In \cref{sec:FirstDerTop} we preemptively introduced the notation $\Theta\fbo^{(1)}$, $\Theta\fbt^{(1)}$, and $\Theta\fbg^{(1)}$ for what we can now call the first derivatives of these functors, and in \eqref{equ:first-derivatives} and \cref{prop:first-derivative} we described the rational homotopy types of these spectra. Determining their rational homotopy types as $\OO(1)$-spectra amounts to understanding the $\OO(1)$-action on their rational homotopy groups, which is as follows. We write $\bfQ^\pm$ for $\bfQ$ acted upon by $\oO(1)$ via multiplication by $\pm1$.

\begin{lem}\label{lem:FirstDerivatives}
We have
\[\pi_*(\Theta\fbo^{(1)})_\bfQ\cong \pi_*(\Theta\fbg^{(1)})_\bfQ\cong\bfQ^+[0]\quad\text{and}\quad \pi_*(\Theta\fto^{(1)})_\bfQ\cong \bfQ^+[0] \oplus (K_{*>0}(\bfZ)_\bfQ \otimes \bfQ^-).\]
\end{lem}
\begin{proof}As just mentioned we already determined these groups without the $\oO(1)$-action, so we are left to show that $\oO(1)$ acts as claimed. By definition the $\OO(1)$-spectrum $\Theta\fbo^{(1)}$ has $d$th space $S^d = \bfR^d\cup\{\infty\}$, on which $\OO(1)$ acts by reflecting each of the $d$ coordinates, so it acts by $(-1)^d$ on the homotopy groups and thus trivially on $\Theta\fbo^{(1)}$ under the untwisting equivalence mentioned in \cref{sec:recollection-derivatives}. The composition $\Theta\fbo^{(1)} \to \Theta\fbt^{(1)} \to \Theta\fbg^{(1)}$ is an equivalence, so it remains to show that the $\OO(1)$-action on $\pi_{*>0}(\Theta\fbt^{(1)})_\bfQ$ is by $(-1)$, or equivalently that the $\OO(1)$-action on $\pi_{*}(\Theta\fto^{(1)})_\bfQ$ is by $(-1)$ where $\fto=\TOP(-)/\oO(-)$. As in \eqref{equ:concordance-stability-smoothing-theory} we have $\Omega^{d+1} \fto^{(1)}(\bfR^d) \simeq \oC(D^d)$, and---taking the untwisting into account---under this equivalence $\OO(1)$ acts on homotopy groups by $(-1)^{d+1}$ times the concordance involution. By \cref{thm:concordances-both-involutions} the concordance involution on $\oC(D^d)$ acts by $(-1)^d$ on rational homotopy groups in degrees $* < d-4$, so the claim follows.
\end{proof}

As for second derivatives, for $\fbo$ we already mentioned in the introduction that the metastable EHP-description of the loop-suspension map $S^d= \fbo^{(1)}(\bfR^d)\to \Omega\fbo^{(1)}(\bfR^{d+1})=\Omega S^{d+1}$ shows that $\Theta\fbo^{(2)} \simeq \bfS^{-1}$, with trivial $\OO(2)$-action (see \cite[Example 2.7]{WeissOrthogonal} for details). Based on \eqref{equ:homotopy-type-g(n)} there are several ways to determine the second derivative of $\fbg$; Reis--Weiss \cite[Proposition 3.5]{ReisWeiss} did it by describing an equivalence
\begin{equation}\label{eq:ThetaBg2}
\Theta\fbg^{(2)} \simeq_\bfQ \Map(S^1_+, \bfS^{-1})
\end{equation}
under which the map $\Theta\fbo^{(2)} \to \Theta\fbg^{(2)}$ is the inclusion of the constant maps and $\OO(2)$ acts on $\Map(S^1_+, \bfS^{-1})$ through the canonical action on $S^1$. We complete this picture with the following.

\begin{thm}\label{thm:second-derivative}
The map $\Theta\fbt^{(2)} \to \Theta\fbg^{(2)}$ induced by inclusion is a rational equivalence.
\end{thm}
\begin{proof}
Comparing the square in \cref{thm:square-top-to-g} with the analogous looped square for $d+1$, it follows that the square
\[
\begin{tikzcd}
\fbt^{(1)}(\bfR^{d})=\TOP(d+1)/\TOP(d)\rar \dar["\sigma",d]& \fbg^{(1)}(\bfR^{d}) =\oG(d+1)/\oG(d)\dar["\sigma",d]\\
\Omega\fbt^{(1)}(\bfR^{d+1})=\Omega (\TOP(d+2)/\TOP(d+1))\rar &\Omega\fbg^{(1)}(\bfR^{d+1})=\Omega (\oG(d+2)/\oG(d+1))
\end{tikzcd}
\]
is rationally $(5\lfloor \tfrac{d}{2}\rfloor-c)$-cartesian for a constant $c$, so taking vertical homotopy fibres we conclude that the map $\fbt^{(2)}(\bfR^d) \to \fbg^{(2)}(\bfR^d)$ is rationally $(5\lfloor \tfrac{d}{2}\rfloor-c)$-connected. As these are the $2d$th spaces of $\Theta\fbt^{(2)}$ and $\Theta\fbg^{(2)}$ and $5\lfloor \tfrac{d}{2}\rfloor-2d$ tends to infinity with $d$, the result follows.
\end{proof}

Combining \cref{lem:FirstDerivatives} with Borel's calculation $K_{*>0}(\bfZ)_\bfQ \cong  \oplus_{k \geq 1} \bfQ[4k+1]$ (see \cref{sec:stable-cohomology-borel}), and \cref{thm:second-derivative} with \eqref{eq:ThetaBg2}, we can summarise our discussion of the first two derivatives of $\fbt$ as follows. This in particular includes \cref{bigthm:second-derivative} from the introduction.

\begin{cor}\label{cor:explicit-rational-derivatives}There are equivalences
\begin{enumerate}
\item $\Theta\fbt^{(1)}_\bfQ\simeq H\bfQ^+\vee \bigvee_{k\ge1} \Sigma^{4k+1}H\bfQ^-$ as $\OO(1)$-spectra, and
\item $\Theta\fbt^{(2)}_\bfQ\simeq \Map(S^1_+, \bfS^{-1})_\bfQ$ as $\OO(2)$-spectra.
\end{enumerate}
\end{cor}

We close this subsection by making good for the postponed proof of \cref{lem:factorisation}.

\subsubsection{Proof of \cref{lem:factorisation}}\label{sec:pf-factorisation}
During the proof we use the Taylor tower of an orthogonal functor as summarised in the introduction (see \cite[Section  9]{WeissOrthogonal} for details). We furthermore consider the orientation preserving analogue $\fbst(V) = \BSTOP(V)$ of $\fbt$, the functor  $\ff(V) = \TOP(V \oplus \bfR)/\TOP(V)$ and the canonical morphism $\ff \to \fbst$. Note that $\ff = \fbst^{(1)}$, and that $\fbst$ and $\fbt$ differ by the constant functor on $B\bfZ^\times$ and thus have the same derivatives, which we continue to denote by $ \Theta \fbt^{(k)}$. Now $T_0 \ff$ has contractible values and by general properties of orthogonal calculus we have a preferred equivalence $\Theta \ff^{(1)} \simeq \mathrm{Ind}_e^{\OO(1)} \mathrm{Res}_e^{\OO(1)}\Theta \fbt^{(1)}$ with respect to which the map of $\OO(1)$-spectra $\Theta \ff^{(1)} \to \Theta \fbt^{(1)}$ is adjoint to the identity map. We consider the diagram
\[
\begin{tikzcd}[column sep=0.4cm, ar symbol/.style = {draw=none,"\textstyle#1" description,sloped},
  equivalent/.style = {ar symbol={\simeq}}]
\TOP(2n+2)/\TOP(2n+1) \dar \rar&[-0.25cm] T_1 \ff(\bfR^{2n+1}) \dar &[-0.25cm] \Omega^\infty(S^{2n+1} \wedge_{\OO(1)} {\Theta \ff^{(1)}}) \arrow[l,equivalent ]&[-0.3cm]\Omega^\infty(S^{2n+1} \wedge {\Theta \fbt^{(1)}}) \dar\arrow[l,equivalent] \\
\BSTOP(2n+1) \rar& T_1 \fbst(\bfR^{2n+1}) \arrow[rr, dashed]& & \Omega^\infty(S^{2n+1} \wedge {\Theta \fbt^{(1)}}/\bfS)_\bfQ
\end{tikzcd}
\]
where upper left equivalence is because $T_0 \ff$ has contractible values and the right because ${\Theta \ff^{(1)}}$ is induced from ${\Theta \fbt^{(1)}}$.  As the top row of this diagram is compatible with the canonical map from the statement, it reduces the problem to producing the dashed map---a problem about first Taylor approximations. To construct the dashed map we consider the diagram
\[
\begin{tikzcd}[column sep=0.5cm]
\Omega^\infty(S^{2n+1} \wedge_{\OO(1)} {\Theta \fbt^{(1)}}) \rar \arrow[rd, dotted]&T_1 \fbst(\bfR^{2n+1}) \rar \arrow[d, dashed] &[-0.5cm] T_0 \fbst(\bfR^{2n+1}) \simeq \BSTOP\\
\Omega^\infty(S^{2n+1} \wedge {\Theta \fbt^{(1)}}) \uar \rar & \Omega^\infty(S^{2n+1} \wedge {\Theta \fbt^{(1)}}/\bfS)_\bfQ
\end{tikzcd}
\]
and first produce a dotted map making the lower triangle commute. From \cref{cor:explicit-rational-derivatives} we get
\begin{align*}
\pi_*(S^{2n+1} \wedge {\Theta \fbt^{(1)}})_\bfQ &\cong \textstyle{\bigoplus_{k \geq 1} \bfQ[2n+1+4k+1]\oplus \bfQ[2n+1]}\\
\pi_*(S^{2n+1} \wedge {\Theta \fbt^{(1)}}/\bfS)_\bfQ &\cong \textstyle{\bigoplus_{k \geq 1} \bfQ[2n+1+4k+1]}\\
\pi_*(S^{2n+1} \wedge_{\OO(1)} {\Theta \fbt^{(1)}})_\bfQ &\cong \textstyle{\bigoplus_{k \geq 1} \bfQ[2n+1+4k+1]},
\end{align*}
so there is indeed a unique homotopy class of dotted infinite loop map making the lower triangle commute. To extend this to a map as indicated by the dashed arrow, we observe that $\Omega^\infty(S^{2n+1} \wedge {\Theta \fbt^{(1)}}/\bfS)_\bfQ\simeq \prod_{k \geq 1} K(\bfQ, 2n+4k+2)$, so the problem is to show that certain cohomology classes in $\oH^{2n+4k+2}(\Omega^\infty(S^{2n+1} \wedge_{\OO(1)} {\Theta \fbt^{(1)}});\bfQ)$ are permanent cycles in the Serre spectral sequence for the top row. But this spectral sequence collapses, as $\BSTOP$ is simply-connected and has rational cohomology supported in even degrees. \qed

\subsection{The algebraic $K$-theory Euler class}\label{sec:KThyEuler}
The purpose of this final subsection is twofold: firstly, we explain how a more sophisticated use of orthogonal calculus yields an improvement on \cref{lem:factorisation} in the form of a \emph{preferred} map
\[\tilde{k} : \BSTOP(2n+1) \lra \Omega^\infty (S^{2n+1} \wedge \Theta\fbt^{(1)}/\bfS)_\bfQ=\Omega^\infty (S^{2n+1} \wedge (\Theta\fbt^{(1)}/\Theta\fbo^{(1)}))_\bfQ,\]
together with a nullhomotopy of its precomposition with $\BSO(2n+1) \to \BSTOP(2n+1)$. Secondly, we show how this provides a reinterpretation of our results on $\BSTOP(d)$ in terms of obstructions to finding vector bundle structures on topological $\bfR^d$-bundles. Throughout this section, we freely use the theory of orthogonal calculus and refer to \cite{WeissOrthogonal} for details.

\subsubsection{The intrinsic Euler class}

There is a homotopy cofibre sequence of based $\OO(1)$-spaces \[S^V \lra S^{\bfR \oplus V} \lra \mathrm{Ind}_e^{\OO(1)} S^1 \wedge S^V,\]
where the $\OO(1)$-action on the left-hand and middle spaces is by negation in $V$ and $\bfR \oplus V$ respectively. If $\fe$ is an orthogonal functor with first derivative the $\OO(1)$-spectrum $\Theta \fe^{(1)}$, then the above gives a fibration sequence
\[\Omega^\infty(S^V \wedge_{\OO(1)} \Theta \fe^{(1)}) \lra \Omega^\infty(S^{\bfR \oplus V} \wedge_{\OO(1)} \Theta \fe^{(1)}) \lra \Omega^\infty(S^1 \wedge S^V \wedge  \Theta \fe^{(1)} ),\]
so as $T_0\fe(V)\simeq T_0\fe(\bfR\oplus V)$ on first Taylor approximations there is a homotopy fibre sequence
\[\Omega^{\infty}(S^V \wedge \Theta \fe^{(1)}) \lra T_1 \fe(V) \lra T_1 \fe(\bfR \oplus V).\]
This can be expressed in the following more useful way. The $\OO(1)$-spectrum $\Theta \fe^{(1)}$ arises as the fibre of a parametrised $\OO(1)$-spectrum $\underline{\Theta \fe}^{(1)}$ over $\fe(\bfR^\infty) \simeq T_0 \fe(\bfR \oplus V)$ (i.e. a parametrised spectrum with a fibrewise $\OO(1)$-action), and there is a homotopy cartesian square 
\begin{equation}\label{eq:The UnivEulerClass}
\begin{tikzcd}
T_1 \fe(V) \rar\dar &  T_0 \fe(\bfR \oplus V) \dar{z}\\
T_1 \fe(\bfR \oplus V) \rar{e_\fe} & \underline{\Omega}^\infty(S^{\bfR \oplus V} \wedge \underline{\Theta \fe}^{(1)}),
\end{tikzcd}
\end{equation}
where $\underline{\Omega}^\infty(-)$ denotes the fibrewise infinite loop space of a parametrised spectrum, $z$ is the basepoint section, and $e_\fe$ is a map covering the standard map $T_1 \fe(\bfR \oplus V) \to T_0 \fe(\bfR \oplus V)$. This square is given by \cite[Corollary 8.5]{WeissOrthogonal}. 

Generally speaking, if $\underline{\mathrm{T}}$ is a parameterised spectrum over a space $X$, then homotopy classes of sections of $\underline{\Omega}^\infty(\underline{\mathrm{T}}) \to X$ may be seen as twisted cohomology classes on $X$ with coefficients in $\underline{\mathrm{T}}$, and the set of them is denoted $H^0(X ; \underline{\mathrm{T}})$. If the parameterised spectrum $\underline{\mathrm{T}}$ is pulled back from a (parametrised) spectrum $\mathrm{T}$ over a point, i.e.\,a regular spectrum, then this notion recovers $\mathrm{T}$-cohomology in the usual sense. The basepoint section $z : X \to \underline{\Omega}^\infty(\underline{\mathrm{T}})$ represents the zero cohomology class. If $(X,Y)$ is a space pair then a section $s$ of $\underline{\Omega}^\infty(\underline{\mathrm{T}}) \to X$ together with a homotopy $s\vert_Y \simeq z\vert_Y$ fibrewise over $X$ may be seen as a relative twisted cohomology class, and the set of homotopy classes of such is denoted $H^0(X, Y ; \underline{\mathrm{T}})$.

In the case at hand, the square \eqref{eq:The UnivEulerClass} may thus be interpreted as a relative twisted cohomology class on the pair $(T_1 \fe(\bfR \oplus V), T_1 \fe(V))$ with coefficients in the pullback along $T_1 \fe(\bfR \oplus V) \to T_0 \fe(\bfR \oplus V)$ of the parametrised spectrum $\underline{\Theta \fe}^{(1)}$ over $T_0 \fe(\bfR \oplus V)$. Neglecting that it is a relative class, and pulling back further along $\fe(\bfR \oplus V) \to T_1 \fe(\bfR \oplus V)$, gives an absolute twisted cohomology class 
\[e_\fe \in \oH^{0}\big(\fe(\bfR \oplus V) ; S^{\bfR \oplus V} \wedge \underline{\Theta \fe}^{(1)}\big).\]
If $f : \fe \to \ff$ is a morphism of orthogonal functors, there is an induced map $\Theta f^{(1)} : \underline{\Theta \fe}^{(1)} \to f(\bfR^\infty)^*\underline{\Theta \ff}^{(1)}$ of first derivatives (the target is pulled back along the map $f(\bfR^\infty) : \fe(\bfR^\infty) \to \ff(\bfR^\infty)$), and by naturality we have \[(T_1 f)^*(e_\ff) = (\Theta f^{(1)})_*(e_\fe)\in \oH^0\big(T_1 \fe(\bfR \oplus V), T_1 \fe(V) ; S^{\bfR \oplus V} \wedge f(\bfR^\infty)^*\underline{\Theta \ff}^{(1)}\big).\] More precisely, there are preferred homotopies making the 3-cube given by the corresponding diagrams \eqref{eq:The UnivEulerClass} commute, so again forgetting that $e_\fe$ and $e_\ff$ are relative classes and pulling back we obtain a twisted cohomology class
\[
(e_\ff, e_\fe) \in \oH^0\big( \ff(\bfR \oplus V),  \fe(\bfR \oplus V) ; S^{\bfR \oplus V} \wedge (\underline{\Theta \ff}^{(1)}, \underline{\Theta \fe}^{(1)})\big).
\]

\subsubsection{Application to $\fbo$ and $\fbt$}

Applying the above discussion to the map $\fbo \to \fbt $ gives
\[(e_\fbt, e_\fbo) \in \oH^d(\BTOP(d), \BO(d) ; (\underline{\Theta \fbt}^{(1)}, \underline{\Theta \fbo}^{(1)}).\]

\begin{rem}
If we have understood \cite[p.\ 6 (1)]{DWW} correctly it suggests that $e_{\fbt}$ is the characteristic class associated to the excisive $A$-theory characteristic as constructed in that paper, and that $e_{\fbo}$ is Becker's stable cohomotopy Euler class, so that the above corresponds to \cite[Equation (0.5)]{DWW}. We think of $e_{\fbt}$ as the algebraic $K$-theory Euler class for this reason.
\end{rem}

The discussion in \cite[Remark 5.2]{ReisWeiss} shows that the parametrised spectrum $\underline{\Theta \fbo}^{(1)}$ over $\BO$, and the pullback of the parametrised spectrum $\underline{\Theta \fbt}^{(1)}$ along $\BO \to \BTOP$, are both pulled back along the map $\BO \to \BG$ from the parametrised spectrum given by the canonical action of $\G = \GL_1(\bfS)$ on  $\Theta \fbo^{(1)} =  \Theta \fbo^{(1)} \wedge \bfS$ and $\Theta \fbt^{(1)} =  \Theta \fbt^{(1)} \wedge \bfS$. Upon rationalising this action factors over $\G \to \bfZ^\times$, so ${\underline{\Theta \fbo}^{(1)}}_\bfQ$ and the pullback of ${\underline{\Theta \fbt}^{(1)}}_\bfQ$ to $\BO$ are the pullbacks of $\Theta \fbo^{(1)}_\bfQ \wedge \underline{\bfS}_\mathrm{sign}$ and  $\Theta \fbt^{(1)}_\bfQ \wedge \underline{\bfS}_\mathrm{sign}$ along the orientation map $\BO \to B\bfZ^\times$, where $\underline{\bfS}_\mathrm{sign}$ is the $\bfZ^\times$-spectrum given by $\bfS$ with the sign involution.  Using that $\OO \to \TOP$ is a rational equivalence, it follows that also $\underline{\Theta \fbt}^{(1)}_\bfQ$ itself is the pullback of $\Theta \fbt^{(1)}_\bfQ \wedge \underline{\bfS}_\mathrm{sign}$ along the orientation map $\BTOP \to B\bfZ^\times$. Extending \eqref{eq:The UnivEulerClass} there is a commutative diagram (similarly for $\fbo$)
\begin{equation*}
\begin{tikzcd}[ar symbol/.style = {draw=none,"\textstyle#1" description,sloped},
  equivalent/.style = {ar symbol={\simeq}}]
T_1 \fbt(V) \rar\dar & T_0 \fbt(\bfR \oplus V)  \arrow[r,equivalent ] \dar{z} &[-1cm] \BTOP \rar &[-1cm] B\bfZ^\times \dar{z}\\
T_1 \fbt(\bfR \oplus V) \rar{e_\fbt} & \underline{\Omega}^\infty(S^{\bfR \oplus V} \wedge \underline{\Theta \fbt}^{(1)})  \arrow[rr] &  & \underline{\Omega}^\infty(S^{\bfR \oplus V} \wedge \Theta \fbt^{(1)}_\bfQ \wedge \underline{\bfS}_\mathrm{sign}),
\end{tikzcd}
\end{equation*}
where the left-hand square is homotopy cartesian, the right-hand square is rationally homotopy cartesian, and $z$ is the basepoint section. In particular the lower map and its analogue for $\fbo$ gives a diagram of the form
\begin{equation*}
\begin{tikzcd}[column sep=0.5cm]
T_1 \fbo(\bfR^d) \dar \rar{e_\fbo} & \underline{\Omega}^\infty(S^d \wedge \underline{\Theta \fbo}^{(1)})  \rar \dar & \underline{\Omega}^\infty(S^{d} \wedge \Theta \fbo^{(1)}_\bfQ \wedge \underline{\bfS}_\mathrm{sign}) \dar \rar &[-0.2cm] B\bfZ^\times \dar{z}\\
T_1 \fbt(\bfR^d) \rar{e_\fbt} & \underline{\Omega}^\infty(S^{d} \wedge \underline{\Theta \fbt}^{(1)})  \rar  & \underline{\Omega}^\infty(S^{d} \wedge \Theta \fbt^{(1)}_\bfQ \wedge \underline{\bfS}_\mathrm{sign}) \rar & \underline{\Omega}^\infty(S^{d} \wedge (\tfrac{\Theta \fbt^{(1)}}{\Theta \fbo^{(1)}})_\bfQ \wedge \underline{\bfS}_\mathrm{sign}),
\end{tikzcd}
\end{equation*}
where the right-hand square is developed by taking cofibres of (parametrised) spectra. When $d$ is even the left-hand map in this diagram is a rational equivalence (over $B\bfZ^\times$), so the composition along the bottom is trivial. When $d$ is odd, we consider the composition
\[\BTOP(d) \xlra{e_{\fbt}} \underline{\Omega}^\infty\big(S^{d} \wedge \underline{\Theta \fbt}^{(1)}\big)  \lra \underline{\Omega}^\infty\big(S^{d} \wedge (\tfrac{\Theta \fbt^{(1)}}{\Theta \fbo^{(1)}}\big)_\bfQ \wedge \underline{\bfS}_\mathrm{sign}) \xlra{\cdot 1/2}\underline{\Omega}^\infty\big(S^{d} \wedge (\Theta \fbt_\bfQ^{(1)}/\Theta \fbo_\bfQ^{(1)})_\bfQ \wedge \underline{\bfS}_\mathrm{sign}\big)\]
which qualifies as $\tilde{k}$ in \cref{lem:factorisation} when precomposed with $\BSTOP(2n+1)\ra \BTOP(2n+1)$ as a result of the following lemma, using $\bfS=\Theta\fbo^{(1)}$.

\begin{lem}
The diagram
\begin{equation*}
\begin{tikzcd}
\TOP(d+1)/\TOP(d) \dar \rar& \Omega^\infty(S^{d} \wedge \Theta \fbt^{(1)}) \rar& \Omega^\infty(S^{d} \wedge (\Theta \fbt^{(1)}/\Theta \fbo^{(1)})_\bfQ) \dar{(1 + (-1)^{d+1}) \cdot \mathrm{inc}}\\
\BTOP(d) \rar & T_1 \fbt(\bfR^{d}) \rar & \underline{\Omega}^\infty(S^{d} \wedge (\Theta \fbt^{(1)}/\Theta \fbo^{(1)})_\bfQ \wedge \underline{\bfS}_\mathrm{sign})
\end{tikzcd}
\end{equation*}
commutes up to homotopy.
\end{lem}
\begin{proof}
This is similar to \cref{sec:pf-factorisation}. We set  $\ff(V) \coloneq \TOP(\bfR \oplus V)/\TOP(V)$ and consider $\ff \to \fbt$, where $\Theta \ff^{(1)} = \mathrm{Ind}_e^{\OO(1)} \mathrm{Res}_e^{\OO(1)}\Theta \fbt^{(1)}$ and the map of $\OO(1)$-spectra $\Theta \ff^{(1)} \to \Theta \fbt^{(1)}$ is adjoint to the identity map. There is a diagram
\begin{equation*}
\begin{tikzcd}[ar symbol/.style = {draw=none,"\textstyle#1" description,sloped},
  equivalent/.style = {ar symbol={\simeq}}]
\Omega^\infty(S^{d} \wedge \underline{\Theta \fbt}^{(1)}) \arrow[r,equivalent ] \arrow{dr} &[-0.5cm]\Omega^\infty(S^{d} \wedge_{\OO(1)} \underline{\Theta \ff}^{(1)}) \arrow[r,equivalent] \dar &[-0.5cm]T_1 \ff(\bfR^d) \rar{e_{\ff}}\dar &  \underline{\Omega}^\infty(S^{d} \wedge \underline{\Theta \ff}^{(1)}) \dar\\
&\Omega^\infty(S^{d} \wedge_{\OO(1)} \underline{\Theta \fbt}^{(1)}) \arrow[r]&T_1 \fbt(\bfR^d) \rar{e_{\fbt}} & \underline{\Omega}^\infty(S^{d} \wedge \underline{\Theta \fbt}^{(1)}),
\end{tikzcd}
\end{equation*}
where the right-hand square comes from naturality of the intrinsic Euler class, the middle square is given by the inclusion of the 1st layers (and its top map is an equivalence as $T_0 \ff \simeq *$), the left-hand equivalence is because $\Theta \ff^{(1)}$ is induced from $\Theta \fbt^{(1)}$, the diagonal map is the canonical quotient map, and the quotient along the bottom is the norm. It follows that the map from $T_1 \ff(\bfR^d) \simeq \Omega^\infty(S^{d} \wedge \underline{\Theta \fbt}^{(1)})$ to the bottom right-hand corner is given by $1 + \tau$ for the involution $\tau$. The lemma now follows as the involution acts by $(-1)$ on $(\Theta \fbt^{(1)}/\Theta \fbo^{(1)})_\bfQ$ by \cref{lem:FirstDerivatives}, so acts by $(-1)^{d+1}$ on $S^d \wedge (\Theta \fbt^{(1)}/\Theta \fbo^{(1)})_\bfQ$.
\end{proof}

\subsubsection{Obstruction-theoretic reinterpretation}\label{sec:obstruction-theory-yoga}
It is now simply a reinterpretation of \cref{thm:BSTOP2nHty} and \cref{thm:BSTOP2nPlus1Hty} that the square
\begin{equation*}
\begin{tikzcd}
 \BO(2n+1) \dar \rar & \BTOP(2n+1) \dar{(p_n - E, \prod_{i>n} p_i, \tilde{k})}\\
 B\bfZ^\times \rar& \prod_{i \geq n} K(\bfQ, 4i) \times \underline{\Omega}^\infty\big(S^{2n+1} \wedge (\Theta \fbt^{(1)}/\Theta \fbo^{(1)})_\bfQ \wedge \underline{\bfS}_\mathrm{sign}\big)
\end{tikzcd}
\end{equation*}
is rationally $(5n-6)$-cartesian if $d = 2n+1 \geq 5$, and similarly for $d=2n \geq 6$ with the final factor omitted, $B\bfZ^\times$ replaced by a point, and $p_n-E$ replaced by $p_n-e^2$. In particular, given $d\ge6$ with $d=2n$ or $d=2n+1$ and a map of pairs \[\phi: \big(X, A\big) \lra \big(\BTOP(d), \BO(d)\big)\] such that $(X,A)$ only has relative cells of dimension $\leq 5n-6$ then a complete set of rational obstructions to compressing $\phi$ (i.e.\,homotoping it into $(\BO(d),\BO(d))$ is: 
\begin{enumerate}[label=(O\arabic*), ]
\item $\phi^*(p_i)$\tabto{1.7cm} $\in\oH^{4i}(X,A)$\tabto{3.7cm}  for $i > n$,
\item $\phi^*(p_n-e^2)$\tabto{1.7cm}$\in \oH^{4n}(X,A)$\tabto{3.7cm} if $d=2n$,
\item $\phi^*(p_n-E)$\tabto{1.7cm}$\in \oH^{4n}(X,A)$\tabto{3.7cm} if $d=2n+1$, 
\item $\phi^*(e_\fbt, e_\fbo)$\tabto{1.7cm}$\in \oH^{d}(X, A ;(\underline{\Theta \fbt}^{(1)}, \underline{\Theta \fbo}^{(1)}))_\bfQ$ if $d=2n+1$ ($(e_\fbt, e_\fbo)$ vanishes for $d=2n$).
\end{enumerate}

\appendix

\section{Some cohomology of $G_{V}$}\label{sec:AppendixA}
As announced and used in \cref{sec:stable-cohomolog-gv} this appendix serves to compute the bigraded groups
\[\oH_*\big(G_V ; S_\lambda({H}_{V}) \otimes S_\mu({H}_V^\vee)\big)\]
in a stable range of homological degrees. To do so we first pass to the ungraded setting and dualise: in \cref{lem:GpCoh} we calculate $\oH^*(G_V; (\bar{H}_{V}^\vee)^{\otimes p} \otimes \bar{H}_{V}^{\otimes q})$ in a stable range as a graded $\bfQ[\Sigma_p \times \Sigma_q]$-module, which we then use in \cref{cor:GradedGpCoh} to answer the original question.

In fact it will be more convenient to calculate the groups $\oH^*(G_V; (\bar{H}_{V}^\vee)^{\otimes p} \otimes \bar{H}_{V}^{\otimes q})$ with arbitrary finite sets $P$ and $Q$ indexing the tensor powers. To state the answer  (which also take the $\oH^*(\GL_\infty(\bfZ);\bfQ)$-module structure induced by the surjection $G_V \to \GL(\bar{H}_V^\bfZ)\ra \GL_\infty(\bfZ)$ followed by the stabilisation map into account) we abbreviate $\underline{k}\coloneq \{1,\ldots,k\}$ for $k\ge0$, write $\mathrm{Bij}(S,T)$ for the set of bijections from $S$ to $T$, and denote $(+)$ and $(-)$ for the trivial and sign representations of the symmetric group $\Sigma_2$.

\begin{dfn}
For finite sets $P$ and $Q$, define vector spaces $C_+(P,Q)$ and $C_-(P,Q)$ by
\[
C_\pm(P,Q)\coloneq \begin{cases}
\Ind^{\Sigma_{2t}}_{\Sigma_t \wr \Sigma_2} ((1^t) \wr (\mp)) \otimes_{\Sigma_{2t}} \bfQ[\mathrm{Bij}(Q \sqcup \underline{2t}, P)]&\text{if } |P|-|Q|=2t \geq 0\\
0&\text{otherwise},
\end{cases}
\]
considered as a graded vector space concentrated in degree $t$. Writing $\cat{FB}$ for the category of finite sets and bijections, this assembles to a functor \[C_\pm(-,-)\colon \cat{FB} \times \cat{FB} \lra \cat{grVect},\] where the action is by $(f, g) \cdot h = g \circ h \circ (f^{-1} \sqcup \underline{2t})$ on $\mathrm{Bij}(Q \sqcup \underline{2t}, P)$. 
\end{dfn}

\begin{rem}$C_{\pm}(P,Q)$ may be interpreted the vector space of injections $Q \hookrightarrow P$, with an ordered matching of the complement, and an ordering of the parts of the matching, modulo 
\begin{enumerate}[(i)]
\item reordering a matched pair and multiplying by $\mp 1$,
\item reordering the parts and multiplying by the sign of the permutation.
\end{enumerate}
\end{rem}

\begin{lem}\label{lem:GpCoh}
There is a natural transformation $\Phi$ of functors from $\cat{FB} \times \cat{FB}$ to the category of $\oH^*(\GL_\infty(\bfZ);\bfQ)$-modules with components
\[
\begin{gathered}
\Phi_{P,Q}: \oH^*(\mathrm{GL}_\infty(\bfZ);\bfQ) \otimes C_{(-1)^n}(P,Q) \lra \oH^*(G_V; (\bar{H}_{V}^\vee)^{\otimes P} \otimes \bar{H}_{V}^{\otimes Q}),
\end{gathered}
\]
which for each fixed $P$ and $Q$ are isomorphisms in a range of degrees increasing with $g$.
\end{lem}
\begin{proof}
The intersection form on $\bar{H}_W$ induces an identification of the  kernel of the projection $\bar{H}_W\ra \bar{H}_V$ with the dual of $\bar{H}_{V}^\vee$, resulting in an extension of $\bfQ[\GL(\bar{H}^\bfZ_V)]$-modules
\begin{equation}\label{eq:HWgExtension}
0 \lra \bar{H}_{V}^\vee \lra \bar{H}_{W} \lra \bar{H}_{V} \lra 0
\end{equation}
which represents a class
\[\epsilon \in \mathrm{Ext}^1_{\bfQ[G_V]}(\bar{H}_{V}, \bar{H}_{V}^\vee) \cong \mathrm{Ext}^1_{\bfQ[G_V]}(\bfQ, (\bar{H}_{V}^\vee)\otimes(\bar{H}_{V}^\vee) ) = \oH^1(G_V ; (\bar{H}_{V}^\vee)^{\otimes 2}).\]
Recall from Lemma \ref{lem:mcg-is-semidirect-product} the semi-direct product decomposition $G_V = M_V^\bfZ \rtimes \GL(\bar{H}_V^\bfZ)$. The Serre spectral sequence for this semi-direct product provides an identification
\[\oH^1(G_V ; (\bar{H}_{V}^\vee)^{\otimes 2}) = [\oH^1(M_V^\bfZ ; (\bar{H}_{V}^\vee)^{\otimes 2}]^{\GL(\bar{H}_V^\bfZ)} \oplus \oH^1( \GL(\bar{H}_V^\bfZ) ; (\bar{H}_{V}^\vee)^{\otimes 2}).\]
The extension \eqref{eq:HWgExtension} is $\GL(\bar{H}_V^\bfZ)$-equivariantly split (see \cref{sec:involution-and-intersection}), so $\epsilon$ decomposes as $(\bar{\epsilon}, 0)$. Under the identification
\[\oH^1(M_V^\bfZ ; (\bar{H}_{V}^\vee)^{\otimes 2})^{\GL(\bar{H}_V^\bfZ)} = \mathrm{Hom}_{\GL(\bar{H}_V^\bfZ)} \Big(M_V^\bfZ = \begin{cases}
 \Lambda^2(\bar{H}_V^\bfZ)^\vee&\text{for }n\text{ even}\\
\Gamma^2(\bar{H}_V^\bfZ)^\vee&\text{for }n\text{ odd},
\end{cases}, (\bar{H}_{V}^\vee)^{\otimes 2}\Big),\]
the class $\bar{\epsilon}$ is given by the dual of the (anti)symmetrisation map from $\bar{H}_V^{\otimes 2}$ to $\Gamma^2(\bar{H}_V)$ or $\Lambda^2(\bar{H}_V)$. In particular, under the $\Sigma_2$-action given by swapping the factors of $(\bar{H}_{V}^\vee)^{\otimes 2}$ the class $\bar{\epsilon}$ transforms as $\mathrm{sign}^{\otimes n+1}$, and hence $\epsilon$ does as well. 

In addition to $\epsilon$, we have a class
$\pi \in \oH^0(G_V; \bar{H}_{V} \otimes \bar{H}_{V}^\vee)$ defined by coevaluation. Taking cup products of copies of $\epsilon$ and $\pi$, and permuting the tensor factors, gives a map
\begin{equation}\label{equ:cupping-epsilon-and-pi}C_{(-1)^n}(\underline{p},\underline{q}) \lra \oH^*(G_V; (\bar{H}_{V}^\vee)^{\otimes p} \otimes \bar{H}_{V}^{\otimes q}),\end{equation}
that is $\Sigma_p \times \Sigma_q$-equivariant, which then extends to the natural transformation $\Phi$ as in the claim, using the $\oH^*(\mathrm{GL}_\infty(\bfZ);\bfQ)$-module structure of the codomain. To verify that this map is an isomorphism in a stable range, we will once more consider the Serre spectral sequence of the semi-direct product $G_V = M_V^\bfZ \rtimes \GL(\bar{H}_V^\bfZ)$
\[E_2^{s,t} = \oH^s\big( \GL(\bar{H}_V^\bfZ) ; \oH^t(M_V^\bfZ ;\bfQ) \otimes (\bar{H}_{V}^\vee)^{\otimes P} \otimes \bar{H}_{V}^{\otimes Q}\big) \Rightarrow \oH^{s+t}\big(G_V; (\bar{H}_{V}^\vee)^{\otimes P} \otimes \bar{H}_{V}^{\otimes Q}\big).\]
Rationally we may identify the divided square with the symmetric square, so $M_V^\bfZ \otimes \bfQ$ agrees with $\Lambda^2(\bar{H}_V)^\vee$ if $n$ is even and with $\mathrm{Sym}^2(\bar{H}_V)^\vee$ if $n$ is odd, and hence 
$\oH^t(M_V^\bfZ ;\bfQ) = \Lambda^t(\Lambda^2(\bar{H}_{V}))$ if $n$ is even and $\oH^t(M_V^\bfZ ;\bfQ) = \Lambda^t(\mathrm{Sym}^2(\bar{H}_{V}))$ if $n$ is odd. Taking  cup products thus gives a map of $\bfQ[\GL_\infty(\bfZ)]$-modules of the form
\begin{equation}\label{eq:BorelMap}
\oH^s(\mathrm{GL}_\infty(\bfZ);\bfQ) \otimes \Big[\Lambda^t\Big(\begin{cases}
\Lambda^2(\bar{H}_{V}) \\
\mathrm{Sym}^2(\bar{H}_{V})
\end{cases}\Big) \otimes (\bar{H}_{V}^\vee)^{\otimes P} \otimes \bar{H}_{V}^{\otimes Q}\Big]^{\GL(\bar{H}_V^\bfZ)} \lra E_2^{s,t}
\end{equation}
which is an isomorphism in a range of $s$ increasing with $g$ by \cref{thm:Borel-vanishing}. We write
\[\Lambda^t\Big(\begin{cases}
\Lambda^2(\bar{H}_{V}) &\text{for }n\text{ even}\\
\mathrm{Sym}^2(\bar{H}_{V}) &\text{for }n\text{ odd}
\end{cases}\Big) = 
\big(\Ind^{\Sigma_{2t}}_{\Sigma_t \wr \Sigma_2} ((1^t) \wr (1^2)^{\otimes n+1})
\big) \otimes_{\Sigma_{2t}} \bar{H}_{V}^{\otimes 2t}.\]
The map
\[\bfQ[\mathrm{Bij}(T, S)] \lra [(\bar{H}_{V}^\vee)^{\otimes S} \otimes \bar{H}_{V}^{\otimes T}]^{\GL(\bar{H}_V^\bfZ)}\]
given by inserting coevaluations is an isomorphism as long as $g > \max\{|S|, |T|\}$, which one sees as follows: The first and second fundamental theorems of invariant theory for $\mathrm{SL}_g(\bfC)$ (see e.g.\ \cite[Section 11.6.1, p.\ 388]{Procesi}) and Zariski density of $\SL_g(\bfZ) \subset \SL_g(\bfC)$ show that this map is an isomorphism when replacing the target with the $\SL(\bar{H}_V^\bfZ)$-invariants, but since inserting coevaluations gives elements that are even $\GL(\bar{H}_V^\bfZ)$-invariant, the $\SL(\bar{H}_V^\bfZ)$ and $\GL(\bar{H}_V^\bfZ)$ invariants in fact agree.

Combining this with the expression above identifies the domain of \eqref{eq:BorelMap} with $\oH^s(\mathrm{GL}_\infty(\bfZ);\bfQ) \otimes C_{(-1)^n}(P,Q)_t$. Thus in a range of $s$ increasing with $g$, the spectral sequence is supported along the row $t= (|P|-|Q|)/2$, so it collapses and there are no extension issues. For this value of $t$ it is clear that the map $C_{(-1)^n}(P,Q)_t \to E_2^{0,t} = \oH^t(G_V; (\bar{H}_{V}^\vee)^{\otimes P} \otimes \bar{H}_{V}^{\otimes Q})$ agrees with \eqref{equ:cupping-epsilon-and-pi}, so the result follows.
\end{proof}

\begin{cor}\label{cor:GradedGpCoh}
Fix partitions $\lambda \vdash p$ and $\mu \vdash q$. In a range of homological degrees increasing with $g$ the bigraded vector space $\oH_*(G_V; S_\lambda({H}_{V}) \otimes S_\mu({H}_{V}^\vee))$ is 
\begin{enumerate}[(i)]
\item trivial if $p-q$ is odd or negative, and
\item isomorphic to $\oH_*(\mathrm{GL}_\infty(\bfZ);\bfQ) \otimes \chi(\lambda, \mu)$ as a $\oH_*(\mathrm{GL}_\infty(\bfZ);\bfQ)$-comodule if $p-q = 2t$. Here 
\[ \chi(\lambda, \mu) \coloneq \Big[ \Ind_{\Sigma_q \times \Sigma_{2t}}^{\Sigma_p} \big( (\mu) \boxtimes \Ind^{\Sigma_{2t}}_{\Sigma_t \wr \Sigma_2} (1^t) \wr (2) \big) \otimes (\lambda) \Big]^{\Sigma_p} \] is placed in bidegree $(t, (p-q)(n-1))$
\end{enumerate}
\end{cor}
\begin{proof}
Adapted to the graded setting \cref{lem:GpCoh} gives an isomorphism 
\[\oH^*(G_V; ({H}_{V}^\vee)^{\otimes p} \otimes {H}_{V}^{\otimes q}) \cong \oH^*(\mathrm{GL}_\infty(\bfZ);\bfQ) \otimes C_{(-1)^n}(\underline{p},\underline{q}) \otimes (((1^p)^{\otimes n+1}) \boxtimes (1^q)^{\otimes n+1}))\otimes\bfQ[0,(q-p)(n-1)]\]
of bigraded $\bfQ[\Sigma_p \times \Sigma_q]$-modules in a range of cohomological degrees increasing with $g$. The right-hand side vanishes unless $p-q$ is even and nonnegative, and if $p-q=2t$ we have
\begin{align*}
C_{(-1)^n}(\underline{p},\underline{q}) \otimes ((1^p)^{\otimes n+1}) \boxtimes (1^q)^{\otimes n+1}) &\cong \Ind_{\Sigma_q \times \Sigma_{2t}}^{\Sigma_p} \left( \bfQ[\Sigma_q] \boxtimes \Ind^{\Sigma_{2t}}_{\Sigma_t \wr \Sigma_2} (1^t) \wr (1^2)^{\otimes n+1} \right) \otimes (1^p)^{\otimes n+1}\\
&\cong \Ind_{\Sigma_q \times \Sigma_{2t}}^{\Sigma_p} \left( \bfQ[\Sigma_q] \boxtimes \Ind^{\Sigma_{2t}}_{\Sigma_t \wr \Sigma_2} (1^t) \wr (2) \right)
\end{align*}
as a $\bfQ[\Sigma_p \times \Sigma_q]$-module using Frobenius reciprocity, supported in cohomological degree $t$. Dualising and applying $[- \otimes (\lambda) \boxtimes (\mu)]^{\Sigma_p \times \Sigma_q}$ gives the claimed formula.
\end{proof}

\section{Relation to the work of Watanabe}\label{app:Watanabe}In \cite{WatanabeI} Watanabe used a characteristic class associated to the ``theta'' graph defined by Kontsevich  in terms of configuration space integrals \begin{equation}\label{equ:Kontsevich-class}\zeta_{2}\in\oH^{2n-2}(\BDiff_\partial^\fr(D^{2n+1});\bfQ)\end{equation} to show that the group $\pi_{2n-2}(\BDiff_\partial(D^{2n+1}))_\bfQ$ is nontrivial for many values of $n\ge2$ (see \cite[Theorem 1, Corollary 2]{WatanabeI}). The purpose of this appendix is twofold:

\begin{enumerate}
\item In Section \ref{sec:who-would-have-thought}, we explain how one can replace Watanabe's use of the Kontsevich class $\zeta_2$ in \cite{WatanabeI} by an invariant with a significantly simpler description. This results a surprisingly elementary proof that $\pi_{2n-2}(\BDiff_\partial(D^{2n+1}))_\bfQ$ is nontrivial for many values of $n\ge2$, which could have been discovered in the '70s. This also gives a simple argument for the odd-dimensional analogue $p_n\neq E\in\oH^{4n}(\BTOP(2n+1);\bfQ)$ of Weiss' theorem $p_n\neq e^2\in\oH^{4n}(\BTOP(2n+1);\bfQ)$ described in the introduction.
\item In Section \ref{sec:Kontsevich-alternative}, we give an alternative description of $\zeta_{2}$ which does not involve configuration space integrals and fits well with our computation of the rational homotopy type of the $(\approx 3n)$-truncation of $\BDiff_\partial^\fr(D^{2n+1})_{\ell}$ in \cref{sec:bdiff-and-btop} (but does not use it).
\end{enumerate}

\subsection{A ``classical'' proof that $\pi_{2n-2}(\BDiff_\partial(D^{2n+1}))_\bfQ$ is often nontrivial}\label{sec:who-would-have-thought}
On close inspection, the proof of the main theorem in \cite{WatanabeI} saying that $\pi_{2n-2}(\BDiff_\partial(D^{2n+1}))_\bfQ$ is nontrivial for ``many'' $n\ge2$ hinges only on very few properties of the Kontsevich class $\zeta_{2}\in\oH^{2n-2}(\BDiff_\partial^\fr(D^{2n+1});\bfQ)$. In fact, the only input needed for the final step in the argument (see Section 5 loc.cit.) to go through is the existence of a morphism
\begin{equation}\label{equ:alternative-invariant}\psi\colon \pi_{2n-2}(\BDiff^\fr_\partial(D^{2n+1})_\ell)\lra \bfQ\end{equation}
with the following two properties:
\begin{enumerate}[leftmargin=*]
\item\label{enum:psi-i} $\psi$ is integral, i.e.\ it has image in $\bfZ\subset\bfQ$.
\item\label{enum:psi-ii} The composition
\[\pi_{4n}(\BO(2n+1))\cong \pi_{2n-2}(\Omega^{2n+1}\OO(2n+1))\xlra{r} \pi_{2n-2}(\BDiff^\fr_\partial(D^{2n+1})_\ell)\xlra{\psi}\bfQ\]
agrees with $\tfrac{1}{4}\cdot p_n\colon\pi_{4n}(\BO(2n+1))\ra \bfQ$. Here $r$ is induced by reframing.
\end{enumerate}
The morphism induced by Kontsevich's class $\xi_2\in \oH^{2n-2}(\BDiff_\partial^\fr(D^{2n+1});\bfQ)$ as described in Section 2.5.2 loc.cit.\ satisfies \ref{enum:psi-i} by construction and \ref{enum:psi-ii} as a result of the proof of Theorem 2 loc.cit.. The class $\xi_2$ is thus a valid choice for $\psi$, but for the purpose of proving the main result in \cite{WatanabeI}, any other $\psi$ satisfying \ref{enum:psi-i} and \ref{enum:psi-ii} would do. Our favourite choice is the following.
\subsubsection{An alternative $\psi$}\label{sec:alternative-choice}Consider the composition
\begin{equation}\label{equ:framed-diff-to-G}\BDiff_\partial^\fr(D^{2n+1})_\ell\lra \Omega^{2n+1}\TOP(2n+1)\lra \Omega^{2n+1}\oG(2n+1)\simeq \Omega^{2n+2}\BSG(2n+1)\end{equation}
where the first map is the ``scanning map'' appearing in smoothing theory and the second map is the canonical map $\TOP(2n+1)\ra\oG(2n+1)$ obtained by viewing $\TOP(2n+1)$ as homeomorphisms fixing the origin and $\oG(2n+1)$ as homotopy automorphisms of $\bfR^{2n+1}\backslash \{0\}$. We construct $\psi$ as the composition of the map induced by \eqref{equ:framed-diff-to-G} on $\pi_{2n-2}(-)$ with an invariant
\begin{equation}\label{equ:varepsilon-alternative}\varepsilon\colon \pi_{4n}(\BSG(2n+1))\lra\bfZ\end{equation} 
defined as follows: given an oriented fibration $\pi\colon E\ra S^{4n}$ together with an identification $\pi^{-1}(*)\simeq S^{2n}$ of the preimage of the basepoint with $S^{2n}$, we set $\varepsilon(\pi)\coloneq \int_E \iota(v)^3\in \bfZ$ where $v\in \oH^{2n}(S^n;\bfZ)$ is the standard generator. Here $\iota\colon  \oH^{2n}(S^{2n};\bfZ)\ra \oH^{2n}(E;\bfZ)$ is the inverse of the map induced by restriction along $S^{2n}\simeq \pi^{-1}(*)\subset E$, which is an isomorphism by the Serre spectral sequence. In the universal case, this defines \eqref{equ:varepsilon-alternative}. The precomposition of $\varepsilon$ with the map induced by \eqref{equ:framed-diff-to-G} satisfies \ref{enum:psi-i} by construction and \ref{enum:psi-ii} by the following lemma.

\begin{lem}\label{lem:C1}
Precomposing $\varepsilon$ with $\pi_{4n}(\BO(2n+1))\ra \pi_{4n}(\BG(2n+1))$ gives $\tfrac{1}{4}\cdot p_n$.
\end{lem}
\begin{proof}Given a pullback diagram
\[
\begin{tikzcd}
E\dar{\pi}\rar{\bar{f}} &\BSO(2n)\dar{p}\\
S^{4n}\rar{f} & \BSO(2n+1)
\end{tikzcd}
\]
where $p$ is the universal linear oriented $S^{2n}$-fibration, the task is to show $\varepsilon(\pi)=\tfrac{1}{4}\cdot p_n(f)$ or equivalently $2^3\cdot \varepsilon(\pi)=2\cdot p_n(f)$. As $\chi(S^{2n})=2$, we have $2\cdot \iota(v^{2n})=\bar{f}^*(e)$ where $e\in\oH^{2n}(\BSO(2n);\bfZ)$ is the Euler class. With this in mind, we compute
\[\textstyle{8\cdot \varepsilon(\pi)=\int_E\bar{f}^*(p_n\cdot e)=\int_{S^{4n}}\int_\pi \bar{f}^*(e)\cdot \pi^*(p_n(f))=\int_{S^{4n}} (\int_\pi \bar{f}^*(e))\cdot p_n(f)=2\cdot p_n(f)},\]
where the first equality uses $e^3=p_n\cdot e$ in $\oH^{4n}(\BSO(2n);\bfZ)$, the second follows by using the fact that $p_n$ pulls back from $\oH^{4n}(\BSO(2n+1);\bfZ)$, the third is a consequence of using the projection formula, and the final uses $\chi(S^{2n})=2$ once more.
\end{proof}

Using this choice of $\psi$ instead of $\xi_2$, the argument in \cite[Section 6]{WatanabeI} shows that $\pi_{2n-2}(\BDiff_\partial(D^{2n+1}))_\bfQ$ is nontrivial for many $n$ without using configuration space integrals, in fact only relying on differential topology as developed in the '60s and '70s. To illustrate the simplicity of this alternative line of argument, we spell out a complete proof of:

\begin{thm}[Watanabe]\label{thm:easy-rational-nontriviality}
$\pi_{2}(\BDiff_\partial(D^5))_\bfQ\neq 0$.
\end{thm}

\begin{proof}
We first consider a general odd dimension $2n+1$. By the exact sequence
\[\ldots\ra\pi_{2n-2}(\Omega^{2n+1}\SO(2n+1))\ra \pi_{2n-2}(\BDiff^\fr_\partial(D^{2n+1})_\ell)\ra \pi_{2n-2}(\BDiff_\partial(D^{2n+1}))\ra \ldots\]
and $\pi_{2n-2}(\Omega^{2n+1}\SO(2n+1))_\bfQ \cong \bfQ$, to see that $\pi_{2n-2}(\BDiff_\partial(D^{2n+1}))_\bfQ$ is nontrivial it suffices to show that $\varepsilon$ and $\Omega^{2n+1} p_n$ are not proportional as morphisms $\pi_{2n-2}(\BDiff^\fr_\partial(D^{2n+1})_\ell) \to \bfQ$. If they were, then by Lemma \ref{lem:C1} the proportionality must be $4 \cdot \varepsilon = \Omega^{2n+1}p_n$, so it suffices to show that the map $\Omega^{2n+1}p_n$ does not land in $4\bfZ$.

To analyse this question, consider the commutative diagram with exact rows
\[
\begin{tikzcd}[column sep=0.3cm, row sep=0.6cm]
& \pi_{2n-2}(\BDiff^\fr_\partial(D^{2n+1})_\ell) \dar \rar & \pi_{2n-2}(\BDiff_\partial(D^{2n+1})) \rar \dar & \pi_{4n-1}(\BO(2n+1)) \arrow[d, equals]\\
&\pi_{4n}(\BTOP(2n+1)) \rar{\partial}\dar&\pi_{4n-1}(\TOP(2n+1)/\OO(2n+1))\dar \rar & \pi_{4n-1}(\BO(2n+1))\\
\pi_{4n}(\BO) \dar{p_n} \rar &\pi_{4n}(\BTOP)\rar{\partial}\dar{p_n}&\pi_{4n-1}(\TOP/\OO)\dar{[p_n]}\cong\Theta_{4n-1} \rar & 0\\
\bfZ \rar{\subset} & \bfQ \rar & \bfQ/\bfZ \rar & 0,
\end{tikzcd}
\]
where the top vertical maps are the ``scanning maps'' appearing in smoothing theory (they are isomorphisms, but we will not use this), the middle vertical maps are given by stabilisation, the two leftmost bottom vertical maps $p_n$ are given by evaluation against the $n$th Pontryagin class, and the map $[p_n]$ is induced by the $p_n$. Identifying $\pi_{4n-1}(\TOP/\OO)$ with the group $\Theta_{4n-1}$ of homotopy $(4n-1)$-spheres by surgery, the map $[p_n]$ can be evaluated homotopy sphere $\Sigma$ by choosing a stable framing and  integrating the topological Pontryagin class $p_n$ over the topological $4n$-manifold $\mathrm{Cone}(\Sigma)$ with boundary $\Sigma$, using the stable framing on the boundary. The resulting rational number is not independent of the choice of framing, but its residue class modulo $\bfZ$ is. On the Milnor sphere $\Sigma_M \in \bP_{4n} \subset \Theta_{4n-1}$ one may compute this as follows. Gluing $\mathrm{Cone}(\Sigma_M)$ to the $E_8$-plumbing $E$ along $\Sigma_M$ gives a topological $4n$-manifold $\bar{W}$ of signature 8, so we have $8 = \int_{\bar{W}} \mathcal{L}_n$. As $W$ is parallelisable, all $p_i \in \oH^{4i}(\bar{W};\bfQ)$ with $i < n$ vanish, and using the stable framing on $\Sigma_M$ induced by a choice of stable framing of $W$ we have \[\textstyle{\int_{\bar{W}} p_n = \int_{(\mathrm{Cone}(\Sigma_M),  \Sigma_M)} p_n}.\] From \[\mathcal{L}_n = \tfrac{2^{2n} (2^{2n-1}-1) |B_{2n}|}{(2n)!} \cdot p_n + \text{decomposables},\] for $B_{2n}$ the $2n$th Bernoulli number, it thus follows that
\[[p_n](\Sigma_M) = \tfrac{(2n)!}{2^{2n-3} (2^{2n-1}-1) |B_{2n}|} \mod \bfZ.\]
In dimension $2n+1 = 5$, Cerf's ``pseudoisotopy implies isotopy'' theorem ensures that the map $\pi_2(\BDiff_\partial(D^5))\ra \Theta_{7}$ is surjective. As $\pi_{4n-1}(\BO(2n+1))=\pi_7(\BO(5))=0$, this implies by the upper row in the diagram that there is an element $x \in \pi_2(\BDiff_\partial^\fr(D^5)_\ell)$ that maps to $\Sigma_M \in \Theta_7$. By the above and $B_4=-\tfrac{1}{30}$ this satisfies the congruence
\[(\Omega^{5}p_2)(x) \equiv \tfrac{5 \cdot 3^2 \cdot 2^3}{7} \mod \bfZ,\] so $(\Omega^{5}p_2)(x) \not\in \bfZ$ and in particular $(\Omega^{5}p_2)(x) \not\in 4\cdot \bfZ$ as required.
\end{proof}

\begin{rem}As $\pi_{4n}(\BSG(2n+1))_\bfQ\cong \bfQ$ and the class $E\in \oH^{4n}(\BSG(2n+1);\bfQ)$ from \cref{sec:transgressed-pontryagin} to $\BSO(2n+1)$ agrees with $p_n$ in $\oH^{4n}(\BSG(2n+1);\bfQ)$, the invariant $\varepsilon$ agrees with the evaluation of $\tfrac{1}{4}\cdot E$. In particular, the argument explained above shows $p_n\neq E\in\oH^{4n}(\BSTOP(2n+1);\bfQ)$ for many values of $n\ge2$ by surprisingly elementary means.
\end{rem}

\subsection{An alternative description of the first Kontsevich class}\label{sec:Kontsevich-alternative}
Recall that there is a unique class $E\in \oH^{4n}(\BSG(2n+1);\bfQ)$ that pulls back to the square of the Euler class $e^2\in \oH^{4n}(\BSG(2n);\bfQ)$ along the stabilisation map $\BSG(2n)\ra \BSG(2n+1)$ (see \cref{sec:transgressed-pontryagin}). After looping this class $(2n+2)$ times, we may pull it back along the composition \eqref{equ:framed-diff-to-G}
to get a class $\Omega^{2n+2}E\in \oH^{2n-2}(\BDiff_\partial^\fr(D^{2n+1});\bfQ)$. It turns out that this class agrees with the Kontsevich class $\zeta_2$ up to the following explicit constant.

\begin{thm}\label{thm:KontsevichClass}
 For $2n+1\ge5$, we have
\[\zeta_{2}=\tfrac{1}{4}\cdot\Omega^{2n+2}E\in \oH^{2n-2}(\BDiff_\partial^\mathrm{fr}(D^{2n+1})_\ell;\bfQ)\] where $\zeta_{2}$ is the Kontsevich class associated to the $\Theta$-graph from \cite[Section 2.5.3]{WatanabeI}.
\end{thm}

\begin{rem}
It is worth mentioning that the class $\zeta_2\in \oH^{2k(n-1)}(\BDiff_\partial^\fr(D^{2n+1});\bfQ)$ as studied in \cite{WatanabeI} does \emph{not} map to the Kontsevich class \[\zeta_{2k,3k}\in \oH^{2k(n-1)}(\BDiff_\partial^\fr(D^{2n+1});\cA_{2k,3k}\otimes \bfR)\] for $k=1$ considered in \cite{WatanabeII} under the map induced by the ``theta graph'' (considered as a map $\Theta\colon \bfQ\ra \cA_{2,3}\otimes \bfR$). Comparing definitions, one sees that it rather maps to $12\cdot \zeta_{2,3}$. 
\end{rem}

\begin{rem}It follows from \cref{thm:KontsevichClass} that Watanabe's invariant \[\hat{Z}_2\colon\pi_{2n-2}(\BDiff_\partial(D^{2n+1}))_\bfQ\lra \bfQ\] from \cite[Theorem 2]{WatanabeI} agrees up to a nontrivial scalar with the invariant \begin{equation}\label{equ:e-pn-tau-on-homotopy}\Omega^{2n+2}(p_n-E)^\tau\colon \pi_{2n-2}(\BDiff_\partial(D^{2n+1}))_\bfQ\lra \bfQ\end{equation} featuring in our \cref{cor:homotopy-diff-odd-disc}. In particular, this implies that $\hat{Z}_2$ is nontrivial for \emph{all} $n\ge5$, which was previously known from Watanabe's work \cite{WatanabeI,WatanabeII} contingent on numerical conditions on $n$ involving Bernoulli numbers. As these conditions are satisfied for $n=2,3,4$, we can also argue the other way to conclude that \eqref{equ:e-pn-tau-on-homotopy} is nontrivial in the cases $n=2,3,4$ not captured by our \cref{cor:homotopy-diff-odd-disc}.
\end{rem}

For the remainder of this section we assume familiarity with the construction of $\zeta_2\in \oH^{2n-2}(\BDiff^\fr_\partial(D^{2n+1});\bfZ)$ as an \emph{integral} class given in \cite[Section 2.5]{WatanabeI}, and freely use the notation of that paper. For convenience we pass between $(D^d, \partial D^d)$-bundles and $\interior(D^d)$- or $\bfR^d$-bundles standard outside of a compact set without further mentioning. As a preparation to \cref{thm:KontsevichClass}, we prove the following.

\begin{lem}\label{lem:Zeta2Prim}
The map representing the class $\zeta_2\in \oH^{2n-2}(\BDiff^\fr_\partial(D^{2n+1});\bfZ)$ 
\[\zeta_2\colon \BDiff_\partial^\mathrm{fr}(D^{2n+1})_\ell\lra K(\bfQ,2n-2)\]
is a map of $H$-spaces; the $H$-space structure on the source is induced by the action of the little $(2n+1)$-discs operad induced by boundary connected sum. 
\end{lem}

\begin{proof}
The claim is equivalent to showing that the function $[B,\BDiff^\fr_\partial(D^{2n+1})]\ra \oH^{2n-2}(B;\bfQ)$ given by evaluating $\zeta_2$ is additive for spaces $B$. A class in the source is represented by a smooth $\interior(D^{2n+1})$-bundle  $\pi\colon E \to B$ with a framing $\tau_E$ of its vertical tangent bundle, which is identified with the trivial framed bundle $B \times \interior(D^{2n+1})$ outside of a compact set. It being the sum of classes $\pi_1$ and $\pi_2$ means the following: there is a vertically framed embedding
\[K \coloneq B \times (\interior(D^{2n+1}) \setminus \big(\tfrac{1}{10}\interior(D^{2n+1}) - \tfrac{1}{2} e_1) \sqcup (\tfrac{1}{10}\interior(D^{2n+1}) + \tfrac{1}{2} e_1)\big) \lra E\]
over $B$ whose complement $\pi_1 \sqcup \pi_2\colon E_1 \sqcup E_2 \to B$ is (after rescaling the framing by a factor of 10) a disjoint union of two bundles representing classes in $[B,\BDiff^\fr_\partial(D^{2n+1})]$. The bundle of compactified two-point configuration spaces
$\overline{C}_2(\pi)\colon E\overline{C}_2(\pi) \ra B$
contains subbudles $E\overline{C}_2(\pi_1)$, $E\overline{C}_2(\pi_2)$, $E_1 \times_B E_2$, $E_2 \times_B E_1$, and the complement of these consists of configurations where both points lie in $K$. The framing $\tau_E$ gives a map
\[p(\tau_E)\colon \partial^\mathrm{fib} E\overline{C}_2(\pi) \cup \partial^\mathrm{fib} E\overline{C}_2(\pi_1) \cup \partial^\mathrm{fib} E\overline{C}_2(\pi_2) \lra S^{2n},\]
and it is easy to see that this extends to a map
$p^\mathrm{ext}(\tau_E)\colon E\overline{C}_2(\pi) \setminus \interior( E\overline{C}_2(\pi_1) \sqcup  E\overline{C}_2(\pi_2)) \ra S^{2n},$
using the framing outside of $E_1 \times_B E_2$ and $E_2 \times_B E_1$, and then using the contractibility of the fibres of the latter bundles. Denoting the dual of the fundamental class of $S^{2n}$ by $v^{2n}$, \cite[Lemma 3]{WatanabeI} provides an extension
\[
\begin{tikzcd}
\partial^\mathrm{fib} E\overline{C}_2(\pi_1) \sqcup \partial^\mathrm{fib} E\overline{C}_2(\pi_2) \rar{p(\tau_E)} \dar & S^{2n} \dar{v^{2n}}\\
E\overline{C}_2(\pi_1) \sqcup E\overline{C}_2(\pi_2) \rar{\tilde{\omega}_1 \sqcup \tilde{\omega}_2} & K(\bfZ, 2n),
\end{tikzcd}
\]
which glued to $v^{2n} \circ p^\mathrm{ext}(\tau_E)$ gives a map $\tilde{\omega}\colon E\overline{C}_2(\pi) \to K(\bfZ, 2n)$ that extends $v^{2n} \circ p(\tau_E)\vert_{\partial^\mathrm{fib} E\overline{C}_2(\pi)}$.

The construction of $\zeta_2(\pi)$ is now as follows. As $(v^{2n})^2 = 0 \in \oH^{4n}(S^{2n};\bfZ)$, there is a canonical refinement of $\tilde{\omega}^2$ to a relative cohomology class $\tilde{\omega}^2 \in \oH^{4n}( E\overline{C}_2(\pi) , \partial^\mathrm{fib} E\overline{C}_2(\pi)  ; \bfZ)$, and hence we can consider $\tilde{\omega}^3$ as a relative cohomology class of degree $6n$ and fibre-integrate it along $\overline{C}_2(\pi)$ to obtain $\zeta_2(\pi) \in \oH^{6n-2(2n+1)}(B;\bfZ)=\oH^{2n-2}(B;\bfZ)$. 

Now, by construction $\tilde{\omega}^2$ in fact canonically refines to a relative cohomology class \[\tilde{\omega}^2 \in \oH^{4n}(E\overline{C}_2(\pi), E\overline{C}_2(\pi) \setminus \interior( E\overline{C}_2(\pi_1) \sqcup  E\overline{C}_2(\pi_2)) ; \bfZ),\] which by excision can be identified with \[\tilde{\omega}_1^2 + \tilde{\omega}_2^2 \in \oH^{4n}(E\overline{C}_2(\pi_1) \sqcup  E\overline{C}_2(\pi_2), \partial^\mathrm{fib}E\overline{C}_2(\pi_1) \sqcup  \partial^\mathrm{fib}E\overline{C}_2(\pi_2) ; \bfZ).\] Further multiplying with $\tilde{\omega}$ and fibre integrating gives
\[\textstyle{\zeta_2(\pi) = \int_{\overline{C}_2(\pi)} \tilde{\omega}^3 = \int_{\overline{C}_2(\pi_1)} \tilde{\omega}_1^3 + \int_{\overline{C}_2(\pi_2)} \tilde{\omega}_2^3 = \zeta_2(\pi_1) + \zeta_2(\pi_2)}\]
as claimed.
\end{proof}

\begin{proof}[Proof of \cref{thm:KontsevichClass}]
The characteristic class $\zeta_2$, whose construction was recalled in the proof of the previous lemma, gives a function
\begin{equation}\label{eq:zeta}
\zeta_2\colon [B, \BDiff_\partial^\fr(D^{2n+1})_\ell] \cong [B, \Omega^{2n+1}_0 \TOP(2n+1)] \lra \oH^{2n-2}(B;\bfZ)
\end{equation}
for any space $B$, using Morlet's equivalence $\Omega^{2n+1}_0 \TOP(2n+1) \simeq B\Diff_c^\mathrm{fr}(\bfR^{2n+1})_\ell$. This is a homomorphism by \cref{lem:Zeta2Prim}, so the claim is equivalent to showing that the rationalisation of this morphism agrees with the morphism induced by $\tfrac{1}{4}\cdot\Omega^{2n+2}E\in\oH^{2n-2}(\Omega_0^{2n+1}\BTOP(2n+1);\bfQ)$. We begin with two preliminaries.

The compactification $\overline{C}_2(\bfR^{2n+1})$ (denoted $\overline{C}_2(S^{2n+1}, \infty)$ in \cite{WatanabeI}) is a topological manifold with boundary, whose interior $C_2(\bfR^{2n+1})$ is the configuration space of two ordered points. The end space\footnote{See \cite[p.\ 1 and Appendix B]{HughesRanicki} for useful references on end spaces.}
$\partial_h C_2(\bfR^{2n+1}) \coloneq \{\gamma\colon [0,\infty) \to C_2(\bfR^{2n+1}) \, | \, \gamma\text{ is proper}\}$ is a homotopy-theoretic model for the boundary in that the zig-zag
\[\partial_h C_2(\bfR^{2n+1}) \overset{\text{res}}\lla \{\bar{\gamma}\colon [0,\infty] \to \overline{C}_2(\bfR^{2n+1}) \, | \, \bar{\gamma}^{-1}(\partial \overline{C}_2(\bfR^{2n+1})) = \{\infty\}\} \overset{\bar{\gamma} \mapsto \bar{\gamma}(\infty)}\lra \partial \overline{C}_2(\bfR^{2n+1})\]
consists of equivalences, which may be easily seen using a choice of collar for the boundary of $\overline{C}_2(\bfR^{2n+1})$. This induces an equivalence of ``pairs''
\[(C_2(\bfR^{2n+1}), \partial_h C_2(\bfR^{2n+1})) \simeq (\overline{C}_2(\bfR^{2n+1}), \partial \overline{C}_2(\bfR^{2n+1})),\]
where the ``inclusion of the boundary'' 
$\partial_h C_2(\bfR^{2n+1})\ra C_2(\bfR^{2n+1})$
is given by evaluation at $0$. The evident action of the group of compactly supported diffeomorphisms $\Diff_c(\bfR^{2n+1})$ on $C_2(\bfR^{2n+1})$ extends to an action on $\overline{C}_2(\bfR^{2n+1})$ and hence induces an action on $\partial \overline{C}_2(\bfR^{2n+1})$. Similarly it extends to an action on $\partial_h C_2(\bfR^{2n+1})$, but the advantage of the homotopy-theoretic boundary is that the $\Diff_c(\bfR^{2n+1})$-action on the pair $(C_2(\bfR^{2n+1}),\partial_h C_2(\bfR^{2n+1}))$ extends to an action of the group $\Homeo_c(\bfR^{2n+1})$, and so is homotopically trivial as $\Homeo_c(\bfR^{2n+1})$ is contractible by the Alexander trick.

It will also be helpful to recall how the first identification in \eqref{eq:zeta} can be made explicit. A class in the source of \eqref{eq:zeta} can be represented by a framed smooth $\bfR^{2n+1}$-bundle $\pi\colon E \ra B$ with structure group $\Diff_c(\bfR^{2n+1})$ (so $E$ is identified with the trivial framed bundle outside a compact set) whose framing is standard over each fibre. The Alexander trick gives a concordance 
$\bfR^{2n+1} \ra E_\text{Alex} \ra [0,1] \times B$
of topological bundles to the trivial bundle, under which the framing gives a trivialisation of the vertical tangent microbundle of $B \times \bfR^{2n+1}$, standard outside a compact set. Comparing this with the standard trivialisation of the vertical tangent microbundle gives 
\[\xi\colon B \lra \Omega^{2n+1}_0 \TOP(2n+1)\] which corresponds to the original framed smooth bundle under the first equivalence in \eqref{eq:zeta}.

We now show the claim. Applying the construction $(C_2(-), \partial_h C_2(-))$ fibrewise to the Alexander concordance defines a concordance from $(C_2(\pi), \partial_h^\mathrm{fib} C_2(\pi))$ to $B \times (C_2(\bfR^{2n+1}), \partial_h C_2(\bfR^{2n+1}))$. On the other hand, using the smooth structure gives an equivalence of pairs
$(C_2(\pi), \partial_h^\mathrm{fib} C_2(\pi)) \simeq (\overline{C}_2(\pi), \partial^\mathrm{fib} \overline{C}_2(\pi))$
over $B$, and similarly
\[B \times (C_2(\bfR^{2n+1}), \partial_h C_2(\bfR^{2n+1})) \simeq B \times (\overline{C}_2(\bfR^{2n+1}), \partial \overline{C}_2(\bfR^{2n+1})).\]
Using the identification of $(\overline{C}_2(\pi), \partial^\mathrm{fib} \overline{C}_2(\pi))$ with $B \times (\overline{C}_2(\bfR^{2n+1}), \partial \overline{C}_2(\bfR^{2n+1}))$, we can transfer the fibre integration calculation in the definition of $\zeta_2$ to this trivial bundle. The map
$p(\tau_E)\colon \partial^\mathrm{fib} \overline{C}_2(\pi) \ra S^{2n}$
given by the framing $\tau_E$ (and the standard behaviour outside of a compact set) corresponds, under this identification, to a map
$p'\colon B \times \partial \overline{C}_2(\bfR^{2n+1}) \ra S^{2n}.$
On the part of the boundary given by tending to $\infty$ in $\bfR^{2n+1}$, the Alexander concordance was constant and hence the map $p'$ agrees with $p(\tau_E)$ here. The remaining part of the boundary corresponds to pairs of points colliding and is given by $B \times S(T\bfR^{2n+1})$. On this part $p'$ is given by the map
\begin{align*}
B \times S(T\bfR^{2n+1}) &\lra C_2(\bfR^{2n+1}) \simeq S^{2n}\\
(b, x) &\longmapsto \xi(\pr(b)) \cdot x,
\end{align*}
where $\pr\colon S(T\bfR^{2n+1}) \ra \bfR^{2n+1}$ is the projection. In this formula we implicitly consider $S(T\bfR^{2n+1}) \subset C_2(\bfR^{2n+1})$ as antipodal points, and use the $\TOP(2n+1)$-action on $C_2(\bfR^{2n+1})$.

Now observe that this map $p'$ does not depend on the element $\xi \in [B, \Omega^{2n+1}_0 \TOP(2n+1)]$ but only on its image in $[B, \Omega^{2n+1}_0 \G(2n+1)]$, as it suffices for $b \mapsto \xi(\pr(b))$ to be a family of homotopy equivalences of the $2n$-sphere. Thus there is a factorisation
\begin{align*}
 [B, \Omega_0^{2n+1}\TOP(2n+1)]_\bfQ &\lra \mathrm{Im}\big([B, \Omega_0^{2n+1} \TOP(2n+1)]_\bfQ \to [B, \Omega_0^{2n+1} \G(2n+1)]_\bfQ\big)\\ 
&\lra \oH^{2n-2}(B;\bfQ)
\end{align*}
of the rationalisation of \eqref{eq:zeta}. As $E\colon \BSG(2n+1)\ra K(\bfQ,4n)$ is a rational equivalence (see \cref{sec:transgressed-pontryagin}), $\Omega^{2n+2} E$ induces an isomorphism
$[B, \Omega^{2n+1}_0 \G(2n+1)]_\bfQ \cong \oH^{2n-2}(B;\bfQ)$, so by naturality we must have $\zeta_2 = A \cdot (\Omega^{2n+2}E)\in \oH^{2n-2}(\Omega^{2n+1}_0\TOP(2n+1);\bfQ)$ for some constant $A \in \bfQ$. To determine this constant, we consider the composition
\vspace{-0.1cm}
\[\pi_{4n}(\BSO(2n+1)) = \pi_{2n-2}(\Omega^{2n+1} \OO(2n+1)) \lra \pi_{2n-2}(\Omega^{2n+1} \TOP(2n+1)) \xlra{\zeta_2} \bfQ\]
which is the map $Z'_2$ from \cite[p.\ 699]{WatanabeI}, and as shown in the proof of Theorem 2 of that paper this map agrees with evaluation against a quarter of $p_n$. As $p_n = E \in \oH^{4n}(\BSO(2n+1);\bfQ)$ (see \cref{sec:transgressed-pontryagin}) it then follows that $A=\tfrac{1}{4}$, as claimed.
\end{proof}

\clearpage

\section*{Glossary of notation}
\begin{table}[h!]
\begin{tabular}[h!]{ c | c}
Notation & Section\\
\hline
 $V_{g}$, $W_g$, $W_{g,1}$ &  \ref{sec:manifold-preliminaries} \\ 
 $\Diff(-)$, $\BlockDiff(-)$ &  \ref{sec:block-diffeos} \\ 
  $\hAut(-)$, $\Bun(-)$ &  \ref{sec:block-diffeos} \\
 $\half\partial$ &  \ref{sec:block-diffeos} \\  
 $\fr$, $\onefr$, $\pm\fr$ &  \ref{sec:block-diffeos} \\  
  $\lambda(-,-)$& \ref{sec:standard-model}\\
  $\bar{H}_V$, $\bar{H}_W$, $\bar{K}_W$& \ref{sec:standard-model}\\
  $\bar{H}_V^\bfZ$, $\bar{H}_W^\bfZ$, $\bar{K}_W^\bfZ$ & \ref{sec:standard-model}\\
    $e_i$, $f_i$ & \ref{sec:standard-model}\\
    $\ell_{V_g}$, $\ell_{W_{g,1}}$, $\ell$, $(-)_\ell$ & \ref{sec:standard-framing}\\
    $\check{\Lambda}_{V}$, $\check{\Lambda}_{W}$ & \ref{sec:MCG}\\
        $G_V$, $G_W$, $\GL(\bar{H}^\bfZ_V)$ & \ref{sec:MCG}\\
     $\OSp_{g}(\bfZ)$& \ref{sec:MCG}\\
     $\Sp^{(q)}_{2g}(\bfZ)$, $\Sp_{2g}(\bfZ)$, $\oO_{g,g}(\bfZ)$ & \ref{sec:MCG}\\
     $M_V^\bfZ$ & \ref{sec:MCG}\\
     $\rho$ & \ref{sec:Reflection}\\
     $\bfQ^\pm$, $\bfZ^{\pm}$ & \ref{sec:involution-and-intersection}\\
   $H_V$, $H_W$, and $K_W$ & \ref{sec:gradings}\\
   $(-)^\vee$ & \ref{sec:gradings}\\
      $(-)[k]$ & \ref{sec:gradings}\\
   $s^k(-)$  & \ref{sec:gradings}\\
   $[-]^{\Sigma_k}$ & \ref{sec:Schur-functors}\\
      $(\mu)$& \ref{sec:Schur-functors}\\
         $S_\mu(-)$& \ref{sec:Schur-functors}
\end{tabular}\quad\quad
\begin{tabular}[h!]{ c | c}
Notation & Section\\
\hline
   $V_\mu(-)$ & \ref{sec:symplectic-schur-functors}\\ 
   $\bfL(-)$& \ref{sec:derivations-reminder}\\ 
   $\Der^f_\omega(-,-)$ & \ref{sec:derivations-reminder}\\
   $\cL ie\dl -,-\dr$, $Lie(-)$, $Lie\dl -\dr$  & \ref{sec:homotopy-groups-haut}\\
   $\kappa_V$, $\kappa_W$ & \ref{sec:kappas}\\
   $\Emb^{(\sim)}(-,-)$ & \ref{section:splitting-handles}\\
   $\Emb^{\fr}(-,-)$ & \ref{sec:framed-embeddings}\\
      $(-)^\sim$ & \ref{sec:HEmbSpace}\\
      $\sigma_i$ & \ref{sec:stable-cohomology-borel}\\
   $K_\mu$& \ref{sec:stable-cohomolog-gv}\\
   $p_i^\tau$, $(p_n-e^2)^\tau$, $(p_n-E)^\tau$& \ref{sec:transgressed-pontryagin}\\
      $E$& \ref{sec:transgressed-pontryagin}\\
      $\G(-)$, $\SG(-)$ &\ref{sec:transgressed-pontryagin}, \ref{sec:FirstDerTop} \\
      $K_i(\bfZ)$, $K\OSp_i(\bfZ)$& \ref{section:homotopy-even-discs}\\
      $\xi_i$, $\xi_i^{\OSp}$, $\xi_i^\tau$, $(\xi_i^{\OSp})^\tau$& \ref{section:homotopy-even-discs}\\
   $\oC(-)$&\ref{sec:concordances}\\
   $\Theta\fbo^{(1)}$, $\Theta\fbt^{(1)}$, $\Theta\fbg^{(1)}$&\ref{sec:FirstDerTop}, \ref{sec:orthogonal-calc}\\
   $\fe$, $\fe^{(i)}$, $\Theta\fe^{(i)}$ & \ref{sec:orthogonal-calc}\\
   $e_{\fe}$& \ref{sec:KThyEuler}\\
   $\underline{\Theta\fe^{k}}$ &\ref{sec:KThyEuler}\\
   $\underline{\Omega}^\infty(-)$, $\underline{\bfS}_\mathrm{sign}$&\ref{sec:KThyEuler}\\
$C_{\pm}(-,-)$ & \ref{sec:AppendixA}\\
   $\zeta_2$, $\varepsilon$ & \ref{app:Watanabe}\\
   \ & \\
\end{tabular}
\end{table}

\bibliographystyle{amsalpha}
\bibliography{literature}

\vspace{0.2cm}

\end{document}